\documentclass[12pt]{article}
\usepackage[reqno]{amsmath}
\usepackage{amsthm}
\usepackage{amssymb,amscd,latexsym,stmaryrd}
\usepackage{paralist}
\usepackage[colorlinks,pdfpagelabels,pdfstartview = FitH,bookmarksopen = true,bookmarksnumbered = true,linkcolor = black,plainpages = false,hypertexnames = false,citecolor = black]{hyperref}
\usepackage{bm}
\usepackage{fonttable}                                                         
\newcommand{\arah}{{\text{\usefont{U}{xnsh}{m}{n}\symbol{104}}}}       
\usepackage{tocloft}
\usepackage{aurical}
\usepackage[T1]{fontenc}
\usepackage[bbgreekl]{mathbbol}
\usepackage[all]{xy}
\newcommand{\A}{{\mathbb{A}}}
\newcommand{\C}{{\mathbb{C}}}
\newcommand{\F}{{\mathbb{F}}}

\newcommand{\N}{{\mathbb{N}}}

\newcommand{\Q}{{\mathbb{Q}}}
\newcommand{\oQ}{\overline{\Q}}
\newcommand{\ovQ}{\overline{Q}}
\newcommand{\Sa}{{\mathbb{S}}}

\newcommand{\uC}{\underline{\C}}
\newcommand{\R}{{\mathbb{R}}}
\newcommand{\Z}{{\mathbb{Z}}}
\newcommand{\uR}{\underline{\R}}
\newcommand{\uQ}{\underline{\Q}}
\newcommand{\uZ}{\underline{\Z}}
\newcommand{\hZ}{\hat{\Z}}

\newcommand{\hK}{\hat{K}}

\newcommand{\hP}{\hat{P}}

\newcommand{\oZ}{\overline{\Z}}

\newcommand{\adele}{\mathrm{adele}}

\newcommand{\car}{\mathrm{char}\,}
\newcommand{\card}{\mathrm{card}\,}
\newcommand{\cont}{\mathrm{cont}}
\newcommand{\colim}{\mathrm{colim}}

\newcommand{\ddiv}{\mathrm{div}\,}
\newcommand{\et}{\mathrm{\acute{e}t}}

\newcommand{\oE}{\overline{E}}
\newcommand{\of}{\overline{f}}

\newcommand{\fin}{\mathrm{fin}}

\newcommand{\id}{\mathrm{id}}
\newcommand{\ind}{\mathrm{ind}}

\newcommand{\perf}{\mathrm{perf}}
\newcommand{\pr}{\mathrm{pr}}
\newcommand{\hpr}{\hat{\pr}}

\newcommand{\rat}{\mathrm{rat}}
\newcommand{\imm}{\mathrm{im}\,}

\newcommand{\oP}{\overline{P}}

\newcommand{\oS}{\overline{S}}
\renewcommand{\mod}{\;\mathrm{mod}\;}

\newcommand{\Mor}{\mathrm{Mor}}
\newcommand{\ord}{\mathrm{ord}}
\newcommand{\per}{\mathrm{per}}
\newcommand{\Quot}{\mathrm{Quot}}

\newcommand{\ring}{\mathrm{ring}}
\newcommand{\res}{\mathrm{res}}

\newcommand{\spec}{\mathrm{spec}\,}
\newcommand{\supp}{\mathrm{supp}\,}
\newcommand{\Aut}{\mathrm{Aut}}

\newcommand{\End}{\mathrm{End}\,}

\newcommand{\Gal}{\mathrm{Gal}}

\newcommand{\Hom}{\mathrm{Hom}}

\newcommand{\Imm}{\mathrm{Im}\,}
\newcommand{\inj}{\mathrm{inj}}

\newcommand{\Ker}{\mathrm{Ker}\,}

\newcommand{\Map}{\mathrm{Map}\,}
\newcommand{\map}{\mathrm{map}}

\newcommand{\cP}{\check{P}}

\newcommand{\tP}{\tilde{P}}
\newcommand{\tQ}{\tilde{Q}}
\newcommand{\tU}{\tilde{U}}

\newcommand{\tz}{\tilde{z}}
\newcommand{\talpha}{\tilde{\alpha}}
\newcommand{\calpha}{\check{\alpha}}
\newcommand{\tchi}{\tilde{\chi}}

\newcommand{\tLambda}{\tilde{\Lambda}}

\newcommand{\tvarphi}{\tilde{\varphi}}
\newcommand{\tpsi}{\tilde{\psi}}
\newcommand{\tPsi}{\tilde{\Psi}}
\newcommand{\htvarphi}{\hat{\tvarphi}}

\newcommand{\czeta}{\check{\zeta}}
\newcommand{\hotimes}{\hat{\otimes}}

\newcommand{\tG}{\tilde{G}}

\newcommand{\RRe}{\mathrm{Re}\,}

\newcommand{\sep}{\mathrm{sep}}
\newcommand{\Spf}{\mathrm{Spf}\,}

\renewcommand{\top}{\mathrm{top}}

\newcommand{\tors}{\mathrm{tors}}

\newcommand{\Ch}{{\mathcal C}}

\newcommand{\Eh}{{\mathcal E}}
\newcommand{\Fh}{{\mathcal F}}
\newcommand{\Gh}{{\mathcal G}}

\newcommand{\Mh}{\mathcal{M}}
\newcommand{\hMh}{\hat{\Mh}}
\newcommand{\Nh}{\mathcal{N}}
\newcommand{\Oh}{{\mathcal O}}
\newcommand{\hOh}{\hat{\Oh}}

\newcommand{\Rh}{{\mathcal R}}
\newcommand{\Sh}{{\mathcal S}}

\newcommand{\Zh}{\mathcal{Z}}

\newcommand{\eb}{\mathfrak{b}}

\newcommand{\eq}{{\mathfrak{q}}}
\newcommand{\emm}{{\mathfrak{m}}}
\newcommand{\oemm}{\overline{\emm}}
\newcommand{\hemm}{\hat{\emm}}
\newcommand{\en}{\mathfrak{n}}
\newcommand{\eo}{\mathfrak{o}}
\newcommand{\heo}{\hat{\eo}}

\newcommand{\hoE}{\hat{\overline{E}}}

\newcommand{\ep}{\mathfrak{p}}

\newcommand{\eU}{\mathfrak{U}}
\newcommand{\eX}{\mathfrak{X}}
\newcommand{\ex}{\text{\normalfont\Fontauri x}}
\newcommand{\ez}{\mathfrak{z}}
\newcommand{\ey}{\text{\normalfont\Fontauri y}}

\newcommand{\oeX}{\overline{\eX}}

\newcommand{\hlambda}{\hat{\lambda}}
\newcommand{\ochi}{\overline{\chi}}

\newcommand{\hochi}{\hat{\ochi}}
\newcommand{\onu}{\overline{\nu}}

\newcommand{\tkappa}{\tilde{\kappa}}
\newcommand{\oeta}{\overline{\eta}}
\newcommand{\opi}{\overline{\pi}}
\newcommand{\osigma}{\overline{\sigma}}

\newcommand{\ozeta}{\overline{\zeta}}

\newcommand{\oK}{\overline{K}}
\newcommand{\onQ}{\overline{Q}}
\newcommand{\oR}{\overline{R}}
\newcommand{\barr}{\overline{r}}
\newcommand{\os}{\overline{s}}
\newcommand{\ot}{\overline{t}}
\newcommand{\ou}{\overline{u}}
\newcommand{\ox}{\overline{x}}
\newcommand{\oX}{\overline{X}}

\newcommand{\oF}{\overline{\mathbb{F}}}

\newcommand{\hA}{\hat{A}}
\newcommand{\ohA}{\overline{\hA}}

\newcommand{\oA}{\overline{A}}
\newcommand{\oa}{\overline{a}}
\newcommand{\ta}{\tilde{a}}
\newcommand{\tX}{\tilde{X}}

\newcommand{\cQ}{\check{Q}}

\newcommand{\cH}{\check{H}}
\newcommand{\ceX}{\check{\eX}}

\newcommand{\cW}{\check{W}}
\newcommand{\cY}{\check{Y}}
\newcommand{\df}{\overset{_{\,\hullet}}{f}}
\newcommand{\dH}{\overset{_{\,\hullet}}{H}}
\newcommand{\deX}{\overset{_{\,\hullet}}{\eX}}
\newcommand{\ddeX}{\overset{_{\,\hullet\!\hullet}}{\eX}}
\newcommand{\hceX}{\overset{\diamond}{\eX}}
\newcommand{\hcY}{\overset{\diamond}{Y}}
\newcommand{\hcP}{\overset{\diamond}{P}}
\newcommand{\cpi}{\check{\pi}}
\newcommand{\dpi}{\overset{_{\,\hullet}}{\pi}}

\newcommand{\oxi}{\overline{\xi}}
\newcommand{\multmap}{\varprojlim_{(\,)^p}}
\newcommand{\tf}{\tilde{f}}
\newcommand{\hoQ}{\hat{\oQ}}

\newcommand{\vK}{\overleftarrow{K}}
\newcommand{\vG}{\overleftarrow{\Gamma}}

\newcommand{\umu}{\underline{\mu}}
\newcommand{\silo}{\xrightarrow{\sim}}

\newcommand{\ent}{\;\widehat{=}\;}
\newcommand{\hullet}{{\scriptscriptstyle \bullet\,}}
\newcommand{\verk}{\mbox{\scriptsize $\,\circ\,$}}
\newcommand{\strich}{\,\!'}

\newcommand{\dcup}{\dot{\cup}}

\newtheorem{theorem}{Theorem}[section]
\newtheorem{lemma}[theorem]{Lemma}
\newtheorem{prop}[theorem]{Proposition}
\newtheorem{defn}[theorem]{Definition}

\newtheorem{cor}[theorem]{Corollary}
\newtheorem{exmp}[theorem]{Example}

\newtheorem{remark}[theorem]{Remark}
\newtheorem*{rem}{Remark}
\newtheorem*{rems}{Remarks}
\newtheorem{claim}[theorem]{Claim}
\newtheorem{supple}[theorem]{Supplement}
\newenvironment{example}{\noindent {\bf Example.}}{}
\newenvironment{proofof}{\noindent {\it Proof of}}{\mbox{}\hfill$\Box$}

\begin{document}
\title{Dynamical systems for arithmetic schemes\\
{\large \it  Dedicated to the memory of Jacob Murre}}
\author{Christopher Deninger\footnote{Funded by the Deutsche Forschungsgemeinschaft (DFG, German Research Foundation) under Germany's Excellence Strategy EXC 2044--390685587, Mathematics M\"unster: Dynamics--Geometry--Structure, the CRC 878 Groups, Geometry \& Actions and the CRC 1442 Geometry: Deformations and Rigidity} }
\date{\ }
\maketitle
\section*{Introduction}

The complex points of a smooth algebraic curve over $\C$ have the structure of a Riemann surface. This is an object with a rich structure on which substantial real and complex analysis can be done. In number theory we have the arithmetic curve $\spec \Z$ to which such profound objects as the Riemann zeta function and its $p$-adic counterpart, the Kubota Leopold zeta function can be associated. The definition and the study of the Riemann zeta function require analytic methods. It would be very useful to have a space of $\C$-valued points of $\spec \Z$ on which serious real and complex analysis on ``manifolds'' could be used to understand the Riemann zeta function. Similarly, one could hope for an interesting space of $\C_p$-valued points of $\spec \Z$. However, classically $(\spec \Z) (\C)$ is just a one point space and even with the most benevolent view, it is only supposed to be responsible for the Euler factor at the infinite place $\pi^{-\frac{s}{2}} \Gamma (\frac{s}{2})$. Thus there are far too few classical $\C$-valued points of $\spec \Z$ for the aim of an ``analysis on manifolds'' approach to the zeta function. 

More generally, for each arithmetic scheme $\eX_0$ there is the Hasse-Weil zeta function $\zeta_{\eX_0} (s)$ which generalizes the Riemann zeta function and one would hope for interesting analytic spaces of $\C$ and $\C_p$-valued points which could help to prove the many difficult conjectures about $\zeta_{\eX_0} (s)$ and its $p$-adic analogue. In fact with a view towards function field arithmetic one might even wonder about an interesting analytic space of $\C_{\infty}$-valued points of $\eX_0$ where $\C_{\infty}$ is the $t$-adic completion of $\overline{\F_p ((t))}$. Let us denote these hypothetical new spaces of $\C$-valued points by $\eX_0 [\C]$, where $\C$ denotes either the complex number field or $\C_p$ or $\C_{\infty}$. Of course $\eX_0 [\C]$ should contain the classical points $\eX_0 (\C)$. For reasons explained below, we also expect the spaces $\eX_0 [\C]$ to carry commuting Frobenius endomorphisms $F_p$ for every prime number $p$. 

For the complex numbers $\C$, the existence of $\eX_0 [\C]$ is suggested by the analogies of arithmetic schemes $\eX_0$ with certain types of one-codimensionally foliated $\R$-dynamical systems $(X_0, \phi^t , \Fh)$ studied in \cite{D1,D2,D3,D4,D5,den-hilb-polya,den-arxiv1,den-arxiv2,AKL,KMNT}. Under some assumptions, it had turned out that a dynamical system $X_0$ conjecturally corresponding to $\eX_0$ had to be a suspension  (quotient by the diagonal $\Q^{> 0}$-action) of the form
\[
X_0 = \ceX_0 [\C] \times_{\Q^{>0}} \R^{>0} \; .
\]
Here $\ceX_0 [\C]$ is a space with an action by the group of positive rational numbers $\Q^{> 0}$ and $\phi^t$ acts on $X_0$ by right multiplication with $e^t$. Such a suspension $X_0$ has a ``foliation'' $\Fh$ by the one-codimensional images of $\ceX_0 [\C] \times \{ u \}$ for $u \in \R^{> 0}$. The leaves of $X_0$ are isomorphic to the connected components of $\ceX_0 [\C]$ and $\Q^{>0}$ is the group of Poincar\'e return maps on $\ceX_0 [\C]$. 

For schemes over a finite field the Frobenius morphism is not invertible in general, so it is natural to assume that there is a space $\eX_0 [\C]$ equipped with commuting (Frobenius\nobreakdash-) endomorphisms for all primes $p$ such that $\ceX_0 [\C]$ with its $\Q^{> 0}$-action is obtained from $\eX_0 [\C]$ by a colimit construction which inverts all Frobenii (on the ring of functions on $\eX_0 [\C]$ this may be viewed as ``tilting''). 

In the papers quoted above we specified many properties that the expected system $(X_0 , \phi^t , \Fh)$ and hence the $\Q^{> 0}$-action on $\ceX_0 [\C]$ should have for arithmetic schemes $\eX_0$. These expected properties are more or less dictated by known or conjectural properties of Hasse-Weil and motivic zeta- and $L$-functions and the idea that these functions should be represented by alternating products of zeta-regularised determinants on suitable (foliation) cohomologies of $(X_0 , \Fh)$.\\
For example, the zeroes and poles of $\zeta_{\eX_0} (s)$ should be the eigenvalues of the infinitesimal generator of the induced $\R$-action on foliation cohomology groups of $(X_0 , \Fh)$ with compact supports. As another example, for any non-empty open $\eX_0 \subset \spec \Z$ the pro-algebraic fundamental group $\pi (X_0 , x_0)$ of $\eX_0$ in the sense of \cite{den-pi} should be closely related to the motivic Galois group of $\eX_0$. In particular, the maximal pro-etale quotient $\pi (X_0 , x_0)^{\et}$ should be the \'etale fundamental group of $\eX_0$. Moreover, for general arithmetic schemes $\eX_0$, the motivic cohomology groups of $\eX_0$ should be expressible by topological cohomology groups of $X_0$ with coefficients in certain local systems.

A much more basic property would be a correspondence between the periodic orbits $\gamma$ of $(X_0 , \phi^t)$ and the closed points $\ex_0$ of $\eX_0$. Here the length of $\gamma$ should be the logarithm of the norm $N (\ex_0)$ of $\ex_0$. Thus the Hasse--Weil zeta function of $\eX_0$ would be the Ruelle zeta function of $(X_0 , \phi^t)$. In particular, for $\eX_0 = \spec \Z$ the periodic orbits $\gamma$ of $X_0$ should correspond to the prime numbers $p$, the length being $l (\gamma) = \log p$. The conditions on the $\Q^{> 0}$-action on a space $\ceX_0 [\C]$ for this to happen were written down a long time ago in \cite[3.1 and 4.9]{D1}, but for too many years I had no idea how to construct natural $\Q^{> 0}$-spaces realizing these conditions.  

In the present paper we construct certain spaces $\deX_0 (\C)$ with commuting injective ``Frobenius'' endomorphisms $F_p$ for all prime numbers $p$. These spaces are an approximation to the expected spaces $\eX_0 [\C]$. For example, if $\ceX_0 (\C)$ denotes the $\Q^{> 0}$-space obtained as colimit over the $F_p$'s on $\deX_0 (\C)$, then the closed points $\ex_0$ of $\eX_0$ correspond bijectively to compact packets $\Gamma_{\ex_0}$ of periodic orbits of length $\log N \ex_0$ on \[
X_0 = \ceX_0 (\C) \times_{\Q^{> 0}} \R^{> 0} \; .
\]
If $p = \car \kappa (\ex_0)$, then $\Gamma_{\ex_0}$ is fibred over the compact group
\[
\Aut (\oF^{\times}_p) / \Aut (\oF_p) = \hZ^{\times}_{(p)} / p^{\hZ} \; ,
\]
with fibres the periodic orbits in $\Gamma_{\ex_0}$. Each periodic orbit of $X_0$ lies in exactly one packet $\Gamma_{\ex_0}$. Thus the correspondence between closed points $\ex_0$ of $\eX_0$ and periodic orbits of $(X_0 , \phi^t)$ is not one-to-one but a bit more involved. The compact packets $\Gamma_{\ex_0}$ are reminiscient of invariant tori. There are no fixed points of the flow. In our construction, the existence of periodic orbits is due to the existence of Frobenius elements for $\ex_0$ in the Galois group of the function field $K_0$ of $\eX_0$. 

One of the main theorems of the present paper asserts that for every integral normal scheme $\eX_0$ which is flat of finite type over $\spec \Z$, contrary to $\deX_0 (\C)$ and $\ceX_0 (\C)$, the space $X_0$ is connected as well. Even this first topological insight about $X_0$ is not easy to show. For $\dim \eX_0 \ge 2$ our proof requires de~Jong's theory of alterations. Moreover the ``bad'' properties of the ad\`ele topology on the id\`eles work to our advantage. We also show the somewhat stronger statement that the zeroth ``foliation cohomology'' of $(X_0 , \Fh)$ is one dimensional. Theorem \ref{t710n} elucidates the general structure of the systems $X$ and $X_0$.

We now make some remarks about the construction of our dynamical systems. There are many considerations in arithmetic geometry leading to relaxing the condition of additivity but keeping multiplicativity: For example the Teichm\"uller lift in Witt vector theory which is only multiplicative, the work on $\F_1$-geometry for which we refer to the survey \cite{L}, ideas in arithmetic topology \cite{Mo} and its dynamical enhancement \cite{D2}, and the work of Kucharczyk and Scholze \cite{KS}. In the first version of this paper we constructed $\deX_0 (\C)$ as a Galois quotient of a space of certain multiplicative maps. This was inspired by the idea to use Frobenius elements to generate periodic orbits of $X_0 = \ceX_0 (\C) \times_{\Q^{> 0}} \R^{> 0}$. This first construction of $\deX_0 (\C)$ was ad hoc and extrinsic, whereas an intrinsic definition only in terms of $\eX_0$ was missing. The following arguments eventually led to the new intrinsic description that we will describe below.

In the desired correspondence $\eX_0 \mapsto \eX_0 [\C]$ to which $\eX_0 \mapsto \deX_0 (\C)$ is an approximation, regular functions on $\eX_0$ should give rise to $\C$-valued functions on $\eX_0 [\C]$. This correspondence between functions can only be multiplicative and not additive, since otherwise we would only obtain constant functions on $\eX_0 [\C]$ if $\eX_0 = \spec \Z$ for example. Thus we expect a ring homomorphism
\[
\C \Gamma (\eX_0 , \Oh) \longrightarrow \Gamma (\eX_0 [\C] , \Oh)
\]
from the complex monoid algebra of $(\Gamma (\eX_0 , \Oh) , \hullet)$ to a $\C$-algebra of functions on $\eX_0 [\C]$. Assume that $\eX_0$ is normal with integral closure $\eX$ in an algebraic closure $K$ of the function field $K_0$ of $\eX_0$. Taking Galois descent into account one sees that there should even be a ring homomorphism
\[
(\C \Gamma (\eX, \Oh))^G \longrightarrow \Gamma (\eX_0 [\C] , \Oh)
\]
where $G = \Aut_{K_0} (K)$. The ring on the left or more precisely its quotient by the ideal $\C (0)$ has an intrinsic description as the complexified rational Witt vectors of the ring $\Gamma (\eX_0 , \Oh)$ as I finally understood when reading \cite{KS}. In that very interesting and beautifully written paper, Kucharczyk and Scholze use the ring of rational Witt vectors of a field to define a certain topological space $X_F$ for every field $F$ of characteristic zero containing all roots of unity. They show that the pro-\'etale group $\pi^{\et}_1 (X_F , x)$ classifying finite covers of $X_F$ is isomorphic to the absolute Galois group of $F$. Moreover they establish isomorphisms
\[
H^{\hullet} (X_F , \Z / n) = H^{\hullet}_{\et} (\spec F , \Z / n) \quad \text{for} \; n \ge 2
\]
and
\[
H^{\hullet} (X_F , \Z) = \Lambda^{\hullet} (F^{\times} \otimes \Q) \; .
\]
Kucharczyk and Scholze also suggest to look for a space $X^M_F$ mapping to $X_F$ such that $H^{\hullet} (X^M_F , \Z)$ is the Milnor $K$-ring of $F$ modulo torsion. Our space $\deX_0 (\C)$ is an approximation to the ``true'' space $\eX_0 [\C]$ in the same way as $X_F$ is an approximation to the ``true'' space $X^M_F$. Motivated by the above argument and by the contruction of $X_F$ in \cite{KS} we propose new ringed spaces and their points with values in schemes as follows:

Let $X = (X_{\top} , \Oh_X)$ be a scheme and denote by $W_{\rat} (\Oh_X)$ the sheafification of the separated presheaf $U \mapsto W_{\rat} (\Oh_X (U))$ of rational Witt vectors of $\Oh_X (U)$. Consider the ringed space
\[
W_{\rat} (X) = (X_{\top} , W_{\rat} (\Oh_X)) \; .
\]
It is equipped with commuting Frobenius endomorphisms for every prime $p$ which are the identity on $X_{\top}$. The definition of $W_{\rat} (X)$ looks somewhat familiar: Since the pioneering work of Serre \cite{serre} the ringed spaces $(X , W_{p^n} (\Oh_X))$ play an important role in the $p$-adic cohomology theory of characteristic $p$ schemes $X$. See \cite{borger1,borger2} for a comprehensive account of the geometric aspects of Witt vectors.

Morphisms of schemes induce local homomorphisms on stalks and this is a very important requirement in the theory. Contrary to a scheme, $W_{\rat} (X)$ is not a locally ringed space. For a scheme $S$ we define a morphism $S \to W_{\rat} (X)$ as a morphism $(f , f^{\sharp})$ of ringed spaces which is ``local'' in the following sense. For every point $s \in S$ the stalk of the sheaf map $f^{\sharp} : f^{-1} W_{\rat} (\Oh_X) \to \Oh_S$ fits into a commutative diagram
\[
\xymatrix{
W_{\rat} (\Oh_{X , f(s)}) \ar[r]^-{f^{\sharp}_s} \ar@{>>}[d] & \Oh_{S,s} \ar@{>>}[d] \\
W_{\rat} (\kappa (f (s))) \ar[r]^-{\tf^{\sharp}_s} & \kappa (s) \; .
}
\]
The existence of the (uniquely determined) maps $\tf^{\sharp}_s$ is our replacement for locality. For example there is a canonical morphism $X \to W_{\rat} (X)$. Let $W_{\rat} (X) (S)$ denote the set of such morphisms $S \to W_{\rat} (X)$. It is equipped with commuting Frobenius endomorphisms for all prime numbers $p$. There is a natural inclusion $X (S) \subset W_{\rat} (X) (S)$. If $S = \spec A$ is affine, we denote these $S$-valued points of $W_{\rat} (X)$ also by $W_{\rat} (X) (A)$. If $\eX_0$ is an integral normal scheme, it turns out that $W_{\rat} (\eX_0) (\C) = \deX_0 (\C)$ as defined in the first version of this paper. The connected components of $W_{\rat} (\spec F) (\C)$ are the spaces $\eX_F$ of \cite{KS} of which the $X_F$ above are retractions. There is a natural multiplicative map
\begin{equation} \label{eq:i1}
\Gamma (X , \Oh) \longrightarrow \Map (W_{\rat} (X) (\C) , \C) \; .
\end{equation}
For example, taking $X = \spec \Z$ we see that numbers i.e.~the elements $n$ of $\Z$ become (highly non-trivial) complex valued functions $f_n$ on $W_{\rat} (X) (\C)$. The zero set of $f_n$ is the union of the packets of $\Q^{> 0}$-orbits which correspond to the prime divisors of $n$. While regular functions on $\eX_0$ induce $\C$-valued functions on $\eX_0 (\C)$, they only give rise to certain ``generalized functions'' on $\ceX_0 (\C)$ and $X_0$. This is explained in sections \ref{sec11n} and \ref{sec:12}. 

While $W_{\rat} (\eX_0) (\C) = \deX_0 (\C)$ satisfies some of the requirements of the ``true'' space $\eX_0 [\C]$, it is not yet the right object. Firstly in the interpretation via multiplicative maps and Galois descent one has to impose additional constraints. We formalize this by adding a condition $\Eh$ on the multiplicative maps which has to satisfy certain axioms. In this way we obtain subsystems $\deX_0 (\C)_{\Eh}$ and $\ceX_0 (\C)_{\Eh}$ and the results on periodic orbits mentioned above apply more precisely to $X_0 = \ceX_0 (\C)_{\Eh} \times_{\Q^{> 0}} \R^{> 0}$. These subsystems are still connected. There is a minimal condition $\Eh$ for which our theorems hold but it does not look natural. So this aspect is unsatisfactory. For all conditions $\Eh$ that we consider the space $\ceX_0 (\C)_{\Eh}$ is infinite dimensional whereas ideally we would want it to be of dimension $2 \dim \eX_0$ if e.g. $\eX_0$ is flat, normal and of finite type over $\spec \Z$. 

We therefore studied the closure of the union of all periodic orbits of $\ceX_0 (\C)_{\Eh}$ in $\ceX_0 (\C)$. Using results by Perucca \cite{perucca} we determined this closure if $\eX_0$ is the spectrum of the ring of integers in a number field. Assuming her results extend to higher dimensions which we formulate as a conjecture we can describe the closure of the union of the periodic orbits in general. The resulting dynamical system is still infinite dimensional. This is the present situation in the search for $\eX_0 [\C]$. 

The second problem with $\deX_0 (\C)$ is that the cohomological problem pointed out in \cite{KS} and which we recalled above persists in our more general context. The first foliation cohomology of $X_0$ could be more or less the right one though. If there is a natural map $\alpha : \eX_0 [\C] \to \deX_0 (\C)$ then the cohomology of $\eX_0 [\C]$ with coefficients in a sheaf $\Gh$ can be expressed as the cohomology of $\deX_0 (\C)$ with coefficients in $R \alpha_* \Gh$. Working with $\deX_0 (\C)$ instead of the ``true'' space $\eX_0 [\C]$ would then involve finding the relevant complexes of sheaves on $\deX_0 (\C)$. 

Since $p$-adic number theory is much better understood than global number theory, in our search for $\eX_0 [\C]$ we looked at the situation for arithmetic schemes $\eX_0$ over $\Z_p$ and $\C_p$-valued points instead of $\C$-valued points. Here we found a very natural $F_p$-invariant subsystem $\hcY\!_0$ of (a version of) $\deX_0 (\C_p) = W_{\rat} (\eX_0) (\C_p)$. For $\eX_0 = \spec \eo_{K_0}$ with $K_0 / \Q_p$ finite, $\hcY\!_0$ is in canonical bijection with the set of closed points of the scheme theoretic Fargues-Fontaine curve, \cite{FF}. However, functions on $\hcY\!_0$ take values in $\C_p$ whereas the functions on the Fargues-Fontaine curve take values in the untilts of $\C^b_p$ -- and those are not all isomorphic \cite[Theorem 1.3]{kedlaya}. We should stress though that we have yet to define analytical structures on our spaces and in the $p$-adic case even a topology. Incidentally, note that the actions of $G$ on $\spec \eo_K$ and $\C_p$ induce a $G \times G$-action on $W_{\rat} (\spec \eo_K) (\C_p)$. This is related to the $G \times G$-action on the universal cover of the Fargues-Fontaine curve. Let us now explain our constructions in the $p$-adic case and the motivation for them in more detail. 

Consider a normal $\Z_p$-scheme $\eX_0$ which maps surjectively to $\spec \Z_p$ and let $\eX$ be its normalization in an algebraic closure $K$ of the function field $K_0$ of $\eX_0$. Let $G$ be the Galois group of $K$ over $K_0$. Let $\eo$ be a $p$-adically complete rank one valuation ring e.g. $\eo = \eo_p$ the valuation ring of $\C_p$. Since reduction $\mod p$ will play an important role it is necessary to study $W_{\rat} (\eX_0) (\eo)$ and not only $W_{\rat} (\eX_0) (\C_p)$. After some work, it turns out that 
\[
W_{\rat} (\eX_0) (\eo) = W_{\rat} (\eX) (\eo) / G \; .
\]
The elements of $\deX (\eo) := W_{\rat} (\eX) (\eo)$ are certain commutative diagrams of multiplicative maps. The classical $\eo$-valued points of $\eX$ over $\spec \Z$ and $\spec \Z_p$ agree because ring-homomorphisms are automatically $p$-adically continuous. This is not the case for maps that are only multiplicative. Hence it is natural to replace $\deX (\eo)$ by a subset $\deX_c (\eo)$ of commutative diagrams involving a $p$-adic continuity condition. Next, the points of $\ceX_c (\eo) = \colim_{F_p} \deX_c (\eo)$ correspond to diagrams of multiplicative maps on projective limits which factor over a projection. It turns out that in these diagrams we have to consider not only those maps but all continuous maps on the projective limits. We view the resulting set of diagrams as a completion of $\ceX_c (\eo)$ and therefore denote it by $\hceX_c (\eo)$. Incidentally, if we consider points of $W_{\rat} (X)$ with values in rings without ``small multiplicative subgroups'' like the complex number field $\C$ this process does not give more points. 

As above the monoid algebra $\Z \Gamma (\eX , \Oh)$ maps naturally to the ring of $\eo$-valued functions on $\deX (\eo) = W_{\rat} (\eX) (\eo)$. Correspondingly the monoid algebra of $\varprojlim_{F_p} \Gamma (\eX, \Oh)^{\wedge}$ becomes an $\eo$-algebra of functions on $\hceX_c (\eo)$. Here $\wedge$ denotes $p$-adic completion. The monoid $\varprojlim_{F_p} \Gamma (\eX , \Oh)^{\wedge}$ is isomorphic to the tilt $\Gamma (\eX , \Oh)^{\wedge b}$ of $\Gamma (\eX , \Oh)^{\wedge}$ and it carries the richer structure of a perfect $\F_p$-algebra. The rings $ A_{\inf} (\eX) := W_p (\Gamma  (\eX , \Oh)^{\wedge b})$ are basic objects in $p$-adic Hodge theory, and we asked ourself if there was a canonical $F_p$-subsystem $\hcY$ of $\hceX_c (\eo)$ on which not only the elements of $\Z \Gamma (\eX , \Oh)^{\wedge b}$ but even the elements of $A_{\inf} (\eX)$ could be viewed as $\eo$-valued functions. The answer is simple, $\hcY$ consists of all the diagrams in $\hceX_c (\eo)$ whose maps are not only multiplicative but $\mod p$ also additive. 

In general, it is very difficult to work with maps on rings that are only multiplicative. The condition of additivity $\mod p$ fortunately turns the diagrams in $\hcY$ into diagrams of {\it ring} homomorphisms on their tilted domains. This fact makes $\hcY$ more amenable to calculation than $\hceX_c (\eo)$. For $\eX_0 = \spec \eo_{K_0}$ with $K_0$ a finite extension of $\Q_p$ we have determined $\hcY \subset \hceX (\eo)$ and $\hcY_0 := \hcY / G$. The argument uses fundamental results about the tilting correspondence, and as mentioned above one obtains the closed points of the scheme theoretic Fargues-Fontaine curve (for $\hcY_0)$ and of its universal cover (for $\hcY$). The precise formulation is given in Theorem \ref{t136}. 

For $\eX_0 = \spec \Z_p$ and $\eo = \eo_p$ we have also determined the subsystem $\cY_0 = \hcY \cap \ceX_0 (\eo_p)$. It consists of the (infinite) Frobenius orbit of the one point space $\eX_0 (\eo) \subset \ceX_0 (\eo)$ and of one additional fixed point under Frobenius. Thus the process of ``completion'' to pass from $\ceX_0 (\eo)$ to $\hceX_0 (\eo)$ was necessary to obtain something interesting. 

In \cite{KS} Kucharczyk and Scholze expressed the hope that the two topological realizations of Galois groups in \cite{KS} and the one by Scholze in \cite{weinstein} using the Fargues-Fontaine curve  should be related. As explained, the present paper puts both constructions into a similar context. 

I would like to thank Maria L\"unnemann, Urs Hartl, Masanori Morishita, Peter Scholze and Wilhelm Singhof very much for helpful conversations and comments at various stages of this work. Moreover I am very grateful to Umberto Zannier for the invitation to the Scuola Normale in Pisa where the first construction of the dynamical systems was found, and to Pedro F. dos Santos for giving me the opportunity to pursue the work at the T\'ecnico in Lissabon. I would also like thank the HIM in Bonn for supporting a fruitful stay. \newpage
\tableofcontents
\newpage
\section{Rational Witt vectors} \label{sec:2_neu}
For a commutative unital ring $A$, the ring $W_{\rat} (A)$ of rational Witt vectors is a subring of the big Witt ring $W (A) = 1 + tA [[t ]]$. It consists of power series $f$ of the form $f = P / Q$ with polynomials $P , Q$ in $A [t]$ such that $P (0) = 1 = Q (0)$. There is a multiplicative unital injective (Teichm\"uller) map
\[ 
[ \, ] : A \longrightarrow W_{\rat} (A) \quad \text{where} \; [a] = 1 - at \; .
\]
It maps $0 \in A$ to $0 \in W_{\rat} (A)$ (which corresponds to the constant power series $1$). The Teichm\"uller map splits the surjective ring homomorphism
\[
W_{\rat} (A) \longrightarrow A \; , \; f \longmapsto -f' (0) / f (0) = - f' (0) \; .
\]
The Frobenius and Verschiebung maps of $W (A)$ leave $W_{\rat} (A)$ invariant and we have $F_{\nu} ([a]) = [a^{\nu}]$ for all $a \in A$. For an injective (surjective) ring map $f : A \to B$ the induced ring map $W_{\rat} (f) : W_{\rat} (A) \to W_{\rat} (B)$ is again injective (surjective). Injectivity follows since $W (f) : W (A) \to W (B)$ is injective. We refer to Almkvist \cite{almkvist} for a conceptual way to understand $W_{\rat} (A)$ in terms of $K$-theory. 

Let $\Z A$ be the monoid ring on $(A, \cdot)$. The elements of $\Z A$ are written as $\sum_a n_a a$ or $\sum_a n_a (a)$ where $a$ runs over $A$ and $n_a \in \Z$ with $n_a = 0$ for almost all $a$. We have $(1) = 1$ and $(0) \neq 0$ in $\Z A$. Let $\uZ A$ be the reduced monoid ring, i.e. the quotient of $\Z A$ by the ideal $\Z (0)$. For another commutative ring $R$, homomorphisms $P : \uZ A \to R$ correspond bijectively to multiplicative maps $P : A \to R$ sending $1$ to $1$ and $0$ to $0$. The Teichm\"uller map \cite{} induces a functorial ring homomorphism
\begin{equation} \label{eq:1n}
\omega = \omega (A) : \uZ A \longrightarrow W_{\rat} (A) \; .
\end{equation}

\begin{prop} \label{t21nn}
a) If $A$ is an integral domain, $\omega$ is injective
\[
\omega : \uZ A \hookrightarrow W_{\rat} (A) \; .
\]
b) For an integrally closed domain $A$ with algebraically closed quotient field, we have an isomorphism:
\[
\omega : \uZ A \silo W_{\rat} (A) \; .
\]
\end{prop}

\begin{remark} \label{t22}
In general, $\omega$ may not be injective. For example, if $A$ contains an element $a \neq 0$ with $a^2 = 0$, then the nonzero element $2 (a) - (2a)$ in $\uZ A$ is mapped to
\[
\frac{(1 - at)^2}{1 - 2at} = 1 \quad \text{the zero element in} \quad W_{\rat} (A) \; .
\]
\end{remark}

\begin{proof}
Let $K$ be an algebraically closed field. Since $K [t]$ is a unique factorization domain in which all polynomials split into linear factors, $\omega (K)$ is an isomorphism. Now a) follows by passing to the algebraic closure of the quotient field of $A$. As for b), injectivity is a special case of a). Consider a polynomial $P (t) = a_n t^n + \ldots + 1 \in A [t]$ with $a_n \neq 0$. Since $A$ is integrally closed in its algebraically closed quotient field, the polynomial 
\[
P^* (t) = t^n P (t^{-1}) = a_n + a_{n-1} t + \ldots + t^n
\]
decomposes into a product
\[
P^* (t) = (t - \alpha_1) \cdots (t - \alpha_n) \quad \text{with} \; \alpha_1 , \ldots , \alpha_n \in A \; .
\]
The $\alpha_i$ are non-zero since $\alpha_1 \cdots \alpha_n = (-1)^n a_n \neq 0$. Hence we have
\[
P (t) = t^n P^* (t^{-1}) = (1 - t \alpha_1) \cdots (1 - t \alpha_n) \quad \text{with} \; \alpha_i \; \text{in} \; A \setminus \{ 0 \} \; .
\]
Thus the map $\omega (A)$ is also surjective. 
\end{proof}

\begin{prop}\label{t21n}
Let $G$ be a group acting on a ring $A$ by ring automorphisms. Assume that one of the following two assumptions holds: \\
a) All orbits of the $G$-action on $A$ are finite.\\
b) The ring $A$ is an integrally closed domain.\\
Then the natural inclusion $W_{\rat} (A^G) \hookrightarrow W_{\rat} (A)^G$ is an isomorphism:
\begin{equation}
\label{eq:21n}
W_{\rat} (A^G) \silo W_{\rat} (A)^G 
\end{equation}
\end{prop}

\begin{proof}
a) Let $f = P / Q$ be in $W_{\rat} (A)^G$ and let $A_0$ be the subring of $A$ generated by the finitely many $G$-translates of the finitely many coefficients of $P$ and $Q$. The group $G$ acts on $A_0$ via its finite (!) image $G_0$ under the map 
\[
G \to \Aut A_0 , \sigma \mapsto \sigma \, |_{A_0} \; .
\]
Replacing $A$ and $G$ by $A_0$ and $G_0$ it therefore suffices to prove the surjectivity of \eqref{eq:21n} in the case where $G$ is finite. This is done as follows in \cite[Example 1.14]{dotto}: $G$-invariance of $f$ means that $\,^{\sigma}\!P Q = P \,^{\sigma}\!Q$ for all $\sigma \in G$. Using this, one checks that the polynomial
\[
\alpha = P \prod_{e \neq \sigma \in G} \,^{\sigma}\!Q
\]
is $G$-invariant. Since $\beta = \prod_{\sigma \in G} \,^{\sigma}\!Q$ is obviously $G$-invariant, it follows that $f = \alpha / \beta$ is a quotient of $G$-invariant polynomials with $\alpha (0) = 1 = \beta (0)$. \\
b) We first consider the special case where $A$ is a field. Given $f = P / Q$ in $W_{\rat} (A)^G$ we may assume that $P$ and $Q$ are coprime. Writing
\[
aP + bQ = 1
\]
for polynomials $a,b \in A [t]$ and using $\,^{\sigma}\!P Q = P\,^{\sigma}\!Q$ we find
\[
\,^{\sigma}\!P = \,^{\sigma}\!P (aP + bQ) = a\,^{\sigma}\!P P + bP\,^{\sigma}\!Q = P (a\,^{\sigma}\!P + b\,^{\sigma}\!Q) \; .
\]
By degree reasons and since $P (0) = 1$, it follows that $\,^{\sigma}\!P = P$ and similarly $\,^{\sigma}\!Q = Q$. Now let $A$ be an integral domain which is integrally closed in its quotient field $K$. The group action by $G$ on $A$ extends uniquely to a $G$-action on $K$. Since $W_{\rat} (K)^G = W_{\rat} (K^G)$ as we saw, it follows that
\[
W_{\rat} (A)^G \subset W_{\rat} (A) \cap W_{\rat} (K^G) \; .
\]
It suffices to show that the right hand side is contained in $W_{\rat} (A^G)$. So let $P / Q = P_1 / Q_1$ where $P , Q \in A [t]$ and $P_1 , Q_1 \in K^G [t]$ are polynomials with contant term $1$ and where we may assume $P_1$ and $Q_1$ to be coprime. Hence there are polynomials $a,b \in K^G [t]$ with $aP_1 + bQ_1 = 1$. Setting $S = aP + bQ \in K[t]$ it follows that $P = P_1 S$ and $Q = Q_1 S$, and we have $S (0) = 1$. Writing $P^* (t) := t^{\deg P} P (t^{-1})$ etc, we obtain decompositions in $K[t]$ of monic polynomials $P^* = P^*_1 S^*$ and $Q^* = Q^*_1 S^*$ where $P^*, Q^* \in A [t]$. Let $\oA$ be the integral closure of $A$ in an algebraic closure of $K$. It follows that all roots of the monic polynomials $P^*_1 , Q^*_1 , S^*$ lie in $\oA$ and therefore $P^*_1 , Q^*_1 \in \oA [t]$. Since
\[
\oA \cap K^G = (\oA \cap K) \cap K^G = A \cap K^G = A^G \; ,
\]
it follows that $P_1 , Q_1 \in A^G [t]$ and hence
\[
P / Q = P_1 / Q_1 \in W_{\rat} (A^G) \; .
\]
\end{proof}

For a field $\kappa$ we can identify $\uZ \kappa$ with the group algebra $\Z \kappa^{\times}$. For an algebraic field extension $\kappa / \kappa_0$ we set
\[
\uZ \kappa / \kappa_0 := \Big\{ \sum_{r \in \kappa^{\times}} n_r (r) \mid d_r \; \text{divides} \; n_r \; \text{for} \; r \in \kappa^{\times} \Big\} \subset \uZ \kappa \; .
\]
Here $d_r$ is the degree of inseparability of $r$ over $\kappa_0$, i.e. the smallest prime power $d_r = p^n$ such that $r^{p^n}$ is separable over $\kappa_0$ if $\car \kappa_0 = p > 0$, and $d_r = 1$ if $\car \kappa_0 = 0$. Alternatively $d_r = [\kappa_0 (r) : \kappa_0]_i$ is the inseparability degree of the extension $\kappa_0 (r) / \kappa_0$. Hence $d_{r_1 r_2}$ divides $d_{r_1} d_{r_2}$ for any $r_1 , r_2 \in \kappa$ and it follows that $\uZ \kappa / \kappa_0$ is a subring of $\uZ \kappa$. For perfect fields $\kappa_0$ we have $\uZ \kappa / \kappa_0 = \uZ \kappa$ by definition. 

Recall that for an algebraically closed field $\kappa$ the Teichm\"uller map induces an isomorphism
\begin{equation}
\label{eq:23n}
\uZ \kappa \silo W_{\rat} (\kappa) \; .
\end{equation}
It commutes with the canonical $\Aut (\kappa)$-actions on both sides. The following result is a restatement of results in \cite[\S\,2]{bryden}.

\begin{prop}
\label{t26n}
If $\kappa$ is an algebraic closure of the field $\kappa_0$ and $H = \Aut_{\kappa_0} (\kappa)$ then \eqref{eq:23n} induces a commutative diagram with horizontal isomorphisms
\[
\xymatrix{
(\uZ \kappa)^H \ar[r]^{\sim} & W_{\rat} (\kappa)^H \ar@{=}[r] & W_{\rat} (\kappa^{\perf}_0) \\
(\uZ \kappa / \kappa_0)^H \ar@{_{(}->}[u] \ar[rr]^{\sim} & & W_{\rat} (\kappa_0) \; . \ar@{_{(}->}[u]
}
\]
Here $\kappa^{\perf}_0$ is the perfect closure of $\kappa_0$ in $\kappa$. If $p = \car \kappa_0$ is positive, the vertical arrows become isomorphisms after tensoring with $\Z [1 / p] \subset \Q$. 
\end{prop}

\begin{proof}
The equality $W_{\rat} (\kappa)^H = W_{\rat} (\kappa^{\perf}_0)$ follows from either a) or b) of Proposition \ref{t21n} since $\kappa^{\perf}_0 = \kappa^H$. The lower isomorphism follows because the irreducible polynomials $\alpha \in \kappa_0 [t]$ with $\alpha (0) = 1$ correspond bijectively to the irreducible polynomials $\talpha \in \kappa^{\perf}_0 [t]$ with $\talpha (0) = 1$ via the relation $\alpha (t) = \talpha (t^d)$ where $d$ is the degree of inseparability of any zero of $\alpha$ in $\kappa$. The last assertion is clear.
\end{proof}

The following result will be needed several times:

\begin{theorem}[\cite{KS}]
\label{t27n}
Let $A$ be a ring on which a profinite group $H$ acts continuously. Then for every algebraically closed field $\C$, the natural map
\[
\Hom (A, \C) / H \longrightarrow \Hom (A^H , \C) \; , \; \alpha \mod H \longmapsto \alpha \, |_{A^H}
\]
is bijective. Moreover, the inclusion $W_{\rat} (A^H) \hookrightarrow W_{\rat} (A)$ induces a bijection
\[
\Hom (W_{\rat} (A) , \C) / H \silo \Hom (W_{\rat} (A^H) , \C) \; .
\]
\end{theorem}

\begin{proof}
The first assertion is \cite{KS} Lemma 4.9. For the second assertion, note that $W_{\rat} (A^H) = W_{\rat} (A)^H$ by Proposition \ref{t21n} a). Since $H$ acts with finite orbits on $A$ it acts with finite orbits on $W_{\rat} (A)$ as well (!) Hence the second assertion follows from the first applied to $W_{\rat} (A)$ instead of $A$. 
\end{proof}

Combining this result with Proposition \ref{t26n} we obtain the following

\begin{cor}
\label{t28n}
Let $\kappa_0$ be a field, $\kappa$ an algebraic closure of $\kappa_0$ and $\kappa^{\sep}_0$ the separable closure of $\kappa_0$ in $\kappa$. Set $H = \Aut_{\kappa_0} (\kappa) = \Gal (\kappa^{\sep}_0 / \kappa_0)$ and let $\C$ be an algebraically closed field. \\
1) The inclusions $W_{\rat} (\kappa_0) \subset \uZ \kappa / \kappa_0 \subset \uZ \kappa$ induce bijections
\begin{align*}
& \Hom (W_{\rat} (\kappa_0) , \C) = \Hom (\uZ \kappa / \kappa_0 , \C) / H = \Hom (\uZ \kappa , \C) / H\\
= & \{ P : \kappa \to \C \mid P \; \text{multiplicative} \; , \; P (1) = 1 , P (0) = 0 \} / H \; .
\end{align*}
2) If $\kappa_0$ and $\C$ both have positive characteristic $p$, the inclusions 
\[
W_{\rat} (\kappa_0) \hookrightarrow \uZ \kappa \hookleftarrow \uZ \kappa_0^{\sep}
\]
also induce bijections
\begin{align*}
\Hom (W_{\rat} (\kappa_0) , \C) & = \Hom (\uZ \kappa , \C) / H = \Hom (\uZ \kappa^{\sep}_0 , \C) / H \\
= & \{ P : \kappa^{\sep}_0 \to \C \mid P \; \text{multiplicative} \; , \; P (1) = 1 , P (0) = 0 \} / H \; .
\end{align*}
\end{cor}

\begin{proof}
For a prime number $p$ set $\Lambda_p = \Z [1 / p]$ and set $\Lambda_0 = \Z$. For $p = \car \kappa_0 \ge 0$ we then have
\[
\uZ \kappa / \kappa_0 \otimes_{\Z} \Lambda_p = \uZ \kappa \otimes_{\Z} \Lambda_p \; .
\]
Using Proposition \ref{t26n} and Theorem \ref{t27n} we have
\begin{align} \label{eq:4nx}
\Hom (W_{\rat} (\kappa_0) , \C) & = \Hom ((\uZ \kappa / \kappa_0)^H , \C) \\
& = \Hom (\uZ \kappa / \kappa_0 , \C) / H \; . \nonumber
\end{align}
If $\car \kappa_0 = 0$ or if $\car \kappa_0 \neq \car \C$ we have
\begin{align*}
\Hom (\uZ \kappa / \kappa_0 , \C) & = \Hom (\uZ \kappa / \kappa_0 \otimes \Lambda_p , \C) \\
& = \Hom (\uZ \kappa \otimes \Lambda_p , \C) \\
& = \Hom (\uZ \kappa , \C) \; .
\end{align*}
Thus assertion 1) follows for such $\kappa_0$ and $\C$. We now assume that $\kappa_0$ and $\C$ both have positive characteristic $p$ and show that 1) and 2) hold in this case. By \eqref{eq:4nx} it suffices to show that the inclusions
\[
\uZ \kappa^{\sep}_0 \hookrightarrow \uZ \kappa / \kappa_0 \hookrightarrow \uZ \kappa
\]
induce bijections
\begin{equation}
\label{eq:4nn}
\Hom (\uZ \kappa , \C) \overset{\alpha}{\silo} \Hom (\uZ \kappa / \kappa_0 , \C) \overset{\beta}{\silo} \Hom (\uZ \kappa^{\sep}_0 , \C) \; .
\end{equation}
The composition $\beta \verk \alpha$ can be identified with the restriction map
\[
\Hom (\kappa^{\times} , \C^{\times}) \longrightarrow \Hom ((\kappa^{\sep}_0)^{\times} , \C^{\times}) \; .
\]
It is a bijection since $\kappa^{\times} / (\kappa^{\sep}_0)^{\times}$ is a $p$-power torsion group and $\C^{\times}$ is uniquely $p$-divisible. Hence it suffices to show that $\beta$ is injective. Let $P : \uZ \kappa / \kappa_0 \to \C$ be a ring homomorphism and for $r \neq 0$ set $\varphi (r) = P (d_r (r))$ where $d_r = [\kappa_0 (r) : \kappa_0]_i$ is the degree of inseparability of $r$. We have
\begin{equation}
\label{eq:5nn}
\varphi (1) = 1 \quad \text{and} \quad \varphi (r) \varphi (s) = \frac{d_r \, d_s} {d_{rs}} \varphi (rs) \quad \text{for} \; r,s \in \kappa^{\times} \; .
\end{equation}
Giving $P$ is equivalent to giving the map $\varphi : \kappa^{\times} \to \C$ satisfying \eqref{eq:5nn}. The relations \eqref{eq:5nn} imply that 
\[
\varphi (r)^2 = \frac{d^2_r}{d_{r^2}} \varphi (r^2) \quad \text{for all} \; r \in \kappa^{\times} \; .
\]
Since $\kappa_0 (r^2) \subset \kappa_0 (r)$ the number $d_{r^2}$ divides $d_r$. Thus for $r \in \kappa \setminus \kappa^{\sep}_0$ i.e. $d_r > 1$ the quotient $d^2_r / d_{r^2}$ is a positive power of $p$. Since $\car \C = p$ by assumption, it follows that $\varphi (r)^2 = 0$ and hence $\varphi (r) = 0$. Hence $P$ is uniquely determined by its restriction to $\uZ \kappa^{\sep}_0$. 
\end{proof}

We now extend Theorem \ref{t27n} to certain valuation rings $\eo$ instead of fields $\C$. This requires the following result:

\begin{lemma}
\label{t212n}
Let $\eo$ be a rank one valuation ring with algebraically closed quotient field $\C$. \\
1) For any integral morphism of schemes $\pi : X \to X_0$ the induced map
\[
X (\eo) \longrightarrow X_0 (\eo) \times_{X_0 (\C)} X (\C)
\]
is surjective. If $\pi$ is surjective, then $X (\eo) \to X_0 (\eo)$ is surjective as well.\\
2) Let $B$ be a ring on which a profinite group $G$ acts continuously. Then the morphism
\[
\pi : X = \spec B \longrightarrow X_0 = \spec B^G
\]
induces a bijection
\[
X (\eo) / G \silo X_0 (\eo) \; .
\]
\end{lemma}

\begin{proof}
1) If the integral morphism $\pi$ is surjective then the map $X (\C) \to X_0 (\C)$ is surjective as well. Namely, given $x_0 \in X_0$ there is some $x \in X$ with $\pi (x) = x_0$, and $\kappa (x)$ is algebraic over $\kappa (x_0)$ since $\pi$ is integral. Since $\C$ is algebraically closed we can extend any embedding $\kappa (x_0) \hookrightarrow \C$ to an embedding $\kappa (x) \hookrightarrow \C$. Hence it suffices to prove the first assertion in 1). A point
\[
(f,g) \in X_0 (\eo) \times_{X_0 (\C)} X (\C)
\]
is a commutative diagram
\[
\xymatrix{
\spec \C \ar[r]^-g \ar[d] & X \ar[d]^{\pi} \\
\spec \eo \ar[r]^-f & X_0
}
\]
Let $\emm$ and $k$ be the maximal ideal and the residue field of $\eo$. Then $\spec \eo = \{ (0) , \emm \}$ by the rank one assumption on $\eo$. Let $x_0$ and $y_0$ be the images under $f$ of $(0)$ and $\emm$. Let $Z_0 =  \overline{\{ x_0 \}} \subset X_0$ with the reduced scheme structure. Then $y_0 \in Z_0$ and the residue field $\kappa (x_0)$ of $x_0$ in $X_0$ is the function field of $Z_0$. The morphism $f$ factors canonically as
\[
f : \spec \eo \longrightarrow Z_0 \hookrightarrow X_0 \; .
\]
The morphism $g$ is given by a point $x \in X$ with $\pi (x) = x_0$ and an embedding $\kappa (x) \hookrightarrow \C$ extending the embedding $\kappa (x_0) \hookrightarrow \C$ given by $\spec \C \to \spec \eo \xrightarrow{f} X_0$. Set $Z = \overline{\{ x \}} \subset X$. Since $\pi$ is integral, the base change $\pi^{-1} (Z_0) \to Z_0$ is integral as well. Since closed immersions are integral, it follows that the composition
\[
\pi_Z : Z \hookrightarrow \pi^{-1} (Z_0) \longrightarrow Z_0
\]
is integral as well. Let $U_0 = \spec A_0$ be an open affine neighborhood of $y_0$ in $Z_0$. Then $x_0 \in U_0$ and the quotient field of $A_0$ is $\kappa (x_0)$. Since $\pi_Z$ is integral, the inverse image $U = \pi^{-1}_Z (U_0) = \spec A$ is affine, for an integral $A_0$-algebra $A$ with quotient field $\kappa (x)$. The morphism $f$ factors as 
\[
f : \spec \eo \longrightarrow U_0 \longrightarrow X_0
\]
since $U_0$ contains the image of $f$. We obtain a commutative diagram
\[
\xymatrix{
 & \kappa (x) \ar@{^{(}->}[ddrr] & & \\
A \ar@{^{(}->}[ur] & & & \\
 & \kappa (x_0) \ar@{^{(}->}[uu] \ar@{^{(}->}[rr] & & \C \\
A_0 \ar@{^{(}->}[rr]\ar@{^{(}->}[uu] \ar@{^{(}->}[ur] &  & \eo \ar@{^{(}->}[ur] & 
}
\]
Here the embedding $A_0 \hookrightarrow \eo$ corresponds to the morphism $f : \spec \eo \to U_0 = \spec A_0$. Every element $a \in A$ satisfies a monic polynomial equation over $A_0$ since $A$ is an integral $A_0$-algebra. Hence the image of $a$ in $\C$ satsifies such an equation over $\eo$ and is therefore an element of $\eo$ since $\eo$ is integrally closed in $\C$. Hence we get commutative diagrams
\[
\vcenter{\xymatrix{
A \ar@{^{(}->}[dr] & \\
A_0 \ar@{^{(}->}[u] \ar@{^{(}->}[r] & \eo
} }
\qquad \text{and} \qquad 
\vcenter{\xymatrix{
 & U \ar[r] \ar[d] & X \ar[d] \\
\spec \eo \ar[ur] \ar[r] & U_0 \ar[r] & X_0 \; .
}}
\]
This defines a point in $X (\eo)$ mapping to $(f,g)$. \\
2) According to \cite{KS} Lemma 4.9 and its proof, $B$ is an integral $B^G$-algebra, $\pi$ is surjective and the induced map
\[
X (\C) / G \silo X_0 (\C)
\]
is bijective. It follows from 1) that the map
\[
X (\eo) / G \longrightarrow X_0 (\eo)
\]
is surjective. Since $X (\eo) \subset X (\C)$ we have a commutative diagram
\[
\xymatrix{
X (\eo) / G \ar@{->>}[r] \ar@{_{(}->}[d] & X_0 (\eo) \ar[d] \\
X (\C) / G \ar[r]^{\sim} & X_0 (\C) \; .
}
\]
Hence the map $X (\eo) / G \to X_0 (\eo)$ is a bijection. 
\end{proof}

\begin{cor} 
\label{t213n}
Let $\eo$ be a rank one valuation ring with algebraically closed quotient field $\C$ and let $A$ be a ring on which a profinite group $G$ acts continuously.\\
1) The inclusion $W_{\rat} (A^G) \to W_{\rat} (A)$ induces a bijection
\[
\Hom (W_{\rat} (A) , \eo) / G \silo \Hom (W_{\rat} (A^G) , \eo) \; .
\]
Any extension $P : W_{\rat} (A) \to \C$ of a homomorphism $P_0 : W_{\rat} (A^G) \to \eo$ takes values in $\eo$. \\
2) For an integrally closed domain $A$ with algebraically closed quotient field we have identifications
\begin{align*}
& \Hom (W_{\rat} (A^G) , \eo) = \Hom (\uZ A, \eo) / G \\
= & \{ P : A \to \eo \mid P \; \text{multiplicative} \; , \; P (1) = 1 , P (0) = 0 \} / G \; .
\end{align*}
\end{cor}

\begin{proof}
1) Set $B = W_{\rat} (A)$. By assumption $G$ acts with finite orbits on $A$ and hence on $B$. Proposition \ref{t21n} a) gives $B^G = W_{\rat} (A^G)$. The bijection in 1) now follows from part 2 of Lemma \ref{t212n}.  The final claim follows because any two extensions $P$ of $P_0$ with values in $\C$ are conjugate under $G$ by the second part of Theorem \ref{t27n}. Since we know that an extension with values in $\eo$ exists, all extensions $P$ must take values in $\eo$. \\
2) This follows from 1) and Proposition \ref{t21nn} b).
\end{proof}

An integral domain $A$ with quotient field $K$ is called Fatou if
\[
W_{\rat} (K) \cap W (A) = W_{\rat} (A) \quad \text{in} \; W (K) \; .
\]
A Noetherian integral domain is Fatou \cite[section 3]{hazewinkel}. 

\begin{theorem} \label{t29}
Let $A$ be a Fatou domain. Then for any $N \ge 1$ we have
\[
V_N (W (A)) \cap W_{\rat} (A) = V_N (W_{\rat} (A)) \; .
\]
\end{theorem}

\begin{proof}
We first show that if $K$ is a field and $f \in W (K)$ satisfies $V_N (f) \in W_{\rat} (K)$ then $f \in W_{\rat} (K)$. If $\car K = 0$ one can check this using a Galois argument. The following argument works in general. Recall that a Laurent series $g \in K ((t))$ is rational if and only if its continued fraction expansion is finite. Explicitely: For $g = \sum^{\infty}_{\nu > - \infty} a_{\nu} t^{\nu}$ set
\[
Tg = \big( \textstyle{\sum^{\infty}_{\nu = 1 }} a_{\nu} t^{\nu} \big)^{-1} \quad \text{if} \quad  \textstyle{\sum^{\infty}_{\nu = 1}} a_{\nu} t^{\nu} \neq 0 
\]
and $Tg = 0$ otherwise. Then $g \in K (t)$ if and only if $T^n g = 0$ for some $n \ge 0$. Namely, for $g \in K (t) = \Quot K [t^{-1}]$, a repeated application of the Euclidean algorithm in $K [t^{-1}]$ shows that $T^n g = 0$ for some $n \ge 0$. The other direction is clear. Extend $V_N$ to a map 
\[
V_N : K ((t)) \longrightarrow K ((t)) \quad \text{by} \; (V_N (g)) (t) = g (t^N) \; .
\]
Then $V_N$ is injective and $T \verk V_N = V_N \verk T$. Hence $g$ is rational if and only if $V_N (g)$ is rational. The above assertion on $f \in W (K)$ follows since
\[
W (K) \cap K (t) = W_{\rat} (K) \quad \text{in} \; K ((t)) \; .
\]
Now let $A$ be a Fatou domain and assume that $f \in W (A)$ satisfies $V_N (f) \in W_{\rat} (A)$. Let $K$ be the quotient field of $A$. Then by what we have seen it follows that $f \in W_{\rat} (K)$ and hence $f \in W_{\rat} (A)$ since $A$ is Fatou. The theorem follows. 
\end{proof}

\section{Rational Witt spaces} \label{sec3nn}

We call an integral scheme $X$ Fatou if for every open affine subset $U$ of $X$ the ring $\Oh_X (U)$ is Fatou. Every integral, locally Noetherian scheme is Fatou. The rational Witt vectors commute with filtered colimits since firstly, a rational function is determined by finitely many coefficients; secondly, for polynomials $P_i , Q_i \in A [t]$ with constant terms $1$ we have $P_1 / Q_1 = P_2 / Q_2$ in $W_{\rat} (A)$ if and only if $P_1 Q_2 = P_2 Q_1$ in $A [t]$. For a scheme $X$ and a point $x \in X$ the stalk of the Zariski presheaf $U \mapsto W_{\rat} (\Oh_X (U))$ in $x$ is therefore given by $W_{\rat} (\Oh_{X,x})$. We denote by $W_{\rat} (\Oh_X)$ the {\it sheafification} of the presheaf $U \mapsto W_{\rat} (\Oh_X (U))$. Stalks of presheaves and of associated sheaves coincide. Hence we have
\[
W_{\rat} (\Oh_X)_x = W_{\rat} (\Oh_{X,x}) \quad \text{for} \; x \in X \; .
\]
The sheaf $W_{\rat} (\Oh_X)$ is naturally a subsheaf of the sheaf $W (\Oh_X)$ defined by $W (\Oh_X) (U) = W (\Oh_X (U))$. 

\begin{prop} \label{t22n}
a) For any scheme $X$ the presheaf $U \mapsto W_{\rat} (\Oh_X (U))$ is separated and hence we have for all open $U \subset X$ that
\[
W_{\rat} (\Oh_X (U)) \subset W_{\rat} (\Oh_X) (U) \; .
\]
b) If $X$ is integral and the domain $\Oh_X (U)$ is Fatou then we have equality
\[
W_{\rat} (\Oh_X (U)) = W_{\rat} (\Oh_X) (U) \; .
\]
c) Let $X$ be Fatou (e.g. integral and locally Noetherian) with function field $K$. Then for every open $U \subset X$ we have
\[
W_{\rat} (\Oh_X) (U) = W_{\rat} (K) \cap W (\Oh_X (U)) \quad \text{in} \;\, W (K) \; .
\]
\end{prop}

\begin{proof}
a) The presheaf $U \mapsto W_{\rat} (\Oh_X (U))$ is a sub-presheaf of the sheaf $W (\Oh_X)$. Hence it is separated.\\
b) Since $U \mapsto W_{\rat} (\Oh_X (U))$ is separated, a section of $W_{\rat} (\Oh_X) (U)$ is given by a family $(f_i)$ of elements $f_i \in W_{\rat} (\Oh_X (U_i))$ for some open cover $\eU = (U_i)$ of $U$ such that $f_i$ and $f_j$ have the same image in $W_{\rat} (\Oh_X (U_i \cap U_j))$ for all indices $i$ and $j$. Let $K$ be the field of rational functions on the integral scheme $X$. It follows that the $f_i$ have the same image $f$ in $W_{\rat} (K)$ for all $i$. On the other hand, the $f_i$ define a section of the sheaf $W (\Oh_X)$ over $U$ i.e. an element of $W (\Oh_X (U))$. Hence using the Fatou property of $\Oh_X (U)$, we find
\[
f \in W_{\rat} (K) \cap W (\Oh_X (U)) = W_{\rat} (\Oh_X (U)) \; .
\]
Thus the inclusion
\[
W_{\rat} (\Oh_X (U)) \subset W_{\rat} (\Oh_X) (U)
\]
is an equality. \\
c) The right hand side defines a sheaf on $X$ which contains the presheaf $U \mapsto W_{\rat} (\Oh_X (U))$ and hence also $W_{\rat} (\Oh_X)$. To show the reverse inclusion, let $f$ be in $W_{\rat} (K) \cap W (\Oh_X (U))$. Choose an open affine covering $U = \bigcup_i U_i$. By assumption each $\Oh_X (U_i)$ is Fatou. Let $f_i$ be the restriction of $f$ to
\[
W_{\rat} (K) \cap W (\Oh_X (U_i)) = W_{\rat} (\Oh_X (U_i)) \subset W_{\rat} (\Oh_X) (U_i) \; .
\]
The restrictions of $f_i$ and $f_j$ to $W_{\rat} (\Oh_X) (U_i \cap U_j)$ are equal because this is true in $W (\Oh_X) (U_i \cap U_j) = W (\Oh_X (U_i \cap U_j))$. Since $W_{\rat} (\Oh_X)$ is a sheaf, $f$ is in $W_{\rat} (\Oh_X) (U)$.
\end{proof}

For any scheme $X$, the Frobenius and Verschiebung maps on the rings $W_{\rat} (\Oh_X (U))$ for $U$ open in $X$ induce Frobenius and Verschiebung morphisms on the presheaf $U \mapsto W_{\rat} (\Oh_X (U))$ and hence on the associated sheaf $W_{\rat} (\Oh_X)$. They satisfy the usual relations and will be denoted by $F_{\nu}$ resp. $V_{\nu}$ for all integers $\nu \ge 1$. 
The Teichm\"uller maps $[\, ] : \Oh_X (U) \to W_{\rat} (\Oh_X (U))$ induce an injective multiplicative map of sheaves sending $1$ to $1$ and $0$ to $0$
\[
[\, ] : \Oh_X \to W_{\rat} (\Oh_X) \; .
\]

\begin{defn} \label{t23n}
Let $X$ be a scheme with underlying topological space $X_{\top}$. The rational Witt space of $X$ is the ringed space
\[
W_{\rat} (X) := (X_{\top} , W_{\rat} (\Oh_X)) \; .
\]
\end{defn}

A morphism of schemes is a special type of morphism of ringed spaces: the induced maps on local rings have to be {\it local} i.e. induce homomorphisms of the residue fields. In our case we propose an analogous condition to account for the locality.

\begin{defn} \label{t24n}
Let $X$ and $Y$ be schemes. A morphism
\[
f : W_{\rat} (Y) \longrightarrow W_{\rat} (X)
\]
is a morphism of ringed spaces
\[
(f , f^{\sharp}) : (Y_{\top} , W_{\rat} (\Oh_Y)) \longrightarrow (X_{\top} , W_{\rat} (\Oh_X))
\]
which is local in the following sense: For each $y \in Y$ there is a commutative diagram
\[
\xymatrix{
W_{\rat} (\Oh_{X , f (y)}) \ar@{->>}[d] \ar[r]^-{f^{\sharp}_y} & W_{\rat} (\Oh_{Y, y}) \ar@{->>}[d] \\
W_{\rat} (\kappa (f (y))) \ar[r]^-{\tf^{\sharp}_y} & W_{\rat} (\kappa (y)) \; .
}
\]
Note that the map $\tf^{\sharp}_y$ is uniquely determined by $f^{\sharp}_y$. 
\end{defn}

For every integer $\nu \ge 1$ we have a Frobenius endomorphism
\[
F_{\nu} : W_{\rat} (X) \longrightarrow W_{\rat} (X) \; .
\]
It is the identity on the underlying space $X_{\top}$ and the $\nu$-th Frobenius on $W_{\rat} (\Oh_X)$. The multiplicative monoid $\Nh$ of positive integers acts on $W_{\rat} (X)$ via $\nu \mapsto F_{\nu}$. 

Any morphism of schemes $\alpha : Y \to X$ induces a morphism 
\[
W_{\rat} (\alpha) : W_{\rat} (Y) \to W_{\rat} (X) \; .
\]
On the underlying topological spaces it is the map $\alpha : Y_{\top} \to X_{\top}$. For $U \subset X$ open, using the map induced by $\alpha$
\[
\Oh_X (U) \longrightarrow \Oh_Y (\alpha^{-1} (U))
\]
we get a composition
\[
W_{\rat} (\Oh_X (U)) \longrightarrow W_{\rat} (\Oh_Y (\alpha^{-1} (U)) \hookrightarrow W_{\rat} (\Oh_Y) (\alpha^{-1} (U)) \; .
\]
This gives a morphism from the presheaf $U \mapsto W_{\rat} (\Oh_X (U))$ to the sheaf $\alpha_* W_{\rat} (\Oh_Y)$ and hence a morphism of sheaves
\[
W_{\rat} (\Oh_X) \longrightarrow \alpha_* W_{\rat} (\Oh_Y) \; .
\]
Consider its adjunction:
\[
W_{\rat} (a^{\sharp}) : \alpha^{-1} W_{\rat} (\Oh_X) \longrightarrow W_{\rat} (\Oh_Y) \; .
\]
Then
\[
W_{\rat} (\alpha) = (\alpha , W_{\rat} (\alpha^{\sharp}))
\]
is the desired morphism. The stalk in $y \in Y$ is
\begin{equation} \label{eq:x}
W_{\rat} (\alpha^{\sharp})_y = W_{\rat} (\alpha^{\sharp}_y) : W_{\rat} (\Oh_{X, \alpha (y)}) \longrightarrow W_{\rat} (\Oh_{Y, y}) \; .
\end{equation}
The morphism $\alpha^{\sharp}_y : \Oh_{X , \alpha (y)} \to \Oh_{Y, y}$ is local and therefore induces a map of residue fields $\talpha^{\sharp}_y : \kappa (\alpha (y)) \to \kappa (y)$ and of rings
\[
W_{\rat} (\talpha^{\sharp}_y) : W_{\rat} (\kappa (\alpha (y)) \longrightarrow W_{\rat} (\kappa (y)) \; .
\]
The locality condition on $W_{\rat} (\alpha)$ is satisfied with $W_{\rat} (\alpha^{\sharp})^{\sim}_y := W_{\rat} (\talpha^{\sharp}_y)$. 

We obtain a functor from the category of schemes to the category of rational Witt spaces by sending $X$ to $W_{\rat} (X)$ and $\alpha : Y \to X$ to $W_{\rat} (\alpha)$. Note that $W_{\rat} (\alpha)$ is equivariant with respect to the $\Nh$-actions via Frobenius on $W_{\rat} (Y)$ and $W_{\rat} (X)$. On stalks it is equivariant with respect to Verschiebung. 

\begin{prop} \label{t27xn}
For all schemes $X$ and $Y$, the map
\[
W_{\rat} : \Mor (Y,X) \longrightarrow \Mor (W_{\rat} (Y) , W_{\rat} (X))
\]
is injective.
\end{prop}

\begin{proof}
Assume that for $\alpha , \beta : Y \to X$ we have $W_{\rat} (\alpha) = W_{\rat} (\beta)$. Then by definition, the underlying maps of topological spaces $Y_{\top} \to X_{\top}$ are equal. Since $W_{\rat} (\alpha)^{\sharp} = W_{\rat} (\beta)^{\sharp}$ by assumption, equation \eqref{eq:x} implies equality for all $y \in Y$ where $z = \alpha (y) = \beta (y)$:
\begin{equation}
\label{eq:2x}
W_{\rat} (\alpha^{\sharp}_y) = W_{\rat} (\beta^{\sharp}_y) : W_{\rat} (\Oh_{X, z}) \longrightarrow W_{\rat} (\Oh_{Y,y}) \; .
\end{equation}
Consider the commutative diagram:
\[
\xymatrix{
W_{\rat} (\Oh_{X,z}) \ar[r]^{W_{\rat} (\alpha^{\sharp}_y)} & W_{\rat} (\Oh_{Y,y}) \\
\Oh_{X,z} \ar[u]^{[\;]} \ar[r]^{\alpha^{\sharp}_y} & \Oh_{Y,y} \; .\ar[u]_{[\;]}
}
\]
Together with the corresponding diagram for $\beta$, equation \eqref{eq:2x} and the injectivity of the Teichm\"uller maps $[\;]$ we find that $\alpha^{\sharp}_y = \beta^{\sharp}_y$ for all $y$, and hence $\alpha^{\sharp} = \beta^{\sharp}$. 
\end{proof}

Thus the functor $W_{\rat}$ is faithful. I have the hope that on a suitable category of anabelian schemes it may be fully faithful.


The Frobenius morphisms $F_{\nu}$ on $W_{\rat} (X)$ are not automorphisms in general. In order to invert them consider $\Nh$ as a directed poset under divisibility and define the $\ind$-space
\[
\cW_{\rat} (X) = \text{``}\colim_{\Nh}\text{''} W_{\rat} (X) \; .
\]
More generally, the same can be done for any submonoid $\Nh_0$ of $\Nh$.

We now define $S$-valued points of $W_{\rat} (X)$ and $\cW_{\rat} (X)$ for any scheme $S$, i.e. morphisms from $S$ into these spaces.

\begin{defn} \label{t25n}
Let $X$ and $S$ be schemes. A morphism
\[
f : S \longrightarrow W_{\rat} (X)
\]
is a morphism of ringed spaces
\[
(f , f^{\sharp}) : (S_{\top} , \Oh_S) \longrightarrow (X_{\top} , W_{\rat} (\Oh_X))
\]
which is local in the following sense. For each $s \in S$, there is a commutative diagram
\[
\xymatrix{
W_{\rat} (\Oh_{X , f (s)}) \ar@{->>}[d] \ar[r]^-{f^{\sharp}_s} & \Oh_{S,s} \ar@{->>}[d] \\
W_{\rat} (\kappa (f (s))) \ar[r]^-{\tf^{\sharp}_s} & \kappa (s) \; .
}
\]
Note that the map $\tf^{\sharp}_s$ is uniquely determined by $f^{\sharp}_s$. 
\end{defn}

For schemes $T$ and $Y$ there are natural (associative) composition maps:
\[
\Mor (T,S) \times \Mor (S, W_{\rat} (X)) \xrightarrow{\verk} \Mor (T , W_{\rat} (X))
\]
and
\[
\Mor (S, W_{\rat} (X)) \times \Mor (W_{\rat} (X) , W_{\rat} (Y)) \xrightarrow{\verk} \Mor (S, W_{\rat} (Y)) \; .
\]
We write $W_{\rat} (X) (S)$ for the set of morphisms from $S$ to $W_{\rat} (X)$. The $\Nh$-action on $W_{\rat} (X)$ induces an $\Nh$-action on $W_{\rat} (X) (S)$ by composition. We define:
\[
\cW_{\rat} (X) (S) = \colim_{\Nh} W_{\rat} (X) (S) \; ,
\]
taking the colimit in the category of sets. As usual, if $S = \spec A$ is affine we will write $W_{\rat} (X) (A)$ and $\cW_{\rat} (X) (A)$ for the $S$-valued points. 

\begin{exmp} \label{t28}
For a scheme $X$ the canonical surjective ring homomorphisms
\[
W_{\rat} (\Oh_X (U)) \longrightarrow \Oh_X (U) \quad \text{for} \; U \subset X \; \text{open} \; ,
\]
induce a surjective morphism of ring sheaves
\[
f^{\sharp} : W_{\rat} (\Oh_X) \longrightarrow \Oh_X \; .
\]
This gives a morphism
\[
f : X \longrightarrow W_{\rat} (X) \; , \; f = (\id , f^{\sharp} )
\]
because the locality condition is satisfied with the canonical projections
\[
\tf^{\sharp}_x : W_{\rat} (\kappa (x)) \longrightarrow \kappa (x) \quad \text{for} \; x \in X \; .
\]
\end{exmp}
For every scheme $S$ we therefore have a canonical injective map, which we will view as an inclusion:
\begin{equation}
\label{eq:8a}
X (S) \hookrightarrow W_{\rat} (X) (S) \; , \; g \longmapsto f \verk g \; .
\end{equation}
For $X = \spec \Z$ the set $X (S)$ consists of one element whereas $W_{\rat} (X) (S)$ can be very big. The following remark follows from standard facts about Witt vectors.

\begin{rem}
For an $\F_p$-scheme $X$ with absolute Frobenius endomorphisms $F_X$ we have a commutative diagram of ringed spaces
\[
\xymatrix{
X \ar[r]^f \ar[d]_{F_X} & W_{\rat} (X) \ar[d]^{F_p} \\
X \ar[r]^f & W_{\rat} (X) \; .
}
\]
Moreover $F_p$ on $W_{\rat} (X)$ is induced by $F_X$ i.e.$F_p = W_{\rat} (F_X)$. 
\end{rem}

The Teichm\"uller maps induce a map of sheaves of monoids on $X$
\[
\Oh^m_X \longrightarrow W_{\rat} (\Oh_X)^m \; .
\]
Here the superscript $m$ stands for ``multiplicative monoid''. This leads to a morphism of ``monoid spaces''
\[
W_{\rat} (X)^m := (X_{\top} , W_{\rat} (\Oh_X)^m) \longrightarrow (X_{\top} , \Oh^m_X) =: X^m \; .
\]
In a way, $W_{\rat} (X)$ is intermediate between the scheme $X$ and the space $X^m$ which may be considered as an object of $\F_1$-geometry. 

A morphism $f : S \to W_{\rat} (X)$ gives a sheaf map $f^{\sharp} : W_{\rat} (\Oh_X) \to f_* \Oh_S$ and hence a map of global sections
\[
\Gamma (f^{\sharp}) : \Gamma (X , W_{\rat} (\Oh_X)) \longrightarrow \Gamma (S , \Oh_S) \; .
\]
We get a pairing:
\begin{equation}
\label{eq:3*}
W_{\rat} (X) (S) \times \Gamma (X, W_{\rat} (\Oh_X)) \longrightarrow \Gamma (S , \Oh_S) \; , \; (f,g) \longmapsto \Gamma (f^{\sharp}) (g) \; .
\end{equation}
It induces maps of sets resp. of rings
\begin{equation}
\label{eq:4*}
W_{\rat} (X) (S) \longrightarrow \Hom (\Gamma (X , W_{\rat} (\Oh_X)) \; , \; \Gamma (S, \Oh_S)) 
\end{equation}
and
\begin{equation}
\label{eq:5*}
\Gamma (X, W_{\rat} (\Oh_X)) \longrightarrow \Map (W_{\rat} (X) (S) \; , \; \Gamma (S , \Oh_S)) \; .
\end{equation}
For $S = \spec R$ the map \eqref{eq:4*} becomes a map
\begin{equation}
\label{eq:6*}
W_{\rat} (X) (R) \longrightarrow (\spec \Gamma (X , W_{\rat} (\Oh_X)) (R) \; .
\end{equation}
For $X = \spec K$, where $K$ is a field, a short argument shows that the map \eqref{eq:6*} is a bijection
\begin{equation}
\label{eq:7*}
W_{\rat} (\spec K) (R) \silo (\spec W_{\rat} (K)) (R) \; .
\end{equation}
In the paper \cite{KS} the space $(\spec W_{\rat} (K)) (\C)$ is studied for the complex number field $\C$. It agrees with $W_{\rat} (\spec K) (\C)$ but the ringed spaces $W_{\rat} (\spec K)$ and $\spec W_{\rat} (K)$ are quite different. The underlying space of the former has only one point whereas the latter has many points.

For $S = \spec R$ the map \eqref{eq:5*} becomes a ring map
\begin{equation}
 \label{eq:8*}
 \Gamma (X , W_{\rat} (\Oh_X)) \longrightarrow \Map (W_{\rat} (X) (R) , R) \; .
\end{equation}
Using the canonical maps
\[
\Gamma (X, \Oh_X) \xrightarrow{[\;]} W_{\rat} (\Gamma (X, \Oh_X)) \hookrightarrow \Gamma (X, W_{\rat} (\Oh_X))
\]
we obtain a ring homomorphism
\begin{equation}
\label{eq:9*} W_{\rat} (\Gamma (X , \Oh_X)) \longrightarrow \Map (W_{\rat} (X) (R) , R) 
\end{equation}
and a {\it multiplicative}, non-additive map
\begin{equation}
\label{eq:10*} 
[\;] : \Gamma (X, \Oh_X) \longrightarrow \Map (W_{\rat} (X) (R), R) \; .
\end{equation}
For a field $R = \C$, the map
\begin{equation}
\label{eq:11*} 
[\;] : \Gamma (X , \Oh_X) \longrightarrow \Map (W_{\rat} (X) (\C) ,\C)
\end{equation}
can be described explicitely as follows. A point $f \in W_{\rat} (X) (\C)$ is given by a pair $(x , \tf^{\sharp}_{(0)})$, where $x \in X$ is the image of $(0) \in \spec \C$ and where \linebreak
$\tf^{\sharp}_{(0)} : W_{\rat} (\kappa (x)) \to \C$ is a homomorphism. The composition
\[
\kappa (x) \xrightarrow{[\;]} W_{\rat} (\kappa (x)) \xrightarrow{\tf^{\sharp}_{(0)}} \C
\]
is multiplicative and maps $1$ to $1$ and $0$ to $0$. It follows that $\kappa (x)^{\times}$ is mapped to $\C^{\times}$. Thus, for $a \in \kappa (x)$ we have $\tf^{\sharp}_{(0)} ([a]) = 0$ if and only if $a = 0$. Going through the construction of the map \eqref{eq:11*} we see that a regular function $\alpha \in \Gamma (X , \Oh_X)$ is mapped to the function
\[
[\alpha] : W_{\rat} (X) (\C) \longrightarrow \C \quad \text{defined by} \; [\alpha] (x , \tf^{\sharp}_{(0)}) = \tf^{\sharp}_{(0)} ([\alpha (x)]) \; .
\]
Here $\alpha (x) \in \kappa (x)$ is the image of $\alpha$ under the evaluation map $\Gamma (X, \Oh_X) \to \kappa (x)$. By what we just saw we have
\begin{equation}
\label{eq:12*} 
[\alpha] (x , \tf^{\sharp}_{(0)}) = 0 \iff \alpha (x) = 0 \; .
\end{equation}
Consider the natural map
\[
\pi : W_{\rat} (X) (\C) \longrightarrow X \; , \; (x , \tf^{\sharp}_{(0)}) \longmapsto x \; .
\]
We have just seen that as sets
\begin{equation}
 \label{eq:13*}
\ddiv [\alpha] = \pi^{-1} (\ddiv \alpha) \; .
\end{equation}
For example for $X = \spec \Z$ consider the multiplicative map \eqref{eq:11*}
\begin{equation}
\label{eq:14*}
[\;] : \Z \longrightarrow \Map (W_{\rat} (\spec \Z) (\C) , \C) \; .
\end{equation}
It exhibits numbers as $\C$-valued functions on $W_{\rat} (\spec \Z) (\C)$. For $\nu \in \Z$ the zero set of the function $[\nu]$ is the union of the inverse images $\pi^{-1} ((p))$'s for the primes $p$ dividing $\nu$. The map $\pi$ is $\Nh$-equivariant if $\Nh$ acts trivially on $X$. Using the existence of Frobenius elements in the absolute Galois group of $\Q$ we will see that the fibres of $\pi$ over the closed points of $\spec \Z$ are the ones with non-trivial stabilizers in $\Nh$. After inverting the $\Nh$-action and suspending to an $\R^{> 0}$-action, they give rise to the periodic orbits in the resulting dynamical system, c.f. Theorem \ref{t6}. 

\begin{remark}
\label{t37}
For schemes $X$ and $S$, the group $\Aut (W_{\rat} (X)) \times \Aut (S)$ operates on $W_{\rat} (X) (S)$ via the formula
\[
(\Sigma , \tau) \hullet f = \Sigma \verk f \verk \tau^{-1} = \big( S \xrightarrow{\tau^{-1}} S \xrightarrow{f} W_{\rat} (X) \xrightarrow{\Sigma} W_{\rat} (X) \big) \; .
\]
Let $\Aut_{\Nh} (W_{\rat} (X))$ denote the subgroup of $\Aut (W_{\rat} (X))$ consisting of automorphisms which commute with the $\Nh$-action. Then $\Aut_{\Nh} (W_{\rat} (X)) \times \Aut (S)$ acts on $W_{\rat} (X) (S)$ by automorphisms which commute with the $\Nh$-action on $W_{\rat} (X) (S)$. In particular, via the natural homomorphism $\Aut (X) \to \Aut_{\Nh} (W_{\rat} (X))$, the group $\Aut (X) \times \Aut (S)$ operates on $W_{\rat} (X) (S)$ by $\Nh$-equivariant automorphisms. Explicitely:
\[
(\sigma , \tau) \hullet f = W_{\rat} (\sigma) \verk f \verk \tau^{-1} = (S \xrightarrow{\tau^{-1}} S \xrightarrow{f} W_{\rat} (X) \xrightarrow{W_{\rat} (\sigma)} W_{\rat} (X)) \; .
\]
\end{remark}

\begin{example}
Let $G$ be the absolute Galois group of $\Q_p$. Then $G \times G$ acts $\Nh$-equivariantly on $W_{\rat} (\spec \oZ_p) (\eo_{\C_p})$. In section \ref{sec:13} we will see that this action is related to the $G \times G$-action on the universal covering of the ``curve'' in Fontaine and Fargues' theory \cite{FF}.
\end{example}
\section{$\C$-valued points of $W_{\rat} (X)$} \label{sec:3n}
Let $\C$ be a field. A $\C$-valued point $f : \spec \C \to W_{\rat} (X)$ of $W_{\rat} (X)$ is given by a point $x \in X$ and a homomorphism
\[
\oP_0 : W_{\rat} (\kappa (x)) \longrightarrow \C \; .
\]
Namely we have $x = f ((0))$ and $\oP_0 = \tf^{\sharp}_{(0)}$ in the notation of Definition \ref{t25n}. The map $f^{\sharp}$ i.e. $f^{\sharp}_{(0)}$ is the composition
\[
P_0 : W_{\rat} (\Oh_{X,x}) \longrightarrow W_{\rat} (\kappa (x)) \xrightarrow{\oP_0} \C \; .
\]
We will need a different description of these points for certain schemes $X$. 

Let $X = \eX_0$ be an integral normal scheme with function field $K_0$ and let $K$ be an algebraic closure of $K_0$. Let $\eX$ be the normalization of $\eX_0$ in $K$. Then the automorphism group $G = \Aut_{K_0} (K)$ of $K$ over $K_0$ acts on $\eX$ from the right. The canonical integral and hence affine morphism $\pi : \eX \to \eX_0$ is surjective. For $\ex_0 \in \eX_0$ let $\ex \in \eX$ be a point over $\ex_0$. Then $\kappa (\ex)$ is an algebraic closure of $\kappa (\ex_0)$. Let $\kappa (\ex_0)^{\sep}$ be the separable closure of $\kappa (\ex_0)$ in $\kappa (\ex)$. Let
\[
G_{\ex} = \{ \sigma \in G \mid \ex^{\sigma} = \ex \}
\]
be the stabilizer group of $\ex$ in $G$. It follows from \cite[Chap. V, \S\;2, n$^o$ 3, Proposition 6]{B} that the group $G$ permutes the points $\ex$ over $\ex_0$ transitively and that the induced homomorphism 
\begin{equation}
\label{eq:22n}
G_{\ex} \longrightarrow \Aut_{\kappa (\ex_0)} (\kappa (\ex)) = \Gal (\kappa (\ex_0)^{\sep} / \kappa (\ex_0))
\end{equation}
is surjective. The following is an immediate consequence of Corollary \ref{t28n}

\begin{cor} \label{t29n}
Let $\C$ be an algebraically closed field. Then in the above situation we have natural bijections
\begin{equation}
\label{eq:6n}
\Hom (W_{\rat} (\kappa (\ex_0)) , \C) \silo \Hom (\kappa (\ex)^{\times} , \C^{\times}) / G_{\ex} \; .
\end{equation}
Moreover, if $\car \kappa (\ex_0) = \car \C$ is positive, we also have a natural bijection
\begin{equation}
\label{eq:7n} 
\Hom (W_{\rat} (\kappa (\ex_0)) , \C) \silo \Hom ((\kappa (\ex_0)^{\sep})^{\times} , \C^{\times}) / G_{\ex} \; .
\end{equation}
\end{cor}

Let $\C$ be an algebraically closed field. We define $\deX (\C)$ to be the set of pairs $(\ex , \oP^{\times})$ where $\ex \in \eX$ and $\oP^{\times} : \kappa (\ex)^{\times} \to \C^{\times}$ is a homomorphism. The group $G$ acts on $\deX (\C)$ from the right by
\[
(\ex , \oP^{\times})^{\sigma} = (\ex^{\sigma} , \oP^{\times} \verk \sigma) \quad \text{for} \; \sigma \in G \; .
\]
The $G$-action commutes with the $\Nh$-action by
\[
F_{\nu} (\ex , \oP^{\times}) = (\ex , \oP^{\times} \verk (\,)^{\nu}) \quad \text{for} \; \nu \in \Nh \; .
\]
Hence the quotient
\[
\deX_0 (\C) = \deX (\C) / G
\]
inherits an $\Nh$-action. The Teichm\"uller map induces isomorphisms
\[
\Z \kappa (\ex)^{\times} \equiv \uZ \kappa (\ex) \silo W_{\rat} (\kappa (\ex)) \quad \text{for} \; \ex \in \eX \; .
\]
Hence we have a canonical $G$- and $\Nh$-equivariant identification
\begin{equation} 
\label{eq:23a}
W_{\rat} (\eX) (\C) = \deX (\C) \; .
\end{equation}

\begin{cor}
\label{t210n}
For any algebraically closed field $\C$ we have a natural $\Nh$-equivariant bijection
\[
\deX_0 (\C) = W_{\rat} (\eX) (\C) / G \silo W_{\rat} (\eX_0) (\C) \; .
\]
It sends $(\ex , \oP^{\times}) G$ to $(\ex_0 , \oP_0)$ where $\ex_0 = \pi (\ex)$ and where $\oP_0 : W_{\rat} (\kappa (\ex_0)) \to \C$ is the homomorphism corresponding to $\oP^{\times} \mod G_{\ex}$ under the bijection \eqref{eq:6n}. For each $\nu \in \Nh$ the Frobenius $F_{\nu}$ is injective on $W_{\rat} (\eX_0) (\C)$.
\end{cor}

\begin{proof}
Using Corollary \ref{t29n} one checks that the map $\deX_0 (\C) \to W_{\rat} (\eX_0) (\C)$ is well defined and surjective. Injectivity follows since $G$ permutes the points $\ex \in \eX$ over any $\ex_0 \in \eX_0$ transitively. The map $F_{\nu}$ is injective on $\deX (\C)$ since $(\;)^{\nu} : \kappa (\ex)^{\times} \to \kappa (\ex)^{\times}$ is surjective for all $\ex \in \eX$. Hence the induced self map $F_{\nu}$ of $\deX_0 (\C) = \deX (\C) / G$ is injective as well. 
\end{proof}

It follows from Corollary \ref{t210n} that descent holds more generally. Let $K_1 \subset K$ be a normal extension of $K_0$ and set $N = \Aut_{K_1} (K)$. Then $N$ is a normal subgroup of $G$ whose quotient $G / N$ is canonically isomorphic to $G_1 = \Aut_{K_0} (K_1)$. Let $\eX_1$ be the normalization of $\eX_0$ in $K_1$. 

\begin{cor} \label{t36}
For any algebraically closed field $\C$ the map
\[
W_{\rat} (\eX_1) (\C) / G_1 \longrightarrow W_{\rat} (\eX_0) (\C)
\]
induced by the morphism $\eX_1 \to \eX_0$ is an $\Nh$-equivariant bijection.
\end{cor}

\begin{proof}
By Corollary \ref{t210n} we have 
\[
W_{\rat} (\eX) (\C) / N = W_{\rat} (\eX_1) (\C) \quad \text{and} \quad W_{\rat} (\eX) (\C) / G = W_{\rat} (\eX_0) (\C) \; .
\]
Hence
\[
W_{\rat} (\eX_1) (\C) / G_1 = (W_{\rat} (\eX) (\C) / N) / G_1 = W_{\rat} (\eX) (\C) / G = W_{\rat} (\eX_0) (\C) \; .
\]
\end{proof}

\begin{remark} \label{t34} 
If $\eX_0 = \spec R_0$ is affine, $\eX = \spec R$, we will identify the points $(\ex , \oP^{\times})$ of $\deX (\C)$ with the multiplicative maps $P : R \to \C$ satisfying the following properties:\\
1) $P (0) = 0 , P (1) = 1$.\\
2) $\ep := P^{-1} (0)$ is additively closed and hence a prime ideal.\\
3) We have a factorization $P : R \to R / \ep \xrightarrow{\oP} \C$.
\end{remark}

Thus $\ep$ is the prime ideal corresponding to $\ex$ and $\oP^{\times}$ is the unique extension of the multiplicative map $\oP : (R / \ep) \setminus \{ 0 \} \to \C^{\times}$ to a homomorphism \\
$\oP^{\times} : \kappa (\ep)^{\times} \to \C^{\times}$. Occasionally we will write $P = (\ep , \oP^{\times})$ and $\ep_P = \ep$. Then we have $P^{\sigma} = (\ep , \oP^{\times})^{\sigma}$ where $P^{\sigma} = P \verk \sigma$.

\begin{prop} \label{t1}
The natural morphism $\pi : \eX \to \eX_0$ induces a bijection $\eX (\C) / G = \eX_0 (\C)$.
\end{prop}

\begin{proof}
If $\eX_0 = \spec R_0$ is affine we have $\eX = \spec R$ where $R$ is the normalization of $R_0$ in $K$. Hence $R_0 = R^G$ and the claim is a special case of \cite[Lemma 4.9]{KS}. The general case follows because by construction $\pi$ is integral and in particular affine: We have
\begin{equation} \label{eq:1}
\eX (\C) = \bigcup_i \pi^{-1} (\eX^i_0 (\C))
\end{equation}
where $\{ \eX^i_0 \}$ is an open affine covering of $\eX_0$. Moreover $\pi^{-1} (\eX^i_0 (\C)) = \eX^i (\C)$ where $\eX^i \to \eX^i_0$ is the normalization of $\eX^i_0$ in $K$. Hence $\pi^{-1} (\eX^i_0 (\C)) / G = \eX^i_0 (\C)$ and therefore the map $\eX (\C) / G \to \eX_0 (\C)$ is bijective.
\end{proof}

A point of $\eX_0 (\C)$ consists of a scheme theoretic point $\ex_0$ of $\eX_0$ together with an embedding of its residue field $\kappa (\ex_0)$ into $\C$. Such field embeddings exist only if $\car \kappa (\ex_0) = \car \C$ and hence $\eX_0 (\C)$ does not ``see'' the other points $\ex_0$ of $\eX_0$. 

We have
\[
\eX (\C) \subset \deX (\C) \quad \text{and} \quad \eX_0 (\C) \subset \deX_0 (\C) = W_{\rat} (\eX_0) (\C) \; .
\]
More generally, for any field $k$ with a multiplicative embedding $i : k^{\times} \hookrightarrow \C^{\times}$ we get a $G$-equivariant inclusion $i_* : \eX (k) \hookrightarrow \deX (\C)$. If $a \in \eX (k)$ corresponds to a point $\ex \in \eX$ and an inclusion $\kappa (\ex) \overset{\alpha}{\hookrightarrow} k$, we set $i_* (a) = (\ex , \oP^{\times})$ where $\oP^{\times}$ is the composition 
\[
\oP^{\times} : \kappa (\ex)^{\times} \overset{\alpha}{\hookrightarrow} k^{\times} \overset{i}{\hookrightarrow} \C^{\times} \; .
\]
If $k$ is algebraically closed, we have $\eX (k) / G =\eX_0 (k)$ by Proposition \ref{t1}, and we obtain an inclusion $i_* : \eX_0 (k) \hookrightarrow \deX_0 (\C)$.

We will now discuss $\cW_{\rat} (\eX_0) (\C)$ and variants. Let $\Nh_0$ be the submonoid of $\Nh$ generated by a set of prime numbers $\car \Nh_0$. Let $\Q^{> 0}_0$ be the subgroup of $\Q^{> 0}$ generated by $\Nh_0$. We always assume that $\car \Nh_0 \supset \car \eX_0$, the set of positive residue characteristics of the points of $\eX_0$. The main cases of interest are $\car \Nh_0 = \car \eX_0$ and $\Nh_0 = \Nh$. We order $\Nh_0$ by divisibility and view it as a directed poset. For $\nu \in \Nh_0$, the injective self maps $F_{\nu}$ of $\deX (\C)$ and $\deX_0 (\C)$ induce bijections $F_{\nu}$ of
\[
\ceX (\C) = \colim_{\Nh_0} \deX (\C) \quad \text{and} \quad \ceX_0 (\C) = \colim_{\Nh_0} \deX_0 (\C) \; .
\]
In this way the $\Nh_0$-actions on $\deX (\C)$ and $\deX_0 (\C)$ extend to $\Q^{> 0}_0$-actions on $\ceX (\C)$ and $\ceX_0 (\C)$. The $G$-action on $\deX (\C)$ extends canonically to a $G$-action on $\ceX (\C)$ commuting with the $\Q^{> 0}_0$-action and we have $\ceX_0 (\C) = \ceX (\C) / G$. The $\Nh_0$-equivariant projection 
\[
\dpi : \deX (\C) \to \deX_0 (\C) = \deX (\C) / G
\]
extends to a $\Q^{> 0}_0$-equivariant projection
\[
\cpi : \ceX (\C) \longrightarrow \ceX_0 (\C) = \ceX (\C) / G \; .
\]
Note that $\deX (\C) \subset \ceX (\C)$ and $\deX_0 (\C) \subset \ceX_0 (\C)$ canonically since the $F_{\nu}$'s are injective. We can write the points of $\ceX (\C)$ in the form $F^{-1}_{\nu} P$ for some $\nu \in \Nh_0$ and $P$ in $\deX (\C)$. Then $F^{-1}_{\nu} P = F^{-1}_{\nu'} P'$ is equivalent to $F_{\nu'} P = F_{\nu} P'$, an equality in $\deX (\C)$. A similar remark applies to the points $F^{-1}_{\nu} P_0$ of $\ceX_0 (\C)$ where $P_0 \in \deX_0 (\C)$. The canonical map 
\[
\pr_{\eX} : \deX (\C) \to \eX , (\ex , \oP^{\times}) \mapsto \ex
\]
extends naturally to a map 
\[
\pr_{\eX} : \ceX (\C) \to \eX \quad \text{by setting} \; \pr_{\eX} (F^{-1}_{\nu} P) = \pr_{\eX} (P)\; .
\]
We may view a point $\cP$ of $\ceX (\C)$ as a pair $(\ex , \cP^{\times})$ where $\ex = \pr_{\eX} (\cP)$ and where
\[
\cP^{\times} : \varprojlim_{\Nh_0} \kappa (\ex)^{\times} \longrightarrow \C^{\times}
\]
is a homomorphism. Explicitely, if $\cP = F^{-1}_{\nu} P$ as above we have
\[
\cP^{\times} ((x_{\mu})_{\mu \in \Nh_0}) = \oP^{\times} (x_{\nu}) \; .
\]
For $\nu \in \Nh_0$, let
\begin{equation}
\label{eq:23rn}
\pr_{\nu} : \varprojlim_{\Nh_0} \kappa (\ex)^{\times} \longrightarrow \kappa (\ex)^{\times}
\end{equation}
denote the projection onto the $\nu$-th component. With this notation we get the following alternative description of $\ceX (\C)$
\begin{multline}
\label{eq:24rn}
\ceX (\C) \; \text{ is the set of pairs $(\ex , \cP^{\times})$ where $\ex \in \eX$ and} \; 
\cP^{\times} : \varprojlim_{\Nh_0} \kappa (\ex)^{\times} \to \C^{\times} \; \text{is a} \\
\shoveleft{\text{homomorphism which factors over $\pr_{\nu}$ for some $\nu \in \Nh_0$.}} \hspace*{\fill}
\end{multline}
It is clear how $G$ and $\Q^{> 0}_0$ act on $\ceX (\C)$ in this description. We will later deal with topological fields $\C$ and also with topologies on $\kappa (\ex)^{\times}$. In section \ref{sec:13} we will see that in a certain situation with $\C = \C_p$ it is important to consider all characters on $\varprojlim_{\Nh_0} \kappa (\ex)^{\times}$ which are continuous for the projective limit topology on the projective limit of the topological groups $\kappa (\ex)^{\times}$. 

\begin{defn}
\label{t36n}
Given topologies on the groups $\kappa (\ex)^{\times}$ for $\ex \in \eX$ and an algebraically closed topological field $\C$ we set
\[
\hceX (\C) = \{ (\ex, \hcP) \mid \ex \in \eX \; \text{and} \; \hcP : \varprojlim_{\Nh_0} \kappa (\ex)^{\times} \to \C^{\times} \; \text{is a continuous character}  \; .\}
\]
If the isomorphisms $\sigma : \kappa (\ex^{\sigma}) \to \kappa (\ex)$ are continuous for all $\sigma \in G$ then $G$ acts naturally on $\hceX (\C)$ and we set
\[
\hceX_0 (\C) := \hceX (\C) / G \; .
\]
\end{defn}

The following proposition shows that $\hceX (\C)$ is actually contained in $\ceX (\C)$ if the group $\C^{\times}$ has no ``small subgroups'', i.e. if $1 \in \C^{\times}$ has an open neighborhood $U$ such that $\{ 1 \}$ is the only subgroup of $\C^{\times}$ contained in $U$.

\begin{prop}
\label{t37n}
In the above situation equip $\kappa (\ex)^{\times}$ with any topology for which it is a Hausdorff topological group and endow $\varprojlim_{\Nh_0} \kappa (\ex)^{\times}$ with the projective limit topology. Let $\C$ be an algebraically closed field such that $\C^{\times}$ has no small subgroups, e.g. the field of complex numbers. Then every continuous character
\[
\chi : \varprojlim_{\Nh_0} \kappa (\ex)^{\times} \longrightarrow \C^{\times}
\]
has the form $\chi = \psi \verk \pr_{\nu}$ for some $\nu \in \Nh_0$ and for some continuous character $\psi : \kappa (\ex)^{\times} \to \C^{\times}$ and $\pr_{\nu}$ as in \eqref{eq:23rn}.
\end{prop}

\begin{proof}
Let $S$ be the set of prime numbers $p \neq \car \kappa (\ex)$ dividing an element of $\Nh_0$. Since $\kappa (\ex)^{\times}$ is divisible, we have an exact sequence
\[
1 \longrightarrow T \longrightarrow \varprojlim_{\Nh_0} \kappa (\ex)^{\times} \xrightarrow{\pr_1} \kappa (\ex)^{\times} \longrightarrow 1 \; .
\]
Since $\kappa (\ex)^{\times}$ is Hausdorff, the points of $\kappa (\ex)^{\times}$ are closed and hence $\mu_{\nu} (\kappa (\ex)) \subset \kappa (\ex)^{\times}$ carries the discrete topology as a subspace. Therefore as topological groups, we have
\[
T = \prod_{p \in S} T_p \mu (\kappa (\ex)) \cong \prod_{p \in S} \Z_p \; .
\]
In particular, $T$ has a neighborhood basis of $1$ consisting of open subgroups. Let $\alpha : T \to \C^{\times}$ be a continuous character and let $U \subset \C^{\times}$ be an open neighborhood of $1$ in $\C^{\times}$ that does not contain a non-trivial subgroup. By continuity, $\alpha$ maps an open subgroup of $T$ into $U$ and hence to $1$. Hence $\alpha$ factors over a finite quotient of $T$. Thus there is some $\nu \in \Nh_0$ such that $\chi \verk (\;)^{\nu} |_T$ is trivial. Hence $F_{\nu} (\chi) = \chi \verk (\;)^{\nu}$ factors over $\pr_1$ and therefore $F_{\nu} (\chi) = \psi \verk \pr_1$ for some (automatically continuous) character $\psi : \kappa (\ex)^{\times} \to \C^{\times}$. This implies that
\[
\chi = F^{-1}_{\nu} (F_{\nu} (\chi)) = F^{-1}_{\nu} (\psi \verk \pr_1) = \psi \verk \pr_{\nu} \; .
\]
\end{proof}

\begin{cor}
\label{t38n}
In the situation of \eqref{eq:24rn} assume that $\C^{\times}$ has no small subgroups. If the $\kappa (\ex)^{\times}$ carry the discrete topology then
\[
\ceX (\C) = \{ (\ex , \cP^{\times}) \mid \ex \in \eX \; \text{and} \; \cP^{\times} : \varprojlim_{\Nh_0} \kappa (\ex)^{\times} \to \C^{\times} \; \text{is a continuous character} \} \; .
\]
\end{cor}

\section{Classes of characters and subsystems} \label{sec:4ggn}
Let the notations be as in the last section with $\C$ an algebraically closed field.  Consider the situation of Corollary \ref{t210n} where
\[
W_{\rat} (\eX_0) (\C) = \deX (\C) / G \; .
\]
The study of periodic orbits and connectedness in the next sections show that $W_{\rat} (\eX_0) (\C)$ and $\cW_{\rat} (\eX_0) (\C) = \colim_{\Nh_0} W_{\rat} (\eX_0) (\C)$ are too large for our purposes. We have to pass to suitable $\Nh_0$- resp. $\Q^{> 0}_0$-subsystems. Recall that $\deX (\C)$ consists of pairs $(\ex , \oP^{\times})$ where $\ex \in \eX$ and $\oP^{\times} : \kappa (\ex)^{\times} \to \C^{\times}$ is a character. We will now define some classes $\Eh$ of characters obtained by imposing suitable conditions. The condition that the extension by zero $\oP : \kappa (x) \to \C$ is additive leads to the usual points $\eX (\C) \subset \deX (\C)$. In this case $\oP$ and hence $\oP^{\times}$ are injective. Asking all characters $\oP^{\times}$ for $(\ex , \oP^{\times})$ in $\deX (\C)$ to be injective loses the Frobenius actions $F_{\nu}$ for all $\nu \in \Nh_0$ which are not a power of $p = \car \kappa (\ex)$ since $\oP^{\times} \verk (\;)^{\nu}$ is not injective. The best condition on the characters $\oP^{\times}$ is not clear to me. However the following two minimal conditions play an important role as we will see later. Let $\chi : \kappa^{\times} \to \C^{\times}$ be a character of an algebraically closed field $\kappa$ into another algebraically closed field $\C$. \\[0.2cm]
({\it Tors}) the group $\ker (\chi)_{\tors} = \ker (\chi \, |_{\mu (\kappa)})$ is finite and $|(\ker \chi)_{\tors}| \in \Nh_0$.\\[0.2cm]
({\it Image}) Only if $\car \kappa > 0$. If $\chi (\kappa^{\times})$ is torsion, then $\kappa^{\times}$ is torsion as well, i.e. $\kappa^{\times} \otimes \Q \neq 0$ implies $\chi (\kappa^{\times}) \otimes \Q \neq 0$. 

Conditions ({\it Tors}) and ({\it Image)} are weakened versions of injectivity of $\chi$. 

\begin{defn}
\label{t41n}
A class $\Eh$ of characters $\chi : \kappa^{\times} \to \C^{\times}$ on algebraically closed fields $\kappa$ is $(\Nh_0-)$admissible if for any $\sigma \in \Aut \kappa$ resp. $\nu \in \Nh_0$ the character $\chi$ is in $\Eh$ if and only if $\chi \verk \sigma$ resp. $\chi^{\nu} = \chi \verk (\;)^{\nu}$ is in $\Eh$. Moreover the characters in $\Eh$ should satisfy ({\it Tors}). 
\end{defn}

The following facts are immediate:

\begin{prop}
\label{t42n}
Given an admissible class of characters $\Eh$ on the $\kappa (\ex)^{\times}$ for $\ex \in \eX$, the set
\[
\deX (\C)_{\Eh} = \{ (\ex , \oP^{\times}) \in \deX (\C) \mid \oP^{\times} \; \text{is in} \; \Eh \} \subset \deX (\C)
\]
is $G$-invariant. It is foreward- and backward invariant under the $\Nh_0$-action, i.e. for $P \in \deX (\C)$ we have $P \in \deX (\C)_{\Eh}$ if and only if $F_{\nu} (P) \in \deX (\C)_{\Eh}$. The set
\[
\ceX (\C)_{\Eh} = \colim_{\Nh_0} \deX (\C)_{\Eh} \subset \ceX (\C)
\]
is $G$- and $\Q^{> 0}_0$-invariant. The quotients
\[
\deX_0 (\C)_{\Eh} = \deX (\C)_{\Eh} / G \; \text{resp.} \; \ceX_0 (\C)_{\Eh} = \ceX (\C)_{\Eh} / G = \colim_{\Nh_0} \deX_0 (\C)_{\Eh} 
\]
are foreward- and backward $\Nh_0$- resp. $\Q^{> 0}_0$-invariant. The monoid $\Nh_0$ acts by injections on $\deX (\C)_{\Eh}$ and $\deX_0 (\C)_{\Eh}$.
\end{prop}

We will need the map
\begin{equation}
\label{eq:4_29}
\pr_{\eX} : \deX (\C)_{\Eh} \longrightarrow \eX \; , \; (\ex , \oP^{\times}) \longmapsto \ex 
\end{equation}
and the induced map
\begin{equation}
\label{eq:4_30}
\pr_{\eX_0} : \deX_0 (\C)_{\Eh} \longrightarrow \eX_0 \; .
\end{equation}
For general admissible classes $\Eh$, these maps may not be surjective, but see Corollary \ref{t4ka}. 
Both $\pr_{\eX}$ and $\pr_{\eX_0}$ are $\Nh_0$-equivariant if we let $\Nh_0$ act trivially on $\eX_0$ and $\eX$. Note that the maps $\pr_{\eX}$ and $\pr_{\eX_0}$ above extend $\Q^{>0}_0$-equivariantly to maps
\[
\pr_{\eX} : \ceX (\C)_{\Eh} \longrightarrow \eX \quad \text{and} \quad \pr_{\eX_0} : \ceX_0 (\C)_{\Eh} \longrightarrow \eX_0 \; .
\]
Here we let $\Q^{> 0}_0$ act trivially on $\eX$ and $\eX_0$. 

\begin{lemma} \label{t3}
Let $k$ be a field with $\card k \le \card \C$ and $\C^{\times} \neq \mu (\C)$. Let $\chi_{\mu} : \mu (k) \to \C^{\times}$ be a homomorphism. Then there is a homomorphism $\chi : k^{\times} \to \C^{\times}$ with $\chi |_{\mu (k)} = \chi_{\mu}$ and $\ker \chi = \ker \chi_{\mu}$. In particular there is an injective homomorphism $k^{\times} \hookrightarrow \C^{\times}$ if $\car \C = 0$ or $\car k = \car \C$ is positive.
\end{lemma}

\begin{proof}
We may assume that $k$ is algebraically closed. The sequences
\[
1 \to \mu (k) \to k^{\times} \to k^{\times} / \mu (k) \to 1 \quad \text{and} \quad 1 \to \mu (\C) \to \C^{\times} \to \C^{\times} / \mu (\C) \to 1
\]
are both split since $\mu (k)$ and $\mu (\C)$ are divisible. Choose splittings and consider the induced isomorphisms
\[
k^{\times} \cong \mu (k) \times (k^{\times} / \mu (k)) \quad \text{and} \quad \C^{\times} \cong \mu (\C) \times (\C^{\times} / \mu (\C)) \; .
\]
Here $k^{\times} / \mu (k)$ and $\C^{\times} / \mu (\C)$ are uniquely divisible abelian groups and therefore $\Q$-vector spaces. Since $\C^{\times} \neq \mu (\C)$ we have
\[
\dim_{\Q} k^{\times} / \mu (k) \le \card k \le \card \C = \dim_{\Q} \C^{\times} / \mu (\C) \; .
\]
Hence there is a $\Q$-linear injection from $k^{\times} / \mu (k)$ into $\C^{\times} / \mu (\C)$. Together with $\chi_{\mu}$ and the above decompositions we obtain a homomorphism $\chi : k^{\times} \to \C^{\times}$ prolonging $\chi_{\mu}$ with $\ker \chi = \ker \chi_{\mu}$.
\end{proof}

In the next corollary we assume the following conditions on $K_0$ and the algebraically closed field $\C$:\\
1) $\C^{\times} \neq \mu (\C)$ i.e. $\C$ is not the algebraic closure of a finite field\\
2) $\card K_0 \le \card \C$\\
3) $\car \C = 0$ or $\car K_0 = \car \C$ is positive.\\
We are thinking of the complex number field or of $\C_p = \hoQ_p$ or of $\C_{\infty} = \widehat{\overline{\F_p ((t))}}$. 

\begin{cor} \label{t4ka}
Let $\Eh$ be an admissible class as in Definition \ref{t41n} containing all injective characters. Then the maps $\pr_{\eX} : \deX (\C)_{\Eh} \to \eX$ and $\pr_{\eX_0} : \deX_0 (\C)_{\Eh} \to \eX_0$ are surjective.
\end{cor}

\begin{proof}
For any point $\ex \in \eX$ we have $\card \kappa (\ex) \le \card K_0 \le \card \C$. Using Lemma \ref{t3} for $k = \kappa (\ex)$ it follows that the map
\[
\pr_{\eX} : \deX (\C)_{\Eh} \longrightarrow \eX \; , \; (\ex , \oP^{\times}) \longmapsto \ex
\]
is surjective. Hence the induced map of orbit spaces $\deX_0 (\C)_{\Eh} \to \eX / G$ is surjective as well. The going-up theorem implies that the normalization map $\eX \to \eX_0$ and hence the induced map $\eX / G \to \eX_0$ are surjective as well. By composition we obtain the surjection $\pr_{\eX_0} : \deX_0 (\C)_{\Eh} \longrightarrow \eX_0$.
\end{proof}

Here are examples of admissible classes $\Eh$ of characters.

\begin{example}
\label{t43n}
1) $\Eh_{\tors}$: ({\it Tors}) holds\\
2) $\Eh_{\max}$: ({\it Tors}) and ({\it Image}) hold\\
3) $\Eh_f$: ({\it Tors}) and $\ker \chi$ is finite. Equivalently: $|\ker \chi| \in \Nh_0$\\
4) $\Eh_{fg}$: ({\it Tors}) and $\ker \chi$ is finitely generated\\
5) $\Eh_{fd}$: ({\it Tors}) and $\ker \chi \otimes \Q$ is finite dimensional\\
6) $\Eh_{fd_0}$: ({\it Tors}) and $(\ker \chi \, |_{\kappa (\ex_0)^{\times}}) \otimes \Q$ is finite dimensional where $\ex_0 = \pi (\ex)$ under the projection $\pi : \eX \to \eX_0$. 
\end{example}

We have inclusions in the appropriate sense
\[
\Eh_f \subset \Eh_{fg} \subset \Eh_{fd} \subset \Eh_{fd_0} \subset \Eh_{\max} \subset \Eh_{\tors} \; .
\]
For $\Eh_{fd_0} \subset \Eh_{\max}$, note that if $\chi (\kappa (\ex)^{\times})$ and hence $\chi (\kappa (\ex_0)^{\times})$ are torsion, then $(\ker \chi \, |_{\kappa (\ex_0)^{\times}}) \otimes \Q = \kappa (\ex_0)^{\times} \otimes \Q$. For $\chi$ in $\Eh_{fd_0}$ it follows that $\kappa (\ex_0)^{\times} \otimes \Q$ is finite dimensional. The structure theory of fields implies that $\kappa (\ex_0) \subset \oF_p$ for some $p$, hence $\kappa (\ex) \cong \oF_p$ and therefore $\kappa (\ex)^{\times}$ is torsion.

\begin{rem}
The topology of the $\R$-dynamical systems that we will build from $\ceX (\C)_{\Eh}$ depends very much on the choice of $\Eh$. For $\Eh = \Eh_f$ the path-connected components are not simply-connected whereas for $\Eh = \Eh_{fd}$ they are contractible. Incidentally, in the $p$-adic case where we will deal with multiplicative maps $\oP$ into a $p$-adic valuation ring and $\Nh_0 = p^{\Z}$, the right condition $\Eh$ is the following: $\oP$ is additive $\mod p$. This can be rephrased in terms of absolute values and translated to the case where $\C$ is the complex number field. However the resulting class $\Eh$ is not $\Nh$-invariant. 
\end{rem}

We call a class $\Eh$ of characters $\chi : \kappa^{\times} \to \C^{\times}$ on algebraically closed fields $\kappa$ {\it stable} if $\chi \in \Eh$ implies that $\chi \, |_{\tkappa^{\times}} \in \Eh$ for all algebraically closed subfields $\tkappa \subset \kappa$. All classes in the example are {\it stable} except for ({\it Image}) and hence $\Eh_{\max}$. Any class $\Eh$ can be extended to a stable class. For example ({\it Image}) becomes:

({\it Stable Image}) Only if $\car \kappa > 0$. If $\chi (\tkappa^{\times})$ is torsion for some algebraically closed subfield $\tkappa \subset \kappa$ then $\tkappa^{\times} \otimes \Q = 0$. 

It is also clear that the stable Galois invariant classes $\Eh$ are the {\it functorial} ones i.e. where $(\chi : \kappa^{\times} \to \C^{\times}) \in \Eh$ implies $(\chi \verk \tau : \tkappa^{\times} \to \C^{\times}) \in \Eh$ for any homomorphism $\tau : \tkappa \to \kappa$ of algebraically closed fields. 

For a morphism of integral normal schemes $f_0 : \eX_0 \to \eX'_0$ let
\[
\df_0 := W_{\rat} (f_0) (\C) : \deX_0 (\C) = W_{\rat} (\eX_0) (\C) \longrightarrow \deX'_0 (\C) = W_{\rat} (\eX'_0) (\C)
\]
be the induced $\Nh_0$-equivariant map.

\begin{prop}
\label{t44n}
Let $\C$ be algebraically closed and let $\Eh$ be a functorial admissible class of characters. If $f_0 : \eX_0 \to \eX'_0$ is a dominant morphism of integral normal schemes, then $\df_0$ maps $\deX_0 (\C)_{\Eh}$ into $\deX'_0 (\C)_{\Eh}$. 
\end{prop}

Thus the map $\eX_0 \mapsto \deX_0 (\C)_{\Eh}$ is functorial on the category of integral normal schemes and dominant morphisms. This restricted functoriality is in line with previous considerations, \cite[4.7]{D1}.  

\begin{proof}
For all points $\ex_0$ of $\eX_0$ we have an inclusion of residue fields
\[
f^{\sharp}_{\ex_0} : \kappa (f (\ex_0)) \hookrightarrow \kappa (\ex_0) \; .
\]
For the generic point $\eta_0$ of $\eX_0$, since the image $f (\eta_0) = \eta'_0$ is the generic point of $\eX'_0$ we get an inclusion of function fields
\[
f^{\sharp}_{\eta_0} : K'_0 \hookrightarrow K_0 \; .
\]
Let $K' = \oK'_0$ and $K = \oK_0$ be algebraic closures and fix an embedding $\tau : K' \hookrightarrow K$ prolonging $f^{\sharp}_{\eta_0}$. We obtain a commutative diagram
\[
\xymatrix{
\eX \ar[r]^{f_{\tau}} \ar[d] & \eX' \ar[d] \\
\eX_0 \ar[r]^{f_0} & \eX'_0
}
\]
where $f_{\tau}$ is the induced morphism between the normalizations $\eX$ and $\eX'$ of $\eX_0$ in $K$ resp. of $\eX'_0$ in $K'$. For any point $\ex \in \eX$, the morphism $f_{\tau}$ gives an inclusion
\[
f^{\sharp}_{\tau, \ex} : \kappa (f_{\tau} (\ex)) \hookrightarrow \kappa (\ex) \; .
\]
Given a multiplicative map $\oP^{\times} : \kappa (\ex)^{\times} \to \C^{\times}$ satisfying $\Eh$, pullback by $f^{\sharp}_{\tau , \ex}$ gives a multiplicative map
\[
f_{\tau} (\oP^{\times}) := \oP^{\times} \verk f^{\sharp}_{\tau , \ex} : \kappa (f_{\tau} (\ex)) \longrightarrow \C^{\times} 
\]
which is of class $\Eh$ as well. Thus we get a $G$- and $\Nh_0$-equivariant map
\[
\df_{\tau} : \deX (\C)_{\Eh} \longrightarrow \deX\strich (\C)_{\Eh} \; , \; (\ex , \oP^{\times}) \longmapsto (f_{\tau} (\ex) , f_{\tau} (\oP^{\times})) \; .
\]
On orbit spaces it induces (the restriction of) our map $\df_0$ above:
\[
\df_0 : \deX_0 (\C)_{\Eh} \longrightarrow \deX\strich_0 (\C)_{\Eh} \; .
\]
\end{proof}

The following lemma will be needed later.

\begin{lemma}
\label{t46n}
Let $R$ be an integral domain with quotient field $K$. Let $N \ge 1$ be an integer for which every element of $R$ is an $N$-th power and such that $\mu_N (R) = \mu_N (K)$. Let $\eo$ be another ring and consider a multiplicative map
\[
P : R \longrightarrow \eo \quad \text{with} \; P (1) = 1 \; \text{and} \; P (0) = 0 \; .
\]
1) There is a multiplicative map $Q : R \to \eo$ with $Q (1) = 1 , Q (0) = 0$ such that $P = Q^N$ if and only if $P \, |_{\mu_N (R)} = 1$. The map $Q$ is uniquely determined if it exists.\\
2) If $R = K$ is an algebraically closed field and $P^{-1} (1)$ is finite then $Q$ as above with $Q^N = P$ exists if and only if $N'$ divides $|P^{-1} (1)|$. Here $N'$ is the prime to $\car K$-part of $N$.
\end{lemma}

\begin{proof}
1) Uniqueness: For $r \in R$ we have $Q (r^N) = Q (r)^N = P (r)$. Since every element of $R$ is an $N$-th power, the map $Q$ is uniquely determined by $P$. The condition $P \, |_{\mu_N (R)} = 1$ is necessary because for $r \in \mu_N (R)$ we have
\[
P (r) = Q (r)^N = Q (r^N) = Q (1) = 1 \; .
\]
Assume that $P \, |_{\mu_N (R)} = 1$. Set $Q (0) = 0$. For $r \in R \setminus \{ 0 \}$ choose some $s \in R \setminus \{ 0 \}$ with $s^N = r$ and set $Q (r) = P (s)$. If $s'^N = r$ as well, we have $s'^N = s^N$ and hence $s' s^{-1} \in \mu_N (K) = \mu_N (R)$. Thus $s' = \zeta s$ for some $\zeta \in \mu_N (R)$. Multiplicativity of $P$ implies that $P (s') = P (\zeta) P (s) = P (s)$. Hence $Q$ is well defined and we have $Q^N = P$. The map $Q$ is multiplicative and sends $1$ to $1$ and $0$ to $0$ by definition.\\
2) The subgroup $P^{-1} (1)$ of $K^{\times}$ is finite by assumption and hence cyclic. The group $\mu_N (K) = \mu_{N'} (K)$ has $N'$ elements and hence $\mu_N (K) \subset P^{-1} (1)$ is equivalent to $N'$ dividing $|P^{-1} (1)|$. 
\end{proof}

\section{The structure of the $\Q^{> 0}_0$-orbits} \label{sec1}

Let $\eX_0$ be an integral normal scheme with function field $K_0$. Note that every open subscheme $\eX'_0$ of $\eX_0$ is again an integral normal scheme with the same function field $K_0$. Let $K$ be an algebraic closure of $K_0$ and set $G = \Aut_{K_0} (K)$. Let $\eX$ be the normalization of $\eX_0$ in $K$. The group $G$ acts on $\eX$ over $\eX_0$. Let $\C$ be an algebraically closed field which satisfies the conditions before Corollary \ref{t4ka}.

Fix an injective homomorphism $\iota : \mu (K) \hookrightarrow \mu (\C)$. It exists by our assumptions on $\car K_0$ and $\car \C$. 

The fibres of $\pr_{\eX_0} : \ceX_0 (\C)_{\Eh_{\tors}} \to \eX_0$ are $\Q^{> 0}_0$-invariant. We will now analyze the structures of the $\Q^{> 0}_0$-sets $C_{\ex_0} = \pr^{-1}_{\eX_0} (\ex_0)$ in $\ceX_0 (\C)_{\Eh_{\tors}}$ for points $\ex_0$ of $\eX_0$ whose residue field $\kappa (\ex_0)$ is finite. For every point $\ex \in \eX$ over $\ex_0$ the residue field $\kappa (\ex)$ is an algebraic closure of $\kappa (\ex_0)$. The reduction map $\Oh_{\eX, \ex} \to \kappa (\ex)$ induces an isomorphism
\begin{equation}
\label{eq:3}
i_{\ex} : \mu_{(p)} (K) = \mu_{(p)} (\Oh_{\eX, \ex}) \silo \kappa (\ex)^{\times} \; .
\end{equation}
Here $p = \car \kappa (\ex_0)$ and $\mu_{(p)} (S)$ denotes the group of prime-to-$p$ order roots of unity in a ring $S$. Consider the commutative diagram where $\pr = \pr_{\eX}$ and $\pr_0 = \pr_{\eX_0}$
\begin{equation}
\label{eq:4}
\xymatrix{ 
\eX \ar[d]_{\pi} & \deX (\C)_{\Eh_{\tors}} \ar@{->>}[l]_{\pr} \ar@{^{(}->}[r] \ar[d]_{\dpi} & \ceX (\C)_{\Eh_{\tors}} \ar[d]^{\cpi} \\
\eX_0 & \deX_0 (\C)_{\Eh_{\tors}} \ar@{->>}[l]_{\pr_0} \ar@{^{(}->}[r] & \ceX_0 (\C)_{\Eh_{\tors}} \; .
}
\end{equation}
The fibre $\pr^{-1}_0 (\ex_0)$ in $\deX_0 (\C)_{\Eh_{\tors}}$ is $\Nh_0$-invariant. Its extension to a $\Q^{> 0}_0$-invariant subset of $\ceX_0 (\C)_{\Eh_{\tors}}$ is the set
\[
C_{\ex_0} = \pr^{-1}_0 (\ex_0) \Q^{> 0}_0 = \bigcup_{\nu \in \Nh_0} F^{-1}_{\nu} \pr^{-1}_0 (\ex_0) \subset \ceX_0 (\C)_{\Eh_{\tors}} \; .
\]
The fibre $\pr^{-1}_0 (\ex_0)$ consists of the $G$-orbits of all pairs $(\ex , \oP^{\times})$ where $\ex$ is a point of $\eX$ over $\ex_0$ and $\oP^{\times} : \kappa (\ex)^{\times} \to \C^{\times}$ satisfies ({\it Tors}). Since $\kappa (\ex)^{\times}$ is torsion this means that $\ker \oP^{\times} = (\ker \oP^{\times})_{\tors}$ is finite hence cyclic and $|\ker \oP^{\times}| \in \Nh_0$. Given $\ex$ and composing the fixed injection $\iota : \mu (K) \hookrightarrow \mu (\C)$ above with \eqref{eq:3} we obtain the injective character $\chi_{\ex} = \iota \verk i^{-1}_{\ex} : \kappa (\ex)^{\times} \hookrightarrow \C^{\times}$ and hence a point $(\ex , \chi_{\ex})$ of $\pr^{-1} (\ex)$. Set $P_{0\ex} := \dpi (\ex , \chi_{\ex})$, an element of $\pr^{-1}_0 (\ex_0)$. The structure of $\pr^{-1}_0 (\ex_0)$ and of $C_{\ex_0}$ as $\Nh_0$-resp. $\Q^{> 0}_0$-sets is determined as follows. Fix a point $\ex$ over $\ex_0$ and let $G_{\ex}$ be the stabilizer subgroup of $\ex$ in $G$. It surjects onto $\Gal (\kappa (\ex) / \kappa (\ex_0))$, c.f. \cite[Chap.V, \S\,2, n$^{\rm o}$ 3, Proposition 6]{B}. The group of automorphisms of the abelian group $\kappa (\ex)^{\times}$ is given by $\hZ^{\times}_{(p)}$ where $\hZ_{(p)} = \prod_{l \neq p} \Z_l$. We have a natural inclusion
\begin{equation}
\label{eq:5}
N \ex^{\hZ}_0 = \Gal (\kappa (\ex) / \kappa (\ex_0)) \hookrightarrow \Aut (\kappa (\ex)^{\times}) = \hZ^{\times}_{(p)} \; .
\end{equation}
Here on the left, $N \ex_0$ corresponds to the Frobenius automorphism $y \mapsto y^{N \ex_0}$ in the Galois group. The monoid $\hZ^{\times}_{(p)} \times \Nh_0$ acts by pre-composition on the set $\Sh$ of homomorphisms $\oP^{\times} : \kappa (\ex)^{\times} \to \C^{\times}$ with finite cyclic kernel of order in $\Nh_0$. We have a $\hZ^{\times}_{(p)} \times \Nh_0$-equivariant surjection:
\begin{equation}
\label{eq:6}
\hZ^{\times}_{(p)} \times \Nh_0 \twoheadrightarrow \Sh \; , \; (a,\nu ) \mapsto \chi_{\ex} \cdot (a, \nu) := \chi_{\ex} \verk (\;)^a \verk (\;)^{\nu} \; .
\end{equation}
Two elements $(a,\nu)$ and $(a' , \nu')$ are in the same fibre of this map if and only if $\nu' = \nu p^n$ and $a = p^n a'$ for some $n \in \Z$. Any point $\ey$ in $\eX$ over $\ex_0$ is conjugate to our chosen point $\ex$ by an element of $G$. This follows from \eqref{eq:1} and \cite[Theorem 5 vi)]{Ma}. Having fixed $\ex$ and $\iota$ we have canonical $\Nh_0$-equivariant bijections
\begin{align*}
\pr^{-1}_0 (\ex_0) & = (\dpi^{-1} (\pr^{-1}_0 (\ex_0))) / G \\
 & = (\pr^{-1} (\pi^{-1} (\ex_0))) / G \\
 & = \{ (\ey , \oP^{\times}) \in \deX (\C)_{\Eh_{\tors}} \mid \pi (\ey) = \ex_0 \} / G \\
 & \cong \{ (\ex , \oP^{\times}) \in \deX (\C)_{\Eh_{\tors}} \} / G_{\ex} \\
 & \cong \Sh / \Gal (\kappa (\ex) / \kappa (\ex_0)) \; .
\end{align*}
Explicitely this gives the $\Nh_0$-equivariant bijection:
\begin{equation}
\label{eq:7}
\Sh / N \ex^{\hZ}_0 \silo \pr^{-1}_0 (\ex_0) \; , \; \oP^{\times} \mod N \ex^{\hZ}_0 \mapsto \dpi ((\ex , \oP^{\times})) \; .
\end{equation}
Composing with \eqref{eq:6} we get the $\Nh_0$-equivariant surjection:
\begin{equation}
\label{eq:8}
(\hZ^{\times}_{(p)} / N \ex^{\hZ}_0) \times \Nh_0 \twoheadrightarrow \pr^{-1}_0 (\ex_0) , (\oa , \nu) \mapsto \dpi ((\ex , i \verk (\,)^a \verk (\,)^{\nu})) \; .
\end{equation}
Two elements $(\oa , \nu)$ and $(\oa' , \nu')$ are in the same fibre if and only if $\nu' = \nu p^n$ and $\oa = p^n \oa'$ for some $n \in \Z$. The latter condition depends only on the class of $n \mod \deg \ex_0$. Passing to the $\Q^{> 0}_0$-extensions i.e. to $\colim_{\Nh_0}$ we obtain a $\Q^{> 0}_0$-equivariant bijection:
\begin{equation}
\label{eq:9}
(\hZ^{\times}_{(p)} / N \ex^{\hZ}_0 ) \times_{p^{\Z}} \Q^{> 0}_0 \silo C_{\ex_0} \; .
\end{equation}
Here, for a right $\Gamma$-module $A$ and a left $\Gamma$-module $B$ we write $A \times_{\Gamma} B$ for the quotient of $A \times B$ by the left $\Gamma$-action given by $\gamma (a,b) = (a \gamma^{-1} , \gamma b)$. The $\Gamma$-orbit of $(a,b)$ is denoted by $[a,b]$. It follows that all points $P_0 \in C_{\ex_0}$ have isotropy subgroup $(\Q^{> 0}_0)_{P_0} = N \ex^{\Z}_0$. Namely, for $q \in \Q^{> 0}_0$ the equation $[\oa , r] q = [\oa , r]$, i.e. $[\oa , rq] = [\oa , r]$ means that for some $n \in \Z$ we have $\oa = p^n \oa$ in $\hZ^{\times}_{(p)} / N \ex^{\hZ}_0$ and $rq = p^n r$. This is equivalent to $p^n \in N \ex^{\hZ}_0$ and $q = p^n$ i.e. to $q \in N \ex^{\Z}_0$. We may also write \eqref{eq:9} as a $\Q^{> 0}_0$-equivariant bijection:
\begin{equation}
\label{eq:10}
(\hZ^{\times}_{(p)} / N \ex^{\hZ}_0) \times_{p^{\Z / \deg \ex_0}} (\Q^{> 0}_0 / N \ex^{\Z}_0) \silo C_{\ex_0} \; .
\end{equation}

The set $C_{\ex_0}$ fibres over the compact group 
\[
\hZ^{\times}_{(p)} / p^{\hZ} = \Aut (\oF^{\times}_p) / \Aut (\oF_p) \; ,
\]
and the fibres are the $\Q^{> 0}_0$-orbits in $C_{\ex_0}$. After this discussion of the $\Q^{> 0}_0$-sets $C_{\ex_0}$ we mention one further item before proving the basic Theorem \ref{t4} below. The maps \eqref{eq:8}, \eqref{eq:9} and the fibration map depend on our choices of $\ex$ and $\iota$. On the other hand, the projection
\begin{equation}
\label{eq:11}
C_{\ex_0} \longrightarrow \Q^{> 0}_0 / p^{\Z}
\end{equation}
resulting from \eqref{eq:9} is canonical, it maps the point $\cpi (F^{-1}_{\nu} (\ex , \oP^{\times}))$ to the class of $|\ker \oP^{\times}| / \nu \mod p^{\Z}$. More generally consider the map
\begin{equation}
\label{eq:12}
\rho : \deX (\C)_{\Eh_{\tors}} \longrightarrow \Nh_0 \; , \; (\ex , \oP^{\times}) \longmapsto |(\Ker \oP^{\times})_{\tors}| \; .
\end{equation}
For $\nu \in \Nh_0$ let $\nu_{\ex}$ be the prime-to-$\car \kappa (\ex)$ part of $\nu$. If $\car \kappa (\ex) = 0$ then $\nu_{\ex} = \nu$. We have
\begin{equation} \label{eq:nr}
|(\Ker \oP^{\times} \verk ( \; )^{\nu})_{\tors}| = \nu_{\ex} |(\Ker \oP^{\times})_{\tors}| \; ,
\end{equation}
or in other words
\begin{equation}
\label{eq:13}
\rho (F_{\nu} (P)) = \nu_{\ex} \rho (P) \quad \text{for} \; \nu \in \Nh_0 \; \text{and} \; P \in \deX (\C)_{\Eh} \; .
\end{equation}
Recall the projection $\pr_{\eX} : \ceX (\C)_{\Eh_{\tors}} \to \eX$ and for a prime $p$ set 
\[
\ceX (\C)_{p , \Eh_{\tors}} = \pr^{-1}_{\eX} (\eX \otimes \F_p) \; .
\]
This is the set of points $F^{-1}_{\nu} (\ex, \oP^{\times})$ in $\ceX (\C)_{\Eh_{\tors}}$ with $\kappa (\ex)$ of characteristic $p$. Similarly,
\[
\ceX (\C)_{\Q , \Eh_{\tors}} = \pr^{-1}_{\eX} (\eX \otimes \Q)
\]
is the set of points with $\kappa (\ex)$ of characteristic zero. We define $\ceX_{0} (\C)_{p , \Eh_{\tors}}$ and $\ceX_{0} (\C)_{\Q , \Eh_{\tors}}$ similarly. Formulas \eqref{eq:12} and \eqref{eq:13} show that $\rho$ induces $\Q^{> 0}_0$-equivariant maps
\begin{equation}
\label{eq:14}
\rho : \ceX (\C)_{p , \Eh_{\tors}} \longrightarrow \Q^{> 0}_0 / p^{\Z}
\end{equation}
and
\begin{equation}
\label{eq:15}
\rho : \ceX (\C)_{\Q , \Eh_{\tors}} \longrightarrow \Q^{> 0}_0 \; 
\end{equation}
We obtain a natural $\Q^{> 0}_0$-equivariant surjective map of sets
\begin{equation}
\label{eq:16}
\rho : \ceX (\C)_{\Eh_{\tors}} =  \ceX (\C)_{\Q , \Eh_{\tors}} \amalg \coprod_p \ceX (\C)_{p , \Eh_{\tors}} \longrightarrow \Q^{> 0}_0 \amalg \coprod_p \Q^{> 0}_0 / p^{\Z}  \; .
\end{equation}
Here $p$ runs over $\car \eX_0$. The map $\rho$ factors over a map $\rho_0$ on $\ceX_0 (\C)_{\Eh_{\tors}}$ and the projection \eqref{eq:11} is the restriction of $\rho_0 \, |_{\ceX_{0p} (\C)} : \ceX_{0} (\C)_{p , \Eh_{\tors}} \to \Q^{> 0}_0 / p^{\Z}$ to $C_{\ex_0} \subset \ceX_{0} (\C)_{p , \Eh_{\tors}}$. 

\begin{prop}
\label{2xn}
A point $\cQ \in \ceX (\C)_{\Eh_{\tors}}$ is mapped by $\rho$ to $1 \in \Q^{> 0}_0$ resp. $1 \in \Q^{> 0}_0 / p^{\Z}$ for some prime $p$ in $\car \eX_0$ if and only if $\cQ = (\ex , \ovQ^{\times}) \in \deX (\C)_{\Eh_{\tors}}$ with $\car \kappa (\ex) = 0$ resp. $\car \kappa (\ex) = p$ and $\ovQ^{\times}$ injective on $\mu (\kappa (\ex))$. A corresponding statement holds for $\rho_0$ and $\cQ_0 \in \ceX_0 (\C)_{\Eh_{\tors}}$. 
\end{prop}

\begin{proof}
The condition is clearly sufficient. For $\cQ = F^{-1}_{\nu} (\ex , \oP^{\times})$ with $(\ex , \oP^{\times}) \in \deX (\C)_{\Eh_{\tors}}$ and $\nu \in \Nh_0$ we have:
\[
\rho (\cQ) = \nu^{-1}_{\ex} |(\ker \oP^{\times})_{\tors}| \quad \text{in $\Q^{> 0}_0$ resp.} \; \Q^{> 0}_0 / p^{\Z} \; .
\]
Hence $\rho (\cQ) = 1$ means that $|(\ker \oP^{\times})_{\tors}| = \nu_{\ex}$ since both numbers are prime to $p = \car \kappa (\ex)$ if $p > 0$. There exists a (unique) point $(\ex , \ovQ^{\times}) \in \deX (\C)_{\Eh_{\tors}}$ with $F_{\nu_{\ex}} (\ex , \ovQ^{\times}) = (\ex , \oP^{\times})$. This follows from Lemma \ref{t46n} for $R = K = \kappa (\ex)$ and $\eo = \C$. Since
\[
|(\ker \ovQ^{\times})_{\tors}| = \nu^{-1}_{\ex} |(\ker \oP^{\times})_{\tors}| = 1
\]
the character $\ovQ^{\times}$ is injective on $\mu (\kappa (\ex))$. Writing $\nu = \nu_{\ex} p^n$ for some $n \ge 0$ (where $n = 0$ if $\car \kappa (\ex) = 0$), we have
\[
\cQ = F^{-n}_p F^{-1}_{\nu_{\ex}} (\ex , \oP^{\times}) = F^{-n}_p (\ex , \ovQ^{\times}) \; .
\]
Noting that $(\;)^p$ is an isomorphism on $\kappa (\ex)^{\times}$ if $\car \kappa (\ex) = p > 0$ we can replace $\ovQ^{\times}$ by $\ovQ^{\times} \verk (\;)^{p^{-n}}$ which is still injective on $\mu (\kappa (\ex))$ and hence the necessity of the above condition follows. 
\end{proof}

Concerning the $\Q^{> 0}_0$-action on $\ceX_0 (\C)_{\Eh}$ we have the following main result:

\begin{theorem}
\label{t4}
Let $\Eh$ be an admissible class with $\Eh \subset \Eh_{\max}$. The following decomposition holds, where $\ex_0$ runs over the points of $\eX_0$ with finite residue field $\kappa (\ex_0)$ and where $C^{\Eh}_{\ex_0} = C_{\ex_0} \cap \ceX_0 (\C)_{\Eh}$
\begin{equation}
\label{eq:17}
\{ P_0 \in \ceX_0 (\C)_{\Eh} \mid (\Q^{> 0}_0)_{P_0} \neq 1 \} = \coprod_{\ex_0} C^{\Eh}_{\ex_0} \; .
\end{equation}
For any point $P_0 \in C^{\Eh}_{\ex_0}$ the isotropy group of $P_0$ is $(\Q^{> 0}_0)_{P_0} = N \ex^{\Z}_0$ where $N \ex_0 = |\kappa (\ex_0)|$. If e.g. $\Eh \supset \Eh_f$ then $C^{\Eh}_{\ex_0} = C_{\ex_0}$. 
\end{theorem}

\begin{rem}
In case $\eX_0$ has no points with finite residue field, the theorem asserts that all isotropy groups of the $\Q^{> 0}_0$-action on $\ceX_0 (\C)_{\Eh}$ are trivial.
\end{rem}

{\it Proof.} 
Replacing $P_0 \in \ceX_0 (\C)_{\Eh}$ with $F_{\nu} (P_0)$ for some $\nu \in \Nh_0$ does not change the isotropy subgroup. Hence it is sufficient to determine the points $P_0 \in \deX_0 (\C)_{\Eh}$ whose isotropy subgroup is non-trivial. The points $P_0$ in $\deX_0 (\C)_{\Eh}$ have the form $P_0 = \dpi ((\ex , \oP^{\times})) = (\ex , \oP^{\times}) \mod G$, where $\oP^{\times} : \kappa (\ex)^{\times} \to \C^{\times}$ is a homomorphism satisfying $\Eh$ and hence $\Eh_{\max}$. Let us write $q = \nu / \nu' \in \Q^{> 0}_0$ with $\nu , \nu' \in \Nh_0$ coprime. The condition $q \in (\Q^{> 0}_0)_{P_0}$ means that $F_q (P_0) = P_0$ i.e. $F_{\nu} (P_0) = F_{\nu'} (P_0)$. This is equivalent to the existence of some $\sigma \in G$ with $\sigma \ex = \ex$ and
\begin{equation}
\label{eq:18}
\oP^{\times} (y^{\nu}) = \oP^{\times} (\osigma (y)^{\nu'}) \quad \text{for all} \; y \in \kappa (\ex)^{\times} \; .
\end{equation}
Here $\osigma$ is the automorphism of $\kappa (\ex)$ over $\kappa (\ex_0)$ induced by $\sigma$ where $\ex_0 = \pi (\ex)$. Condition \eqref{eq:18} means that the subgroup
\[
H = \{ y^{\nu} \osigma (y)^{-\nu'} \mid y \in \kappa (\ex)^{\times} \}
\]
is contained in the kernel of $\oP^{\times}$. For $n \ge 1$ let $\mu_{(n)} (K)$ be the group of roots of unity in $K$ of order prime to $n$. To proceed with the proof we need the following result:

\begin{lemma}
\label{t5}
For coprime $\nu \neq \nu'$ in $\Nh$ the following map is surjective where $\emm_{\ex}$ is the maximal ideal in $\Oh_{\eX,\ex}$
\begin{equation}
\label{eq:19}
\mu_{(\nu \nu')} (K) = \bigcup_{i \ge 1} \mu_{\nu^i - \nu^{'i}} (K) \longrightarrow \kappa (\ex)^{\times} / H \; , \; \zeta \longmapsto (\zeta \mod \emm_{\ex}) H \; .
\end{equation}
In particular $\kappa (\ex)^{\times} / H$ is a prime to $\nu \nu'$ torsion group.
\end{lemma}

\begin{rem}
The map \eqref{eq:19} is in fact an isomorphism as shown in the proof of Theorem \ref{t6n} b) below.
\end{rem}

The proof of the lemma will be given below. Recall that because of the assumption $\Eh \subset \Eh_{\max}$, every character in $\Eh$ satisfies conditions ({\it Tors}) and ({\it Image}) from section \ref{sec:4ggn}. Assume that $q \neq 1$. Using the lemma it follows that $\mu_{l^{\infty}} (K)$ maps to $1$ in $\kappa (\ex)^{\times} / H$ for all prime divisors $l$ of $\nu \nu'$. Since $\oP^{\times}$ factors over $\kappa (\ex)^{\times} / H$, we see that
\[
\Imm (\mu_{l^{\infty}} (K) \longrightarrow \kappa (\ex)^{\times}) \subset (\ker \oP^{\times})_{\tors} 
\]
is finite because $\oP^{\times}$ satisfies ({\it Tors}). Because of $\nu \nu' > 1$ it follows that $p = \car \kappa (\ex)$ is positive. More precisely $\nu \nu'$ must have exactly one prime divisor $l$, namely $l = p = \car \kappa (\ex)$. Note here that the image of $\mu_{p^{\infty}} (K)$ in $\kappa (\ex)^{\times}$ is trivial. Because of the lemma the quotient 
\[
\kappa (\ex)^{\times} / H \twoheadrightarrow \oP^{\times} (\kappa (\ex)^{\times})
\]
is a torsion group. Since $p = \car \kappa (\ex)$ is positive and since $\oP^{\times}$ satisfies ({\it Image}), it follows that $\kappa (\ex)^{\times}$ is torsion. By the structure theory of fields we conclude that $\kappa (\ex) \subset \oF_p$ for some $p$ and hence $\kappa (\ex) = \oF_p$. Since $\kappa (\ex)^{\times}$ is torsion, we have $H \subset \ker \oP^{\times} = (\ker \oP^{\times})_{\tors}$ and ({\it Tors}) implies that $H$ is finite. Since $\kappa (\ex) \cong \oF_p$, we have
\[
\osigma = p^a \in p^{\hZ} = \Aut (\kappa (\ex)) \subset \Aut (\kappa (\ex)^{\times}) = \hZ^{\times}_{(p)} \; .
\]
The endomorphism ring of the abelian group $\kappa (\ex)^{\times}$ is 
\[
\End \kappa (\ex)^{\times} = \End (\bigoplus_{l \neq p} \Q_l / \Z_l) = \hZ_{(p)} \; .
\]
The group $H$ is the image of the endomorphism of $\kappa (\ex)^{\times}$ corresponding to $\nu - \nu' p^a \in \hZ_{(p)}$. Since $\kappa (\ex)^{\times}$ is divisible it follows that $H$ is divisible as well, and being finite, $H$ is trivial. Thus $\nu = \nu' p^a$ in $\hZ_{(p)}$, and using $q \neq 1$ it follows that $0 \neq a \in \Z$. Since $\kappa (\ex_0)$ is contained in the fixed field of $\osigma = p^a$, writing $n = |a| \ge 1$ we have $\kappa (\ex_0) \subset \F_{p^n}$. Thus our assumption $(\Q^{> 0}_0)_{P_0} \neq 1$ eventually implies that $\kappa (\ex_0)$ is finite and $P_0 \in \pr^{-1}_0 (\ex_0) \subset C_{\ex_0}$. The remaining assertion in the theorem has already been shown in our discussion of $C_{\ex_0}$. Thus, modulo the lemma, Theorem \ref{t4} is proved. 

\begin{proofof}
{\it Lemma \ref{t5}} For $\nu \neq \nu'$, the union of the groups $\mu_{\nu^i - \nu^{'i}} (K)$ for $i \ge 1$ is a subgroup of $\mu (K)$ since $\nu^i - \nu^{'i}$ divides $\nu^{ki} - \nu^{'ki}$ for $k \ge 1$. We have
\begin{equation}
\label{eq:20}
\mu_{(\nu \nu')} (K) = \bigcup_{i \ge 1} \mu_{\nu^i - \nu^{'i}} (K) \; .
\end{equation}
Namely, if a prime number $l$ divides $\nu^i- \nu^{'i}$, then it must be prime to $\nu$ and $\nu'$ because otherwise it would divide them both. Hence the right hand side is contained in $\mu_{(\nu  \nu')} (K)$. Now assume that $\zeta$ has order $N$ where $N$ is prime to $\nu$ and $\nu'$. Then for $i = |(\Z / N)^{\times}|$, the classes $\onu$ and $\onu'$ in $(\Z / N)^{\times}$ satisfy $\onu^i = \onu^{'i}$ which means that $N$ divides $\nu^i - \nu^{'i}$. Thus $\zeta \in \mu_{\nu^i - \nu^{'i}} (K)$ and the reverse inclusion follows. 

To prove the surjectivity of the map to $\kappa (\ex)^{\times} / H$ fix $c \in \kappa (\ex)^{\times}$ and choose $i \ge 1$ such that $\osigma^i (c) = c$. Since $\nu \neq  \nu'$ there is some $z \in \kappa (\ex)^{\times}$, such that
\[
z^{\nu^i - \nu^{'i}} = c \; .
\]
We find 
\[
\osigma^i (z)^{\nu^i - \nu^{'i}} = \osigma^i (c) = c = z^{\nu^i - \nu^{'i}} \; .
\]
Hence we have
\[
\osigma^i (z) = \eta z
\]
for some $\eta \in \mu_{\nu^i - \nu^{'i}} (\kappa (\ex))$. This implies: 
\[
\frac{z^{\nu^i}}{\osigma^i (z)^{\nu^{'i}}} = \frac{z^{\nu^i}}{\eta^{\nu^{'i}} z^{\nu^{'i}}} = \eta^{- \nu^{'i}} z^{\nu^i - \nu^{'i}} = \eta^{- \nu^{'i}} c \; .
\]
In the commutative group ring $\Z \langle \osigma \rangle$ of the cyclic group $\langle \osigma \rangle$ generated by $\osigma$ we have
\[
\nu^i - \nu^{'i} \osigma^i = \omega (\nu - \nu' \osigma)
\]
for $\omega = \sum^{i-1}_{\alpha = 0} \nu^{i - 1 - \alpha} \nu^{'\alpha} \osigma^{\alpha} \in \Z \langle \osigma \rangle$. 

The group ring $\Z \langle \osigma \rangle$ acts on $\kappa (\ex)^{\times}$. Setting $y = z^{\omega}$ we obtain:
\[
\frac{y^{\nu}}{\osigma (y)^{\nu'}} = y^{\nu - \nu' \osigma} = z^{\omega (\nu - \nu' \osigma)} = z^{\nu^i - \nu^{'i} \osigma^i} = \frac{z^{\nu^i}}{\osigma^i (z)^{\nu^{'i}}} \; .
\]
This implies that
\[
\eta^{- \nu^{'i}} c = \frac{y^{\nu}}{\osigma (y)^{\nu'}} \in H \; .
\]
Thus we have
\[
cH = \zeta H \; ,
\]
for $\zeta = \eta^{\nu^{'i}} \in \mu_{\nu^i - \nu^{'i}} (\kappa (\ex))$. This implies the Lemma and hence Theorem \ref{t4}.
\end{proofof}

\begin{remark} \em
For classes $\Eh$ satisfying the stronger condition $\Eh \subset \Eh_{fd_0}$ instead of $\Eh \subset \Eh_{\max}$ only, the proof of Theorem \ref{t4} does not require Lemma \ref{t5}. As before \eqref{eq:18} implies $H \subset \ker \oP^{\times}$ and hence $H \cap \kappa (\ex_0)^{\times} \subset \ker \oP^{\times} \, |_{\kappa (\ex_0)^{\times}}$. Since $\osigma$ acts trivially on $\kappa (\ex_0)^{\times}$, we have
\[
(\kappa (\ex_0)^{\times})^{\nu - \nu'} := \{ y^{\nu - \nu'} \mid y \in \kappa (\ex_0)^{\times} \} \subset H \cap \kappa (\ex_0)^{\times} \subset \ker \oP^{\times} \, |_{\kappa (\ex_0)^{\times}} \; .
\]
Because $\oP^{\times}$ satisfies $\Eh_{fd_0}$ it follows that $\dim_{\Q} ((\kappa (\ex_0)^{\times})^ {\nu - \nu'} \otimes \Q) < \infty$. If $q \neq 1$ i.e. $\nu \neq \nu'$ this implies that $\dim_{\Q} \kappa (\ex_0)^{\times} \otimes \Q < \infty$. It follows that $\kappa (\ex_0)$ is contained in $\oF_p$ for some $p$, and hence $\kappa (\ex) = \oF_p$. We now finish the proof as before. 
\end{remark}

The condition ({\it Tors}) i.e. that $(\ker \oP^{\times})_{\tors}$ is finite seems to be fundamental, in particular for the later connectedness results. However by itself it is not enough to ensure that Theorem \ref{t4} holds. As we have seen in the proof of Theorem \ref{t4}, ({\it Tors}) does force the points $P_0 \in \ceX_0 (\C)_{\Eh}$ over characteristic zero points of $\eX_0$ to have trivial stabilizer $(\Q^{> 0}_0)_{P_0} = 1$. However over all positive characteristic points of $\eX_0$ that condition still allows points $P_0$ with $(\Q^{> 0}_0)_{P_0} \neq 1$ as we will now see, whereas we want this to happen only over the points of $\eX_0$ with finite residue field.

\begin{theorem}
\label{t6n}
For $P_0 \in \ceX_0 (\C)_{\Eh_{\tors}}$ let $\ex_0 = \pr_{\eX_0} (P_0) \in \eX_0$. Then the following assertions hold:\\
a) If $(\Q^{> 0}_0)_{P_0} \neq 1$ then the characteristic $p$ of $\kappa (\ex_0)$ is positive and $(\Q^{> 0}_0)_{P_0} \subset N^{\Z}_{\ex_0}$. Here $N_{\ex_0} = |\kappa (\ex_0) \cap \oF_p|$ where $\oF_p$ is the algebraic closure of $\F_p$ in $\kappa (\ex)$. \\
b) For any point $\ex_0 \in \eX_0$ with $p = \car \kappa (\ex_0) > 0$ there exist points $P_0 \in \pr^{-1}_0 (\ex_0) \subset \ceX_0 (\C)_{\Eh_{\tors}}$ with $(\Q^{> 0}_0)_{P_0} \neq 1$. 
\end{theorem}

\begin{proof} 
a) For $q \in (\Q^{> 0}_0)_{P_0}$ with $q \neq 1$ write $q = \nu / \nu'$ with $\nu , \nu' \in \Nh_0$ coprime, $\nu \neq \nu'$. In the proof of Theorem \ref{t4}, after Lemma \ref{t5} we have seen that $p = \car \kappa (\ex) > 0$ and that $p$ is the only prime divisor of $\nu \nu' > 1$. Since $\nu \neq \nu'$ are coprime it follows that $\nu = p^n$ and $\nu' = 1$ or $\nu = 1$ and $\nu' = p^n$ for some $n \ge 1$. Thus we have $q = \nu / \nu' \in p^{\Z}$. We now assume that $\nu = p^n , \nu' = 1$, the following argument being similar in the case $\nu = 1 , \nu' = p^n$. We know that the divisible group
\[
\{ \zeta^{p^n} \sigma (\zeta)^{-1} \mid \zeta \in \mu_{(p)} (K) \} \hookrightarrow H_{\tors} \; , \; \zeta \longmapsto \zeta \mod \emm_{\ex}
\]
is finite since $H_{\tors} \subset (\ker \oP^{\times})_{\tors}$ and therefore trivial. It follows that $\sigma (\zeta) = \zeta^{p^n}$ for all $\zeta \in \mu_{(p)} (K)$. Since $\osigma$ is trivial on $\kappa (\ex_0)$ and hence on $\kappa (\ex_0) \cap \oF_p$ it follows that the index $(\kappa (\ex_0) \cap \oF_p : \F_p)$ divides $n$ i.e. that $q = \nu / \nu' = p^n \in N^{\Z}_{\ex_0}$. 

b) For a point $\ex_0 \in \eX_0$ of positive characteristic $p = \car \kappa (\ex_0)$ let $\spec R_0$ be an open affine subscheme of $\eX_0$ containing $\ex_0$. Let $R$ be the normalization of $R_0$ in $K$. Then the open affine subscheme $\spec R$ of $\eX$ contains the points $\ex \in \eX$ over $\ex_0$. Choose such an $\ex$ and let $\ep$ be the corresponding prime ideal in $R$ and let $\ep_0 = \ep \cap R_0$ be the prime ideal of $R_0$ corresponding to $\ex_0$. Let $\overline{\osigma} \in \Gal (\oF_p / \kappa (\ep_0) \cap \oF_p)$  be the Frobenius automorphism $y \mapsto y^{N_{\ex_0}}$. Lift $\overline{\osigma}$ to an automorphism $\osigma$ of $\kappa (\ep)$ over $\kappa (\ep_0)$. According to \cite[Chap. V, \S\,2, n$^{0}$ 3, Proposition 6]{B} there is an extension of $\osigma$ to an automorphism $\sigma \in G = \Aut (K / K_0)$ with $\sigma (\ex) = \ex$. Consider the subgroup
\[
H = \{ y^{N_{\ex_0}} \osigma (y)^{-1} \mid y \in \kappa (\ex)^{\times} \} \subset \kappa (\ex)^{\times} \; .
\]
By Lemma \ref{t5} the natural map
\[
\mu_{(p)} (K) \cong \mu (\kappa (\ep)) \twoheadrightarrow \kappa (\ep)^{\times} / H
\]
is surjective. We now show that it is an isomorphism. Its kernel consists of elements $\zeta = y^{N_{\ex_0}} \osigma (y)^{-1} \in \mu (\kappa (\ep))$ for some $y \in \kappa (\ep)^{\times}$. Choose some $N \ge 1$ prime to $p$ with $\zeta^N = 1$. For $z = y^N$ we then have $\osigma (z) = z^{N_{\ex_0}}$. Letting $i$ be such that $\osigma^i (z) = z$ we see that $z$ is an $(N^i_{\ex_0} - 1)$-th root of unity in $\kappa (\ep)^{\times}$. Hence $z$ and therefore also $y$ lie in $\mu (\kappa (\ep)) = \oF_p^{\times}$. But then, by the choice of $\sigma$ we have $\osigma (y) = y^{N_{\ex_0}}$ and therefore $\zeta = 1$. Using the isomorphism $\mu_{(p)} (K) \silo \kappa (\ep)^{\times} / H$, any character $\chi : \mu_{(p)} (K) \to \mu (\C)$ with finite kernel gives a character $\oP^{\times} : \kappa (\ex)^{\times} \to \mu (\C) \subset \C^{\times}$ of class $\Eh_{\tors}$ with $\oP^{\times} (H) = 1$ i.e. with $\oP^{\times} (y^{N_{\ex_0}}) = \oP^{\times} (\osigma (y))$ for all $y \in \kappa (\ex)^{\times}$. This gives a point $P_0 = (\ex , \oP^{\times}) \mod G$ in $\ceX_0 (\C)_{\Eh_{\tors}}$ over $\ex_0$ with $F_{N_{\ex_0}} (P_0) = P_0$. 
\end{proof}

\section{The $\R^{> 0}$-dynamical system $X_0$} \label{sec:3}

Let $\eX_0$ be an integral normal scheme and let $\Eh$ be an admissible class of characters in the sense of Definition \ref{t41n}. Let $\C$ be an algebraically closed field which satisfies the conditions before Corollary \ref{t4ka}. 

Let $\R^{> 0}$ be the group of positive real numbers under multiplication and consider the suspension
\[
X_0 = \ceX_0 (\C)_{\Eh} \times_{\Q^{> 0}_0} \R^{> 0} \; .
\]
It is the quotient of $\ceX_0 (\C)_{\Eh} \times \R^{> 0}$ by the right $\Q^{> 0}_0$-action given by 
\[
(P_0 , u) q = (P_0 q , q^{-1} u) = (F_q (P_0) , q^{-1} u) \quad \text{for} \; q \in \Q^{> 0}_0 \; .
\]
The $\Q^{> 0}_0$-orbit of $(P_0 , u)$ is denoted by $[P_0 , u]$. The group $\R^{> 0}$ acts on $X_0$ via the second factor:
\[
[P_0 , u] \cdot v = [P_0 , uv] \quad \text{for} \; v \in \R^{> 0} \; .
\]
We may also view $X_0$ as a continuous time dynamical system by letting $t \in \R$ act by $e^t$. We write $\phi^t$ for this action i.e. $\phi^t ([P_0 ,u]) = [P_0 , ue^t]$. For a point $\ex_0$ of $\eX_0$ with finite residue field of characteristic $p$ set
\[
\Gamma_{\ex_0} = C_{\ex_0} \times_{\Q^{> 0}_0} \R^{> 0} \subset X_0 \; .
\]
The $\Q^{> 0}_0$-bijection \eqref{eq:10} induces an $\R^{> 0}$-bijection
\[
(\hZ^{\times}_{(p)} / N \ex^{\hZ}_0) \times_{p^{\Z / \deg \ex_0}} \R^{> 0} / N \ex^{\Z}_0 \silo \Gamma_{\ex_0} \; . 
\]
Thus all $\R^{> 0}$-orbits in $\Gamma_{\ex_0}$ are circles $\R^{> 0} / N \ex^{\Z}_0$ and $\Gamma_{\ex_0}$ fibres over $\hZ^{\times}_{(p)} / p^{\hZ}= \Aut (\oF^{\times}_p) / \Aut (\oF_p)$ with fibres the $\R^{> 0}$-orbits in $\Gamma_{\ex_0}$. We set $\Gamma^{\Eh}_{\ex_0} = C^{\Eh}_{\ex_0} \times_{\Q^{> 0}_0} \R^{> 0}$ where $C^{\Eh}_{\ex_0} = C_{\ex_0} \cap \ceX_0 (\C)_{\Eh}$. If e.g. $\Eh_f \subset \Eh$ then $\Gamma^{\Eh}_{\ex_0} = \Gamma_{\ex_0}$. 

The next result is an immediate consequence of Theorem \ref{t4}.

\begin{theorem}
\label{t6}
Let $\Eh$ be an admissible class with $\Eh \subset \Eh_{\max}$. The following decomposition holds, where $\ex_0$ runs over the points of $\eX_0$ with finite residue fields $\kappa (\ex_0)$
\[
\{ x_0 \in X_0 \mid (\R^{> 0})_{x_0} \neq 1 \} = \coprod_{\ex_0} \Gamma^{\Eh}_{\ex_0} \; .
\]
For any point $x_0 \in \Gamma^{\Eh}_{\ex_0}$ the isotropy group of $x_0$ is $(\R^{> 0})_{x_0} = N \ex^{\Z}_0$.
\end{theorem}

The theorem asserts that the points $\ex_0$ of $\eX_0$ with finite residue field correspond bijectively to the ``packets'' $\Gamma^{\Eh}_{\ex_0}$ of periodic orbits of length $\log N \ex_0$ in the $\R$-dynamical system $X_0$. Any periodic orbit $\gamma$ in $X_0$ is contained in $\Gamma^{\Eh}_{\ex_0}$ for a uniquely determined point $\ex_0$ of $\eX_0$ with finite residue field. 

\begin{rem}
We may also describe $X_0$ using $\deX_0 (\C)_{\Eh}$ instead of $\ceX_0 (\C)_{\Eh}$. For a (right) action of the monoid $\Nh_0$ on a set $Y$ we get an equivalence relation on $Y$ as follows: $y \sim y'$ if and only if there are $\nu , \nu' \in \Nh_0$ with $y \cdot \nu = y' \cdot \nu'$. Let $Y / \Nh_0$ denote the set of equivalence classes. Let $\nu \in \Nh_0$ act on $\deX_0 (\C)_{\Eh} \times \R^{> 0}$ by setting $(P_0 , u) \cdot \nu := (F_{\nu} (P_0) , \nu^{-1} u)$. We set
\[
\deX_0 (\C)_{\Eh} \times_{\Nh_0} \R^{> 0} := (\deX_0 (\C)_{\Eh} \times \R^{>0}) / \Nh_0 \; .
\]
The set $\deX_0 (\C)_{\Eh} \times_{\Nh_0} \R^{> 0}$ has a natural $\R^{> 0}$-action by multiplication on the second factor. It is straightforward to check that the natural map
\[
\deX_0 (\C)_{\Eh} \times \R^{> 0} \longrightarrow \ceX_0 (\C)_{\Eh} \times_{\Q^{> 0}_0} \R^{> 0} \; , \; (m,u) \mapsto [m,u]
\]
induces an $\R^{> 0}$-equivariant bijection
\[
\deX_0 (\C)_{\Eh} \times_{\Nh_0} \R^{> 0} \silo \ceX_0 (\C)_{\Eh} \times_{\Q^{> 0}_0} \R^{> 0} \; .
\]
\end{rem}
Now assume that $\eX_0$ is of finite type over $\spec \Z$ and $\C$ is the complex number field. In the following observation we omit $\Eh$ from the notation. One prediction in \cite{D3} was that $X_0$ should have a ``compactification'' $\oX_0$ corresponding to an Arakelov compactification $\oeX_0$ of $\eX_0$. The fixed points of the $\R$-action on $\oX_0$ should be the set $\eX_0 (\C) / F_{\infty}$ where $F_{\infty} : \eX_0 (\C) \to \eX_0 (\C)$ is induced by complex conjugation on $\C$. The isotropy groups in $\Q^{> 0}$ of the classical points $\eX_0 (\C) \subset \ceX_0 (\C)$ are trivial e.g. by Theorem \ref{t4}. More precisely, for $P_0 \in \eX_0 (\C)$ and $1 \neq q \in \Q^{> 0}$ the point $F_q (P_0)$ is no-longer a classical point. Namely, for an inclusion of fields $i : \kappa (\ex) \subset \C$ the map $i \verk (\;)^{\nu}$ cannot be conjugate to $i' \verk (\;)^{\nu'}$ for some other field embedding $i' : \kappa (\ex) \subset \C$ and $\nu , \nu' \in \Nh$ unless $\nu = \nu'$. Note here that $\nu = |\ker i \verk (\;)^{\nu} \verk \osigma|$ for all $\osigma \in \Aut \kappa (\ex)$. Hence the map
\[
\eX_0 (\C) \times \R^{> 0} \longrightarrow X_0 \; , \; (P_0 , u) \longmapsto [P_0 , u] = [P_0 , 1] \cdot u
\]
is an $\R^{> 0}$-equivariant inclusion. Letting $F_{\infty}$ act on $\R$ by $F_{\infty} (t) = -t$ we have an $\R^{> 0}$-equivariant isomorphism where $\R^{> 0}$ acts on $\R^{\times}$ by multiplication:
\[
\eX_0 (\C) \times \R^{> 0} \silo \eX_0 (\C) \times_{F_{\infty}} \R^{\times} \; .
\]
This gives an $\R^{> 0}$-immersion:
\[
\eX_0 (\C) \times_{F_{\infty}} \R^{\times} \hookrightarrow \ceX_0 (\C) \times_{\Q^{> 0}} \R^{> 0} = X_0 \; .
\]
It should extend to an $\R^{> 0}$-immersion where $\R = \R^{\times} \cup \{ 0 \}$
\[
\eX_0 (\C) \times_{F_{\infty}} \R \hookrightarrow \oX_0 \; .
\]
The fixed point set of the $\R^{>0}$-action on the left is $\eX_0 (\C) \times_{F_{\infty}} \{ 0 \} = \eX_0 (\C) / F_{\infty}$. Moreover for $z \in \eX_0 (\C) \setminus \eX_0 (\R)$ there are two orbits tending to the fixed point $[z,0]$, namely $[z , \pm 1] u$ for $u \to 0+$. For $z \in \eX (\R)$ there is one orbit with this property, namely $[z , 1] u$ for $u \to 0+$. In the case $\eX_0 = \spec \eo_k$, where $k / \Q$ is a number field, we had deduced this type of orbit structure from a consideration of $\Gamma$-factors, c.f. \cite[\S\,3]{D2}. Of course, it is unclear, why $\oX_0$ would not contain many more fixed points obtained in the same way from the points in $\ceX_0 (\C) \setminus \eX_0 (\C)$. The system $X_0$ may have to be replaced by a much smaller system: Is there a sub-dynamical system $Y_0 \subset X_0 = \ceX_0 (\C) \times_{\Q^{> 0}} \R^{> 0}$ or at least one which maps to $X_0$ such that $\dim Y_0 = 2d + 1$ where $d = \dim \eX_0$ and such that $Y_0$ contains at least one periodic orbit in $\Gamma_{\ex_0}$ for every closed point $\ex_0$ of $\eX_0$? If $d = 1$, is there such a $Y_0$ which is a Riemann surface lamination in the sense of \cite{G}? 
\section{Topology} \label{sec:4n}

We now introduce topologies on our spaces. We only consider integral normal schemes $\eX_0$ whose function field $K_0$ is countable. For brevity we call them {\it arithmetic schemes}. If as usual $K$ denotes an algebraic closure of $K_0$ and $\eX$ is the normalization of $\eX_0$ in $K$, then $\eX$ is an arithmetic scheme as well. In this section, $\C$ is an algebraically closed field with a valuation $|\;|$ and the corresponding topology. We begin with the affine case $\eX_0 = \spec R_0$ and write $\eX = \spec R$. Viewing $\deX (\C)$ as a set of multiplicative maps $P : R \to \C$ as in Remark \ref{t34} we give $\deX (\C)$ the topology of pointwise convergence. It is the subspace topology induced by the Tychonov topology of $\C^R = \prod_R \C$ on $\deX (\C)$ via the inclusion $\deX (\C) \subset \C^R , P \mapsto (P (r))_{r \in R}$. Since $R$ is countable, $\deX (\C)$ is a metrizable topological space.

\begin{lemma} \label{t7}
For affine arithmetic schemes $\eX_0$, the natural map
\[
\pr_{\eX} : \deX (\C) \longrightarrow \eX \; , \; (\ex , \oP^{\times}) \longmapsto \ex \quad \text{or} \quad P \longmapsto \ep = P^{-1} (0)
\]
is continuous.
\end{lemma}

\begin{proof}
A closed subset $A$ of $\eX$ has the form $A = \{ \ep \supset I \}$ for some ideal $I$ in $R$. Consider a convergent sequence $P_n \to P$ in $\deX (\C)$ where $P_n \in \pr^{-1}_{\eX} (A)$ for all $n$, i.e. $\ep_n = P^{-1}_n (0) \supset I$. Since $P_n (r) \to P (r)$ for all $r \in R$, it follows that $P (r) = 0$ for $r \in I$ and hence $\ep = P^{-1} (0) \supset I$ i.e. $P  \in \pr^{-1}_{\eX} (A)$. Hence $\pr^{-1}_{\eX} (A)$ is closed and therefore $\pr_{\eX}$ is continuous.
\end{proof}

Let $\eX'_0 = \spec R'_0 \subset \eX_0 = \spec R_0$ be the inclusion of an open affine (hence arithmetic) subscheme of $\eX_0$. Then we have $R_0 \subset R'_0 \subset K_0$. Let $\eX' = \spec R'$ be the normalization of $\eX'_0$ in $K$. It is an open subscheme of $\eX = \spec R$ and we have $R \subset R' \subset K$. There are commutative diagrams
\[
\vcenter{\xymatrix{
\eX' \ar@{}[r]|{\subset}  \ar[d]_{\pi'} & \eX \ar[d]^{\pi} \\
\eX'_0  \ar@{}[r]|{\subset} & \eX_0 
} }
\quad \text{and} \quad 
\vcenter{\xymatrix{
\deX\strich (\C) \ar@{}[r]|{\subset}  \ar[d]_{\pr_{\eX'}} & \deX (\C) \ar[d]^{\pr_{\eX}} \\
\eX'  \ar@{}[r]|{\subset} & \eX \, .
}}
\]
The inclusion map $\deX\strich (\C) \subset \deX (\C)$ sends $(\ex , \oP^{\times})$ to $(\ex , \oP^{\times})$ or in terms of multiplicative maps, $P'$ is sent to $P' \, |_R$. It follows that the inclusion is continuous.

\begin{lemma} \label{t8}
The space $\deX\strich (\C)$ is an open subspace of $\deX (\C)$.
\end{lemma}

\begin{proof}
The inclusion being continuous it remains to show that it is an open map. Let $O \subset \deX\strich (\C)$ be an open subset. The set $O$ is open in $\deX (\C)$ if
\[
A' = (\deX (\C) \setminus \deX\strich (\C)) \cup A \quad \text{with} \; A = \deX\strich (\C) \setminus O
\]
is closed in $\deX (\C)$. Since $\deX\strich (\C) = \pr^{-1}_{\eX} (\eX')$ is open in $\deX (\C)$ by Lemma \ref{t7} it suffices to show that $\oA \subset A'$, where $\oA$ is the closure of $A$ in $\deX (\C)$. So, let $P_n \in A \subset \deX (\C)$ be a sequence with $P_n \to P \in \deX (\C)$. The inclusion $\deX\strich (\C) \subset \deX (\C)$ is given by  the restriction $P' \mapsto P' \, |_R$. Hence $P_n = P'_n \, |_R$ for a sequence of points $P'_n \in A \subset \deX\strich (\C)$. If $P \notin \deX\strich (\C)$, then $P \in A'$ and we are done. If $P \in \deX\strich (\C)$ then $P = P' \, |_R$. For the multiplicative maps $P'_n , P' : R' \to \C$ we have $P'_n (r) \to P' (r)$ for each $r \in R$ as $n \to \infty$. Let $\ep' = P^{'-1} (0) \subset R'$ and $\ep = P^{-1} (0) \subset R$ be the prime ideals belonging to $P'$ resp. $P$. We have $\ep = \ep' \cap R$. Since $\eX' = \spec R' \subset \eX = \spec R$ is an open immersion, the local rings of $\eX'$ in $\ep'$ and $\eX$ in $\ep$ are canonically isomorphic. Thus the inclusion $R \subset R'$ induces a commutative diagram:
\[
\xymatrix{
R_{\ep} \ar[dr]_P \ar@{=}[rr] & & R'_{\ep'} \ar[dl]^{P'} \\
 & \C & 
}
\]
Here $P , P'$ also denote the unique multiplicative extensions of $P$ and $P'$ to $R_{\ep}$ resp. $R'_{\ep'}$. We have $R' \subset R'_{\ep'} = R_{\ep}$. Given $r' \in R'$ write $r' = r/ s$ with $r \in R , s \in R \setminus \ep$. Since $P_n (s) \to P (s) \neq 0$ for $n \to \infty$ we have $P_n (s) \neq 0$ for $n \ge n_0 (s)$. Pointwise convergence $P_n \to P$ on $R$ gives:
\[
P'_n (r') = P_n (r) / P_n (s) \longrightarrow P (r) / P (s) = P' (r') 
\]
for $n \to \infty$. Thus we have $P'_n \to P'$ pointwise on $R'$ and therefore $P' \in A \subset A'$ since $P_n \in A$ and $A$ is closed in $\deX\strich (\C)$. Hence we have shown that $\oA \subset A'$.
\end{proof}

Now let $\eX_0$ be any arithmetic scheme. For any affine open (and hence arithmetic) subscheme $\eX'_0$ of $\eX_0$ consider the inclusion $\deX\strich (\C) \subset \deX (\C)$ mapping $(\ex' , \oP^{\times})$ to $(\ex', \oP^{\times})$ via the identification $\Oh_{\eX', \ex'} = \Oh_{\eX , \ex'}$. We give $\deX (\C)$ the topology for which $O \subset \deX (\C)$ is open if and only if $O \cap \deX\strich (\C)$ is open in $\deX\strich (\C)$ for any $\eX'_0$. Given any covering $\{ \eX^i_0\}_{i \in I}$ of $\eX_0$ by affine open subschemes $\eX^i_0$, the set $O$ is open if and only if $O \cap \deX\!\,^i (\C)$ is open for all $i$. To see the non-trivial direction, one covers $\eX^i_0 \cap \eX'_0$ by open affine subschemes $\eX^{ij}_0$ and notes that $\deX\!\,^{ij} (\C)$ is an open subspace of $\deX\!\,^i (\C)$ and $\deX\strich (\C)$ by Lemma \ref{t8}. Thus, if $O \cap \deX\!\,^i (\C)$ is open in $\deX\!\,^i (\C)$ for all $i$, it follows that $O \cap \deX\!\,^{ij} (\C)$ is open in $\deX\!\,^{ij} (\C)$ and hence in $\deX\strich (\C)$. Hence
\[
O \cap \deX\strich (\C) = \bigcup_{i,j} O \cap \deX\!\,^{ij} (\C)
\]
is open in $\deX\strich (\C)$. Let $G = \Aut_{K_0} (K)$. We equip $\deX_0 (\C) = \deX (\C) / G$ with the quotient topology. Using Lemma \ref{t7} one sees that
\[
\pr_{\eX} : \deX (\C) \longrightarrow \eX \; \text{and hence also} \; \pr_{\eX_0} : \deX_0 (\C) \longrightarrow \eX_0
\]
are continuous.

\begin{lemma} \label{t9}
For any arithmetic scheme $\eX_0$, the group $G$ acts by homeomorphisms on $\deX (\C)$ and the injective maps $F_{\nu} : \deX (\C) \hookrightarrow \deX (\C)$ for $\nu \in \Nh$ are continuous, closed and open. In particular $F_{\nu} (\deX (\C))$ is closed and open in $\deX (\C)$. 
\end{lemma}

\begin{proof}
We only need to prove the properties of $F_{\nu}$. First assume that $\eX_0 = \spec R_0$ is affine. Lemma \ref{t46n} implies the equations
\begin{equation}
\label{eq:21}
F_{\nu} (\deX (\C)) = \{ P \in \deX (\C) \mid P (\mu_{\nu} (K)) = 1 \} \subset \deX (\C) \; ,
\end{equation}
and
\[
\deX (\C) \setminus F_{\nu} (\deX (\C)) = \bigcup_{\chi \neq 1} \{ P \in \deX (\C) \mid P \, |_{\mu_{\nu} (K)} = \chi \} \; .
\]
Here $\chi$ is in the finite group $\Hom (\mu_{\nu} (K) , \C^{\times})$. It follows that $F_{\nu} (\deX (\C))$ is both closed and open in $\deX (\C)$. Now let $A$ be a closed subset of $\deX (\C)$. If $P_n \in F_{\nu} (A)$ and $P_n \to P$ in $\deX (\C)$, it follows that $P \in F_{\nu} (\deX (\C))$. Writing $P_n = F_{\nu} (Q_n)$ and $P = F_{\nu} (Q)$ it follows that $Q_n (r) \to Q (r)$ for any $r \in R$ since any $r$ has the form $r = s^{\nu}$ for some $s \in R$. Since $A$ is closed we see that $Q \in A$ and hence $P \in F_{\nu} (A)$. Thus $F_{\nu} (A)$ is closed. The continuity of $F_{\nu}$ being clear, it follows that $F_{\nu}$ induces a homeomorphism onto its image $F_{\nu} (\deX (\C))$. Since $F_{\nu} (\deX (\C))$ is open in $\deX (\C)$, it follows that the map $F_{\nu} : \deX (\C) \to \deX (\C)$ is also open. If $\eX_0$ is an arbitrary arithmetic scheme and $\eX'_0 \subset \eX_0$ an affine open (hence arithmetic) subscheme, we have
\begin{equation}
\label{eq:22}
F_{\nu} (Z) \cap \deX\strich (\C) = F_{\nu} (Z \cap \deX\strich (\C)) \; ,
\end{equation}
for any subset $Z \subset \deX (\C)$ since $F_{\nu} (\ex , \oP^{\times}) = (\ex , F_{\nu} (\oP^{\times}))$. If $Z$ is closed (open) in $\deX (\C)$, then by definition of the topology of $\deX (\C)$, the set $Z \cap \deX\strich (\C)$ is closed (open) in $\deX\strich (\C)$. From the affine case we know that $F_{\nu} (Z \cap \deX\strich (Z))$ is closed (open) in $\deX\strich (\C)$. Since this holds for all $\eX'_0$, using \eqref{eq:22} it follows that $F_{\nu} (Z)$ is closed (open) in $\deX (\C)$. 
\end{proof}

As usual let $\Nh_0$ be the submonoid of $\Nh$ generated by a set of prime numbers $\car \Nh_0 \supset \car \eX_0$ and let $\Q^{> 0}_0$ be the subgroup of $\Q^{> 0}$ generated by $\Nh_0$. Recall that the maps $F_{\nu}$ extend to bijections $F_{\nu} : \ceX (\C) \silo \ceX (\C)$. We give $\ceX (\C) = \colim_{\Nh_0} \deX (\C)$ the inductive limit topology. It is the finest topology such that for all $\nu \in \Nh_0$ the inclusions
\begin{equation}
\label{eq:23}
F^{-1}_{\nu} \, |_{\deX (\C)} : \deX (\C) \hookrightarrow \ceX (\C)
\end{equation}
are continuous. Thus $Z \subset \ceX (\C)$ is closed, resp. open if and only if $F_{\nu} (Z) \cap \deX (\C)$ is closed, resp. open in $\deX (\C)$ for all $\nu \in \Nh_0$.

\begin{prop} \label{t10}
a) $\deX (\C)$ is a closed and open subspace of $\ceX (\C)$.\\
b) $F_q : \ceX (\C) \to \ceX (\C)$ is a homeomorphism for every $q \in \Q^{> 0}_0$.\\
c) The group $G$ acts by homeomorphisms on $\ceX (\C)$. 
\end{prop}

\begin{proof}
a) By \eqref{eq:23} for $\nu = 1$, the inclusion $\deX (\C) \subset \ceX (\C)$ is continuous. We show that it is a closed (open) map. Thus let $Z \subset \deX (\C)$ be closed (open). Then $F_{\nu} (Z) \cap \deX (\C) = F_{\nu} (Z)$ is closed (open) in $\deX (\C)$ for every $\nu$ by Lemma \ref{t9}. Hence $Z$ is closed (open) in $\ceX (\C)$. \\
b) Fix $\nu \in \Nh_0$. If $O \subset \ceX (\C)$ is open, then for all $\mu \in \Nh_0$,
\[
F_{\mu} (F_{\nu} (O)) \cap \deX (\C) = F_{\mu\nu} (O) \cap \deX (\C)
\]
is open in $\deX (\C)$ by definition. Hence $F_{\nu} (O)$ is open in $\ceX (\C)$ as well. On the other hand
\begin{equation}
\label{eq:24}
F_{\mu} (F^{-1}_{\nu} (O)) \cap \deX (\C) = F^{-1}_{\nu} (F_{\mu} (O)) \cap \deX (\C) = (F_{\nu} \, |_{\deX (\C)})^{-1} (F_{\mu} (O)) \; .
\end{equation}
The map
\[
F_{\nu} \, |_{\deX (\C)} : \deX (\C) \longrightarrow \deX (\C) \subset \ceX (\C)
\]
is continuous and we have seen that $F_{\mu} (O)$ is open. Hence the subsets in \eqref{eq:24} are open for every $\mu$ and therefore $F^{-1}_{\nu} (O)$ is open. In conclusion, $F_{\nu}$ is a homeomorphism for every $\nu \in \Nh_0$ and therefore also for every $\nu \in \Q^{> 0}_0$. Assertion c) is clear.
\end{proof}

We give $\ceX_0 (\C) = \ceX (\C) / G$ the quotient topology. Then $\ceX_0 (\C)$ is homeomorphic to $\colim_{\Nh_0} \deX_0 (\C)$ with the inductive limit topology. The projections
\[
\dpi : \deX (\C) \longrightarrow \deX_0 (\C) \quad \text{and} \quad \cpi : \ceX (\C) \longrightarrow \ceX_0 (\C)
\]
are continuous and since $G$ acts by homeomorphisms, also open. Moreover the projections
\[
\pr_{\eX} : \ceX (\C) \longrightarrow \eX \quad \text{and} \quad \pr_{\eX_0} : \ceX_0 (\C) \longrightarrow \eX_0
\]
are continuous. It suffices to show that $\pr_{\eX}$ is continuous, i.e. that
\[
\pr_{\eX} \verk (F^{-1}_{\nu} \, |_{\deX (\C)}) : \deX (\C) \longrightarrow \eX
\]
is continuous for each $\nu \in \Nh_0$. Since $\pr_{\eX} = \pr_{\eX} \verk F_{\nu}$, this follows from the continuity of $\pr_{\eX} : \deX (\C) \to \eX$ which was noted before Lemma \ref{t9}.

Next we discuss the action of $G$ on $\deX (\C)$ and show that $\deX_0 (\C)$ is metrizable and in particular Hausdorff if $\eX_0$ is affine or more generally, if $\eX_0$ carries an invertible ample sheaf.

\begin{prop}
\label{t85}
Let $\eX_0$ be an arithmetic scheme. Then the right-action map
\[
\deX (\C) \times G \longrightarrow \deX (\C)
\]
is continuous.
\end{prop}

\begin{proof}
$\deX (\C)$ is covered by the open $G$-invariant subspaces $\deX\strich (\C)$ with $\eX'_0 \subset \eX_0$ an open affine subscheme. Since the assertion is local on $\deX (\C) \times G$ we may assume that $\eX_0 = \spec R_0$ is affine. For a sequence $(P_n , \sigma_n)$ in $\deX (\C) \times G$ with $(P_n , \sigma_n) \to (P , \sigma)$ we have to show that for all $r \in R$ we have
\[
P_n (\sigma_n (r)) \longrightarrow P (\sigma (r)) \quad \text{for} \; n \to \infty \; .
\]
Since $\sigma_n \sigma^{-1} \to e$ in $G$ and since the stabilizer of $r$ in $G$ is open, we have $P_n (\sigma_n (r)) = P_n (\sigma (r))$ for $n$ large enough. For $n \to \infty$ this tends to $P (\sigma (r))$.
\end{proof}

\begin{prop}
\label{t86}
For affine arithmetic schemes $\eX_0 = \spec R_0$, the space $\deX (\C)$ carries a $G$-invariant metric inducing the topology. If $\C$ is separable, the space $\deX (\C)$ is second countable and separable.
\end{prop}

\begin{proof}
It suffices to show that $\C^R = \map (R , \C)$ has a $G$-invariant metric inducing the topology. Since $R$ is countable by assumption, it is known that the metrics
\begin{equation}
\label{eq:861}
d (P , P') = \sum_{r \in R} a_r \frac{|P (r) - P' (r)|}{1 + |P (r) - P' (r)|} \quad \text{with} \; a_r > 0 \; , \; \sum_{r \in R} a_r < \infty
\end{equation}
induce the product topology on $\C^R$. The $G$-orbits $\eo$ of the $G$-action on $R$ are finite. Choose $b_{\eo} > 0$ with $\sum_{\eo} |\eo| \eb_{\eo} < \infty$. Then the metric 
\[
d (P , P') = \sum_{\eo} b_{\eo} \sum_{r \in \eo} \frac{|P (r) - P' (r)|}{1 + |P (r) - P' (r)|} 
\]
satisfies $d (P^{\sigma} , P^{'\sigma}) = d (P , P')$ for all $\sigma \in G$ and it is of the form \eqref{eq:861} with $a_r = b_{\eo}$ for $r \in \eo$. Note that
\[
\sum_r a_r = \sum_{\eo} \big( \sum_{r \in \eo} 1) b_{\eo} = \sum_{\eo} |\eo| b_{\eo} < \infty \; .
\]
If $\C$ is separable, the metric space $\C^R$ is separable as well and hence second countable. Hence the subspace $\deX (\C)$ is second countable and hence separable as well.
\end{proof}

\begin{prop}
\label{t87}
Let $(X , d)$ be a metric space with an action of a compact group $G$ of isometries and such that the maps $G \to X , \sigma \to x^{\sigma}$ are continuous for all $x \in X$. Then we obtain a metric $\delta$ on $X / G$ by setting
\[
\delta (xG , yG) = \min_{\sigma , \tau} d (x^{\sigma} , y^{\tau}) = \min_{\sigma} d (x^{\sigma} , y) = \min_{\tau} d (x , y^{\tau}) \; .
\]
The metric $\delta$ induces the quotient topology. In particular $X / G$ is Hausdorff. 
\end{prop}

\begin{proof}
The map $G \times G \to \R , (\sigma , \tau) \mapsto d (x^{\sigma} , y^{\tau})$ is continuous and $G \times G$ is compact. Hence the first minimum is obtained and similarly the other two. The equalities follow since $G$ acts by isometries. In particular $\delta (x G , yG) = 0$ implies that $d (x^{\sigma} , y^{\tau}) = 0$ for some $\sigma , \tau$, i.e. $x^{\sigma} = y^{\tau}$ and hence $xG = yG$. Thus $\delta$ is positive definite. We have
\[
d (x^{\sigma} , z^{\tau}) \le d (x^{\sigma} , y) + d (y , z^{\tau}) \; .
\]
Hence
\[
\inf_{\sigma} d (x^{\sigma} , z^{\tau}) \le \inf_{\sigma} d (x^{\sigma} , y) + d (y , z^{\tau}) = \delta (xG , yG) + d (y , z^{\tau}) 
\]
and therefore: 
\begin{align*}
\delta (xG , zG) & = \inf_{\sigma , \tau} d (x^{\sigma} , z^{\tau}) \le \delta (xG , yG) + \inf_{\tau \in G} d (y , z^{\tau}) \\
& = \delta (xG , yG) + \delta (yG , zG) \; .
\end{align*}
The estimate $\delta (xG , yG) \le d (x,y)$ shows that the projection map $X \to (X / G)_{\delta}$ is continuous, where $(X / G)_{\delta}$ is $X / G$ with the topology induced by $\delta$. Hence the map $X / G \to (X / G)_{\delta}$ is continuous. To show that $(X / G)_{\delta} \to X / G$ is continuous, since $(X / G)_{\delta}$ is first countable, we can argue with sequences. Assume $\delta (x_n G , xG) \to 0$ for $n \to \infty$. This means that there are $\sigma_n \in G$ with $d (x^{\sigma_n}_n , x) \to 0$ i.e. $x^{\sigma_n}_n \to x$ in $X$. It follows that $x_n G = x^{\sigma_n}_n G \to xG$ in $X / G$.
\end{proof}

We now obtain the desired result about the topology of $\deX_0 (\C)$:

\begin{cor}
\label{t88}
For any affine arithmetic scheme $\eX_0 = \spec R_0$ there are a $G$-invariant metric $d$ on $\deX (\C)$ and a metric $\delta$ on $\deX_0 (\C) = \deX (\C) / G$ inducing the topology on $\deX (\C)$ and the (quotient-)topology on $\deX_0 (\C)$. If $\pi : \deX (\C) \to \deX_0 (\C)$ denotes the quotient map, we have
\[
\delta (\pi (P) , \pi (P')) \le d (P, P') \quad \text{for all} \; P ,  P' \in \deX (\C) \; .
\]
The topological space $\deX_0 (\C)$ is metrizable and separable and in particular Hausdorff.
\end{cor}

This result has the following consequence:

\begin{cor} \label{t89n}
Let $\eX_0$ be an arithmetic scheme which carries an ample invertible sheaf. Then the spaces $\deX (\C) , \deX_0 (\C) , \ceX (\C)$ and $\ceX_0 (\C)$ are Hausdorff.
\end{cor}

\begin{proof}
By \cite[Lemma 28.29.5, (3)]{stacks} any two points of $\eX_0$ lie in a common open affine subscheme of $\eX_0$. Hence, given points $(\ex , \oP^{\times})$ and $(\ey , \onQ^{\times})$ of $\deX (\C)$, there is an open affine $\eX'_0 \subset \eX_0$ containing the images of $\ex$ and $\ey$ in $\eX_0$. It follows that $(\ex , \oP^{\times})$ and $(\ey , \onQ^{\times})$ lie in the Hausdorff open subspace $\deX' (\C)$ of $\deX (\C)$. Hence $\eX (\C)$ is Hausdorff. The points $(\ex , \oP^{\times}) G$ and $(\ey , \onQ^{\times}) G$ of $\deX_0 (\C)$ lie in $\deX'_0 (\C)$ which is Hausdorff by Corollary \ref{t88}. Hence $\deX_0 (\C)$ is Hausdorff as well. Since any two points of $\ceX (\C)$ resp. $\ceX_0 (\C)$ can be mapped to the open subspace $\deX (\C)$ resp. $\deX_0 (\C)$ by a suitable Frobenius homeomorphism $F_r$, it follows that $\ceX (\C)$ and $\ceX_0 (\C)$ are Hausdorff as well. 
\end{proof}

\begin{rem}
We do not know, whether $\deX (\C)$ or $\deX_0 (\C)$ are Hausdorff if we only assume that $\eX_0$ is separated: The functor $X \mapsto W_{\rat} (X) (\C)$ is far from commuting with cartesian products.
\end{rem}

The following theorem contributes a basic insight into the structure of the $\R^{> 0}$-dynamical systems studied in section \ref{sec:3}.

Consider the natural continuous maps from $\deX (\C) , \deX_0 (\C) , \ceX (\C)$ and $\ceX_0 (\C)$ to $\spec \Z$ and denote their fibres over $\spec \F_p$ resp. $\spec \Q$ by $\deX (\C)_p$ resp. $\deX_0 (\C)_{\Q}$ etc. Consider the $G$-invariant subspace
\[
\deX (\C)_{in} = \{ (\ex , \oP^{\times}) \in \deX (\C) \mid \oP^{\times} \, |_{\mu (\kappa (\ex))} \; \text{is injective} \, \} \; .
\]
We give
\[
\deX (\C)_{p , in} = \deX (\C)_{in} \cap \deX (\C)_p \quad \text{and} \quad \deX (\C)_{\Q , in} = \deX (\C)_{in} \cap \deX (\C)_{\Q}
\]
the subspace topologies inside $\deX (\C)$. The quotient topologies on 
\[
\deX_0 (\C)_{p , in} := \deX (\C)_{p , in} / G \quad \text{and} \quad \deX_0 (\C)_{\Q , in} := \deX (\C)_{\Q , in} / G
\]
agree with the subspace topologies inside of $\deX_0 (\C)$. For a prime number $p \in \car \eX_0$ the Frobenius endomorphisms $F_p$ of $\deX (\C)$ and $\deX_0 (\C)$ restrict to homeomorphisms $F_p$ of $\deX (\C)_p$ and $\deX_0 (\C)_p$. Hence the subgroup $p^{\Z}$ of $\Q^{> 0}_0$ acts on these spaces with $p$ acting by $F_p$, as usual. 

\begin{theorem}
\label{t710n}
The following canonical $\R^{> 0}$-equivariant maps are continuous bijections: 
\[
\deX (\C)_{\Q , in} \times \R^{> 0} \amalg \coprod_{p \in \car \eX_0} \deX (\C)_{p , in} \times_{p^{\Z}} \R^{> 0} \silo X = \ceX (\C)_{\Eh_{\tors}} \times_{\Q^{> 0}_0} \R^{> 0}
\]
and
\[
\deX_0 (\C)_{\Q , in} \times \R^{> 0} \amalg \coprod_{p \in \car \eX_0} \deX_0 (\C)_{p , in} \times_{p^{\Z}} \R^{> 0} \silo X_0 = \ceX_0 (\C)_{\Eh_{\tors}} \times_{\Q^{> 0}_0} \R^{> 0} \; .
\]
\end{theorem}

\begin{proof}
By Proposition \ref{2xn}, every $\Q^{> 0}_0$-orbit in $\ceX (\C)_{\Q , \Eh_{\tors}} \times \R^{> 0}$ contains a unique representative in $\deX (\C)_{\Q , in} \times \R^{> 0}$. Hence the natural map
\[
\deX (\C)_{\Q , in} \times \R^{> 0} \longrightarrow \ceX (\C)_{\Q , \Eh_{\tors}} \times_{\Q^{> 0}_0} \R^{> 0}
\]
is a continuous bijection. By the same proposition, the natural map
\[
\deX (\C)_{p , in} \times_{p^{\Z}} \R^{> 0} \longrightarrow \ceX (\C)_{p , \Eh_{\tors}} \times_{\Q^{> 0}_0} \R^{> 0}
\]
is a continuous bijection as well. Composing these maps with the inclusions into $X$ we obtain a continuous $G$-equivariant bijection from the coproduct into $X$. The second continuous bijection in the theorem follows by passing to the $G$-orbit spaces.
\end{proof}

\begin{rems}
1) The $\R^{> 0}$-equivariant bijections in Theorem \ref{t710n} make it obvious that any isotropy group of the $\R^{> 0}$-action on $X_0$ is either trivial or a subgroup of $p^{\Z}$ for some $p \in \car \eX_0$. A more precise statement follows from Theorem \ref{t6n} a): If $(\R^{> 0})_{x_0} \neq 1$ for some point $x_0 = [(\ex , \oP^{\times}) G , u] \in X_0$ then $p = \car \kappa (\ex_0) > 0$ where $\ex_0$ is the image of $\ex$ in $\eX_0$ and $(\R^{> 0})_{x_0} \subset N^{\Z}_{\ex_0}$ where $N_{\ex_0} = |\kappa (\ex_0) \cap \oF_p|$.\\
2) The continuous bijections in Theorem \ref{t710n} are not homeomorphisms in general. If $\eX_0$ is flat of finite type over $\spec \Z$ and $\C$ is the field of complex numbers it follows from the last remark in section \ref{sec:10n} or the results in section \ref{sec:5n} that the fibre $X_{\eta}$ over $\eta = \spec K$ of the space on the right is connected. The corresponding fibre on the left is the space
\[
\{ \oP^{\times} \in \Hom (K^{\times} , \C^{\times}) \mid \oP^{\times} \, |_{\mu (K)} \; \; \text{injective} \, \} \times \R^{> 0} \; ,
\]
whose uncountably many connected components are parametrized by $\Hom_{\inj} (\mu (K) , S^1)$, use e.g. \cite[Proposition 5.1]{KS}. It also follows that in the second continuous bijection the fibre $X_{0 \eta_0}$ of $X_0$ over $\eta_0 = \spec K_0$ is connected, whereas on the left the fibre over $\eta_0$ has $d_{\mu} = (\hZ^{\times} : \kappa (G))$ connected components. Here $\kappa$ is the cyclotomic character defined in \eqref{eq:57} below.
\end{rems}

Given an admissible class $\Eh$ as in Definition \ref{t41n} we equip $\deX (\C)_{\Eh}$ and $\deX_0 (\C)_{\Eh}$ with the subspace topologies of $\deX (\C)$ and $\deX_0 (\C)$. Equip $\deX (\C)_{\Eh} / G$ with the quotient topology. It is clear that the natural bijection $\deX (\C)_{\Eh} / G \silo \deX_0 (\C)_{\Eh}$ is continuous. It is also open since $G$ acts by homeomorphisms and hence we may identify $\deX (\C)_{\Eh} / G$ and $\deX_0 (\C)_{\Eh}$ as topological spaces. We give
\[
\ceX (\C)_{\Eh} = \colim_{\Nh_0} \deX (\C)_{\Eh} \quad \text{and} \quad \ceX_0 (\C)_{\Eh} = \colim_{\Nh_0} \deX_0 (\C)_{\Eh}
\]
the inductive limit topologies. They agree with the subspace topologies via $\ceX (\C)_{\Eh} \subset \ceX (\C)$ and $\ceX_0 (\C)_{\Eh} \subset \ceX_0 (\C)$ because the subspaces $F^{-1}_{\nu} \deX (\C)$ and $F^{-1}_{\nu} \deX_0 (\C)$ are open in $\ceX (\C)$ resp. $\ceX_0 (\C)$ for all $\nu \in \Nh_0$. As above, the natural continuous bijection $\ceX (\C)_{\Eh} / G \silo \ceX_0 (\C)_{\Eh}$ is a homeomorphism. All preceding results in this section remain true if we replace $\deX (\C)$ etc. by $\deX (\C)_{\Eh}$ etc. In the rest of the section, for simplicity, we assume that $\eX_0 \to \spec \Z$ is surjective and hence $\car K_0 = 0$ and $\Nh_0 = \Nh$. We discuss a relation of $\ceX (\C)$ with objects studied by Bost and Connes \cite{BC}. Fix an isomorphism $\iota : \mu (K) \silo \mu (\C)$ and consider $\dH = \Hom (\mu (K) , \mu (\C))$ with the topology of pointwise convergence. Using $\iota$, we have a topological isomorphism
\begin{equation}
\label{eq:55n}
\hZ \silo \dH \; , \; a \mapsto (\zeta \mapsto \iota (\zeta)^a) \; .
\end{equation}
The group $G$ acts continuously on $\dH$ via its action on $\mu (K)$. Let
\begin{equation}
\label{eq:57}
\kappa : G \twoheadrightarrow \Aut (K_0 (\mu (K)) / K_0) = \Aut (\Q (\mu (K)) / K_0 \cap \Q (\mu (K)) \subset \hZ^{\times}
\end{equation}
be the cyclotomic character. Then we have $\chi^{\sigma} = \chi^{\kappa (\sigma)}$ pointwise for $\chi \in \dH$ and $\sigma \in G$. The $\Nh$-action by exponentiation of characters is just the action by multiplication on the abelian group $\dH$. Consider
\[
\cH = \colim_{\Nh} \dH = \dH \otimes_{\Z} \Q
\]
with the inductive limit topology and the continuous $G \times \Q^{> 0}$-action. Using $\iota$ we get a topological isomorphism of $\cH$ with the ring of finite ad\`eles
\begin{equation}
\label{eq:58}
\A_f = \hZ \otimes \Q \silo \cH \; .
\end{equation}
For $\ex \in \eX$ consider the map $i_{\ex} : \mu (K) = \mu (\Gamma (\eX , \Oh) \to \mu (\kappa (\ex))$. The $G$- and $\Nh$-equivariant map
\begin{equation}
\label{eq:59}
r : \deX (\C) \to \dH \; , \; (\ex , \oP^{\times}) \mapsto (\oP^{\times} \, |_{\mu (\kappa (\ex)}) \verk i_{\ex}
\end{equation}
is continuous. It suffices to check this in the affine case $\eX = \spec R$, where continuity is clear since $r (P) = P \, |_{\mu (K)}$ for $P \in \deX (\C)$. Passing to the colimit over $\Nh$ we obtain a continuous $G$- and $\Q^{> 0}$-equivariant map
\begin{equation}
\label{eq:60}
r : \ceX (\C) \to \cH \; .
\end{equation}
Using \eqref{eq:55n} and \eqref{eq:58} and $\iota$, we obtain continuous $G$- and $\Nh$-resp. $\Q^{> 0}$-equivariant maps
\begin{equation}
\label{eq:61}
r : \deX (\C) \to \hZ \quad \text{and} \quad r : \ceX (\C) \to \A_f = \hZ \otimes \Q \; .
\end{equation}
Here $\sigma \in G$ acts on the right by multiplication with $\kappa (\sigma) \in \hZ^{\times}$. Thus $\tX = \ceX (\C) \times \R^{> 0}$ maps naturally and $\Q^{> 0}$-equivariantly to $\A^{> 0} = \A_f \times \R^{> 0} \subset \A$ with its $\Q^{> 0}$-action. The work of Bost-Connes \cite{BC} and Connes on the $\Q^{> 0}$-action on $\A$ is very well known, c.f. \cite{C} for an overview. One may wonder if it is possible to study the $\Q^{> 0}$-action on $\tX$ from their non-commutative point of view and whether that could lead to a higher dimensional analogue of their work. 

Consider the $G$-invariant subsets $\dH_{\fin}$ and $\dH_{\fin , p}$ of $\dH$ consisting of the characters $\chi$ with finite kernel, resp. with $\chi \, |_{\mu_{p^{\infty}} (K)} = 1$ and $\ker (\chi \, |_{\mu_{(p)} (K)})$ finite. The sets $\dH_{\fin}$ and $\dH_{\fin , p}$ are $\Nh$-foreward and backward invariant. We give
\begin{equation}
\label{eq:62}
\dH_{\Eh_{\tors}} := \dH_{\fin} \; \dot{\cup} \; \dot{\bigcup}_{p \in \car \eX_0} \dH_{\fin, p} \subset \dH
\end{equation}
the subspace topology. Using Lemma \ref{t46n}, via $\iota$ we have an identification
\begin{equation}
\label{eq:63}
\dH_{\Eh_{\tors}} \overset{\iota}{=} \Nh \hZ^{\times} \; \dot{\cup} \; \bigcup_p \Nh (\hZ^{\times}_{(p)} \times 0) \subset \Q^{>0} \hZ = \A_f \; .
\end{equation}
Set $\cH_{\Eh_{\tors}} = \colim_{\Nh} \dH_{\Eh_{\tors}} \subset \cH$ with the colimit topology which is the same as the subspace topology in $\cH$ since $\dH$ is open in $\cH$. Then
\begin{equation}
\label{eq:64}
\cH_{\Eh_{\tors}} \overset{\iota}{=} \Q^{> 0} \hZ^{\times} \; \dot{\cup} \; \dot{\bigcup_p} \Q^{> 0} (\hZ_{(p)} \times 0) \subset \A_f \; .
\end{equation}
The topology induced on $\Q^{> 0}\hZ^{\times} = \A^{\times}_f$ by $\A_f$ is not the id\`ele topology, as is well known. Note that the canonical map
\begin{equation}
\label{eq:65}
\tau : \cH_{\Eh_{\tors}} \longrightarrow \spec \Z \; ,
\end{equation}
sending $\Q^{> 0} \hZ^{\times}$ to $(0)$ and $\Q^{> 0} (\hZ^{\times}_{(p)} \times 0)$ to $(p)$ is continuous. By definition of an admissible class $\Eh$, the map \eqref{eq:60} restricts to a continuous map
\begin{equation}
\label{eq:66}
r : \ceX (\C)_{\Eh} \longrightarrow \cH_{\Eh_{\tors}} \overset{\iota}{=} \Q^{> 0} \hZ^{\times} \; \dot{\cup} \; \dot{\bigcup_p} \Q^{> 0} (\hZ^{\times}_{(p)} \otimes 0) \subset \A_f
\end{equation}
Passing to the quotient by the $\hZ^{\times}$-action on the right hand side and making the identification
\[
\Q^{> 0} (\hZ^{\times}_{(p)} \times 0) / \hZ^{\times} = \Q^{> 0}_{(p)} = \Q^{> 0} / p^{\Z} \; ,
\]
we obtain a continuous map which is independent of $\iota$,
\[
\rho : \ceX (\C)_{\Eh} \longrightarrow \Q^{> 0} \; \dot{\cup} \; \dot{\bigcup_p} \Q^{> 0} / p^{\Z} \; .
\]
Here the right hand side carries the quotient topology of $\cH_{\Eh_{\tors}}$ or equivalently the subspace topology of $\A_f / \hZ^{\times}$. For $\Eh = \Eh_{\tors}$ this is the same map as in \eqref{eq:16}. The map $r$ fits into the commutative diagram
\begin{equation}
\label{eq:67}
\vcenter{\xymatrix{
\ceX (\C)_{\Eh} \ar[r]^r \ar[d] & \cH_{\Eh_{\tors}} \ar[d]^{\tau} \\
\eX \ar[r] & \spec \Z \; .
}}
\end{equation}
The map \eqref{eq:66} induces a $G$- and $\Q^{> 0}$-equivariant continuous map
\begin{equation}
\label{eq:68}
\tX = \ceX (\C)_{\Eh} \times \R^{> 0} \longrightarrow \cH_{\Eh_{\tors}} \times \R^{> 0} \overset{\iota}{=} \A^{> 0} \;  \dot{\cup} \; \dot{\bigcup_p} \A^{> 0}_{(p)} \subset \A \; .
\end{equation}
Here $\A^{> 0} = \A^{\times}_f \times \R^{> 0}$ and $\A^{> 0}_{(p)} = \Q^{> 0} (\hZ^{\times}_{(p)} \times 0) \times \R^{> 0}$. The $\Q^{> 0}$-action on $\cH_{\Eh_{\tors}} \times \R^{> 0}$ is not properly discontinuous. In section \ref{sec:10n}, we will see that this works to our advantage.
\section{The closure of the periodic orbits in $X_0$} \label{sec9_nneu}
In this section $\C$ denotes the complex number field, and we take $\Nh_0 = \Nh$. 

For an integral normal scheme $\eX_0$ of finite type over $\spec \Z$ with $\dim \eX_0 \ge 1$ the dynamical system $X_0 = \ceX_0 (\C) \times_{\Q^{> 0}} \R^{> 0}$ is infinite dimensional, whereas we are searching for a system of dimension $2 \dim \eX_0 + 1$, e.g. $3$ in the case of $\eX_0 = \spec \Z$. Since we want to keep periodic orbits for every closed point of $\eX_0$ one idea would be to replace $X_0$ by the dynamical system $Y_0$ obtained as the topological closure of the union of {\it all} periodic orbits coming from closed points of $\eX_0$. In this section we will show that the system $Y_0$ is still infinite-dimensional: Namely, for one-dimensional $\eX_0$, flat over $\spec \Z$ and conditionally for all $\eX_0$ we have
\[
Y_0 = \ceX_0 (S^1) \times_{\Q^{> 0}} \R^{> 0} \; .
\]
Here $\ceX_0 (S^1) = \ceX (S^1) / G$ where $\ceX (S^1) = \colim_{\Nh} \deX (S^1)$ and $\deX (S^1)$ is the subspace of $\deX (\C)$ consisting of points $(\ex , \oP^{\times})$ with $\oP^{\times} : \kappa (\ex)^{\times} \to S^1$ a unitary character. This follows from Theorem \ref{t92nn} below which depends on the following assertion which is known for number rings $A$ and that we expect to hold in general. 

\begin{claim} \label{t91nn}
Let $A$ be a finitely generated integrally closed domain with quotient field $\kappa$ and let $x_1 , \ldots , x_n \in \kappa^{\times}$ be elements whose images $x_1 \otimes 1 , \ldots , x_n \otimes 1$ are $\Q$-linearly independent in $\kappa^{\times} \otimes \Q$.  Furthermore let $N \ge 1$ be an integer prime to $\car \kappa$, $l \nmid N$ a prime number with $l \neq \car \kappa$ and $\nu_1 , \ldots , \nu_n \ge 1$ integers. Let $\Mh$ be the set of maximal ideals $\en$ of $A$ with the following properties:\\
1) $N \notin \en$ and $l \notin \en$\\
2) $x_1 , \ldots , x_n \in A_{\en} \setminus \en A_{\en}$ and the orders $o (\ox_j)$ of the elements $\ox_j = x_j \mod \en A_{\en}$ in $(A_{\en} / \en A_{\en})^{\times}$ are prime to $N$ for $1 \le j \le n$.\\
3) $\ord_l (o (\ox_j)) = \nu_j$ \quad for $1 \le j \le n$.\\
Then $\Mh$ is non-empty.
\end{claim}
If $\dim A = 0$ i.e. if $A = \kappa$ is a finite field, the claim holds trivially.

If $\spec A$ is a non-empty open subscheme of $\spec \eo_{\kappa}$ for some number field $\kappa$, then the set $\Mh$ in Claim \ref{t91nn} is always infinite, and it even has a positive Dirichlet density. This follows from \cite[Theorem 1]{perucca}.

Consider the subset $\ceX (\C)_{\per} = \colim_{\Nh} \deX (\C)_{\per} \subset \ceX (\C)$ where
\[
\deX (\C)_{\per} = \{ (\ex , \oP^{\times}) \in \deX (\C) \mid \kappa (\ex) \cong \oF_p \; \text{for some $p$ and} \; \ker \oP^{\times} \; \text{is finite} \} \; .
\]

\begin{theorem} \label{t92nn}
Let $\eX_0$ be an integral normal scheme of finite type over $\spec \Z$. If $\dim \eX_0 \ge 2$ or $\dim \eX_0 = 1$ and $\car K_0 > 0$ assume that Claim \ref{t91nn} holds. Then we have
\[
\overline{\ceX (\C)_{\per}} = \ceX (S^1) \quad \text{and} \quad \overline{\ceX_0 (\C)_{\per}} = \ceX_0 (S^1) \; .
\]
\end{theorem}

\begin{proof}
It suffices to show the first assertion. Any character of the torsion group $\oF^{\times}_p$ takes values in the roots of unity. Hence $\deX (\C)_{\per} \subset \deX (S^1)$ and $\ceX (\C)_{\per} \subset \ceX (S^1)$. Since $\deX (S^1)$ is closed in $\deX (\C)$ the subspace $\ceX (S^1)$ is closed in $\ceX (\C)$ as well and $\overline{\ceX (\C)_{\per}} \subset \ceX (S^1)$. Equality will follow if we show that $\deX (\C)_{\per}$ is dense in $\deX (S^1)$. For this we may assume that $\eX_0 = \spec R_0$ and hence $\eX = \spec R$ are affine. Let $(\ex , \oP^{\times})$ be a point of $\deX (S^1)$. Let $\ep$ be the prime ideal of $R$ corresponding to $\ex$ and set $\oR = R / \ep$. We have $\kappa (\ex) = \Quot \oR$ and $\overline{\{ \ex \} } = \spec \oR$, where $\overline{\{ \ex \} }$ carries the reduced scheme structure. By Remark \ref{t34} we may view the point $(\ex , \oP^{\times}) \in \deX (S^1)$ as a multiplicative map $P : R \to S^1_0 = S^1 \cup \{ 0 \}$ with $P^{-1} (0) = \ep$ which factors as $P : R \to \oR \xrightarrow{\oP} S^1_0$. Then $\oP^{\times}$ is the unique multiplicative extension of $\oP \, |_{\oR \setminus 0}$ to a character $\kappa (\ex)^{\times} \to S^1$. 
We have to show that every neighborhood of $P = (\ex , \oP^{\times})$ in $\deX (\C)$ contains a point from $\deX (\C)_{\per}$. This means that for every $\varepsilon > 0$ and every finite set $S \subset R$ there is a point $Q \in \deX (\C)_{\per}$ with
\begin{equation}
 \label{eq:9nn1}
 |P (r) - Q (r)| < \varepsilon \quad \text{for all} \; r \in S \; .
\end{equation}
In the situation of Theorem \ref{t92nn} and with the preceeding notations we have the following result:

\begin{lemma} \label{t93nn}
Let $\psi : \kappa (\ex)^{\times} \to S^1$ be a character. Fix a finite subset $T \subset \kappa (\ex)^{\times}$ and some $\varepsilon > 0$. Then there are a maximal ideal $\oemm$ of $\oR$ and a character
\[
\chi : (\oR_{\oemm} / \oemm R_{\oemm})^{\times} = (\oR / \oemm)^{\times} \longrightarrow S^1
\]
with finite kernel, such that $T \subset \oR_{\oemm} \setminus \oemm \oR_{\oemm}$ and
\begin{equation}
\label{eq:9nn2}
|\psi (t) - \chi (t \mod \oemm)| < \varepsilon \quad \text{for all} \; t \in T \; .
\end{equation}
\end{lemma}

Using the lemma, we find some $Q \in \deX (\C)_{\per}$ with \eqref{eq:9nn1} as follows. Take $\psi = \oP^{\times}$ and $T = \oS \setminus 0$ where $\oS$ is the image of $S$ in $\oR$. Let $\emm$ be the inverse image of $\oemm$ in $R$. It is a maximal ideal with $R / \emm = \oR / \oemm$. Using the character $\chi$ of the lemma we obtain a point $Q = (\emm , \chi) \in \deX (\C)_{\per}$. For $r \in S \setminus \ep$ i.e. $\barr \in \oS \setminus 0 = T$ we have $P (r) = \oP^{\times} (\barr) = \psi (\barr)$ and $Q (r) = \chi (\barr \mod \oemm)$. For $r \in S \setminus \ep$, estimate \eqref{eq:9nn2} therefore implies \eqref{eq:9nn1}. For $r \in S \cap \ep$ we have $P (r) = 0 = Q (r)$ and \eqref{eq:9nn1} holds trivially.
\end{proof}

\begin{proofof} {\it Lemma \ref{t93nn}.} 
Recall that $\kappa (\ex) = \Quot \oR$. Since $\kappa (\ex)^{\times}$ is divisible we have a natural exact sequence
\[
1 \longrightarrow \mu (\kappa (\ex)) \longrightarrow \kappa (\ex)^{\times} \longrightarrow \kappa (\ex)^{\times} \otimes \Q \longrightarrow 0 \; .
\]
Let $x_1 , x_2 , \ldots$ be preimages in $\kappa (\ex)^{\times}$ of a basis of $\kappa (\ex)^{\times} \otimes \Q$. Then the following facts hold:\\
A) If $x^{\nu_1}_{i_1} \cdots x^{\nu_k}_{i_k} \in \mu (\kappa (\ex))$ for some $i_1 < \ldots < i_k$ and $\nu_1 , \ldots , \nu_k \in \Z$ then $\nu_1 = \ldots = \nu_k = 0$. In particular $x_1 , x_2 , \ldots$ are $\Z$-linearly independent in $\kappa (\ex)^{\times}$.\\
B) Every element $x \in \kappa (\ex)^{\times}$ has a unique representation
\begin{equation}
\label{eq:9nn3}
x = \zeta x^{\nu_1}_1 x^{\nu_2}_2 \cdots \quad \text{with} \; \zeta \in \mu (\kappa (\ex)) \; , \; \nu_j \in \Z \; , \; \nu_j = 0 \; \text{for} \; j \gg 0 \; .
\end{equation}
The elements of $T$ can be written in the form \eqref{eq:9nn3} where only finitely many $x_1 , \ldots , x_n$ occur with non-zero exponent, and only finitely many roots of unity $\zeta$. The subgroup of $\kappa (\ex)^{\times}$ generated by these $\zeta$'s is cyclic. Let $\eta$ be a generator. Since $\psi$ is multiplicative, it suffices to prove the lemma for sets $T$ of the form $T = \{ \eta , x_1 , \ldots , x_n \}$. Let $\ex_0$ be the image of $\ex$ in $\eX_0$, corresponding to the prime ideal $\ep_0 = \ep \cap R_0$ of $R_0$. In $\kappa (\ex)$ consider the finite extension $\kappa = \kappa (\ex_0) (\eta , x_1 , \ldots , x_n)$ generated by $T$. Let $A$ be the normalization of the (usually non-normal) ring $\oR_0 = R_0 / \ep_0$ in $\kappa$. We claim that $\overline{ \{ x \} } = \spec \oR$ is normal, i.e. that $\oR$ is integrally closed in its quotient field $\kappa (\ex)$. Namely, let $\onQ (T) \in \oR [T]$ be a monic polynomial. Lift $\onQ$ to a monic polynomial $Q \in R [T]$ of the same degree. Since $R$ is integrally closed in its algebraically closed quotient field the polynomial $Q$ decomposes in $R [T]$ into a product of monic linear factors. Therefore $\onQ$ decomposes into monic linear factors in $\oR [T]$ as well. Hence all solutions $x \in \kappa (\ex)$ of $\onQ (x) = 0$ lie in $\oR$, and we have shown that $\oR$ is integrally closed. It follows that $\oR$ is the normalization of $\oR_0$ in $\kappa (\ex)$. Hence we have inclusions
\[
\oR_0 \subset A \subset \oR \quad \text{and} \quad T \subset \kappa^{\times} \; .
\]
Since $\Z$ is Nagata, the ring $A$ is finite over $\oR_0$ and hence of finite type over $\Z$. Note also that
\[
\dim A = \dim \oR_0 \le \dim R_0 = \dim \eX_0 \; .
\]
By assumption we may apply Claim \ref{t91nn} to our ring $A$, the chosen elements $x_1 , \ldots , x_n \in T$ and the order $N$ of the root of unity $\eta$ above. Note that $N$ is prime to $\car \kappa = \car \kappa (\ex)$. Moreover we let $\nu_j = j$ for $1 \le j \le n$ and choose the prime number $l \nmid N$ such that $l \neq \car \kappa$ and $2 \pi < \varepsilon l$. Let $\en$ be a maximal ideal of $A$ for which 1)--3) in Claim \ref{t91nn} hold with these choices. Then $k = A_{\en} / \en A_{\en} = A / \en$ is a finite field and the order of $\ox_j = x_j \mod \en$ in $k^{\times}$ is prime to $N$ for each $1 \le j \le n$. Let $k^{\times} (N)$ resp. $k^{\times} (N')$ denote the subgroups of $k^{\times}$ consisting of elements of order dividing $N^{\infty}$ resp. of order prime to $N$. Then $k^{\times}$ is the direct product $k^{\times} = k^{\times} (N) k^{\times} (N')$, and $\eta \in k^{\times} (N)$ and $\ox_1 , \ldots , \ox_n \in k^{\times} (N')$. Since $N \notin \en$, reduction gives an isomorphism
\[
\mu_N (A) = \langle \eta \rangle \silo \mu_N (k) = \langle \oeta \rangle \; .
\]
On $\langle \oeta \rangle \subset k^{\times} (N)$ we therefore obtain a character $\chi : \langle \oeta \rangle \to S^1$ by setting $\chi (\oeta^i) = \psi (\eta^i)$. We extend $\chi$ to a character $\chi : k^{\times} (N) \to S^1$. In order to define $\chi$ on $k^{\times} (N')$, we decompose $k^{\times} (N')$ further:
\begin{equation}
\label{eq:9nn4}
k^{\times} (N') = k^{\times} ((Nl)') k^{\times} (l) \; .
\end{equation}
Let
\[
\ox_j = \ox'_j \ox''_j
\]
be the corresponding decomposition of $\ox_j$ for $1 \le j \le n$. Next we define $\chi$ on the subgroup $\mu_{l^n} (k) \subset k^{\times} (l)$. By the choice of $\en$, we have $\ord_l (o (\ox_j)) = j$ for $1 \le j \le n$. Hence $\ox''_j$ has order $l^j$ and it follows that
\begin{equation}
\label{eq:9nn5}
\ox''_j = \ox^{l^{n-j}}_n \quad \text{for} \; 1 \le j \le n \; .
\end{equation}
Choose $\theta_j \in \R / \Z$ s.t.
\begin{equation}
\label{eq:9nn6}
\psi (x_j) = \exp 2 \pi i \theta_j \quad \text{for} \; 1 \le j \le n \; .
\end{equation}
For $\alpha \in \R / \Z$ we set $((\alpha)) = \min_{\nu \in \Z} |\alpha - \nu|$. We will now define an integer
\[
B = b_0 + b_1 l + \ldots + b_{n-1} l^{n-1} \quad \text{with} \; b_i \in \{ 0 , 1 , \ldots , l-1 \} 
\]
such that
\begin{equation}
 \label{eq:9nn7}
\big( \big( \frac{B}{l^j} - \theta_j \big) \big) \le \frac{1}{j} \quad \text{for} \; 1 \le j \le n \; .
\end{equation}
For $j = 1$ we have
\[
\big( \big( \frac{B}{l} - \theta_1 \big) \big) = \big( \big( \frac{b_0}{l} - \theta_1 \big) \big) \le \frac{1}{l} 
\]
for a suitable $b_0 \in \{ 0 , \ldots , l-1 \}$. Next:
\[
\big( \big( \frac{B}{l^2} - \theta_2 \big) \big) = \big( \big( \frac{b_1}{l} - \big( \theta_2 - \frac{b_0}{l^2} \big) \big) \big) \le \frac{1}{l}
\]
for a suitable $b_1$, and so on. Thus in the $n$-th and last step we find $b_{n-1}$ such that
\[
\big( \big( \frac{B}{l^n} - \theta_n \big) \big) = \big( \big( \frac{b_{n-1}}{l} - \big( \theta_n - \frac{b_0}{l^n} - \ldots - \frac{b_{n-2}}{l^2} \big) \big) \big) \le \frac{1}{l} \; .
\]
For $\alpha , \beta \in \R$ we have the estimate
\[
|\exp 2 \pi i \alpha - \exp 2 \pi i \beta| \le 2 \pi ((\alpha - \beta)) \; .
\]
It follows from \eqref{eq:9nn6}, \eqref{eq:9nn7} and our choice of $l$ that
\begin{equation}
\label{eq:9nn8}
\big| \psi (x_j) - \exp \frac{2 \pi i B}{l^j} \big| \le \frac{2\pi}{l} < \varepsilon \quad \text{for} \; 1 \le j \le n \; .
\end{equation}
We define $\chi$ on $\mu_{l^n} (k) = \langle \ox''_n\rangle$ by setting
\[
\chi (\ox''_n) = \exp \frac{2 \pi i B}{l^n} \; .
\]
This is possible since $\ox''_n$ is a primitive $l^n$-th root of unity. Using equation \eqref{eq:9nn5} it follows that
\begin{equation}
\label{eq:9nn9}
\chi (\ox''_j) = \exp \frac{2\pi iB}{l^j} \quad \text{for} \; 1 \le j \le n \; .
\end{equation}
We now extend $\chi$ from $\mu_{l^n} (k)$ to a character on $k^{\times} (l)$, and using the decomposition \eqref{eq:9nn4} to a character $\chi : k^{\times} (N') \to S^1$ by setting $\chi = 1$ on $k^{\times} ((Nl)')$. By construction we then have $\chi (\ox_j) = \chi (\ox''_j)$ and using \eqref{eq:9nn8} and \eqref{eq:9nn9} we get
\begin{equation}
\label{eq:9nn10}
|\psi (x_j) - \chi (\ox_j)| < \varepsilon \quad \text{for} \; 1 \le j \le n \; .
\end{equation}
On $k^{\times} (N)$ we already defined $\chi$ in such a way that $\chi (\oeta) = \psi (\eta)$. Thus we get a character $\chi : k^{\times} \to S^1$ which satisfies
\begin{equation}
\label{eq:9nn11}
|\psi (t) - \chi (t \mod \en)| < \varepsilon \quad \text{for all} \; t \in T \; .
\end{equation}
Since $\oR$ is integral over $A$ there is a prime ideal $\oemm$ of $\oR$ over $\en$ which must be maximal. Using Lemma \ref{t3}, we can extend $\chi$ from the finite field $k^{\times} = (A / \en)^{\times}$ to a character $\chi : (\oR / \oemm)^{\times} \to S^1$ with finite kernel. Using \eqref{eq:9nn11}, it follows that $\chi$ satisfies the estimate \eqref{eq:9nn2} in Lemma \ref{t93nn} and hence the lemma is proved.
\end{proofof}

\section{Connectedness of the $\R^{> 0}$-dynamical system} \label{sec:5n}
In this section $\C$ is the complex number field and our base scheme $\eX_0$ will be integral normal and of finite type over $\spec \Z$. As usual we fix a sub-monoid $\Nh_0 \subset \Nh$ generated by a set of prime numbers $\car \Nh_0 \supset \car \eX_0$ and form $\ceX (\C) = \colim_{\Nh_0} \deX (\C)$. Let $\deX (\C)' \subset \deX (\C)$ denote the subspace of points $(\ex , \oP^{\times})$ where $(\ker \oP^{\times})_{\tors}$ is a possibly infinite $\Nh_0$-primary group. Set $\deX_0 (\C)' = \deX (\C)' / G$ and
\[
\ceX (\C)' = \colim_{\Nh_0} \deX (\C)' \quad \text{and} \quad \ceX_0 (\C)' = \ceX (\C)' / G \; .
\]
We have $\ceX (\C)_{\Eh_{\tors}} \subset \ceX (\C)'$ etc. and $\ceX (\C)' = \ceX (\C)$ etc. if $\Nh_0 = \Nh$.

\begin{prop}
\label{t91xn}
$\deX (\C)'$ is the topological closure of $\deX (\C)_{\Eh_{\tors}}$ in $\deX (\C)$. 
\end{prop}

\begin{proof}
For subsets $S \subset T$ of a topological space $X$, the assertion $\oS = T$ is local. Hence we may assume that $\eX_0 = \spec R_0$ is affine. A point $(\ex , \oP^{\times})$ of $\deX (\C)$ corresponds via Remark \ref{t34} to a multiplicative map $P : R \to \C$. The condition $(\ex , \oP^{\times}) \in \deX (\C)'$ means that $\ker P \, |_{\mu (R)}$ is $\Nh_0$-primary. Note here that $\car \kappa (\ex) \in \Nh_0$. If $P_n \in \deX (\C)'$ converge to $P \in \deX (\C)$ and if $P (\zeta) = 1$ for some $\zeta \in \mu_N (R)$ then $P_n (\zeta) = 1$ for all $n \gg 0$ since $\mu_N (\C)$ is discrete in $\C^{\times}$. It follows that $\zeta \in \mu_{\Nh_0} (R)$ and hence $\ker P \, |_{\mu(R)}$ is $\Nh_0$-primary. Thus $\deX (\C)'$ is closed in $\deX (\C)$. Given a point $P = (\ex , \oP^{\times})$ of $\deX (\C)'$, since $\mu (\kappa (\ex))$ is divisible, we can choose a splitting $\beta$ of the exace sequence
\[
\xymatrix{
1 \ar[r] & \mu (\kappa (\ex)) \ar[r] & \kappa (\ex)^{\times} \ar[r]_-{\pi} & \kappa (\ex)^{\times} \otimes \Q \ar[r] \ar@/_/[l]_-{\beta} & 1 \; .
}
\]
It induces a topological isomorphism of groups
\begin{align*}
\Hom (\kappa (\ex)^{\times} , \C^{\times}) & \xrightarrow{\;\sim\;} \Hom (\mu (\kappa (\ex)) , \C^{\times}) \times \Hom (\kappa (\ex)^{\times} \otimes \Q , \C^{\times}) \\
\chi & \longmapsto (\chi \, |_{\mu (\kappa (\ex))} , \chi \verk \beta) \\
\tchi_1 \cdot \tilde{\tchi}_2 & \longmapsfrom (\chi_1 , \chi_2)
\end{align*}
where $\tchi_1 (x) = \chi_1 (x \beta (x \otimes 1)^{-1})$ and $\tilde{\tchi}_2 (x) = \chi_2 (x \otimes 1)$. Here the $\Hom$ groups carry the topologies of pointwise convergence. Since $\ker \oP^{\times}$ is $\Nh_0$-primary, we can find a sequence of characters $\psi_n \in \Hom (\mu (\kappa (\ex)) , \C^{\times})$ with $|\ker \psi_n| \in \Nh_0$ and $\psi_n \to \oP^{\times} \, |_{\mu (\kappa (\ex))}$ pointwise: On $\mu_{(\Nh_0)} (\kappa (\ex))$, the roots of unity of order prime to $\Nh_0$ set $\psi_n = \oP^{\times}$, which is injective. Choose a sequence $N_1 \mid N_2 \mid \ldots$ of integers in $\Nh_0$ such that every $\nu \in \Nh_0$ divides some $N_n$. Set $\psi_n \, |_{\mu_{N_n} (\kappa (\ex))} := \oP^{\times} \, |_{\mu_{N_n} (\kappa (\ex))}$ and extend  $\psi_n$ to a character of $\mu_{\Nh_0} (\kappa (\ex))$ with finite kernel. This gives the desired sequence of characters $\psi_n$ on $\mu (\kappa (\ex))$ converging to $P \, |_{\mu (\kappa (\ex))}$. Define $\oP^{\times}_n : \kappa (\ex)^{\times} \to \C^{\times}$ by $\oP^{\times}_n = \tpsi_n (\oP^{\times}_n \verk \alpha)^{\approx}$. Then $(\ex , \oP^{\times}_n) \in \deX (\C)_{\Eh_{\tors}}$ and $\oP^{\times}_n \to \oP^{\times}$ on $\kappa (\ex)^{\times}$. For the corresponding multiplicative maps $P_n : R \to \C$ we have $P^{-1}_n (0) = P^{-1} (0)$ for all $n$ and therefore $P_n \to P$ on $R$. Hence $(\ex , \oP^{\times}_n) \to (\ex , \oP^{\times})$ in $\deX (\C)$ and we have shown that $\deX (\C)_{\Eh_{\tors}}$ is dense in $\deX (\C)'$. 
\end{proof}

\begin{theorem}
\label{t11}
Let $\eta = \spec K$ and $\eta_0 = \spec K_0$ be the generic points of $\eX$ resp. $\eX_0$. Then the fibres of $\ceX (\C)_{\Eh_f}$ and $\ceX_0 (\C)_{\Eh_f}$ over $\eta$ resp. $\eta_0$ are dense in $\ceX (\C)'$ resp. $\ceX_0 (\C)'$. 
\end{theorem}

\begin{proof}
It suffices to show the first assertion. Consider the projection $\pr : \deX (\C)_{\Eh_f} \to \eX$, c.f. Corollary \ref{t4ka}. Since $\ceX (\C)'$ and the fibre of $\ceX (\C)_{\Eh_f}$ over $\eta$ are $\Q^{> 0}_0$-invariant and since $\ceX (\C)'$ is covered by the $\Q^{> 0}_0$-translates of the subspace $\deX (\C)'$, it is enough to show that $\pr^{-1} (\eta)$ is dense in $\deX (\C)'$. By Proposition \ref{t91xn} it suffices to show that $\pr^{-1} (\eta)$ is dense in $\deX (\C)_{\Eh_{\tors}}$. For a covering of $\eX_0$ by open affine subschemes $\eX^i_0$, the space $\deX (\C)$ is covered by the open subspaces $\deX\,\!^i_0 (\C)$ and each of them contains $\pr^{-1} (\eta)$. Hence it suffices to show that $\pr^{-1} (\eta)$ is dense in $\deX (\C)_{\Eh_{\tors}}$ if $\eX_0 = \spec R_0$ is affine. Let $\Hom_f (K^{\times} , \C^{\times})$ be the set of homomorphisms $\chi : K^{\times} \to \C^{\times}$ with finite kernel such that $|\ker \chi| \in \Nh_0$. We have an identification
\[
\Hom_f (K^{\times} , \C^{\times}) \silo \pr^{-1} (\eta) \; , \; \chi \longmapsto (\eta , \chi) \; .
\]
We will also view $\chi$ as a multiplicative map $\chi : R \to \C$ by extending $\chi \, |_{R \setminus 0}$ by $\chi (0) = 0$ to $R$. Now let $P = (\ep , \oP^{\times})$ be a point of $\deX (\C)_{\Eh_{\tors}}$. We can extend the multiplicative map $P : R \to \C$ uniquely to a multiplicative map on $R_{\ep}$ as the composition $P : R_{\ep} \to \kappa (\ep) \xrightarrow{\oP} \C$ where $\oP$ is the extension of $\oP^{\times}$ by $0 \mapsto 0$. We have to show that every neighborhood of $P$ contains a point from $\pr^{-1} (\eta)$. Since $P (0) = 0 = \chi (0)$, this means that given any finite subset $T \subset R \setminus 0$ and any $\varepsilon > 0$ there is a homomorphism $\chi : K^{\times} \to \C^{\times}$ with $|\ker \chi| \in \Nh_0$ such that
\begin{equation}
\label{eq:25}
|\chi (r) - P (r)| < \varepsilon \quad \text{for all} \; r \in T \; .
\end{equation}
We first define a character $\chi_{\mu} : \mu (K) \to \C^{\times}$, which will be the restriction of $\chi$ to  $\mu (K)$. If the characteristic $p$ of $\kappa (\ep)$ is zero or if $\car K = \car \kappa (\ep) = p$ is positive, we have $\mu (K) \hookrightarrow \kappa (\ep)^{\times}$ and we set
\[
\chi_{\mu} (\zeta) = P (\zeta) = \oP^{\times} (\zeta \mod \ep) \quad \text{for} \; \zeta \in \mu (K) \; .
\]
If $p$ is non-zero and $\car K_0 = 0$, we have a direct product decomposition
\[
\mu (K) = \mu_{(p)} (K) \cdot \mu_{p^{\infty}} (K)
\]
and it suffices to define $\chi_{\mu}$ on each factor. We set
\[
\chi_{\mu} (\zeta) = P (\zeta) = \oP^{\times} (\zeta \mod \ep) \quad \text{for} \; \zeta \in \mu_{(p)} (K) \; ,
\]
noting the inclusion $\mu_{(p)} (K) \hookrightarrow \kappa (\ep)^{\times}$. On $\mu_{p^{\infty}} (K)$ the map $P$ is identically $1$ and we cannot set $\chi_{\mu} = P$ because then $\chi_{\mu}$ and hence $\chi$ would have a kernel with infinite torsion. Instead we choose some $M \ge 0$ which is bigger than the $p$-valuations of the orders of all the roots of unity in the finite set $T$. Fix an isomorphism $\iota : \mu_{p^{\infty}} (K) \silo \mu_{p^{\infty}} (\C)$ and set
\begin{equation}
\label{eq:26}
\chi_{\mu} (\zeta)  = \iota (\zeta^M) \quad \text{for} \; \zeta \in \mu_{p^{\infty}} (K) \; .
\end{equation}
In all cases the so defined character $\chi_{\mu} : \mu (K) \to \C^{\times}$ has a finite kernel with $|\ker \chi_{\mu}| \in \Nh_0$ and we have
\begin{equation}
\label{eq:27}
\chi_{\mu} (\zeta) = P (\zeta) \quad \text{for all} \; \zeta \in \mu (K) \cap T \; .
\end{equation}
At a later stage we may have to increase the chosen $M$. The characters $\chi$ that we construct are of the following type. Since $\mu (K)$ is divisible the sequence of abelian groups
\[
1 \longrightarrow \mu (K) \longrightarrow K^{\times}  \longrightarrow K^{\times} / \mu (K) \longrightarrow 1
\]
splits. Let $\alpha : K^{\times} \to \mu (K)$ and $\beta : K^{\times} / \mu (K) \to K^{\times}$ be corresponding splittings i.e. $\alpha (\zeta) = \zeta$ for $\zeta \in \mu (K)$ and $\beta (\ox) = x \alpha (x)^{-1}$ where we set $\ox = x \mod \mu (K)$. Note that $\alpha (\beta (\ox)) = 1$. Set $\Z (1) = 2 \pi i \Z$ and $\Q (1) = 2 \pi i \Q$ inside of $\C$ and choose a complement $\C'$ to $\Q (1)$ as $\Q$-vector spaces, i.e. $\C = \Q (1) \oplus \C'$. Then the homomorphism
\[
\mu (\C) \times \C' \longrightarrow \C^{\times} \; , \; (\zeta , z) \longmapsto \zeta \exp z
\]
is injective. Our characters $\chi$ are defined by composition
\begin{equation}
\label{eq:28}
\xymatrix{
K^{\times} \ar[r]^-{\sim} \ar[d]_{\chi} & \mu (K) \times K^{\times} / \mu (K) \ar[d]^{(\chi_{\mu} , \chi')} \\
\C^{\times} & \mu (\C) \times \C' \ar@{_{(}->}[l]
}
\end{equation}
Here $\chi' : K^{\times} / \mu (K) \to \C'$ is an injective linear map of $\Q$-vector spaces to be constructed and the top horizontal map sends $x$ to $(\alpha (x) , \ox)$. Explicitely, we have:
\[
\chi (x) = \chi_{\mu} (\alpha (x)) \exp \chi' (\ox) \quad \text{for} \; x \in K^{\times} \; .
\]
It is then immediate that
\[
\chi \, |_{\mu (K)} = \chi_{\mu} \quad \text{and} \quad \ker \chi = \ker \chi_{\mu} \; \text{is finite with} \; |\ker \chi| \in \Nh_0 \; .
\]
In defining $\chi'$, so that the resulting character $\chi$ is close to $P$ on $T$, the elements $t \in T \cap \ep$ are the problematic ones since $P (t) = 0$. The naive approach which works for $t \in T \setminus \ep$, using elementary linear algebra after tensoring with $\Q$ runs into the problem that $\frac{0}{0}$ is not defined. Instead we have to work with multiplicative maps on the rings themselves and use resolution of singularities to arrive at rings whose multiplicative structure is understood. Let $K_1$ be the finite extension of $K_0$ generated by $T$ and let $R_1 = R \cap K_1$ be the integral closure of $R_0$ in $K_1$. Since $\Z$ is universally Japanese, $R_1$ is a finite $R_0$-algebra and in particular of finite type over $\Z$. As an integral Noetherian scheme, $\eX_1 = \spec R_1$ has an alteration, cf. \cite{dJ}. This is a proper, dominant hence surjective morphism $f : \Zh \to \eX_1$ from an integral regular scheme $\Zh$ whose function field $L$ is a finite extension of $K_1$ in $K$. Let $z \in \Zh$ be a point with $f (z) = \ep_1 := \ep \cap R_1$. Then we have inclusions, where $\Oh_z = \Oh_{\Zh ,z}$,
\[
\begin{array}{ccccc}
 & & \Oh_z & \subset & L \\
& & \cup && \cup \\
R_1 & \subset & R_{1 , \ep_1} & \subset & K_1 \; .
\end{array}
\]
Let $\emm_z$ be the maximal ideal in the local ring $\Oh_z$. The map $R_{1 , \ep_1} \subset \Oh_z$ is local, i.e.
\[
\ep_1 R_{1 , \ep_1} = \emm_z \cap R_{1 , \ep_1}
\]
and hence $\ep_1 = \emm_z \cap R_1$. By the Auslander--Buchsbaum theorem, the regular, Noetherian local ring $\Oh_z$ is a unique factorization domain. We therefore have a decomposition of multiplicative monoids
\begin{equation}
\label{eq:29}
\Oh_z \setminus 0 = \Oh^{\times}_z \oplus \pi^{\N_0}_1 \oplus \pi^{\N_0}_2 \oplus \ldots
\end{equation}
Here $\pi_1 , \pi_2 , \ldots$ are a complete set of pairwise non-associated prime elements of $\Oh_z$. If $T \cap \ep = T \cap \ep_1 \subset \emm_z$ is non-empty we can write its elements $s_1 , \ldots , s_m$ in the form
\begin{equation}
\label{eq:30}
s_i = u_i \pi^{\nu_{i1}}_1 \cdots \pi^{\nu_{iN}}_N \quad \text{for} \; 1 \le i \le m \; .
\end{equation}
Here $N \ge 1 , \nu_{ij} \ge 0$ and for each $i$ there is some $1 \le j \le N$ with $\nu_{ij} \ge 1$. Moreover $u_i \in \Oh^{\times}_z$ for $1 \le i \le m$. We will come back to this expression later. If $T \setminus \ep = T \setminus \ep_1$ is non-empty its elements $t_1 , \ldots , t_n$ lie in $R^{\times}_{1, \ep_1} \subset \Oh^{\times}_z$. 

Let $\Gamma$ be the subgroup of $R^{\times}_{1 , \ep_1} / \mu (K_1) \subset K^{\times} / \mu (K)$ generated by $\ot_1 , \ldots , \ot_n$. Since $\Gamma$ is finitely generated and torsion-free, it is free of some rank $l \ge 1$. Let $\gamma_1 , \ldots , \gamma_l$ be a $\Z$-basis of $\Gamma$. Then we have
\[
\ot_i = \gamma^{k_{i1}}_1 \cdots \gamma^{k_{il}}_l \quad \text{for some} \; k_{ij} \in \Z \; \text{and} \; 1 \le i \le n \; .
\]
Setting $\xi_i = \beta (\gamma_i) \in K^{\times}$, we have $\xi_i \in \mu (K) R^{\times}_{1 , \ep_1} \subset R^{\times}_{\ep}$ and 
\[
t_i = \zeta_i \xi^{k_{i1}}_1 \cdots \xi^{k_{il}}_l \quad \text{in $R^{\times}_{\ep}$ for some} \; \zeta_i \in \mu (K) \; \text{and} \; 1 \le i \le n \; .
\]
Note that $P (\xi_j) \neq 0$ for all $1 \le j \le l$. We need $\chi_{\mu}$ and $P$ to agree on all $\zeta_i$ for $1 \le i \le n$. This will be the case if the integer $M \ge 0$ in \eqref{eq:26} is bigger than the $p$-valuations of the orders of all the roots of unity $\zeta_i$ for $1 \le i \le n$. Choosing some $M$ that satisfies this condition as well, $\chi_{\mu}$ is finally defined. We have
\[
P (t_i) = P (\zeta_i) P (\xi_1)^{k_{i1}} \cdots P (\xi_l)^{k_{il}} = \chi_{\mu} (\zeta_i) P (\xi_1)^{k_{i1}} \cdots P (\xi_l)^{k_{il}} \; .
\]
Thus, if $\chi'$ is defined in such a way, that for the corresponding $\chi$ the values $P (\xi_j)$ and $\chi (\xi_j)$ are very close to each other, it will follow that $P (t_i)$ and
\[
\chi_{\mu} (\zeta_i) \chi (\xi_i)^{k_{i1}} \cdots \chi (\xi_l)^{k_{il}} = \chi (\zeta_i \xi^{k_{i1}}_1 \cdots \xi^{k_{il}}_l) = \chi (t_i)
\]
are very close for $1 \le i \le n$. We have
\[
\chi (\xi_j) = \chi_{\mu} (\alpha (\xi_j)) \exp \chi' (\oxi_j) = \exp \chi' (\gamma_j) \; .
\]
Note here that $\alpha (\xi_j) = \alpha (\beta (\gamma_j)) = 1$. Choose $z_j \in \C$ with $\exp z_j = P (\xi_j)$ and choose $\Q$-linearly independent complex numbers $w_1 , \ldots , w_l \in \C'$ such that each $w_j$ is very close to $z_j$ for each $j$. Note that $\C'$ is dense in $\C$ since it has to contain two $\R$-linearly independent vectors. We obtain an injective $\Q$-linear map
\[
\chi' : \Gamma \otimes \Q \hookrightarrow \C'
\]
by setting $\chi' (\gamma_j) = w_j$ for $1 \le j \le l$. Then $P (\xi_j)$ and $\chi (\xi_j)$ will be close to each other as desired. We extend $\chi'$ from $\Gamma \otimes \Q \subset (R^{\times}_{1 , \ep_1} / \mu (K_1)) \otimes \Q \subset (\Oh^{\times}_z / \mu (L)) \otimes \Q$ to an injective $\Q$-linear map
\[
\chi' : (\Oh^{\times}_z / \mu (L)) \otimes \Q \hookrightarrow \C' \; .
\]
The isomorphism \eqref{eq:29} gives a decomposition of $\Q$-vector spaces
\[
(L^{\times} / \mu (L)) \otimes \Q = ((\Oh^{\times}_z / \mu (L)) \otimes \Q) \oplus \opi^{\Q}_1 \oplus \opi^{\Q}_2 \oplus \ldots
\]
The values of $\chi'$ on the first summand on the right are already determined. In the formula resulting from \eqref{eq:30},
\[
\os_i = \ou_i \opi^{\nu_{i1}}_1 \cdots \opi^{\nu_{iN}}_N \quad \text{in} \; L^{\times} / \mu (L)
\]
the values $\chi' (\ou_i)$ are therefore given. We extend $\chi'$ to a $\Q$-linear injection
\[
\chi' : (L^{\times} / \mu (L)) \otimes \Q \hookrightarrow \C'
\]
in such a way that the values $\RRe \chi' (\opi_1) , \ldots , \RRe \chi' (\opi_N)$ are close to $-\infty$. Then the real parts of
\[
\chi' (\os_i) = \chi' (\ou_i) + \sum^N_{j=1} \nu_{ij} \chi' (\opi_j)
\]
will be close to $-\infty$ as well. It is decisive here that all $\nu_{ij} \ge 0$ and for each $i$ there is an index $j$ with $\nu_{ij} \ge 1$. In this way we can make
\[
|\chi (s_i)| = |\chi_{\mu} (\alpha (s_i))| \exp \RRe \chi' (\os_i) = \exp \RRe \chi' (\os_i)
\]
as small as we wish and hence $\chi (s_i)$ as close to $P (s_i) = 0$ as desired. Having thus made the necessary choices to ensure that \eqref{eq:25} will hold, we now choose any extension of $\chi'$ from $(L^{\times} / \mu (L)) \otimes \Q$ to a $\Q$-linear injection $\chi' : K^{\times} / \mu (K) \hookrightarrow \C'$. Note that $K^{\times} / \mu (K)$ is a $\Q$-vector space. Hence we have found a character $\chi \in \Hom_{\Eh_f} (K^{\times} , \C^{\times})$ satisfying \eqref{eq:25}. 
\end{proof}

%

The same proof but with some simplifications (e.g. $\chi_{\mu} = P \, |_{\mu (K)}$ in all cases) shows the following result:

\begin{theorem}
\label{t93n}
The fibres of $\ceX (\C)$ and $\ceX_0 (\C)$ over $\eta$ resp. $\eta_0$ are dense in $\ceX (\C)$ resp. $\ceX_0 (\C)$. 
\end{theorem}

We now investigate connectedness properties of our spaces. Let $V$ be a countable $\Q$-vector space and let $H = \Hom (V , \C^{\times})$ be the set of group homomorphisms $\varphi : V \to \C^{\times}$. We give $H$ the topology of pointwise convergence i.e. the subspace topology of the inclusion
\[
H \subset \prod_{v \in V} \C^{\times} \; , \; \varphi \longmapsto (\varphi (v))_{v \in V} \; .
\]
Let $H_{\inj} \subset H$ be the subspace of injective $\varphi$'s. 

\begin{lemma} \label{t13}
For any $\varphi_0 , \varphi_1 \in H$ and neighborhoods $U_0$ of $\varphi_0$ and $U_1$ of $\varphi_1$ there are $\psi_0 , \psi_1 \in H_{\inj}$ with $\psi_0 \in U_0$ and $\psi_1 \in U_1$ that can be connected by a continuous path in $H_{\inj}$.  
\end{lemma}

\begin{proof}
Choose a finite set $\emptyset \neq S \subset V$ and some $\varepsilon > 0$ such that
\[
U'_i := H \cap \big( \prod_{s \in S} U_{\varepsilon} (\varphi_i (s)) \times \prod_{s \in V \setminus S} \C^{\times} \big) \subset U_i \quad \text{for} \; i = 0,1 \; .
\]
The open sets $U'_i$ are non-empty since $\varphi_i \in U'_i$ for $i = 0,1$. Choose a $\Z$-basis $\eb = \{ v_1 , \ldots , v_n \}$, $n \ge 0$ of the finitely generated torsion free (hence free) subgroup of $V$ generated by $S$. For small enough $\delta > 0$ any $\varphi \in H$ with $\varphi (v) \in U_{\delta} (\varphi_0 (v))$ for $v \in \eb$ satisfies $\varphi (s) \in U_{\varepsilon} (\varphi_0 (s))$ for all $s \in S$, and similarly for $\varphi_1$ instead of $\varphi_0$. This implies that for small $\delta > 0$ and $i = 0,1$
\[
\varphi_i \in U''_i := H \cap \big( \prod_{v \in \eb} U_{\delta} (\varphi_i (v)) \times \prod_{v \in V \setminus \eb} \C^{\times}\big) \subset U'_i \subset U_i \; .
\]
Choose a $\Q$-vector space complement $\C_0 \subset \C$ to the countable $\Q$-vector space generated by $\exp^{-1} (\imm \varphi_0$). Then $\C_0$ is dense in $\C$ since it has to contain two $\R$-linear independent vectors. Hence there is an injective $\Q$-linear map $\psi'_0 : V \to \C_0 \subset \C$ such that $\exp \psi'_0 (v) \in U_{\delta} (\varphi_0 (v))$ for all $v \in \eb$. Since $2 \pi i \Z = \exp^{-1} (1) \subset \exp^{-1} (\imm \varphi_0)$ we have $2 \pi i \Q \cap \C_0 = 0$ and hence $\psi_0 = \exp \verk \psi'_0 \in U_0 \cap H_{\inj}$. Similarly we find an injective $\Q$-linear map $\psi'_1 : V \to \C_1 \subset \C$ with $2 \pi i \Q \cap \C_1 = 0$ such that $\psi_1 = \exp \verk \psi'_1 \in U_1 \cap H_{\inj}$. For a continuous path $\alpha : [0,1] \to \C$ with $\alpha (0) = 0 , \alpha (1) = 1$ set
\[
\psi'_t = (1 - \alpha (t)) \psi'_0 + \alpha (t) \psi'_1 \quad \text{and} \quad \psi_t = \exp \verk \psi'_t \in H \; .
\]
We want $\psi_t$ to be in $H_{\inj}$ for all $0 \le t \le 1$. This means that $\psi_t (v) \neq 1$ or equivalently $\psi'_t (v) \notin 2 \pi i \Z$ for all $0 \neq v \in V$ and $0 \le t \le 1$. In other words, $\alpha$ has to avoid the set
\[
\Omega = \{ z \in \C \mid (1-z) \psi'_0 (v) + z \psi'_1 (v) \in 2 \pi i \Z \; \text{for some} \; 0 \neq v \in V \} \; .
\]
For $0 \neq v \in V$ with $\psi'_0 (v) = \psi'_1 (v)$ we have
\[
(1-z) \psi'_0 (v) + z \psi'_1 (v) = \psi'_0 (v) \notin 2 \pi i \Z
\]
by the construction of $\psi'_0$. Hence $\Omega$ consists of the complex numbers
\[
z = (2 \pi i \nu - \psi'_0 (v)) / ( \psi'_1 (v) - \psi'_0 (v))
\]
with $\nu \in \Z$ and $0 \neq v \in V$ with $\psi'_0 (v) \neq \psi'_1 (v)$. Thus $\Omega \subset \C$ is countable. By construction, $\alpha (0) = 0$ and $\alpha (1) = 1$ are not in $\Omega$. The complement of any countable set in $\C$ is path connected. In fact, given two points $z_1 \neq z_2$ in the complement, for all but countably many slopes the lines through one of them lie in the complement. Taking two such lines through $z_1$ resp. $z_2$ which intersect we get a path from $z_1$ to $z_2$. Hence a path $\alpha$ avoiding $\Omega$ exists. The corresponding continuous path $t \mapsto \psi_t$ joins $\psi_0$ with $\psi_1$ in $H_{\inj}$. 
\end{proof}

\begin{defn} \label{t95n}
We call a topological space $Z$ almost path-connected if for any two points $z_0 , z_1 \in Z$ and any two neighborhoods $U_0$ of $z_0$ and $U_1$ of $z_1$ there are two points $\tz_0 \in U_0$ and $\tz_1 \in U_1$ and a continuous map $\gamma : [0,1] \to Z$ with $\gamma (0) = \tz_0$ and $\gamma (1) = \tz_1$. 
\end{defn}

The continuous image of an almost path connected space is again almost path connected. An almost path connected space $Z$ is connected. If not, there would be a decomposition $Z = U_0 \cup U_1$ into non-empty disjoint open subsets and a continuous map $\gamma : [0,1] \to Z$ with $\gamma ([0,1]) \cap U_i \neq \emptyset$ for $i = 0,1$. Then $\gamma^{-1} (U_0)$ and $\gamma^{-1} (U_1)$ would be non-empty disjoint open subsets covering $[0,1]$. This is a contradiction. 

Let $\Eh$ be an admissible class as in Definition \ref{t41n}. We give
\[
X = \ceX (\C)_{\Eh} \times_{\Q^{> 0}_0} \R^{> 0} \quad \text{and} \quad X_0 = \ceX_0 (\C)_{\Eh} \times_{\Q^{> 0}_0} \R^{> 0} 
\]
the quotient topologies. The canonical $G \times \R^{> 0}$-action on $X$ and the $\R^{> 0}$-action on $X_0$ are continuous. The canonical $\R^{> 0}$-equivariant projection $X \to X_0$ is continuous and open and identifies $X_0$ with $X / G$ as topological spaces. In the following we will assume that $\car K_0 = 0$. In the case $\car K_0 = p$ similar results hold for the fibres of the map $X_0 \to S_0$ induced by $\eX_0 \to \spec \F_p$ where $S_0 = (\spec \F_p)^{\vee} (\C)_{\Eh} \times_{ \Q^{ > 0}_0} \R^{> 0}$. By assumption $K_0$ is a finitely generated field and hence the subfield $K_0 \cap \Q (\mu (K))$ is finitely generated as well and hence a finite extension of $\Q$. Let  
\[
d_{\mu} = (K_0 \cap \Q (\mu (K)) : \Q)
\]
be its degree. Then $d_{\mu} = 1$ if and only if $K_0$ is disjoint from the maximal cyclotomic extension of $\Q$ in $K$. Let $\pr_{\eX} : \ceX (\C)_{\Eh} \to \eX$ and $\pr_{\eX_0} : \ceX_0 (\C)_{\Eh} \to \eX_0$ be the projections and consider:
\[
X_{\eta} = \pr^{-1}_{\eX} (\eta) \times_{\Q^{> 0}_0} \R^{> 0} \subset X \quad \text{and} \quad X_{0 \eta_0} = \pr^{-1}_{\eX_0} (\eta_0) \times_{\Q^{> 0}_0} \R^{> 0} = X_{\eta} / G \subset X_0 \;.
\]

\begin{theorem}
\label{t14}
Assume that $\eX_0$ is integral, normal and flat of finite type over $\spec \Z$. For every admissible class $\Eh \supset \Eh_f$, the spaces $X_{\eta}$ and $X_{0 \eta_0}$ are connected. The space $X_{0 \eta_0}$ is the disjoint union of $d_{\mu}$ almost path connected subspaces. 
\end{theorem}

\begin{proof}
Let $\Hom_f (\mu (K) , \C^{\times}) = \Hom_f (\mu (K) , S^1)$ be the space of homomorphisms $\mu (K) \to \C^{\times}$ with finite kernel of order in $\Nh_0$, equipped with the topology of pointwise convergence. Fix an isomorphism $i : \mu (K) \silo \mu (\C)$ and consider the induced topological isomorphism
\begin{equation}
\hZ \silo \Hom (\mu (K) , \C^{\times}) = \Hom (\mu (K) , S^1)  \; , \; a \longmapsto \chi_a := (\zeta \mapsto i (\zeta)^a) \; . \label{eq:33} 
\end{equation}
The isomorphism \eqref{eq:33} restricts to a homeomorphism:
\begin{equation}
\label{eq:34}
\Nh_0 \hZ^{\times} \silo \Hom_f (\mu (K) , \C^{\times}) \; .
\end{equation}
Here $\Nh_0 \hZ^{\times} \subset \hZ$ carries the subspace topology. Note that $\Nh_0 \cap \hZ^{\times} = \{ 1 \}$ in $\hZ$. Let $\Hom_{\Eh} (K^{\times} , \C^{\times})$ be the set of homomorphisms $\chi : K^{\times} \to \C^{\times}$ of class $\Eh$. Consider the restriction map
\begin{equation}
\label{eq:35}
r : \Hom_{\Eh} (K^{\times} , \C^{\times}) \longrightarrow \Hom_f (\mu (K) , \C^{\times}) \; , \;  \chi \longmapsto \chi \, |_{\mu (K)}\; .
\end{equation}
It is continuous for the topologies of pointwise convergence, and surjective by Lemma \ref{t3} since $\Eh_f \subset \Eh$. Choose a splitting $\alpha : K^{\times}\to \mu (K)$ of the exact sequence 
\begin{equation} \label{eq:75n}
1 \to \mu (K) \to K^{\times} \xrightarrow{\pi} K^{\times} / \mu (K) \to 1 \; .
\end{equation}
Given a character $\chi_{\mu} \in \Hom_f (\mu (K) , \C^{\times})$, the following map is a homeomorphism onto its image with the subspace topology
\begin{equation}
\label{eq:36}
f : r^{-1} (\chi_{\mu}) \hookrightarrow \Hom (K^{\times} \otimes \Q , \C^{\times}) \; , \; \chi \longmapsto \psi = \chi (\chi^{-1}_{\mu} \verk \alpha) \; .
\end{equation}
Here we view $\psi = \chi (\chi^{-1}_{\mu} \verk \alpha)$ which is trivial on $\mu (K)$ as a character on $K^{\times} / \mu (K)= K^{\times} \otimes \Q$. The inverse map $f^{-1}$ defined on $\imm f$ sends $\psi : K^{\times} \otimes \Q \to \C^{\times}$ to the character $\chi$ defined by $\chi (x) = \psi (x \otimes 1) \chi_{\mu} (\alpha (x))$ for $x \in K^{\times}$. The continuity of $f$ and $f^{-1}$ can be checked on pointwise convergent sequences of characters. Moreover, we claim that $f$ restricts to a homeomorphism 
\begin{equation} \label{eq:76n}
\Hom_{\Eh_f} (K^{\times} , \C^{\times}) \cap r^{-1} (\chi_{\mu}) \silo \Hom_{\inj} (K^{\times} \otimes \Q , \C^{\times}) \; .
\end{equation} 
Recall here that $\Eh_f \subset \Eh$. Namely, consider the commutative diagram
\[
\xymatrix{
K^{\times} \otimes \Q \ar[r]^{\psi} \ar[dr]_{\chi \otimes \id} & \C^{\times} \ar[d] \\
 & \C^{\times} \otimes \Q \; .
}
\] 
It shows that $\Ker \chi = (\Ker \chi)_{\tors}$ i.e. $\Ker (\chi \otimes \id) = 0$ if and only if $\psi$ is injective. 
Lemma \ref{t13} and \eqref{eq:36}, \eqref{eq:76n} show that the fibres of the restriction map $r$ are almost path connected. For the projection $\pr : \deX (\C)_{\Eh} \to \eX$ we have
\[
\pr^{-1} (\eta) = \Hom_{\Eh} (K^{\times} , \C^{\times}) \quad \text{via} \; (\eta , \oP^{\times}) \mapsto \oP^{\times} \; .
\]
Using \eqref{eq:34} which depends on the choice of the isomorphism $i : \mu (K) \silo \mu (\C)$ we may therefore view $r$ as an $\Nh_0$-equivariant continuous surjection with almost path-connected fibres
\begin{equation}
\label{eq:38}
r : \pr^{-1} (\eta) \twoheadrightarrow \Nh_0 \hZ^{\times} \subset \hZ \; .
\end{equation}
Passing to the colimit over $\Nh_0$ we obtain a $\Q^{> 0}_0$-equivariant continuous surjection with almost path-connected fibres
\begin{equation}
\label{eq:39}
r : \pr^{-1}_{\eX} (\eta) \twoheadrightarrow \Q^{> 0}_0 \hZ^{\times} \subset \hZ \otimes \Q_0 \; .
\end{equation}
Here $\Q_0 = \Z [p^{-1} \mid p \in \car \Nh_0] \subset \Q$ and $\Nh_0 \hZ^{\times}$ resp. $\Q^{> 0}_0 \hZ^{\times}$ carry the subspace topology of $\hZ$ resp. $\hZ \otimes \Q_0$ i.e. the adele and not the idele topology. The maps \eqref{eq:38} and \eqref{eq:39} are $G$-equivariant if the $G$-action on the right is given by multiplication with the cyclotomic character $\kappa$ of $G$, c.f. \eqref{eq:57}.

The quotient
\[
\hZ^{\times} / \kappa (G) = \Gal (K_0 \cap \Q (\mu (K)) / \Q)
\]
is finite of order $d_{\mu}$. Passing to the orbit spaces $\mod G$ in \eqref{eq:39} we obtain a continuous surjective map
\begin{equation}
\label{eq:40}
r_0 : \pr^{-1}_{\eX_0} (\eta_0) \twoheadrightarrow \Q^{> 0}_0 \hZ^{\times} / \kappa (G)
\end{equation}
whose fibres are continuous images of the fibres of $r$ in \eqref{eq:39} and are therefore almost path connected. Incidentally, note that the quotient topology on $\pr^{-1}_{\eX_0} (\eta_0) = \pr^{-1}_{\eX} (\eta) / G$ equals the subspace topology within $\ceX_0 (\C)_{\Eh}$, the latter being equipped with the quotient topology via $\ceX_0 (\C)_{\Eh} = \ceX (\C)_{\Eh} / G$. 

For $a \in \hZ^{\times}$ we have a continuous bijection
\begin{equation}
\label{eq:81n}
r^{-1} (a) \times \R^{> 0} \silo (X_{\eta})_a := r^{-1} (\Q^{> 0}_0 a) \times_{\Q^{> 0}_0} \R^{> 0} \; , \; (f , u) \mapsto [f, u] \; .
\end{equation}
It follows that $(X_{\eta})_a$ is almost path connected. Similarly, for $\oa \in \hZ^{\times} / \kappa (G)$ the continuous bijection
\begin{equation}
\label{eq:82n}
r^{-1}_0 (\oa) \times \R^{> 0} \silo (X_{\eta_0})_{\oa} = r^{-1}_0 (\Q^{> 0}_0 \oa) \times_{\Q^{> 0}_0} \R^{> 0} \; , \; (f_0 , u) \mapsto [f_0 , u] 
\end{equation}
shows that $(X_{0 \eta_0})_{\oa}$ is almost path-connected. This also follows from the fact that $(X_{0 \eta_0})_{\oa} = (X_{\eta})_a / G$. In particular, $X_{0 \eta_0}$ is the disjoint union of the $d_{\mu}$ almost path-connected subspaces $(X_{0 \eta_0})_{\oa}$ for $\oa \in \hZ^{\times} / \kappa (G)$. This shows the second assertion of Theorem \ref{t14}. In order to prove that $X_{\eta}$ and hence $X_{0 \eta_0} = X_{\eta} / G$ are connected, we first show that the restriction map $r$ of \eqref{eq:39} is open.

Recall the splitting $\alpha$ of the exact sequence \eqref{eq:75n} and let $\beta : K^{\times} / \mu (K) \to K^{\times}$ be the corresponding splitting $\beta (\overline{r}) = r \alpha (r)^{-1}$. We obtain a homeomorphism
\begin{equation}
\label{eq:83n}
\Hom (K^{\times} , \C^{\times}) \silo \Hom (\mu (K) , \C^{\times}) \times \Hom (K^{\times} / \mu (K) , \C^{\times}) \; ,
\end{equation}
sending $\chi$ to $(\chi \, |_{\mu (K)} , \chi \verk \beta)$. Here as usual the $\Hom$-spaces carry the topology of pointwise convergence. The inverse homeomorphism sends $(\chi_1 , \chi_2)$ to $(\chi_1 \verk \alpha) (\chi_2 \verk \pi)$. Straightforeward arguments show that \eqref{eq:83n} restricts to a homeomorphism
\begin{equation}
\label{eq:84n}
\Hom_{\Eh_f} (K^{\times} , \C^{\times}) \silo \Hom_f (\mu (K) , \C^{\times}) \times \Hom_{\inj} (K^{\times} / \mu (K) , \C^{\times}) \; .
\end{equation}
In particular the restriction map from $\Hom_{\Eh_f} (K^{\times} , \C^{\times})$ to $\Hom_f (\mu (K) , \C^{\times})$ is a projection and therefore open. We claim that more generally, for $\Eh \supset \Eh_f$ the restriction map
\begin{equation}
\label{eq:85n}
r : \Hom_{\Eh} (K^{\times} , \C^{\times}) \longrightarrow \Hom_f (\mu (K) , \C^{\times})
\end{equation}
is open. Let $\chi_1 \in r (O)$ where $O$ is open. Then $\chi_1 = r (\chi)$ with $\chi \in O$. Via \eqref{eq:83n}, there are open subsets $O_1 \subset \Hom_f (\mu (K) , \C^{\times})$ and $O_2 \subset \Hom (K^{\times} / \mu (K) , \C^{\times})$ with
\[
\chi \ent (\chi_1 , \chi_2) \in (O_1 \times O_2) \cap \Hom_{\Eh} (K^{\times} , \C^{\times}) \subset O \; .
\]
Lemma \ref{t13} implies that $O'_2 = O_2 \cap \Hom_{\inj} (K^{\times} / \mu (K) , \C^{\times})$ is non-empty. We have
\[
O_1 \times O'_2 \subset (O_1 \times O_2) \cap \Hom_{\Eh_f} (K^{\times} , \C^{\times}) \subset O
\]
and therefore $\chi_1 \in O_1 = r (O_1 \times O'_2) \subset r (O)$. Hence $r (O)$ is open. Since the map $r$ of \eqref{eq:85n} or equivalently the map \eqref{eq:38} are open, the induced map $r$ in \eqref{eq:39} obtained by passing to colimits is open as well. Note here that for all $\nu \in \Nh_0$ the subspaces $F^{-1}_{\nu} \pr^{-1} (\eta)$ and $\nu^{-1} \hZ^{\times}$ are open in $\pr^{-1}_{\eX} (\eta)$ resp. in $\Q^{> 0} \hZ^{\times}$ with the subspace topology of $\hZ \otimes \Q_0$. Since $r$ is open, the composition
\[
\pr^{-1}_{\eX} (\eta) \times \R^{> 0} \twoheadrightarrow \pr^{-1}_{\eX} (\eta) \overset{r}{\twoheadrightarrow} \Q^{> 0}_0 \hZ^{\times}
\]
is open too. Since $\Q^{> 0}_0$ acts by homeomorphisms on these spaces, we obtain a continuous open and surjective map
\[
\pi : X_{\eta} \twoheadrightarrow Y := \Q^{> 0}_0 \hZ^{\times} / \Q^{> 0}_0 \; .
\]
As an abstract group, we may identify $Y$ with $\hZ^{\times}$ and using \eqref{eq:81n} it follows that all fibres of $\pi$ are almost path connected. We will show below that the only open sets of $Y$ are $\emptyset$ and $Y$. This implies that $X_{\eta}$ is connected. Namely, let $\emptyset \neq O \subset X_{\eta}$ be open and closed. Then $O \cap \pi^{-1} (y)$ is open and closed in $\pi^{-1} (y)$ as well, and since $\pi^{-1} (y)$ is connected it follows that $O \cap \pi^{-1} (y) = \emptyset$ or $\pi^{-1} (y)$. Hence $O$ has the form
\[
O = \bigcup_{y \in T} \pi^{-1} (y)
\]
for some subset $T \subset Y$. Since $\pi$ is open, the non-empty subset $T = \pi (O)$ of $Y$ is open and hence $T = Y$ which implies $O = X_{\eta}$. The fact that $Y$ carries the coarse topology follows from strong approximation for $\Q$: Since $\eX_0$ is flat of finite type over $\spec \Z$ its image in $\spec \Z$ is open. Hence $\car \eX_0$ consists of all but finitely many prime numbers. By our standing assumption $\car \Nh_0 \supset \car \eX_0$, it follows that there is a finite set $S$ of prime numbers such that $\Nh_0$ is generated by all $p \notin S$. Multiplication by an idele is a homeomorphism in the adele topology. Hence multiplication by an element of $Y$ is a homeomorphism on $Y$ as well. For proving that $Y$ carries the coarse topology it therefore suffices to show that if $U \subset Y$ is open with $1 \in U$, then $U = Y$. Identifying $Y$ with $\hZ^{\times}$ as abstract groups, there exist a finite set of prime numbers $S' \neq \emptyset$ and some $0 < \varepsilon < 1$ such that $U$ contains a set of the form
\begin{align*}
U' & = \{ a \in \hZ^{\times} \mid \text{ex.} \, r \in \Q^{>0}_0 \, \text{s.t.}\, |r a_p - 1|_p < \varepsilon \; \text{for} \; p  \in S' \; \text{and} \; |ra_p -1 |_p \le 1 \; \text{for} \; p \notin S' \} \\
 & = \{ a \in \hZ^{\times} \mid \text{ex.} \, r \in \Q^{>0}_0 \, \text{s.t.}\, |r - a^{-1}_p|_p < \varepsilon \; \text{for} \; p  \in S' \; \text{and} \; |r|_p \le 1 \; \text{for} \; p \notin S' \} \; .
\end{align*}
Given any $a \in \hZ^{\times}$, strong approximation (excluding the infinite place, i.e. the Chinese remainder theorem) gives an $r \in \Q$ with
\begin{equation} 
\label{eq:100n}
|r - a^{-1}_p|_p < \varepsilon \; \text{for} \; p \in S' \, , \, |r - 1|_p < 1 \; \text{for} \; p \in S \; \text{and} \; |r|_p \le 1 \; \text{for} \; p \notin S' \cup S \; .
\end{equation}
This implies that $r \in \pm \Nh_0$. Replacing $r$ by $r + \prod_{p \in S \cup S'} p^N$ for some large enough $N$ we can achieve that $r \in \Nh_0$ in \eqref{eq:100n}. Hence every $a \in \hZ^{\times}$ is contained in $U'$ and therefore $U' = Y$ and $U = Y$. 
\end{proof}

\begin{cor} \label{t15}
Let $\eX_0$ be an integral normal scheme which is flat of finite type over $\spec \Z$ and let $\Eh$ be an admissible condition with $\Eh \supset \Eh_f$. Then the spaces $X = \ceX (\C)_{\Eh} \times_{\Q^{> 0}} \R^{> 0}$ and $X_0 = \ceX_0 (\C)_{\Eh} \times_{\Q^{> 0}} \R^{> 0}$ are connected.
\end{cor}

\begin{proof}
By Theorem \ref{t14} the spaces $X_{\eta}$ and $X_{0 \eta_0}$ are connected. By Theorem \ref{t11} they are dense in $X$ resp. $X_0$. Hence $X$ and $X_0$ are connected as well.  
\end{proof}

\begin{rem} \em
For $\eX_0$ as in the corollary and $\Eh = \Eh_{\tors}$ the spaces $X_{\eta} , X , X_{0 \eta_0}$ and $X_0$ are connected. This follows by the same arguments as before, except for Lemma \ref{t13} which is not needed. However, by Theorem \ref{t6n} the space $X_0$ has far too many periodic orbits.
\end{rem}

The following result is due to Kucharczyk and Scholze.

\begin{theorem} \label{t98n}
Consider the map $r$ of \eqref{eq:35}. Let $\Hom_{\inj} (\mu (K) , \C^{\times})$ be the set of injective homomorphisms. Then the Galois group $G$ acts freely on
\[
r^{-1} (\Hom_{\inj} (\mu (K) , \C^{\times})) \subset \pr^{-1} (\eta) \; .
\]
\end{theorem}

\begin{proof}
$P = P^{\sigma}$ implies that $r (P) = r (P^{\sigma}) = r (P)^{\sigma} = \kappa (\sigma) r (P)$ under the identification $\Hom_{\inj} (\mu (K) , \C^{\times}) = \hZ^{\times} \subset \hZ$ which follows from \eqref{eq:34}. Hence $\kappa (\sigma) = 1$ and therefore $\sigma \in G_{\infty} = \ker (\kappa : G \to \hZ^{\times})$. Now the claim follows from \cite[Lemma 4.12]{KS} since $K_0$ is of characteristic zero. 
\end{proof}

\section{The $0$-th leafwise cohomology of the $\R^{> 0}$-dynamical system} \label{sec:10n}
The following definitions are modeled on a cohomology theory for foliated spaces called the leafwise cohomology in \cite{alvarez1} and the tangential cohomology in \cite{mooreschochet}. Let $Q \subset \Q^{> 0}$ be a subgroup and let $M$ be a topological space with a $Q$-action. We do not assume $M$ to be connected. Consider the projection:
\[
\pi : \tX = M \times \R^{> 0} \longrightarrow X = M \times_Q \R^{> 0} \; .
\]
For an open subset $\tU \subset \tX$ we let $\Rh_{\tX} (\tU)$ be the $\R$-algebra of continuous real-valued functions $\tf : \tU \to \R$ which are locally constant in the first variable in the following sense: For every point $(m,u) \in \tU$ there is a neighborhood $V = V_u$ of $m$ in $M$ with $V \times \{ u \} \subset \tU$ such that $\tf (\_ , u)$ is constant on $V$. In this way we obtain a sheaf of $\R$-algebras $\Rh_{\tX}$ on $\tX$. For a topological space $Y$, let $\Ch^0_Y$ denote the sheaf of real-valued continuous functions on $Y$. By construction we have $\Rh_{\tX} \subset \Ch^0_{\tX}$. Consider the following subsheaf $\Rh_X$ of $\Ch^0_X$:
\[
\Rh_X = (\pi_* \Rh_{\tX})^Q \subset (\pi_* \Ch^0_{\tX})^Q = \Ch^0_X \; .
\]
We view $\Rh_X$ as the sheaf of continuous functions on $X$ which are locally constant along the leaves of the $1$-codimensional ``foliation'' $\Fh$ of $X$ by the images of the spaces $M \times \{ u \}$ in $X$. Note that in general the continuous bijection
\[
\pi \, |_{M \times \{ u \} } : M \times \{ u \} \silo \pi (M \times \{ u \})
\]
will not be a homeomorphism if $\pi (M \times \{ u \} )$ is equipped with the subspace topology of $X$. If $Q$ acts properly discontinuously on $M \times \R^{> 0}$ and if $M$ is a manifold, then $\Fh$ is an actual $1$-codimensional foliation. In general however the partition of $X$ into the disjoint spaces $\pi (M \times \{ u \})$ for $u \in \R^{> 0} \mod Q$ will not be locally trivial.

\begin{defn}
\label{t101}
The leafwise cohomology of $X$ is the graded commutative unital cohomology $\R$-algebra
\[
H^{\hullet}_{\Fh} (X) := H^{\hullet} (X , \Rh_X) \; .
\]
The natural $\R^{> 0}$-actions on $X$ and $\tX$ induce an $\R^{> 0}$-action on $H^{\hullet}_{\Fh} (X)$. Any homeomorphism of $M$ which commutes with the $Q$-action induces an isomorphism of $H^{\hullet}_{\Fh} (X)$. The constant functions on $X$ embed as $\R \cdot 1$ into $H^0_{\Fh} (X)$. 
\end{defn}

In $C^{\infty}$-foliation theory one usually assumes the sections of $\Rh$ to be smooth in the transversal direction (here: in the $\R^{> 0}$-variable). In favourable cases the maximal Hausdorff quotient $\overline{H}^{\hullet}_{\Fh} (X)$ in a suitable Fr\'echet topology then has an interpretation as the global $i$-forms along $\Fh$ which are harmonic along the leaves \cite{alvarez2}.

We apply Definition \ref{t101} to the following spaces. Let the notations and assumptions be as in the previous section. In particular $\car \Nh_0 \supset \car \eX_0$ consists of almost all prime numbers. We let $Q = \Q^{> 0}_0$ and consider the $\Q^{> 0}_0$-spaces $M = \pr^{-1}_{\eX} (\eta) , \ceX (\C) , \pr^{-1}_{\eX_0} (\eta_0)$ and $\ceX_0 (\C)$. The corresponding $\R^{> 0}$-systems are denoted by $X_{\eta} , X , X_{0 \eta_0}$ and $X_0$. We then have natural $G \times \R^{> 0}$-actions on $H^{\hullet}_{\Fh} (X_{\eta})$ and $H^{\hullet}_{\Fh} (X)$, and $\R^{> 0}$-actions on $H^{\hullet}_{\Fh} (X_{0 \eta_0})$ and $H^{\hullet}_{\Fh} (X_0)$. 

\begin{theorem}
\label{t102}
In the situation of Corollary \ref{t15}, the vector spaces $H^0_{\Fh} (X_{\eta}) , H^0_{\Fh} (X) , H^0_{\Fh} (X_{0\eta_0})$ and $H^0_{\Fh} (X_0)$ consist of constant functions only. Thus they are canonically isomorphic to $\R$ with trivial action of $G \times \R^{> 0}$ resp. $\R^{> 0}$. 
\end{theorem}

\begin{proof}
We have $H^0_{\Fh} (X_0) = H^0_{\Fh} (X)^G$ and $H^0_{\Fh} (X_{0 \eta_0}) = H^0_{\Fh} (X_{\eta})^G$. Hence it suffices to treat $X_{\eta}$ and $X$. The restriction map $H^0_{\Fh} (X) \to H^0_{\Fh} (X_{\eta})$ is injective since by Theorem \ref{t11} the subspace $X_{\eta}$ is dense in $X$. Hence it suffices to show that $H^0_{\Fh} (X_{\eta}) = \R$. By definition a section $f \in H^0_{\Fh} (X_{\eta})$ is a continuous $\Q^{> 0}_0$-invariant function $f : \pr^{-1}_{\eX} (\eta) \times \R^{> 0} \to \R$ such that for all $u \in \R^{> 0}$ the function $f ( \_ , u)$ is locally constant on $\pr^{-1}_{\eX} (\eta)$. Hence $f (\_ , u)$ is constant on the (connected!) fibres of the continuous map $r : \pr^{-1}_{\eX} (\eta) \twoheadrightarrow \Q^{> 0}_0 \hZ^{\times}$ of \eqref{eq:39}. Hence $f$ factors over a continuous $\Q^{> 0}_0$-invariant map
\[
\of : \Q^{> 0}_0 \hZ^{\times} \times \R^{> 0} \longrightarrow \R \; 
\]
and hence over a continuous map
\[
\tf : \Q^{> 0}_0 \hZ^{\times} \times_{\Q^{> 0}_0} \R^{> 0} \longrightarrow \R \; .
\]
By the next proposition, $\tf$ and hence $f$ are constant.
\end{proof}

Recall that a topological space $Y$ is called irreducible if it is not the union of two proper closed subsets or equivalently if any two non-empty open subsets have non-empty intersection. All continuous maps from an irreducible topological space $Y$ to a Hausdorff space are constant. In particular, all continuous real valued functions on $Y$ are constant.

\begin{prop}
\label{t103}
Assume that $\car \Nh_0$ contains almost all prime numbers. If $\Q^{> 0}_0 \hZ^{\times} \times \R^{> 0}$ carries the adele topology i.e. the subspace topology of $\Q^{> 0}_0 \hZ \times \R$, then the quotient space $Y = \Q^{> 0}_0 \hZ^{\times} \times_{\Q^{> 0}_0} \R^{> 0}$ is irreducible.
\end{prop}

\begin{rem} \em
By \cite[Lemma 3.1]{laca}, the orbits of the $\Q^{> 0}$-action on $\Q^{> 0} \hZ^{\times} \times \R^{> 0}$ are closed. The same argument works for $\Q^{> 0}_0$ instead of $\Q^{> 0}$ and it follows that the points of $Y$ are closed, i.e. $Y$ is a $T_1$-space. 
\end{rem}

\begin{proof}
The continuous homomorphism of topological groups
\[
\hZ^{\times} \times \R^{> 0} \longrightarrow \Q^{> 0}_0 \hZ^{\times} \times_{\Q^{> 0}_0} \R^{> 0} \; , \; (a,u) \mapsto [a,u]
\]
is an isomorphism of abstract groups. We write $(\hZ^{\times} \times \R^{> 0})^{\adele}$ for the group $\hZ^{\times} \times \R^{> 0}$ equipped with the topology induced by the right hand side. We have to show that \\
$(\hZ^{\times} \times \R^{> 0})^{\adele}$ is irreducible. Since multiplication by an element of $\hZ^{\times} \times \R^{> 0}$ is a homeomorphism of $(\hZ^{\times} \times \R^{> 0})^{\adele}$, it suffices to show that for any open neighborhood $U$ of $(1,1) \in (\hZ^{\times} \times \R^{> 0})^{\adele}$ and any $(b,v) \in \hZ^{\times} \times \R^{> 0}$ we have $U \cap (b,v) U \neq \emptyset$. For a finite set $S'$ of prime numbers and $0 < \varepsilon < 1$ consider the set $U_{S' , \varepsilon}$ consisting of all $(a,u) \in \hZ^{\times} \times \R^{> 0}$ such that there exists some $r \in \Q^{> 0}_0$ with
\[
|r a_p - 1|_p < \varepsilon \; \text{for} \; p \in S' \; , \; |ra_p-1|_p \le 1 \; \text{for} \; p \notin S' \; \text{and} \; |r^{-1} u -1 |_{\infty} < \varepsilon \; .
\]
The subsets $U_{S' , \varepsilon}$ form a basis of neighborhoods of $(1,1) \in (\hZ^{\times} \times \R^{> 0})^{\adele}$, and we may assume that $U = U_{S' , \varepsilon}$ for some $S' , \varepsilon$. We have
\[
U_{S' , \varepsilon} = \{ (a,u) \in \hZ^{\times} \times \R^{> 0} \mid \text{ex.} \; \nu \in \Nh_0 \; \text{s.t.} \; |\nu a_p -1 |_p < \varepsilon \; \text{for} \; p \in S' , |\nu^{-1} u-1|_{\infty} < \varepsilon \} \; . 
\]
We will now construct an element $(a,u) \in U_{S' , \varepsilon}$ such that $(ab , uv) \in U_{S' , \varepsilon}$, and hence $U_{S' , \varepsilon} \cap (b,v) U_{S' , \varepsilon} \neq \emptyset$. Let $S$ be the finite set of prime numbers which are not in $\car \Nh_0$ and choose $\nu_0 \in \Z$ such that $|\nu_0 - 1|_p < \varepsilon$ for all $p \in S' \cup S$. It follows that $\nu_0 \in \pm \Nh_0$. There are positive integers $M_0$ and $A_0$ such that for all integers $A \ge A_0$ the numbers
\[
\nu = \nu_0 + A \prod_{p \in S' \cup S} p^{M_0}
\]
are in $\Nh_0$ and satisfy $|\nu - 1|_p < \varepsilon$ for all $p \in S' \cup S$. It follows that $(1 , \nu) \in U_{S' , \varepsilon}$ for all $A \ge A_0$. The desired element $(a,u)$ will be of the form $(1, \nu)$ for suitable $A$. It needs to satisfy the condition $(b , \nu v) \in U_{S' , \varepsilon}$, i.e. there should be some $\nu' \in \Nh_0$ such that
\[
|\nu' b_p -1 |_p < \varepsilon \quad \text{for} \; p \in S' \; \text{and} \; |\nu^{'-1} \nu v-1|_{\infty} < \varepsilon \; .
\]
Choose some $\nu'_0 \in \Z$ such that $|\nu'_0 b_p - 1|_p < \varepsilon$ for $p \in S' \cup S$. Then for all integers $A' \ge A_0$ the numbers
\[
\nu' = \nu'_0 + A' \prod_{p \in S' \cup S} p^{M_0}
\]
are in $\Nh_0$ and satisfy $|\nu' b_p -1|_p < \varepsilon$ for all $p \in S' \cup S$. Choose sequences $(A_n) , (A'_n)$ of positive integers $\ge A_0$ with $A_n \to \infty , A'_n \to \infty$ such that $A'_n A^{-1}_n \to v$ for $n \to \infty$. For the corresponding numbers $\nu_n , \nu'_n \in \Nh_0$ we then have $\nu^{' -1}_n \nu_n \to v^{-1}$ and hence for large enough $n$, we have
\[
|\nu'_n b_p - 1|_p < \varepsilon \; \text{for} \; p \in S' \quad \text{and} \quad |\nu^{'-1}_n \nu_n v-1 |_{\infty} < \varepsilon \; .
\]
Hence $(1 , \nu_n) \in U_{S' , \varepsilon} \cap (b,u) U_{S' , \varepsilon}$ for large enough $n$. 
\end{proof}

\begin{rems} \em
1) We have
\[
U_{S' , \varepsilon} = \bigcup_{\nu \in \Nh_0} \{ (a,u) \in \hZ^{\times} \times \R^{> 0} \mid |\nu a_p - 1| < \varepsilon \; \text{for} \; p \in S' , |\nu^{-1} u-1|_{\infty} < \varepsilon \} \; .
\]
Thus the topology of $(\hZ^{\times} \times \R^{> 0})^{\adele}$ is reminiscient of the subspace topology induced by the ambient space on a dense leaf of a foliation. The foliation of the torus by lines of irrational slope is the simplest example. The ordinary topology on $\hZ^{\times} \times \R^{> 0}$ could then be viewed as the leaf topology which is finer.\\
2) The inclusion of sheaves $\uR \subset \Rh_X$ implies that $H^0 (X , \uR) \subset H^0 (X , \Rh_X) = H^0_{\Fh} (X) = \R$. Hence $H^0 (X , \uR) = \R$ and similarly for $X_{\eta} , X_{0 \eta_0}$ and $X_0$. It follows again that these spaces are connected, c.f. Theorem \ref{t14}.\\
3) The assertion of Theorem \ref{t102} also holds for $\Eh = \Eh_{\tors}$ with the same proof, except that the connectedness of the fibres of the map $r$ in \eqref{eq:39} is easier to see since Lemma \ref{t13} is not needed. 
\end{rems}

\section{(Generalized) functions on $\ceX (\C)$ and $\ceX_0 (\C)$} \label{sec11n}
Let $\C$ denote the field of complex numbers. As in section \ref{sec:4n}, let $\eX_0$ be an arithmetic scheme i.e. integral normal and with a countable function field $K_0$. Via the map \eqref{eq:11*}, functions on $\eX$ and $\eX_0$ give rise to complex valued functions on $W_{\rat} (\eX) (\C) = \deX (\C)$ resp. $W_{\rat} (\eX_0) (\C) = \deX_0 (\C) = \deX (\C) / G$. By the definition of the topologies on $\deX (\C)$ and $\deX_0 (\C)$, these functions are continuous and we obtain multiplicative injective maps (use Lemma \ref{t3})
\begin{equation}
\label{eq:11n1}
[\;] : \Gamma (\eX , \Oh) \longrightarrow C^0 (\deX (\C) , \C) \quad \text{and} \quad [\;] : \Gamma (\eX_0 , \Oh) \longrightarrow C^0 (\deX_0 (\C) , \C) \; .
\end{equation}
Consider the induced $\C$-algebra map on the reduced monoid algebra of $(\Gamma (\eX , \Oh) , \cdot)$ with $\C$-coefficients
\[
\uC \Gamma (\eX, \Oh) \longrightarrow C^0 (\deX (\C) , \C) \; .
\]
Passing to the $G$-fixed points gives a $\C$-algebra map
\begin{equation}
\label{eq:11n2}
W_{\rat} (\Gamma (\eX_0 , \Oh)) \otimes_{\Z} \C = \uC (\Gamma (\eX , \Oh))^G \longrightarrow C^0 (\deX_0 (\C) , \C) \; . 
\end{equation}
Note here that $W_{\rat} (\Gamma (\eX_0 , \Oh)) = \uZ (\Gamma (\eX , \Oh))^G$ by Propositions \ref{t21nn} and \ref{t21n} since $\Gamma (\eX , \Oh)$ is the integral closure of $\Gamma (\eX_0 , \Oh)$ in $K$. Moreover the natural inclusion
\[
\uZ \Gamma (\eX , \Oh)^G \otimes_{\Z} \C \longrightarrow \uC \Gamma (\eX , \Oh)^G
\]
is an isomorphism since $G$ acts with finite orbits on $\Gamma (\eX , \Oh)$. Since $\ceX (\C)$ and $\ceX_0 (\C)$ carry the inductive limit topologies, from \eqref{eq:11n1} we get continuous multiplicative injective maps
\begin{equation}
\label{eq:11n3}
\varprojlim_{\Nh_0} \Gamma (\eX , \Oh) \longrightarrow C^0 (\ceX (\C) , \C) \quad \text{and} \quad \varprojlim_{\Nh_0} \Gamma (\eX_0 , \Oh) \longrightarrow C^0 (\ceX_0 (\C) , \C) \; .
\end{equation}
For simplicity, from now on we only consider the monoid $\Nh_0 = \Nh$. Then the latter map is without interest if $\eX_0$ is of finite type over $\spec \Z$ since then $\varprojlim_{\Nh} \Gamma (\eX_0 , \Oh) = \{ 0,1 \}$ as monoids. One gets interesting algebras of functions on $\ceX (\C)$ and $\ceX_0 (\C)$ via the monoid algebra of the topological monoid $\varprojlim_{\Nh} \Gamma (\eX , \Oh)$ with its pro-discrete topology. Here is the construction: The topological group $\overleftarrow{K}^{\times} = \varprojlim_{\Nh} K^{\times}$ fits into an exact sequence
\begin{equation}
\label{eq:11n4}
1 \longrightarrow T \mu_K \longrightarrow \overleftarrow{K}^{\times} \xrightarrow{\pr_1} K^{\times} \longrightarrow 1 \;. 
\end{equation}
Here $T \mu_K = \varprojlim_n \mu_n (K)$ is the Tate-module of $K$ and $\pr_1$ is the projection onto the first component. The compact subgroup $T \mu_K$ is open and hence every point of $\overleftarrow{K}^{\times}$ has an open compact neighborhood. Thus $\overleftarrow{K}^{\times}$ is zero-dimensional and locally compact. Let $\mu$ be the Haar-measure on $\overleftarrow{K}^{\times}$ with $\mu (T \mu_K) = 1$. Let $C_c (\overleftarrow{K}^{\times})$ be the non-unital $\C$-algebra of compactly supported continuous complex valued functions on $\overleftarrow{K}^{\times}$ with the convolution product
\begin{equation}
\label{eq:11n5}
(f \ast g) (r) = \int_{\overleftarrow{K}^{\times} } f (s) g (s^{-1} r) \, d\mu (s) \quad \text{for} \; r \in \overleftarrow{K}^{\times} \; .
\end{equation}
The subring $\Gamma (\eX , \Oh)$ of $K$ contains $\mu (K)$. Hence the submonoid
\begin{equation}
\label{eq:11n6}
\overleftarrow{\Gamma}^{\times} := \varprojlim_{\Nh} (\Gamma (\eX , \Oh) \setminus 0) = \pr^{-1}_1 (\Gamma (\eX , \Oh) \setminus 0) 
\end{equation}
of $\overleftarrow{K}^{\times}$ contains $T \mu_K$. Since $\overleftarrow{\Gamma}^{\times}$ is open and closed in $\overleftarrow{K}^{\times}$ its characteristic function $\chi = \chi_{\overleftarrow{\Gamma}^{\times}}$ is continuous on $\overleftarrow{K}^{\times}$. We can extend continuous functions on $\overleftarrow{\Gamma}^{\times}$ by zero to continuous functions on $\overleftarrow{K}^{\times}$. This gives an injective map
\begin{equation}
\label{eq:11n7}
C_c (\overleftarrow{\Gamma}^{\times}) \hookrightarrow C_c (\overleftarrow{K}^{\times}) \; , \; f \mapsto f \chi \; ,
\end{equation}
where $C_c (\overleftarrow{\Gamma}^{\times})$ denotes the $\C$-vector space of $\C$-valued compactly supported continuous functions on $\overleftarrow{\Gamma}^{\times}$. 

\begin{lemma} \label{t11n1}
The vector subspace $C_c (\overleftarrow{\Gamma}^{\times})$ becomes a non-unital subalgebra of $C_c (\overleftarrow{K}^{\times})$ via the inclusion \eqref{eq:11n7}. Explicitely, for $f,g \in C_c (\overleftarrow{\Gamma}^{\times})$, the induced product is given by the formula
\begin{equation}
\label{eq:11n8}
(f \ast g) (r) = \int_{\overleftarrow{\Gamma}^{\times}} f (s) g (s^{-1} r) \chi (s^{-1} r) \, d\mu (s) \; .
\end{equation}
\end{lemma}

\begin{proof}
For $f,g \in C_c (\overleftarrow{\Gamma}^{\times})$ and $r \in \vK^{\times}$, we have
\begin{equation}
\label{eq:11n9}
(f \chi \ast g \chi) (r) = \int_{\vG^{\times}} f (s) g (s^{-1} r) \chi (s^{-1} r) \, d\mu (s) \; .
\end{equation}
If the integral is non-zero for some $r \in \vK^{\times}$, then there existes some $s \in \vG^{\times}$ with $s^{-1} r \in \vG^{\times}$, and hence $r = ss^{-1} r \in \vG^{\times}$. Thus there is a uniquely determined function $h : \vG^{\times} \to \C$ with
\begin{equation}
\label{eq:11n10}
f\chi \ast g\chi = h \chi \; .
\end{equation}
Since $h \chi$ is in $C_c (\vK^{\times})$, it follows that $h \in C_c (\vG^{\times})$. Identifying $f$ with $f \chi$ etc, formula \eqref{eq:11n8} follows from \eqref{eq:11n9} and \eqref{eq:11n10}.
\end{proof}

Consider the map
\begin{equation}
\label{eq:11n11}
\varphi : C_c (\vG^{\times}) \longrightarrow C^0 (\ceX (\C) , \C)
\end{equation}
defined by 
\begin{equation}
\label{eq:11n12}
\varphi (f) = \int_{\vG^{\times}} f (r) r \, d\mu (r) \; .
\end{equation}
Here we view $r \in \vG^{\times} = \varprojlim_{\Nh} (\Gamma (\eX , \Oh) \setminus 0)$ as a function on $\ceX (\C)$ via the map \eqref{eq:11n3} which we integrate pointwise, i.e.
\begin{equation}
\label{eq:11n13}
\varphi (f) (\cP) = \int_{\vG^{\times}} f (r) \cP (r) \, d\mu (r) \quad \text{for} \; \cP \in \ceX (\C) \; .
\end{equation}
Note that we have $r (\cP) :=\cP (r)$. The group $G = \Aut_{K_0} (K)$ acts continuously and measure preservingly on $\vK^{\times}$ and $\vG^{\times}$. We define a left-action of $G$ on $C_c (\vK^{\times})$ and $C_c (\vG^{\times})$ by setting $(\sigma f) (r) = f (\sigma^{-1} (r))$ for $\sigma \in G$ and $r \in \vK^{\times}$ resp. $\vG^{\times}$. The right action of $G$ on $\ceX (\C)$ gives a natural left action on $C^0 (\ceX (\C) , \C)$. 

\begin{lemma} \label{t11n2}
The map $\varphi$ in \eqref{eq:11n11} is a $G$-equivariant $\C$-algebra homomorphism.
\end{lemma}

\begin{proof}
For $f,g \in C_c (\vG^{\times})$ and $\cP \in \ceX (\C)$ we have with $\chi = \chi_{\vG^{\times}}$
\begin{align*}
\varphi (f \ast g) (\cP) & = \int_{\vK^{\times}} (f \ast g) (r) \chi (r) \cP (r) \, d \mu (r) \\
& = \int_{\vK^{\times}} \int_{\vK^{\times}} f(s) \chi (s) g (s^{-1} r) \chi (s^{-1} r) \chi (r) \cP (r) \, d \mu (s) \, d\mu (r) \\
& = \int_{\vK^{\times}} \int_{\vK^{\times}} f (s) \chi (s) g (r) \chi (r) \chi (sr) \cP (sr) \, d\mu (r) \, d\mu(s) \\
& = \int_{\vG^{\times}} \int_{\vG^{\times}} f(s) g(r) \cP (s) \cP (r) \, d\mu (r) \, d\mu (s) \\
& = \varphi (f) (\cP) \varphi (g) (\cP) = (\varphi (f) \varphi (g)) (\cP) \; .
\end{align*}
$G$-equivariance follows because $G$ preserves the measure $\mu$:
\begin{align*}
(\sigma \varphi (f)) (\cP) & = \varphi (f) (\cP^{\sigma}) = \int_{\vG^{\times}} f (r) \cP (\sigma (r)) \, d\mu (r) \\
& = \int_{\vG^{\times}} f (\sigma^{-1} (r)) \cP (r) \, d \mu (r) = \varphi (\sigma f) (\cP) \; .
\end{align*}
\end{proof}

Using the lemma we obtain a natural homomorphism of $\C$-algebras
\begin{equation}
\label{eq:11n14}
\varphi_0 : C_c (\vG^{\times})^G \longrightarrow C^0 (\ceX (\C) , \C)^G = C^0 (\ceX_0 (\C) , \C) \; .
\end{equation}
This is the analogue for $\ceX_0 (\C)$ of the map \eqref{eq:11n2}. It would be worthwhile to find an intrinsic description of $C_c (\vG^{\times})^G$ in terms of the scheme $\eX_0$ just like $\uC (\Gamma (\eX , \Oh))^G$ could be expressed as $W_{\rat} (\Gamma (\eX_0 , \Oh)) \otimes_{\Z} \C$.

In the rest of this section let $\eX_0$ be an integral normal scheme of finite type over $\spec \Z$. 

We would like to view the elements of $\Gamma (\eX_0 , \Oh)$ as continuous functions not only on $\deX_0 (\C)$ but also on $\ceX_0 (\C)$. Since $\deX_0 (\C)$ is open and closed in $\ceX_0 (\C)$ we can achieve this by extending functions on $\deX (\C)$ by zero. This gives a multiplicative map
\[
\Gamma (\eX_0 , \Oh) \longrightarrow C^0 (\ceX_0 (\C) , \C) \; .
\]
It does not send $1$ to $1$ and it is not compatible with taking divisors. As it turns out, functions on $\eX_0$ do not correspond naturally to functions on $\ceX_0 (\C)$ but to the more general objects in Definition \ref{t11n4} below. For a function $f$ on a topological space let $\supp f = \overline{\{ f \neq 0 \} }$ be its support.

\begin{defn} \label{t11n3}
Let $C (\ceX)$ be the multiplicative monoid consisting of $f = 0$ and all $f \in C^0 (\ceX (\C) , \C)$ such that\\
a) $\supp f = \ceX (\C)$\\
b) $F^*_{\nu}  f = f^{\nu}$ for all $\nu \in \Nh$\\
c) The $G$-orbit of $f_n = f \, |_{F^{-1}_n \deX (\C)}$ is finite for every $n \in \Nh$.\\
\end{defn}
Similarly, we define $C (\ceX_0)$ as the monoid of continuous functions $f_0 : \ceX_0 (\C) \to \C$ with $\supp f_0 = \ceX_0 (\C)$ and $F^*_{\nu} f_0 = f^{\nu}_0$ for all $\nu \in \Nh$. The natural projection $\pi : \ceX (\C) \to \ceX_0 (\C)$ induces an isomorphism of monoids
\begin{equation}
\label{eq:11n15}
\pi^* : C (\ceX_0) \silo C (\ceX)^G \; .
\end{equation}
Let $f \in C (\ceX)$ be a function with $f \, |_{\deX (\C)} = 1$. Then, on $F^{-1}_n \deX (\C)$ the function $f$ takes values in $\mu_n (\C)$ and hence being continuous, $f$ is locally constant. Since $\ceX (\C)$ is the union of the open (and closed) subsets $F^{-1}_n \deX (\C)$ for $n \in \Nh$, it follows that $f$ is a locally constant function $f : \ceX (\C) \to \mu (\C)$. 

Let $\kappa : G \to \hZ^{\times}$ be the cyclotomic character from \eqref{eq:57}. We define a subgroup of $C (\ceX)$ by setting
\begin{align}
\label{eq:11n18}
T_{\kappa} (\ceX) & = \{ f \in C (\ceX) \mid f (\deX (\C)) = 1 \; \text{and} \; \sigma (f) = f^{\kappa (\sigma)} \; \text{for all} \; \sigma \in G \} \\
& = \{ f : \ceX (\C) \to \mu (\C) \mid f \; \text{locally constant} , F^*_{\nu} f = f^{\nu} \; \text{for} \; \nu \in \Nh \; ,  \nonumber \\ 
& f (\deX (\C)) = 1 , \sigma (f) = f^{\kappa (\sigma)} \; \text{for} \; \sigma \in G \} \nonumber 
\end{align}
Note that since $\eX_0$ is of finite type over $\spec \Z$ we have
\begin{equation}
\label{eq:11n19}
T_{\kappa} (\ceX)^G = 1 \; .
\end{equation}
Namely, since the index $d$ of $\kappa (G)$ in $\hZ^{\times}$ is finite, any $f \in T_{\kappa} (\ceX)^G$ takes values in $\mu (\C)^{\kappa (G)} = \mu_d (\C)$. Hence for every $\cP \in \ceX (\C)$ \eqref{eq:11n19} follows:
\[
f (\cP) = (F^*_d f) (F^{-1}_d \cP) = f (F^{-1}_d \cP)^d = 1 \; .
\]
The group $T_{\kappa} (\ceX)$ acts on $C (\ceX)$ by multiplication and on $C (\ceX) \setminus 0$ this action is free because of the condition $\supp f = \ceX (\C)$ for $0 \neq f \in C (\ceX)$. For $n \in \Nh$ define groups
\[
T_n (\ceX) = \{ f : F^{-1}_n \deX (\C) \to \mu_n (\C) \mid f \; \text{locally constant}, F^*_{\nu} f = f^{\nu} \; \text{for} \; \nu \in \Nh \}
\]
and
\[
T_{\kappa , n} (\ceX) = \{ f \in T_n (\ceX) \mid \sigma (f) = f^{\kappa_n (\sigma)} \; \text{for all} \; \sigma \in G \} \; .
\]
Here $\kappa_n = \kappa \mod n : G \to (\Z / n)^{\times}$ is the $\mod n$ reduction of the cyclotomic character. The restriction maps induce a canonical isomorphism
\begin{equation}
\label{eq:11n20}
T_{\kappa} (\ceX) = \varprojlim_n T_{\kappa , n} (\ceX) \; .
\end{equation}
We give $T_{\kappa , n} (\ceX)$ the discrete topology. Then $G$ acts continuously on $T_{\kappa , n} (\ceX)$ because the stabilizers of points contain the open subgroup $\Ker \kappa_n$ of $G$. We give $T_{\kappa} (\ceX)$ the pro-discrete topology resulting from \eqref{eq:11n20}. Then $G$ acts continuously on the Hausdorff topological group $T_{\kappa} (\ceX)$ and we may form the continuous cohomology group $H^1_{\cont} (G , T_{\kappa} (\ceX))$. 

\begin{defn}
\label{t11n4}
We define a multiplicative monoid by setting
\[
\Eh (\ceX_0) = (C (\ceX) / T_{\kappa} (\ceX))^G \; .
\]
\end{defn}
This is the natural target for $\Gamma (\eX_0 , \Oh)$ on $\ceX_0 (\C)$ as we will see. 

{\bf Construction} of a multiplicative map
\[
\delta : \Eh (\ceX_0) \setminus 0 \longrightarrow H^1_{\cont} (G , T_{\kappa} (\ceX)) \; .
\]
For an element $0 \neq \varepsilon = e T_{\kappa} (\ceX) \in \Eh (\ceX_0)$ and $\sigma \in G$, we have $\sigma (e) = e\, c_{\sigma}$ for a uniquely determined element $c_{\sigma} \in T_{\kappa} (\ceX)$. Recall here that $T_{\kappa} (\ceX)$ acts freely on $C (\eX) \setminus 0$. It follows that $c_{\sigma \tau} = c_{\sigma} \sigma (c_{\tau})$ for all $\sigma , \tau \in G$. If $\varepsilon' = e' T_{\kappa} (\ceX)$, writing $\sigma (e') = e' c'_{\sigma}$ and $e' = et$ for uniquely determined elements $c'_{\sigma} , t \in T_{\kappa} (\ceX)$ we get $c'_{\sigma} = c_{\sigma} \sigma (t) t^{-1}$. For a function $f$ on $\ceX (\C)$ let $f_n$ be its restriction to $F^{-1}_n \deX (\C)$. By property c) in Definition \ref{t11n3} the $G$-orbit of $e_n$ is finite. Hence the map $G \to C (\ceX) , \sigma \mapsto \sigma (e_n) = e_n c_{\sigma , n}$ is locally constant. Since $\supp e_n = F^{-1}_n \deX (\C)$, it follows that the map $G \to T_{\kappa , n} (\ceX) , \sigma \mapsto c_{\sigma , n}$ is locally constant. Hence $G \to T_{\kappa} (\ceX) , \sigma \mapsto c_{\sigma}$ is continuous and therefore $(c_{\sigma})$ is a continuous cocycle whose class $\delta (\varepsilon)$ in $H^1_{\cont} (G , T_{\kappa} (\ceX))$ depends only on $\varepsilon$ and not on the choice of $e$.

\begin{prop}
\label{t11n5}
The map $\pi^* : C (\ceX_0) \to \Eh (\ceX_0)$ is injective and an element $\varepsilon \in \Eh (\ceX_0)$ is in the image of $\pi^*$ if and only if $\delta (\varepsilon) = 0$ in $H^1_{\cont} (G , T_{\kappa} (\ceX))$.
\end{prop}

\begin{proof}
If $\pi^* (e_0) = \pi^* (e_1)$ in $\Eh (\ceX_0)$ then $\pi^* (e_0) = \pi^* (e_1) t$ in $C (\ceX)$ for a unique $t \in T_{\kappa} (\ceX)$. Since $\pi^* (e_0)$ and $\pi^* (e_1)$ are $G$-invariant, $t \in T_{\kappa} (\ceX)^G = 1$ by \eqref{eq:11n19}, and therefore $e_0 = e_1$. If $\varepsilon$ is in the image of $\pi^*$ then $\varepsilon = \pi^* (e_0) T_{\kappa} (\ceX)$ for some $e_0 \in C (\ceX_0)$. The cocycle corresponding to $\pi^* (e_0)$ is trivial and hence $\delta (\varepsilon) = 0$. If on the other hand $\delta (\varepsilon) = 0$ for some $\varepsilon = e T_{\kappa} (\ceX) \in \Eh (\ceX_0)$, then $\sigma (e) = e \sigma (t) t^{-1}$ for some $t \in T_{\kappa} (\ceX)$ and all $\sigma \in G$. Then $e' = et^{-1} \in C (\ceX)$ represents $\varepsilon$ and $e' \in C (\ceX)^G$ i.e. $e' = \pi^* (e'_0)$ for some $e'_0 \in C (\ceX_0)$ by \eqref{eq:11n15}. Thus $\varepsilon = \pi^* (e'_0)$ is in the image of $\pi^*$. 
\end{proof}

We will now construct a natural injective multiplicative map
\[
[\;] : \Gamma (\eX_0 , \Oh) \longrightarrow \Eh (\ceX_0) \; .
\]
Recall the multiplicative map
\begin{equation}
\label{eq:11n24}
[\;] : \varprojlim_{\Nh} \Gamma (\eX , \Oh) \longrightarrow \varprojlim_{\Nh} C^0 (\deX (\C) , \C) = C^0 (\ceX (\C) , \C) \; .
\end{equation}
The domain of $[\;]$ is only a monoid whereas the target is a $\C$-algebra. Namely $(\;)^{\nu}$ is multiplicative but not additive on $\Gamma (\eX, \Oh)$ whereas $F^*_{\nu}$ is an algebra endomorphism of $C^0 (\deX (\C) , \C)$. Explicitely the map \eqref{eq:11n24} is given as follows. For $P \in \deX (\C)$ and $n \ge 1$ we have
\begin{equation}
 \label{eq:11n25}
 [(\alpha_{\nu})] (F^{-1}_n P) = [\alpha_n] (P) = \oP (\alpha_n (\ex)) \; .
\end{equation}
Here $\oP : \kappa (\ex) \to \C$ is the extension by $0$ of the character $\oP^{\times} : \kappa (\ex)^{\times} \to \C^{\times}$ and $\alpha (\ex) = \alpha \mod \emm_{\eX , \ex}$. The map $[\;]$ is $G$-equivariant. Recall that the monoid $\varprojlim_{\Nh} \Gamma (\eX , \Oh)$ contains the Tate module $T \mu_K$ as a subgroup. The natural maps
\begin{equation}
\label{eq:11n26}
T\mu_K \longrightarrow \varprojlim_{\Nh} \Gamma (\eX , \Oh) \xrightarrow{\pr_1} \Gamma (\eX , \Oh)
\end{equation}
turn $\varprojlim_{\Nh} \Gamma (\eX , \Oh) \setminus 0$ into a principal homogeneous $T \mu_K$-set over $\Gamma (\eX , \Oh) \setminus 0$. 

\begin{prop}
\label{t11n6}
The map \eqref{eq:11n24} gives rise to a commutative diagram of $\Nh$- and $G$-equivariant injective multiplicative maps, where the lower horizontal map is continuous
\[
\xymatrix{
\varprojlim_{\Nh} \Gamma (\eX , \Oh) \ar@{^{(}->}[r]^-{[\;]} & C (\ceX)\\
T \mu_K \ar@{^{(}->}[u] \ar@{^{(}->}[r]^{[\;]} & T_{\kappa} (\ceX) \ar@{^{(}->}[u]
}
\]
The induced map of continuous cohomology groups is injective as well
\[
H^1_{\cont} (G , T\mu_K) \hookrightarrow H^1_{\cont} (G , T_{\kappa} (\ceX)) \; .
\]
Finally, if $\calpha \in \varprojlim_{\Nh} \Gamma (\eX , \Oh) \setminus 0$ satisfies $[\calpha] \in T_{\kappa} (\ceX)$, then $\calpha \in T\mu_K$. 
\end{prop}

\begin{proof}
Let $\eta = \spec K$ be the generic point of $\eX$ and choose an injective character \\
$\oP^{\times} : \kappa (\eta)^{\times} = K^{\times} \hookrightarrow \C^{\times}$ which exists by Lemma \ref{t3}. Then $(\eta , \oP^{\times}) \in \deX (\C)$ and the map $\alpha \mapsto [\alpha] (\eta, \oP^{\times}) = \oP (\alpha)$ is injective. Hence the map $[\;]$ in \eqref{eq:11n24} is injective. We now check that $[\;]$ takes values in $C (\ceX)$. We may assume that $\calpha = (\alpha_{\nu}) \in \varprojlim_{\Nh} \Gamma (\eX , \Oh) \setminus 0$. For the points $\cP \in \ceX_{\eta} (\C)$ we have $\cP = F^{-1}_n P$ for some $P = (\eta , \oP^{\times})$ where $\oP^{\times} : \kappa (\eta)^{\times} = K^{\times} \to \C^{\times}$ is a character. Hence
\[
[\calpha] (\cP) = [\alpha_n] (P) = \oP^{\times} (\alpha_n) \neq 0 \; ,
\]
and it follows that
\[
\ceX_{\eta} (\C) \subset \supp [\calpha] \; .
\]
By Theorem \ref{t93n} the fibre $\ceX_{\eta} (\C)$ is dense in $\ceX (\C)$ and therefore $\supp [\calpha] = \ceX (\C)$. Thus $f = [\calpha]$ satisfies condition a) in Definition \ref{t11n3}. The $\Nh$- and $G$-equivariance of the maps $[\;]$ and condition b) is clear. For $\cP = F^{-1}_n P$ with $P \in \deX (\C)$ we therefore find
\begin{equation}
\label{eq:11n28}
(\sigma [\calpha]) (\cP) = [\sigma (\alpha_n)] (P) \; .
\end{equation}
Since the $G$-orbit of $\alpha_n$ is finite, the $G$-orbit of $[\calpha] \, |_{F^{-1}_n \deX (\C)}$ is finite as well. Hence condition c) in Definition \ref{t11n3} is satisfied, and $[\calpha] \in C (\ceX)$ follows. For $\czeta = (\zeta_{\nu})$ in $T \mu_K$ we have for $P \in \deX (\C)$ that
\[
[\czeta] (F^{-1}_n P) = [\zeta_n] (P) = \oP^{\times} (\zeta_n) \in \mu_n (\C) \; .
\]
It follows that $[\czeta] \, |_{F^{-1}_n \deX (\C)} \in T_n (\ceX)$.

For $\sigma \in G$ we have
\[
\sigma [\czeta] = [\sigma (\czeta)] = [\czeta^{\kappa (\sigma)}]
\]
and hence
\[
(\sigma [\czeta]) (F^{-1}_n P) = \oP^{\times} (\zeta^{\kappa_n (\sigma)}_n) = \oP^{\times} (\zeta_n)^{\kappa_n (\sigma)} = ([\czeta] (F^{-1}_n P))^{\kappa_n (\sigma)} \; .
\]
It follows that
\[
[\czeta] \, |_{F^{-1}_n \deX (\C)} \in T_{\kappa, n} (\ceX) \quad \text{and} \quad [\czeta] \in T_{\kappa} (\ceX) \; .
\]
By the definition of the topology on $T_{\kappa} (\ceX)$ the map $[\;] : T\mu_K \to T_{\kappa} (\ceX)$ is therefore continuous. Next, we show that the map on cohomology is injective. Assume that for the continuous cocycle $c : G \to T \mu_K$ we have $[c_{\sigma}] = \sigma (t) t^{-1}$ for some $t \in T_{\kappa} (\ceX)$. Choose a point $P = (\eta , \oP^{\times})$ of $\ceX_{\eta} (\C)$ with injective character $\oP^{\times} : K^{\times} \hookrightarrow \C^{\times}$. Writing $c_{\sigma} = (c_{\sigma , \nu})$ with $c_{\sigma , \nu} \in \mu_{\nu} (K)$, we have
\[
[c_{\sigma}] (F^{-1}_n P) = c_{\sigma , n} (P) = \oP^{\times} (c_{\sigma ,n}) \in \mu_n (\C)
\]
and since $\sigma (t) = t^{\kappa (\sigma)}$, setting $\eta_n := t (F^{-1}_n P) \in \mu_n (\C)$ we find
\[
\sigma (t) (F^{-1}_n P) = \eta^{\kappa_n (\sigma)}_n \; .
\]
For $n$ prime to $\car K$ there is a unique $\zeta_n \in \mu_n (K)$ such that $\oP^{\times} (\zeta_n) = \eta_n$ and hence
\[
\eta^{\kappa_n (\sigma)}_n = \oP^{\times} (\zeta_n)^{\kappa_n (\sigma)} = \oP^{\times} (\zeta^{\kappa_n (\sigma)}_n) = \oP^{\times} (\sigma (\zeta_n)) \; .
\]
Evaluating the function $[c_{\sigma}] = \sigma (t) t^{-1}$ on $F^{-1}_n P$ we therefore get
\[
\oP^{\times} (c_{\sigma, n}) = \oP^{\times} (\sigma (\zeta_n)) \oP^{\times} (\zeta^{-1}_n)
\]
and since $\oP^{\times}$ is injective
\[
c_{\sigma , n} = \sigma (\zeta_n) \zeta^{-1}_n \quad \text{for all $n$ prime to} \; \car K \; .
\]
Since
\[
\oP^{\times} (\zeta^{\nu}_{n\nu}) = t (F^{-1}_{n \nu} P)^{\nu} = (F^*_{\nu} t) (F^{-1}_{n \nu} P) = t (F_{\nu} F^{-1}_{n\nu} P) = \oP^{\times} (\zeta_n)
\]
we find $\zeta^{\nu}_{n \nu} = \zeta_n$ for all $n , \nu$ prime to $\car K$. Hence $\czeta = (\zeta_{\nu}) \in T\mu_K$ satisfies $c_{\sigma} = \sigma (\czeta) \czeta^{-1}$ for all $\sigma \in G$. Thus the cohomology class of the cocycle $c$ in $H^1_{\cont} (G , T\mu_K)$ is trivial. Consider $\calpha = (\alpha_{\nu}) \in \varprojlim_{\Nh} \Gamma (\eX , \Oh) \setminus 0$ with $[\calpha] \in T_{\kappa} (\ceX)$. Arguing as above for $t = [\calpha]$, we find $\czeta = (\zeta_{\nu}) \in T\mu_K$ with $\oP^{\times} (\alpha_n) = \oP^{\times} (\zeta_n)$ for all $n \ge 1$, and hence $\calpha = \czeta \in T\mu_K$. 
\end{proof}

From the proposition, we get an injective $G$-equivariant map of monoids
\begin{equation}
\label{eq:11n29}
\Gamma (\eX , \Oh) \xleftarrow{\overset{\pr_1}{\sim}} (\varprojlim_{\Nh} \Gamma (\eX , \Oh)) / T\mu_K \hookrightarrow C (\ceX) / T_{\kappa} (\ceX) \; .
\end{equation}
Passing to $G$-fixed submonoids we get the desired injective multiplicative map
\begin{equation}
\label{eq:11n30}
[\;] : \Gamma (\eX_0 , \Oh) \hookrightarrow \Eh (\ceX_0) \; .
\end{equation}
The Kummer sequence gives an isomorphism, where $n$ is prime to $\car K$
\begin{equation} 
\label{eq:11n31}
\widehat{K^{\times}_0} := \varprojlim_n K^{\times}_0 / (K^{\times}_0)^n \silo \varprojlim_n H^1 (G , \mu_n (K)) = H^1_{\cont} (G , T\mu_K) \; .
\end{equation}
Since $K_0$ is finitely generated we have a canonical inclusion $K^{\times}_0 \subset \hK^{\times}_0$.
Here the second isomorphism holds since the groups $H^0 (G , \mu_n (K))$ are finite. A straightforeward calculation shows that our maps fit into a commutative diagram
\[
\xymatrix{
\Gamma (\eX_0 , \Oh) \setminus 0 \ar@{^{(}->}[r]^{[\;]} \ar@{^{(}->}[d]^{\eqref{eq:11n31}} & \Eh (\ceX_0) \setminus 0 \ar[d]^{\delta} \\
H^1_{\cont} (G , T\mu_K) \ar@{^{(}->}[r] & H^1_{\cont} (G , T_{\kappa} (\ceX)) \; .
}
\]
Together with Proposition \ref{t11n5} it follows that
\[
[\alpha_0] \notin \pi^* C (\ceX_0) \quad \text{for all} \; 0 \neq \alpha_0 \in \Gamma (\eX_0 , \Oh) \; .
\]
We define the divisor $\ddiv \varepsilon_0 \subset \ceX_0 (\C)$ of an element $\varepsilon_0 \in \Eh (\ceX_0)$ with $\varepsilon_0 = e T_{\kappa} (\ceX) , e \in C (\ceX)$ as
\begin{equation}
\label{eq:11n32}
\ddiv \varepsilon_0 = \{ \cP \in \ceX (\C) \mid e (\cP) = 0 \} / G \; .
\end{equation}
Note that since $\varepsilon_0$ is fixed by $G$ the zero set of $e$ is $G$-invariant. It is clear that $\ddiv \varepsilon_0$ does not depend on the choice of representative $e$ for $\varepsilon_0$. 

\begin{prop}
\label{t11n7}
Let $\pr_{\eX_0} : \ceX_0 (\C) \to \eX_0$ be the natural projection, and consider the multiplicative map \eqref{eq:11n30},
\[
[\;] : \Gamma (\eX_0 , \Oh) \hookrightarrow \Eh (\ceX_0) \; .
\]
Then as sets we have
\[
\ddiv [\beta] = \pr^{-1}_{\eX_0} (\ddiv \beta) \quad \text{for all} \; 0 \neq \beta \in \Gamma (\eX_0 , \Oh) \; .
\]
\end{prop}

\begin{proof}
Choose $\calpha = (\alpha_{\nu}) \in \varprojlim_{\Nh} \Gamma (\eX , \Oh) \setminus 0$ with $\alpha_1 = \pi^* (\beta)$. Then $[\beta] \in \Eh (\ceX_0)$ is represented by $[\calpha] \mod T_{\kappa} (\ceX)$. Hence
\begin{equation}
\label{eq:11n33}
\ddiv [\beta] = \{ \cP \in \ceX (\C) \mid [\calpha] (\cP) = 0 \} / G \; .
\end{equation}
For $\cP = F^{-1}_n P$ with $P = (\ex , \oP^{\times}) \in \deX (\C)$ we have
\[
[\calpha] (\cP) = [\alpha_n] (P) = \oP (\alpha_n (\ex)) \; .
\]
Hence $[\calpha] (\cP) = 0$ if and only if $\alpha_n \in \emm_{\eX , \ex}$ i.e. $\alpha_1 \in \emm_{\eX , \ex}$ or equivalently $\ex \in \ddiv \alpha_1$ since $\alpha^n_n = \alpha_1$. Using the projection $\pr_{\eX} : \ceX (\C) \to \eX$ we find 
\begin{align*}
\ddiv [\beta] & = \{ F^{-1}_n (\ex , \oP^{\times}) \mid \ex \in \ddiv \alpha_1 \; , \; n \ge 1 \} / G \\
& = (\pr^{-1}_{\eX} (\ddiv \pi^* (\beta))) / G = (\pr^{-1}_{\eX} \pi^{-1} (\ddiv \beta)) / G \\
& = \pr^{-1}_{\eX_0} (\ddiv \beta) \; .
\end{align*}
Here we have used the commutative diagram
\[
\xymatrix{
\ceX (\C) \ar[r] \ar[d]_{\pr_{\eX}}& \ceX (\C) / G \ar@{=}[r] & \ceX_0 (\C) \ar[d]^{\pr_{\eX_0}} \\
\eX \ar[rr]^{\pi} && \eX_0 \; .
}
\]
\end{proof}

\begin{rem}
More generally, given a monoid $\Nh_0$ generated by a set of prime numbers $\car \Nh_0 \supset \car \eX_0$, all assertions in this section continue to hold if we require $n \in \Nh_0$ instead of $n \in \Nh$ in our definitions of $\overleftarrow{K}^{\times} , T\mu_K , C (\ceX) , T (\ceX) , T_{\kappa} (\ceX)$. We can also restrict to subsystems $\ceX (\C)_{\Eh}$ etc. as long as $\Eh$ is admissible with $\Eh \supset \Eh_f$, so that the fibre over $\eta$ is dense by Theorem \ref{t11} (this denseness was needed in the proof of Proposition \ref{t11n6}). 
\end{rem}

\section{Invertible functions on $\eX , \eX_0$ as sections on $X , X_0$} \label{sec:12}
We refer to \cite[Ch. 8]{claus} for the concept and basic theory of $G$-sheaves for a topological group $G$. Given a topological space $X$ with a continuous right $G$-action, let $\Ch^{sm} = \Ch^{sm}_X$ denote the maximal $G$-subsheaf of the $G^{\delta}$-sheaf $\Ch_X$ of $\C$-valued continuous functions on $X$. For an open subset $U \subset X$ the $\C$-algebra $\Ch^{sm}_X (U)$ consists of all continuous functions $f : U \to \C$ such that for every point $x \in X$ there are neighborhoods $x \in V \subset U$ and $e \in N \subset G$ such that $VN \subset U$ and $f (y^{\sigma}) = f(y)$ for all $y \in V$ and $\sigma \in N$. 

Let $\C , \eX_0 , K_0 , \eX , K$ and $G$ be as in the last section and let $\Nh_0$ be the monoid generated by a set of prime numbers $\car \Nh_0 \supset \car \eX_0$. In this section we explain a way to view invertible regular functions on $\eX$ and $\eX_0$ as ``twisted invertible functions'' on $X$ and $X_0$ -- at least up to $\Nh_0$-torsion. We could also lift the monoid formalism of the last section to $X$ and $X_0$, which would associate ``twisted generalized functions'' on $X$ resp. $X_0$ to all regular functions on $\eX$ resp. $\eX_0$. However, for simplicity we restrict ourselves to the groups of invertible regular functions on $\eX_0$ and $\eX$. Set
\[
\ceX (\C) = \colim_{\Nh_0} \deX (\C) \quad \text{and} \quad \ceX_0 (\C) = \colim_{\Nh_0} \deX_0 (\C) = \ceX (\C) / G \; ,
\]
and consider $\tX = \ceX (\C) \times \R^{> 0}$ with the continuous right $\tG = G \times \Q^{> 0}_0$-action
\begin{equation}
\label{eq:12:32}
(\cP , u) (\sigma , q) := (F_q (\cP^{\sigma}) , q^{-1} u) \; .
\end{equation}
Here $\Q^{> 0}_0$ carries the discrete topology, so that $\tG$ is a totally disconnected locally compact group. For $X_0 = \ceX_0 (\C) \times_{\Q^{> 0}_0} \R^{> 0}$ we have a natural continuous and open projection
\[
\pi : \tX \longrightarrow X_0 = \tX / \tG \; .
\]
The unit group of the ring
\[
\Q_0 = \Z [ p^{-1} \mid p \in \car \Nh_0 ] \subset \Q
\]
is $\Q^{\times}_0 = \pm \Q^{> 0}_0$. For $n \in \Z$ let $\Q_0 (n)$ be the $\Q_0 [\Q^{> 0}_0]$-module given by $\Q_0$ as an abelian group and where $q \in \Q^{> 0}_0$ acts by multiplication with $q^{-n}$. We give $\Q_0 (n)$ the discrete topology. Similarly, for $s \in \R$ resp. $s \in \C$ we define $\R (s)$ resp. $\C (s)$ to be $\R$ resp. $\C$ with their natural topologies and with $q \in \Q^{> 0}_0$ acting by multiplication with $q^{-s}$. As $\Q_0 [\Q^{> 0}_0]$-modules we have inclusions
\[
\Q_0 (n) \subset \R (n) \subset \C (n) \; .
\]
Consider the $\tG$-bundle $\uQ_0 (n)_{\tX}$ of free rank one $\Q_0$-modules on the $\tG$-space $\tX$ defined by
\begin{equation}
\label{eq:12:33}
\uQ_0 (n)_{\tX} = \tX \times \Q_0 (n) \longrightarrow \tX \; .
\end{equation}
Here $\tG = G \times \Q^{> 0}_0$ acts via the second factor on $\Q_0 (n)$ and as in \eqref{eq:12:32} on $\tX$. By $\uQ_0 (n)$ we also denote the $\tG$-sheaf of locally constant sections of $\uQ_0 (n)_{\tX}$. 

Let $\Oh_{\Fh c}$ be the $\tG$-subsheaf of $\Ch^{sm}_{\tX}$ consisting of local sections of $\Ch^{sm}_{\tX}$ which are locally constant in the $\R^{> 0}$-variable. The first multiplicative injective map in \eqref{eq:11n3} induces an injective homomorphism by pullback from $\ceX (\C)$ to $\tX$
\begin{equation}
\label{eq:12:36}
[\;] : \varprojlim_{\Nh_0} \Gamma (\eX , \Oh^{\times}) \longrightarrow \{ \alpha \in \Gamma (\tX , \Oh^{\times}_{\Fh c}) \mid F^*_{\nu} \alpha = \alpha^{\nu} \quad \text{for all} \; \nu \in \Nh_0 \} \; .
\end{equation}
Note here that for affine $\eX_0$, a point $\cP = F^{-1}_n P$ with $P \in \deX (\C)$ and $f = (f_{\nu}) \in \varprojlim_{\Nh_0} \Gamma (\eX , \Oh^{\times})$ the neighborhood $V = F^{-1}_n \deX (\C)$ of $\cP$ is invariant under the open stabilizer $N$ of $f_n$ in $G$. Moreover $[f]^{\sigma} = [f]$ on $V$ for $\sigma \in N$ since $[f] (F^{-1}_n Q) = Q (f_n)$ for $Q \in \deX (\C)$. Hence we see that $[f] \in \Ch^{sm} (\ceX (\C))^{\times}$ and hence $[f]$ pulls back to a function in $\Gamma (\tX , \Oh^{\times}_{\Fh_c})$. The formula $F^*_{\nu} [f] = [f]^{\nu}$ for $\nu \in \Nh_0$ is clear. 

Since the left hand side of \eqref{eq:12:36} is a $\Q_0$-module we have an induced homomorphism where $\otimes = \otimes_{\uZ}$ in the category of abelian sheaves
\begin{equation}
\label{eq:12:37}
[\;] : \varprojlim_{\Nh_0} \Gamma (\eX , \Oh^{\times}) \longrightarrow \{ \alpha \in \Gamma (\tX , \Oh^{\times}_{\Fh c} \otimes \uQ_0 (0)) \mid  F^*_{\nu} \alpha = \alpha^{\nu} \; \text{for} \; \nu \in \Nh_0 \} \; .
\end{equation}
Here we have written $\nu$-multiplication on $\Oh^{\times}_{\Fh c} \otimes \uQ_0 (0)$ as exponentiation. Since $\Q^{> 0}_0 \subset \Q^{\times}_0$, we may replace the condition $F^*_{\nu} \alpha = \alpha^{\nu}$ for $\nu \in \Nh_0$ in \eqref{eq:12:37} by $F^*_q \alpha = \alpha^q$ for $q \in \Q^{> 0}_0$. Replacing $\uQ_0 (0)$ by $\uQ_0 (1)$ which differs only by the $\tG$-action, we may therefore rewrite \eqref{eq:12:37} as a homomorphism
\begin{equation}
\label{eq:12:38}
[\;] : \varprojlim_{\Nh_0} \Gamma (\eX , \Oh^{\times}) \longrightarrow \Gamma (\tX , \Oh^{\times}_{\Fh c} \otimes \uQ_0 (1))^{\Q^{> 0}_0} \; .
\end{equation}
Consider the exact sequence
\[
1 \longrightarrow T_{\Nh_0} \mu (K) \longrightarrow \varprojlim_{\Nh_0} \Gamma (\eX , \Oh^{\times}) \xrightarrow{\pr_1} \Gamma (\eX , \Oh^{\times}) \longrightarrow 1 \; .
\]
The map \eqref{eq:12:36} sends elements of $T_{\Nh_0} \mu (K)$ to sections over $\tX$ of 
\[
\umu_{\Nh_0} (\C) = \varinjlim_{n \in \Nh_0} \umu_n (\C) \subset \Oh^{\times}_{\Fh c} \; .
\]
Since $\umu_{\Nh_0} (\C)$ is the kernel of the map $\Oh^{\times}_{\Fh c} \to \Oh^{\times}_{\Fh c} \otimes \uQ_0 (0)$, it follows that $\varprojlim_{\Nh_0} \mu_{\Nh_0} (K)$ which contains $T_{\Nh_0} \mu (K)$ is the kernel of the maps \eqref{eq:12:37} and hence \eqref{eq:12:38}. Hence \eqref{eq:12:38} induces a homomorphism
\begin{equation}
\label{eq:12:39}
[\;] : \Gamma (\eX , \Oh^{\times}) \longrightarrow \Gamma (\tX , \Oh^{\times}_{\Fh c} \otimes \uQ_0 (1))^{\Q^{> 0}_0}
\end{equation}
whose kernel is $\mu_{\Nh_0} (K)$. 

For a better understanding of the map \eqref{eq:12:39} note that the image of $T_{\Nh_0} \mu (K)$ under $[\;]$ in $\Gamma (\tX , \Oh^{\times}_{\Fh_c})$ is {\it not} a torsion group. Hence we had to tensor with $\uQ (0)$ on the level of sheaves to get a map which is defined on $\Gamma (\eX , \Oh^{\times})$. 

Passing to $\tG$-invariants and noting that the image of the map \eqref{eq:12:39} is a $\Q_0$-module we get an injective homomorphism
\begin{equation}
\label{eq:12:40}
[\;] : \Gamma (\eX_0 , \Oh^{\times}) \otimes \Q_0 \longrightarrow \Gamma (\tX , \Oh^{\times}_{\Fh c} \otimes \uQ_0 (1))^{\tG} \; .
\end{equation}
Hence, up to roots of unity, invertible regular functions on $\eX_0$ can be viewed as sections of the sheaf $\pi^{\tG}_* (\Oh^{\times}_{\Fh_c} \otimes \uQ (1))$ on $X_0$. This achieves the first aim of this section. Next, we define a ``regulator'' map on $\Gamma (\eX_0 , \Oh^{\times})$ and show that it is essentially injective. For simplicity we work with $\tG^{\delta}$-sheaves which is sufficient for our purpose. 

Set $e (z) = \exp (2 \pi i z)$ and consider the exact sequence of $\tG^{\delta}$-sheaves on $\tX$
\[
0 \longrightarrow \uZ \longrightarrow \Ch \xrightarrow{\; e \; } \Ch^{\times} \longrightarrow 1 \; .
\] 
Tensoring with $\uQ_0 (1)$ we get an exact sequence of abelian $\tG^{\delta}$-sheaves
\begin{equation}
\label{eq:12:42}
0 \longrightarrow \uQ_0 (1) \longmapsto \Ch \otimes \uQ_0 (1) \xrightarrow{e \otimes \id} \Ch^{\times} \otimes \uQ_0 (1) \longrightarrow 1 \; .
\end{equation}

From \eqref{eq:12:42} we obtain a canonical $\tG$-equivariant map
\begin{equation}
\label{eq:12:43}
\delta : \Gamma (\tX , \Ch^{\times} \otimes \uQ_0 (1)) \longrightarrow H^1 (\tX , \uQ_0 (1)) \; .
\end{equation}
Using the inclusion $\Oh^{\times}_{\Fh_c} \hookrightarrow \Ch^{\times}$ and composing with \eqref{eq:12:40} we obtain a ``regulator'' map
\begin{equation}
\label{eq:12:44}
r : \Gamma (\eX_0 , \Oh^{\times}) \otimes \Q_0 \longrightarrow H^1 (\tX , \uQ_0 (1))^{\tG} \; .
\end{equation}
The group of the right may be viewed as a replacement for $H^1 (X_0 , \uQ_0 (1))$ which one would expect as the target group of a ``regulator'' map on $\Gamma (\eX_0 , \Oh^{\times})$. It appears because the action of $\Q^{> 0}_0$ on $\check{X}_0 (\C) \times \R^{> 0}$ is not properly discontinuous if the rank of $\Q_0$ is greater than one. Note that we have a canonical isomorphism of abelian $\tG^{\delta}$-sheaves on $\tX$
\[
\Ch \silo \Ch \otimes \uQ_0 (1) \quad \text{where} \; \alpha \mapsto u^{-1} \alpha \otimes 1 \; .
\]
Here $u$ is the coordinate function of $\R^{> 0}$. Thus we can rewrite the sequence \eqref{eq:12:42} on $\tX$ as follows:
\[
0 \longrightarrow \Q_0 (1) \longrightarrow \Ch \xrightarrow{\; E \;} \Ch^{\times} \otimes_{\uZ} \uQ_0 (1) \longrightarrow 1
\]
where $E (\alpha) = e (u^{-1} \alpha) \otimes 1$. Incidentally, we obtain a map
\begin{equation}\label{eq:138n}
E : \Ch (X_0) = \Ch (\tX)^{\tG} \longrightarrow \Gamma (\tX , \Ch^{\times} \otimes \uQ_0 (1))^{\tG} \; ,
\end{equation}
whose image is contained in $\Ker \delta$ contrary to the image of $\Oh^{\times} (\eX_0)$ under the map $r$ of \eqref{eq:12:44}, as we will now see. Let $\mu_{(\Nh_0)} (K_0)$ be the group of roots of unity in $K_0$ of order prime to $\Nh_0$. 

\begin{prop}
\label{t12.1}
The kernel of the regulator map \eqref{eq:12:44} 
\[
r : \Gamma (\eX_0 , \Oh^{\times}) \otimes \Q_0 \longrightarrow H^1 (\tX , \uQ_0 (1))^{\tG}
\]
is contained in $\mu_{(\Nh_0)} (K_0) = \mu (K) \otimes \Q_0$. 
\end{prop}

\begin{rem}
I expect $r$ to be injective if $\eX_0$ is of finite type over $\Z$. For this, it remains to show that $r (\zeta) \neq 0$ for $1 \neq \zeta \in \mu_{(\Nh_0)} (K_0)$. I think this can be done working with suitable periodic orbits on $X_0 = \tX / \tG$. 
\end{rem}

\begin{proof}
Let $\pi : K^{\times} \to K^{\times} \otimes \Q = K^{\times} / \mu (K)$ be the projection and consider the natural continuous inclusion where $\eta = \spec K$
\[
i : Y = \Hom (K^{\times} \otimes \Q , S^1) \overset{\pi^*}{\hookrightarrow} \Hom (K^{\times} , \C^{\times}) = \overset{_{\,\hullet}}{\eta} (\C) \subset \deX (\C) \hookrightarrow \tX
\]
where the last map sends $P$ to $(P, 1)$. We have a commutative diagram
\[
\xymatrix{
\Gamma (\eX_0 , \Oh^{\times}) \otimes \Q_0 \ar[r]^r \ar@{^{(}->}[d] & H^1 (\tX , \uQ_0) \ar[d]^{i^*} \\
K^{\times} \otimes \Q_0 \ar[r]^{\delta \verk [\;]} & H^1 (Y , \uQ_0) \; .
}
\]
Here $\delta$ and $[\;]$ are defined as follows. For $f \otimes 1 \in K^{\times} \otimes \Q$ let $[f \otimes 1] \in \Gamma (Y , \Ch^{\times})$ be defined by $[f \otimes 1] (\chi) = \chi (f \otimes 1)$. We get a homomorphism
\[
[\;] : K^{\times} \otimes \Q_0 \longrightarrow K^{\times} \otimes \Q \xrightarrow{[\;]} \Gamma (Y , \Ch^{\times}) \longrightarrow \Gamma (Y , \Ch^{\times} \otimes \uQ_0) \; .
\]
The map
\[
\delta : \Gamma (Y , \Ch^{\times} \otimes \uQ_0) \longrightarrow H^1 (Y , \uQ_0)
\]
is the connecting homomorphism for the exact sequence on $Y$
\[
0 \longrightarrow \uQ_0 \longrightarrow \Ch \otimes \uQ_0 \xrightarrow{e \otimes \id} \Ch^{\times} \otimes \uQ_0 \longrightarrow 0 \; .
\]
The topological space $Y$ is a pro-torus and it follows from \cite[Proposition 3.8]{KS} that $H^{\hullet} (Y , \uQ_0)$ is canonically isomorphic to the exterior algebra $\Lambda^{\hullet} (K^{\times} \otimes \Q)$ over $\Q$. Under this identification the map $\delta \verk [\;]$ becomes the map
\[
K^{\times} \otimes \Q_0 \longrightarrow K^{\times} \otimes \Q = H^1 (Y , \uQ_0 ) \; .
\]
Hence the kernel of $r$ is contained in $\mu_{(\Nh_0)} (K_0)$. We can prove this conclusion more directly. For $f \in K^{\times} \setminus \mu (K)$ the element $f \otimes 1$ is non-zero in $K^{\times} \otimes \Q$. We want to see that $\delta ([f \otimes 1]) \neq 0$ in $H^1 (Y , \uQ_0)$. Choose a $\Q$-linear splitting $\lambda : K^{\times} \otimes \Q \to \Q (f \otimes 1)$ of the inclusion $\Q (f \otimes 1) \hookrightarrow K^{\times} \otimes \Q$. It gives a continuous inclusion:
\[
\omega : \Sa^1 := \varprojlim_{\Nh} S^1 \hookrightarrow Y \; \text{via} \; (z_n) \mapsto (K^{\times} \otimes \Q \xrightarrow{\lambda} \Q (f \otimes 1) \xrightarrow{\mu} \C^{\times})
\]
where $\mu (f \otimes \frac{n}{m}) = z^n_m$. By construction, we have $\omega ((z_n)) (f \otimes 1) = z_1$. Consider the commutative diagram where in the lower sequence $[f \otimes \frac{n}{m}] ((z_n)) := z^n_m$
\[
\xymatrix{
K^{\times} \otimes \Q \ar[r]^-{[\;]} \ar[d]_{\lambda} & \Gamma (Y , \Ch^{\times} \otimes \Q_0) \ar[r]^{\delta} \ar[d]_{\omega^*} & H^1 (Y , \Q_0) \ar[d]^{\omega^*} \\
\Q (f \otimes 1) \ar[r]^-{[\;]} & \Gamma (\Sa^1 , \Ch^{\times} \otimes \uQ_0) \ar[r]^{\delta} & H^1 (\Sa^1 , \Q_0) \; .
}
\]
We find that
\[
(\omega^* \verk \delta \verk [\;]) (f \otimes 1) = \delta (\pr_1) \in H^1 (\Sa^1 , \Q_0) \; ,
\]
where $\pr_1 : \Sa^1 \to S^1$ is the projection onto the first component. It remains to show that $\delta (\pr_1) \neq 0$ in
\[
H^1 (\Sa^1 , \Q_0) = \colim_{\Nh} H^1 (S^1 , \Q_0) \silo \colim_{\Nh} \Q_0 = \Q \; .
\]
This follows from the commutative diagram
\[
\xymatrix{
\Gamma (S^1 , \Ch^{\times}) \ar[r]^{\delta} \ar[d]^{\pr^*_1} & H^1 (S^1 , \Z) \ar[r]^-{\sim} \ar[d]^{\pr^*_1} & \Z \ar@{^{(}->}[d] \\
\Gamma (\Sa^1 , \Ch^{\times} \otimes \uQ_0) \ar[r]^{\delta} & H^1 (\Sa^1 , \Q_0) \ar[r]^-{\sim} & \Q \; .
}
\]
Namely, for the inclusion $z : S^1 \hookrightarrow \C^{\times}$ in $\Gamma (S^1 , \Ch^{\times})$ we have $\pr_1 = \pr^*_1 (z)$ and hence $\delta (\pr_1) = \delta (z) = 1 \neq 0$, the obstruction to the existence of a continuous logarithm of $z$ on $S^1$, and hence of $\pr_1$ on $\Sa^1$. Incidentally, note that continuous $n$-th roots of $\pr_1$ exist on $\Sa^1$ by definition since $(\pr_n)^n = \pr_1$. 
\end{proof}

Here is the reason for our notation $\Oh_{\Fh c}$. There should be a notion of analytic function on $\ceX (\C)$ so that we could redefine $\Oh_{\Fh c}$ on $\tX$ by additionally imposing an analyticity condition on the $\ceX (\C)$-coordinate i.e. ``along the leaves''. The above considerations should hold for the new $\Oh_{\Fh c}$. 

\section{$\eo$-valued points of $W_{\rat} (X)$} \label{sec:4gn}
In this section we describe points of $W_{\rat} (X)$ with values in certain valuation rings. Let $\eo$ be a rank one valuation ring with quotient field $\C$, maximal ideal $\emm$ and residue field $k = \eo / \emm$. We have
\[
S := \spec \eo = \{ s , \eta \} \quad \text{where} \; s \ent \emm \; \text{and} \; \eta \ent (0) \; .
\]
Let $X$ be a scheme. Via $f (s) = y , f (\eta) = x$, continuous maps $f : S_{\top} \to X_{\top}$ correspond to pairs of points $y,x$ of $X$ with $y \in \overline{\{ x \}}$. We now discuss sheaf maps
\[
f^{\sharp} : \Fh = f^{-1} W_{\rat} (\Oh_X) \longrightarrow \Oh_S \; .
\]
Every presheaf on $S$ is a sheaf. For open $V \subset S$ we therefore have
\[
\Fh (V) = \colim_{U \supset f (V)} W_{\rat} (\Oh_X) (U) \quad \text{where} \; U \subset X \; \text{is open} \; .
\]
The non-empty open subsets $V \subset S$ are $V = \{ \eta \}$ and $V = S$. We get
\[
\Fh (\{ \eta \}) = \colim_{U \ni x} W_{\rat} (\Oh_X) (U) = W_{\rat} (\Oh_{X,x}) \; .
\]
Since $U \supset \{ x,y \}$ is equivalent to $U \ni y$ because $y \in \overline{\{ x \}}$ we find
\[
\Fh (S) = \colim_{U \supset \{ x,y \}} W_{\rat} (\Oh_X) (U) = \colim_{U \ni y} W_{\rat} (\Oh_X) (U) = W_{\rat} (\Oh_{X,y}) \; .
\]
The only non-trivial restriction map of $\Fh$ is the map
\[
\Fh (S) = W_{\rat} (\Oh_{X,y}) \longrightarrow \Fh (\{ \eta \}) = W_{\rat} (\Oh_{X,x})
\]
induced by the map $\Oh_{X,y} \to \Oh_{X,x}$ due to $y \in \overline{\{ x \}}$. Noting that
\[
\Oh_S (\{ \eta \} ) = \C = \Oh_{S, \eta} \quad \text{and} \quad \Oh_S (S) = \eo = \Oh_{S,s} \; 
\]
we see that $f^{\sharp}$ is given by a commutative diagram of ring-homomorphisms:
\begin{equation}
\label{eq:10nn}
\xymatrix{
W_{\rat} (\Oh_{X,y}) \ar[r]^-{f^{\sharp}_s} \ar[d] & \eo \ar@{_{(}->}[d] \\
W_{\rat} (\Oh_{X,x}) \ar[r]^-{f^{\sharp}_{\eta}} & \C \; .
}
\end{equation}
Here $f^{\sharp}_{\eta}$ and $f^{\sharp}_s$ are the stalks of $f^{\sharp}$.

For $(f , f^{\sharp})$ to define a morphism $f : S \to W_{\rat} (X)$, the maps $f^{\sharp}_{\eta} , f^{\sharp}_s$ have to induce maps
\[
\tf^{\sharp}_{\eta} : W_{\rat} (\kappa (x)) \longrightarrow \kappa (\eta) = \C \quad \text{and} \quad \tf^{\sharp}_s : W_{\rat} (\kappa (y)) \longrightarrow \kappa (s) = k
\]
such that the following diagrams commute: 
\begin{equation}
\label{eq:11nn}
\xymatrix{
W_{\rat} (\Oh_{X,x}) \ar[r]^-{f^{\sharp}_{\eta}} \ar@{->>}[d] & \C \ar@{=}[d] \\
W_{\rat} (\kappa (x)) \ar[r]^-{\tf^{\sharp}_{\eta}} & \C
} 
\qquad
\xymatrix{
W_{\rat} (\Oh_{X,y}) \ar[r]^-{f^{\sharp}_s} \ar@{->>}[d] & \eo \ar@{->>}[d] \\
W_{\rat} (\kappa (y)) \ar[r]^-{\tf^{\sharp}_s} & k \; .
}
\end{equation}

\begin{prop}
\label{t211n}
Let $X$ be a scheme and $\eo$ a rank one valuation ring. With notations as above, the preceeding discussion sets up a bijection between $W_{\rat} (X) (\eo)$ and tuples $(x,y, \tf^{\sharp}_{\eta})$. Here $x,y \in X$ are points with $y \in \overline{\{ x \}}$ and $\tf^{\sharp}_{\eta}$ is a homomorphism which fits into a commutative diagram as follows (and where $f^{\sharp}_s , \tf^{\sharp}_s$ and $f^{\sharp}_{\eta}$ are uniquely determined by $\tf^{\sharp}_{\eta}$)
\[
\xymatrix{
W_{\rat} (\kappa (y)) \ar[r]^-{\tf^{\sharp}_s} & k \\
W_{\rat} (\Oh_{X,y}) \ar@{->>}[u] \ar[r]^-{f^{\sharp}_s} \ar[d] & \eo \ar@{->>}[u] \ar@{_{(}->}[d] \\
W_{\rat} (\Oh_{X,x}) \ar[r]^-{f^{\sharp}_{\eta}} \ar@{->>}[d] & \C \ar@{=}[d] \\
W_{\rat} (\kappa (x)) \ar[r]^-{\tf^{\sharp}_{\eta}} & \C \; .
}
\]
\end{prop}

As in the discussion of $W_{\rat} (X) (\C)$ we will now give a different description of $W_{\rat} (X) (\eo)$ for certain schemes $X$. 

Let $X = \eX_0$ be an integral normal scheme with function field $K_0$. Let $\eX$ be the integral closure of $\eX_0$ in an algebraic closure $K$ of $K_0$ and set $G = \Aut_{K_0} (K)$. 

\begin{defn} \label{t42nn}
Let $\deX (\eo)$ be the set of triples $(\ex , \ey , \oP^{\times})$ with $\ex \in \eX , \ey \in \overline{\{ \ex \} }$ and \\
$\oP^{\times} : \kappa (\ex)^{\times} \to \C^{\times}$ a homomorphism whose extension by zero $\oP$ fits into a commutative diagram of multiplicative maps which send $1$ to $1$ and $0$ to $0$
\begin{equation}
\label{eq:29n}
\vcenter{\xymatrix{
\kappa (\ey) \ar[r]^{\tP_{\ey}} & k \\
\Oh_{\eX , \ey} \ar[r]^{P_{\ey}} \ar@{->>}[u] \ar@{^{(}->}[d] & \eo \ar@{->>}[u] \ar@{_{(}->}[d] \\
\Oh_{\eX , \ex} \ar[r]^P \ar@{->>}[d] & \C \ar@{=}[d] \\
\kappa (\ex) \ar[r]^{\oP} & \C
} }
\quad \text{or equivalently} \quad 
\vcenter{\xymatrix{
\kappa (\ey) \ar[r]^{\tP_{\ey}} & k \\
\Oh_{\overline{\{\ex \} } , \ey} \ar[r]^{P'_{\ey}} \ar@{->>}[u] \ar@{_{(}->}[d] & \eo \ar@{->>}[u] \ar@{_{(}->}[d] \\
\kappa (\ex) \ar[r]^{\oP} & \C \; .
}}
\end{equation}
\end{defn}

Here $\overline{ \{ \ex \} }$ is considered as a closed reduced subscheme of $\eX$. Note that
\[
\Oh_{\overline{ \{ \ex \} } , \ey} = \Imm ( \Oh_{\eX , \ey} \to \Oh_{\eX , \ex} \to \kappa (\ex) ) \; .
\]
The maps $P_{\ey}$ and $P'_{\ey}$ determine each other by the factorization
\[
P_{\ey} : \Oh_{\eX , \ey} \twoheadrightarrow \Oh_{\overline{ \{ \ex \} } , \ey} \xrightarrow{P'_{\ey} } \eo \; .
\]
It is clear that the maps $P , P_{\ey} , P'_{\ey} , \tP_{\ey}$ and hence the diagrams \eqref{eq:29n} are uniquely determined by $\oP$. We have
\begin{equation}
\label{eq:148n}
P^{-1}_{\ey} (0) = \Ker (\Oh_{\eX , \ey} \longrightarrow \kappa (\ex))
\end{equation}
since $\oP (\kappa (\ex)^{\times}) \subset \C^{\times}$ which follows from $\oP (1) = 1 , \oP (0) = 0$. Note that since $\kappa (\ey)$ is a field and $\tP_{\ey} (1) = 1$ we have $\tP_{\ey} (\kappa (\ey)^{\times}) \subset k^{\times}$. It follows that
\begin{equation}
\label{eq:35n}
P^{-1}_{\ey} (\emm) = \emm_{\eX , \ey} \; .
\end{equation}

The group $G$ and the monoid $\Nh$ act on $\deX (\eo)$ by setting
\[
(\ex , \ey , \oP^{\times})^{\sigma} = (\ex^{\sigma} , \ey^{\sigma} , \oP^{\times} \verk \sigma) \quad \text{for} \; \sigma \in G 
\]
and
\[
F_{\nu} (\ex , \ey , \oP^{\times}) = (\ex , \ey , \oP^{\times} \verk (\;)^{\nu}) \quad \text{for} \; \nu \in \Nh \; .
\]
The two actions commute and
\begin{equation}
\label{eq:30n}
\deX_0 (\eo) := \deX (\eo) / G
\end{equation}
inherits an $\Nh$-action. For each $\nu \in \Nh$ the Frobenius $F_{\nu}$ is injective on $\deX (\eo)$ and hence on $\deX_0 (\eo)$ since $(\;)^{\nu} : \kappa (\ex) \to \kappa (\ex)$ is surjective for all $\ex \in \eX$. Combining Propositions \ref{t21nn} b) and \ref{t211n} we get a canonical $G$- and $\Nh$-equivariant bijection
\begin{equation}
\label{eq:31n}
W_{\rat} (\eX) (\eo) = \deX (\eo) \; .
\end{equation}
The morphism $\eX \to \eX_0$ induces a canonical map
\begin{equation}
\label{eq:32}
W_{\rat} (\eX) (\eo) \longrightarrow W_{\rat} (\eX_0) (\eo) \; .
\end{equation}

\begin{theorem} \label{t53}
In the above situation, assume that $K_0$ is perfect and that $\C$ and hence $k$ are algebraically closed. Then the following map induced by \eqref{eq:32} 
\begin{equation}
\label{eq:33n}
W_{\rat} (\eX) (\eo) / G = \deX_0 (\eo) \silo W_{\rat} (\eX_0) (\eo)
\end{equation}
is an $\Nh$-equivariant bijection. More generally, if $K_1 \subset K$ is a normal extension of $K_0$ and $G_1 = \Aut_{K_0} (K_1)$ and if $\eX_1$ is the normalization of $\eX_0$ in $K_1$ the morphism $\eX_1 \to \eX_0$ induces an $\Nh$-equivariant bijection
\[
W_{\rat} (\eX_1) (\eo) / G_1 \silo W_{\rat} (\eX_0) (\eo) \; .
\]
\end{theorem}
The proof of \eqref{eq:33n} will be given after the next Proposition and Corollary. The second assertion of Theorem \ref{t53} follows formally from the isomorphism \eqref{eq:33n} as in the proof of Corollary \ref{t36}. Here is a different description of $\deX (\eo)$ focusing on the map $P_{\ey}$ instead of $\oP$.

\begin{prop}
\label{t44}
Mapping $(\ex , \ey , \oP^{\times})$ to $(\ex , \ey , P_{\ey})$ resp. $(\ex , \ey , P'_{\ey})$ identifies $\deX (\eo)$ with the set of triples $(\ex , \ey , P_{\ey})$ resp. $(\ex , \ey , P'_{\ey})$ where $\ey \in \overline{\{ \ex \} }$ and $P_{\ey} : \Oh_{\eX , \ey} \to \eo$ resp. $P'_{\ey} : \Oh_{ \overline{ \{ \ex \} } , \ey} \to \eo$ is a multiplicative map sending $1$ to $1$ and $0$ to $0$ and such that the following conditions hold:\\
1) For $f \in \Oh_{\eX, \ey}$ with $f (\ex) \neq 0$ in $\kappa (\ex)$ we have $P_{\ey} (f) \neq 0$.\\
2) For $f \in \Oh_{\eX , \ey}$ the value $P_{\ey} (f)$ depends only on $f (\ex) \in \kappa (\ex)$. \\
3) For $f \in \Oh_{\eX , \ey}$ the value of $P_{\ey} (f) \mod \emm$ in $k$ depends only on $f (\ey) \in \kappa (\ey)$.\\
resp.\\
1') For $0 \neq f' \in \Oh_{\overline{\{ \ex \} } , \ey}$ we have $P'_{\ey} (f') \neq 0$.\\
3') For $f' \in \Oh_{\overline{ \{ \ex \} } , \ey}$ the value of $P'_{\ey} (f') \mod \emm$ in $k$ depends only on $f' (\ey) \in \kappa (\ey)$. 
\end{prop}

Here we have written $f (\ex)$ for the image of $f$ under the composition $\Oh_{\eX , \ey} \hookrightarrow \Oh_{\eX , \ex} \to \kappa (\ex)$. 

\begin{proof}
For $(\ex , \ey , \oP^{\times}) \in \deX (\eo)$ the triples $(\ex , \ey , P_{\ey})$ resp. $(\ex , \ey , P'_{\ey})$ have the properties 1)--3) resp. 1'), 3'). Conversely, it follows from 1) and 2) that $P^{-1}_{\ey} (0) = \Ker (\Oh_{\eX , \ey} \longrightarrow \kappa (\ex))$. Noting that $\Oh_{\eX , \ex}$ is the localization of $\Oh_{\eX , \ey}$ with respect to the multiplicative subset of $f$'s with $f (\ex) \neq 0$, conditions 1), 2) therefore allow to build the two lower squares in the left diagram of \eqref{eq:29n}. Condition 3) gives the top square. The argument for $(\ex , \ey , P'_{\ey})$ is similar.
\end{proof}

Here is a description of $\deX (\eo)$ in the affine case.

\begin{cor}
\label{t45}
For $\eX = \spec R$, we may identify $\deX (\eo)$ with the set of multiplicative maps $Q : R \to \eo$ sending $1$ to $1$ and $0$ to $0$ which satisfy the following conditions:\\
a) $\ep = Q^{-1} (0)$ is additively closed, hence a prime ideal of $R$ and there is a factorization $Q : R \to R / \ep \xrightarrow{\overline{Q}} \eo$.\\
b) $\eq  = Q^{-1} (\emm)$ is additively closed, hence a prime ideal of $R$ and there is a factorization $Q \mod \emm : R \to R / \eq \xrightarrow{\tQ} k$. 

Conditions a) resp. b) are equivalent to a') resp. b') where

a') For $f , f' \in R$ with $Q (f') = 0$ we have $Q (f + f') = Q (f)$.\\
b') For $f , f' \in R$ with $Q (f') \in \emm$ we have $Q (f + f') \equiv Q (f) \mod \emm$. 
\end{cor}

\begin{proof}
The equivalence of a) with a') and of b) with b') is evident. We define a bijection $Q \mapsto (\ex , \ey , P_{\ey})$ sending $Q$'s with a), b) to triples as in Proposition \ref{t44}. Let $\ex , \ey$ be the points of $\eX = \spec R$ corresponding to $\ep$ and $\eq$. Since $\ep \subset \eq$, we have $\ey \in \overline{\{ \ex \} }$. For $s \in R \setminus \eq$, we have $Q (s) \in \eo \setminus \emm = \eo^{\times}$. Hence $Q$ extends uniquely to a multiplicative map $Q_{\eq} : R_{\eq} \to \eo$ with $Q^{-1}_{\eq} (0) = \ep R_{\eq}$ and $Q^{-1}_{\eq} (\emm) = \eq R_{\eq}$. In particular, for $f \in R_{\eq}$ with $f (\ex) = f \mod \ep R_{\eq} \neq 0$ we have $Q_{\eq} (f) \neq 0$. Assume that $f (\ex) = f' (\ex)$ for $f , f' \in R_{\eq}$, i.e. $f - f' \in \ep R_{\eq}$. Writing $f - f' = rs^{-1}$ for $r \in \ep$ and $s \in R \setminus \eq$ where $sf, sf' \in R$ we have $sf - sf' \in \ep$ and hence by a)
\[
Q (sf) = \overline{Q} (sf \mod \ep) = \overline{Q} (sf' \mod \ep) = Q (sf') \; .
\]
Thus we have $Q_{\eq} (f) = Q_{\eq} (f')$ by multiplicativity, since $Q (s) \in \eo^{\times} \subset \C^{\times}$. Next assume that $f (\ey) = f' (\ey)$ for $f , f' \in R_{\eq}$, i.e. $f - f' \in \eq R_{\eq}$. Writing $f - f' = rs^{-1}$ for $r \in\eq$ and $s \in R \setminus \eq$ such that $sf , sf' \in R$ we have $sf - sf' \in \eq$ and therefore by b),
\[
Q (sf) \mod \emm = \tQ (sf \mod \eq) = \tQ (sf' \mod \eq) = Q (sf') \mod \emm \; .
\]
By multiplicativity, and since $Q (s) \in \eo^{\times}$ this implies that
\[
Q_{\eq} (f) \mod \emm = Q_{\eq} (f') \mod \emm \; .
\]
For the map $P_{\ey} = Q_{\eq}$ on $\Oh_{\eX , \ey} = R_{\eq}$ we have therefore verified conditions 1)--3) in Proposition \ref{t44} and hence defined our map $Q \mapsto (\ex , \ey , P_{\ey})$. We define a map in the other direction by sending $(\ex , \ey , P_{\ey})$ to $Q := P_{\ey} \, |_R$ with $R \subset \Oh_{\eX , \ey}$. If $\ep$ and $\eq$ are the prime ideals of $R$ corresponding to $\ex$ and $\ey$ we have by \eqref{eq:148n}
\[
P^{-1}_{\ey} (0) = \Ker (\Oh_{\eX , \ey} \to \kappa (\ex)) = \ep R_{\eq} \; .
\]
Hence $Q^{-1} (0) = P^{-1}_{\ey} (0) \cap R = \ep R_{\eq} \cap R = \ep$. Moreover \eqref{eq:35n} asserts that $P^{-1}_{\ey} (\emm) = \emm_{\eX , \ey}$ and we get
\[
Q^{-1} (\emm) = P^{-1}_{\ey} (\emm) \cap R = \eq R_{\eq} \cap R = \eq \; .
\]
Straightforeward arguments show that $Q$ satisfies the conditions a) and b) of Corollary \ref{t45} and that our maps $Q \mapsto (\ex , \ey , P_{\ey})$ and $(\ex , \ey , P_{\ey}) \mapsto Q$ are inverse to each other.
\end{proof}

We will now prove Theorem \ref{t36}. In order to do so we first introduce a variant $\ddeX (\eo)$ of $\deX (\eo)$ on which the $G$-action is more transparent. 

For a point $\ex_0 \in \eX_0$ consider the base change diagram
\[
\xymatrix{
\pi^{-1} (\spec \Oh_{\eX_0 , \ex_0}) \ar[r] \ar[d]_{\pi_{\ex_0}} & \eX \ar[d]^{\pi} \\
\spec \Oh_{\eX_0 , \ex_0} \ar[r] & \eX_0 \; .
}
\]
Since $\pi$ is integral and surjective, $\pi_{\ex_0}$ has these properties too and in particular $\pi^{-1} (\spec \Oh_{\eX_0 , \ex_0})$ is affine,
\[
\pi^{-1} (\spec \Oh_{\eX_0 , \ex_0}) = \spec \Oh_{\eX , \ex_0} \quad \text{where} \; \Oh_{\eX , \ex_0} := \Gamma (\pi^{-1} (\spec \Oh_{\eX_0 , \ex_0}) , \Oh) \; .
\]
The $\Oh_{\eX_0 , \ex_0}$-algebra $\Oh_{\eX , \ex_0}$ is a normal domain with quotient field $K$.

For every point $\ex \in \eX$ with $\pi (\ex) = \ex_0$ there are natural maps
\[
\Oh_{\eX, \ex_0} \hookrightarrow \Oh_{\eX , \ex} \twoheadrightarrow \kappa (\ex) \; .
\]
For a point $\ey_0 \in \eX_0$ with $\ey_0 \in \overline{\{ \ex_0 \} }$ we have a commutative diagram (of inclusions)
\[
\xymatrix{
\Oh_{\eX , \ey_0} \ar[r] & \Oh_{\eX , \ex_0} \\
\Oh_{\eX_0 , \ey_0} \ar[u]\ar[r] & \Oh_{\eX_0  , \ex_0} \ar[u] \; .
}
\]

\begin{defn}
\label{t217n}
In the above situation we define $\ddeX (\eo)$ to be the set of triples $(\ex , \ey_0 , \oP^{\times})$ where $\ex \in \eX , \ey_0 \in \overline{\{ \ex_0 \} } \subset \eX_0$ with $\ex_0 = \pi (\ex)$ and $\oP^{\times} : \kappa (\ex)^{\times} \to \C^{\times}$ a homomorphism whose extension by zero $\oP$ fits into a commutative diagram of multiplicative maps where $1 \mapsto 1$ and $0 \mapsto 0$ and $\ey \in \pi^{-1} (\ey_0)$
\begin{equation}
\label{eq:13n}
\xymatrix{
\kappa (\ey) \ar[r]^{\tP_\ey} & k \\
\Oh_{\eX , \ey_0} \ar[u] \ar[r]^{P_{\ey_0}} \ar[d] & \eo \ar@{->>}[u] \ar@{_{(}->}[d]  \\
\Oh_{\eX , \ex_0} \ar[r]^P \ar[d] & \C \ar@{=}[d] \\
\kappa (\ex) \ar[r]^{\oP} & \C \; .
}
\end{equation}
\end{defn}

\begin{supple} \label{t218n}
The point $\ey \in \pi^{-1} (\ey_0)$ and the diagram \eqref{eq:13n} are uniquely determined by $(\ex , \ey_0 , \oP^{\times})$. In particular given $(\ex , \ey_0 , \oP^{\times})$, the maps $P , P_{\ey_0}$ and $\tP_{\ey}$ are unique. We have $\ey \in \overline{\{ \ex \} }$. 
\end{supple}

\begin{proof}
It is clear that $P$ and $P_{\ey_0}$ are uniquely determined by $\oP$. Set
\[
\eq = \Ker (\Oh_{\eX , \ey_0} \longrightarrow \kappa (\ey)) \; .
\]
All maps in \eqref{eq:13n} send $1$ to $1$ and $0$ to $0$. Hence $\tP_\ey (\kappa (\ey)^{\times}) \subset k^{\times}$ and $\tP^{-1}_\ey (0) = 0$. With the canonical map $\lambda : \Oh_{\eX , \ey_0} \to \kappa (\ey)$ we get
\begin{equation} \label{eq:30-x}
\eq = \lambda^{-1} (0) = \lambda^{-1} (\tP^{-1}_\ey (0)) = P^{-1}_{\ey_0} (\emm) \; .
\end{equation}
Hence $\eq$ is uniquely determined by $P_{\ey_0}$ and hence by $\oP$. Setting $\kappa = \Quot (\Oh_{\eX , \ey_0} / \eq)$, the map $\Oh_{\eX , \ey _0} \to \kappa (\ey)$ induces an isomorphism $\kappa \silo \kappa (\ey)$. Therefore, the image of the composition
\[
\spec \kappa \longrightarrow \spec \Oh_{\eX , \ey_0} / \eq  \longrightarrow \spec \Oh_{\eX , \ey_0} = \pi^{-1} (\spec \Oh_{\eX_0 , \ey_0}) \longrightarrow \eX
\]
is the point $\ey$. Moreover, since $\tP_\ey$ induces a homomorphism $\kappa (\ey)^{\times} \to k^{\times}$, the map $\tP_{\ey}$ is uniquely determined by its values on the image of $\Oh_{\eX , \ey_0}$ in $\kappa \equiv \kappa (\ey)$. Hence $P_{\ey_0}$ and hence $\oP$ determine $\tP_\ey$ uniquely. Finally, set $\ep = \emm_{\eX , \ex} \cap \Oh_{\eX , \ey_0}$. Using \eqref{eq:148n} and \eqref{eq:30-x} we find
\[
\ep = P^{-1} (0) \cap \Oh_{\eX , \ey_0} = P^{-1}_{\ey_0} (0) \subset P^{-1}_{\ey_0} (\emm) = \eq \; .
\]
Thus $\eq \in \overline{\{ \ep \} }$ in $\spec \Oh_{\eX , \ey_0}$. Under the natural map $\spec \Oh_{\eX , \ey_0} \to \eX$ we have $\ep \mapsto \ex$ and $\eq \mapsto \ey$, and therefore $\ey \in \overline{\{ \ex \} }$. 

\begin{prop}
\label{t311n}
Mapping $(\ex, \ey_0 , \oP^{\times})$ to $(\ex , \ey_0 , P_{\ey_0})$ identifies $\ddeX (\eo)$ with the set of triples $(\ex , \ey_0 , P_{\ey_0})$ where $\ey_0 \in \overline{\{ \ex_0\} }$ and $P_{\ey_0} : \Oh_{\eX , \ey_0} \to \eo$ is a multiplicative map sending $1$ to $1$ and $0$ to $0$ and such that the following three conditions hold:\\
1) For all $f \in \Oh_{\eX , \ey_0}$ with $f (\ex) \neq 0$ in $\kappa (\ex)$ we have $P_{\ey_0} (f) \neq 0$.\\
2) For all $f \in \Oh_{\eX , \ey_0}$ the value $P_{\ey_0} (f)$ depends only on $f (\ex) \in \kappa (\ex)$.\\
3) There are a point $\ey \in \pi^{-1} (\ey_0)$ and a multiplicative map $\tP_{\ey}$, both uniquely determined such that the following diagram commutes
\[
\xymatrix{
\kappa (\ey) \ar[r]^{\tP_{\ey}} & k \\
\Oh_{\eX, \ey_0} \ar[u] \ar[r]^{P_{\ey_0}} & \eo \; .\ar@{->>}[u]
}
\]
\end{prop}

\begin{proof}
For $(\ex , \ey_0 , \oP^{\times}) \in \ddeX (\eo)$ the triple $(\ex, \ey_0 , P_{\ey_0})$ has the properties 1)--3). For the converse note that the ring $\Oh_{\eX , \ex_0}$ is the localization of $\Oh_{\eX , \ey_0}$ with respect to the multiplicative set of $f_0 \in \Oh_{\eX_0 , \ey_0}$ with $f_0 (\ex_0) \neq 0$ in $\kappa (\ex_0)$. Hence condition 1) implies that $P_{\ey_0}$ can be extended to a multiplicative map $P : \Oh_{\eX , \ex_0} \to \C$. By conditions 1) and 2), $P$ factors over a multiplicative map $\oP : \kappa (\ex) \to \C$. Together with 3) we obtain a diagram as in \eqref{eq:13n} and hence a point of $\ddeX (\eo)$. The uniqueness of $\ey$ and $\tP_{\ey}$ were shown in Supplement \ref{t218n}.
\end{proof}

The group $G$ acts on $\ddeX (\eo)$ by setting
\[
(\ex , \ey_0 , \oP^{\times})^{\sigma} = (\ex^{\sigma} , \ey_0 , \oP^{\times} \verk \sigma) \quad \text{for} \; \sigma \in G \; .
\]
The corresponding point in $\pi^{-1} (\ey_0)$ is $\ey^{\sigma}$ and the corresponding maps in \eqref{eq:30-x} are $\tP_\ey \verk \sigma , P_{\ey_0} \verk \sigma$ and $P \verk \sigma$. We set
\begin{equation}
\label{eq:14n}
\ddeX_0 (\eo) = \ddeX (\eo) / G \; .
\end{equation}
The monoid $\Nh$ acts on $\ddeX (\eo)$ via
\[
F_{\nu} (\ex , \ey_0 , \oP^{\times}) = (\ex , \ey_0 , \oP^{\times} \verk (\;)^{\nu}) \quad \text{for} \; \nu \in \Nh \; .
\]
The point $\ey \in \pi^{-1} (\ey_0)$ is invariant under $F_{\nu}$ and the maps in \eqref{eq:13n} are $\tP_\ey \verk (\;)^{\nu} , P_{\ey_0} \verk (\;)^{\nu}$ and $P \verk (\;)^{\nu}$. The $G$- and $\Nh$-actions on $\ddeX (\eo)$ commute and $\ddeX_0 (\eo)$ inherits an $\Nh$-action. For each $\nu \in \Nh$ the Frobenius $F_{\nu}$ is injective on $\ddeX (\eo)$ and hence on $\ddeX_0 (\eo)$ since $(\;)^{\nu} : \kappa (\ex)^{\times} \to \kappa (\ex)^{\times}$ is surjective for all $\ex \in \eX$.

We now define maps
\[
\deX (\eo) \overset{\displaystyle \xrightarrow{\;\;\alpha\;\;}}{\underset{\beta}{\longleftarrow}} \ddeX (\eo)
\]
using the descriptions of $\deX (\eo)$ and $\ddeX (\eo)$ in Propositions \ref{t44} and \ref{t311n}. The map $\alpha$ sends $(\ex , \ey , P_{\ey})$ to $(\ex , \ey_0 , P_{\ey_0})$ where $\ey_0 = \pi (\ey)$ and $P_{\ey_0}$ is the restriction of $P_{\ey}$ to $\Oh_{\eX , \ey_0} \subset \Oh_{\eX , \ey}$. The map $\beta$ sends $(\ex , \ey_0 , P_{\ey_0})$ to $ (\ex , \ey , P_{\ey})$ where $\ey$ is the unique point specified in Proposition \ref{t311n}, 3) and $P_{\ey}$ is the unique multiplicative extension of $P_{\ey_0}$ to a map $P_{\ey} : \Oh_{\eX , \ey} \to \eo$ (see below).

\begin{prop}
\label{t49}
The maps $\alpha$ and $\beta$ are well defined mutually inverse $G$- and $\Nh$-equivariant bijections between $\deX (\eo)$ and $\ddeX (\eo)$. They induce $\Nh$-equivariant bijections between $\deX_0 (\eo)$ and $\ddeX_0 (\eo)$. 
\end{prop}

\begin{proof}
The map $\alpha$ is well defined since conditions 1)--3) of Proposition \ref{t44} on $P_{\ey}$ imply the corresponding conditions of Proposition \ref{t311n} on $P_{\ey_0} = P_{\ey} \, |_{\Oh_{\eX , \ey_0}}$. As for $\beta$, given $(\ex , \ey_0 , P_{\ey_0})$ first note that $\Oh_{\eX , \ey}$ is the localization of $\Oh_{\eX , \ey_0}$ with respect to the multiplicative set
\[
S = \Oh_{\eX , \ey_0} \setminus (\Oh_{\eX , \ey_0} \cap \emm_{\eX , \ey}) = \varphi^{-1} (\kappa (\ey)^{\times})
\]
where $\varphi$ is the map $\Oh_{\eX , \ey_0} \to \kappa (\ey)$ sending $f$ to $f (\ey)$. For $f \in \Oh_{\eX , \ey_0}$ we have by condition 3) of Proposition \ref{t311n}
\[
P_{\ey_0} (f) \mod \emm = \tP_{\ey} (f (\ey)) \quad \text{in} \; k = \eo / \emm \; .
\]
For $f \in S$ it follows that $P_{\ey_0} (f) \mod \emm \in k^{\times}$ i.e. $P_{\ey_0} (f) \in \eo^{\times}$. Hence the multiplicative map $P_{\ey_0} : \Oh_{\eX , \ey_0} \to \eo$ has a unique multiplicative extension to a map $P_{\ey} : \Oh_{\eX , \ey} \to \eo$. It is straightforeward to check conditions 1)--3) of Proposition \ref{t44} for $P_{\ey}$. Note here that besides $f (\ey) \in \kappa (\ey)^{\times}$ for $f \in S$ we also have $f (\ex) \in \kappa (\ex)^{\times}$ because $\ey \in \overline{\{ \ex \} }$ implies that
\[
\Oh_{\eX , \ey_0} \cap \emm_{\eX , \ex} \subset \Oh_{\eX , \ey_0} \cap \emm_{\eX , \ey} \; .
\]
It is clear that the maps $\alpha$ and $\beta$ are inverse to each other and $G$- and $\Nh$-equivariant. 
\end{proof}

We will now construct a canonical $\Nh$-equivariant map
\begin{equation}
\label{eq:15n}
\varphi : \ddeX_0 (\eo) \longrightarrow W_{\rat} (\eX_0) (\eo) \; .
\end{equation}
Given a point $(\ex , \ey_0 , \oP^{\times}) \in \ddeX (\eo)$, consider the corresponding uniquely determined diagram \eqref{eq:13n}. Using Proposition \ref{t211n} for $X = \eX_0$, the point
\[
\varphi (\ex , \ey_0 , \oP^{\times}) \in W_{\rat} (\eX_0) (\eo)
\]
is determined by $\ex_0 = \pi (\ex) , \ey_0$ and the outer square of the following commutative diagram determined by \eqref{eq:13n}. Here we have used Proposition \ref{t21nn} b) for the identifications in the middle
\[
\xymatrix{
W_{\rat} (\kappa (\ey_0)) \ar[r] & W_{\rat} (\kappa (\ey)) \ar@{=}[r] & \uZ \kappa (\ey) \ar[r]^{\tP_\ey} & k \\
W_{\rat} (\Oh_{\eX_0 , \ey_0}) \ar@{->>}[u] \ar[d] \ar[r] & W_{\rat} (\Oh_{\eX , \ey_0}) \ar[u] \ar[d] \ar@{=}[r] & \uZ \Oh_{\eX , \ey_0} \ar[u] \ar[d] \ar[r]^{P_{\ey_0}} & \eo \ar@{_{(}->}[d] \ar@{->>}[u]\\
W_{\rat} (\Oh_{\eX_0 , \ex_0}) \ar[r] \ar@{->>}[d] & W_{\rat} (\Oh_{\eX , \ex_0}) \ar@{=}[r] \ar[d] & \uZ \Oh_{\eX , \ex_0} \ar[r]^{P} \ar[d] & \C \ar@{=}[d] \\
W_{\rat} (\kappa (\ex_0)) \ar[r] & W_{\rat} (\kappa (\ex)) \ar@{=}[r] & \uZ \kappa (\ex) \ar[r]^{\oP} & \C \; .
}
\]
It is clear that the map $\varphi : \ddeX (\eo) \to W_{\rat} (\eX_0) (\eo)$ is $\Nh$-equivariant and that it factors over $\ddeX_0 (\eo)$. 

There is also a canonical $\Nh$-equivariant map in the other direction
\begin{equation}
\label{eq:16n}
\psi : W_{\rat} (\eX_0) (\eo) \longrightarrow \ddeX_0 (\eo) \; .
\end{equation}
For simplicity we construct it only in the case where $K_0$ is perfect, so that $K / K_0$ is Galois with Galois group $G$. In this case we have
\begin{equation}
 \label{eq:17n}
\Oh_{\eX_0 , \ez_0} = \Oh^G_{\eX , \ez_0} \quad \text{for any} \; \ez_0 \in \eX_0 \; .
\end{equation}
Using Propositions \ref{t21nn} and \ref{t21n} we get identifications
\begin{equation}
\label{eq:18n}
W_{\rat} (\Oh_{\eX_0 , \ex_0}) = W_{\rat} (\Oh_{\eX , \ex_0})^G = (\uZ \Oh_{\eX , \ex_0})^G
\end{equation}
and similarly
\begin{equation}
\label{eq:19n}
W_{\rat} (\Oh_{\eX_0 , \ey_0}) = (\uZ \Oh_{\eX , \ey_0})^G \; .
\end{equation}
Let us now assume, that we are given a point $(\ex_0 , \ey_0 , \oP_0)$ in $W_{\rat} (\eX_0) (\eo)$ with corresponding diagram c.f. Proposition \ref{t211n}
\[
\xymatrix{
W_{\rat} (\kappa (\ey_0)) \ar[r]^-{\tP_{0 \ey_0}} & k \\
W_{\rat} (\Oh_{\eX_0 , \ey_0}) \ar@{->>}[u] \ar[d] \ar[r]^-{P_{0 \ey_0}} & \eo \ar@{_{(}->}[d] \ar@{->>}[u]\\
W_{\rat} (\Oh_{\eX_0 , \ex_0}) \ar[r]^-{P_0} \ar[d] & \C \ar@{=}[d] \\
W_{\rat} (\kappa (\ex_0)) \ar[r]^-{\oP_0} & \C \; .
}
\]
Choose a point $\ex \in \eX$ with $\pi (\ex) = \ex_0$. By Corollary \ref{t29n} there is a multiplicative map $\oP : \kappa (\ex) \to \C$ with $\oP (1) = 1$ and $\oP (0) = 0$ such that $\oP_0$ is the composition
\begin{equation}
\label{eq:20n}
\oP_0 : W_{\rat} (\kappa (\ex_0)) \longrightarrow W_{\rat} (\kappa (\ex)) = \uZ \kappa (\ex) \xrightarrow{\oP} \C \; .
\end{equation}
The map $\oP$ is unique $\mod G_{\ex}$. Define $P$ as the composition
\[
P : \Oh_{\eX , \ex_0} \longrightarrow \kappa (\ex) \xrightarrow{\oP} \C \; .
\]
Then $P_0$ is the composition
\begin{equation}
\label{eq:21nn}
P_0 : W_{\rat} (\Oh_{\eX_0 , \ex_0}) \longrightarrow W_{\rat} (\Oh_{\eX , \ex_0}) = \uZ \Oh_{\eX , \ex_0} \xrightarrow{P} \C \; .
\end{equation}
Let $P_{\ey_0}$ be the composition 
\[
P_{\ey_0} : \Oh_{\eX , \ey_0} \longrightarrow \Oh_{\eX , \ex_0} \xrightarrow{P} \C \; .
\]
Then using \eqref{eq:17n} for $\ez_0 = \ey_0$, we have a commutative diagram
\[
\xymatrix{
W_{\rat} (\Oh_{\eX_0, \ey_0}) \ar@{=}[r] \ar[d]_{P_{0 \ey_0}} & W_{\rat} (\Oh^G_{\eX , \ey_0}) \ar[r] & W_{\rat} (\Oh_{\eX , \ey_0}) \ar@{=}[r] & \uZ \Oh_{\eX , \ey_0} \ar[d]^{P_{\ey_0}} \\
\eo \ar@{^{(}->}[rrr] &&& \C \; .
}
\]
Part 1) of Corollary \ref{t213n} for $A = \Oh_{\eX , \ey_0}$ now implies that $P_{\ey_0}$ takes values in $\eo$. Clearly $P_{0 \ey_0}$ is the composition
\begin{equation}
\label{eq:22nn}
P_{0 \ey_0} : W_{\rat} (\Oh_{\eX_0 , \ey_0}) \longrightarrow W_{\rat} (\Oh_{\eX , \ey_0}) = \uZ \Oh_{\eX , \ey_0} \xrightarrow{P_{\ey_0}} \eo \; .
\end{equation}
By construction the following diagram commutes
\begin{equation}
\label{eq:23nn}
\xymatrix{
\Oh_{\eX , \ey_0} \ar[r]^{P_{\ey_0}} \ar[d] & \eo \ar@{_{(}->}[d] \\
\Oh_{\eX , \ex_0} \ar[r]^P \ar[d] & \C \ar@{=}[d] \\
\kappa (\ex) \ar[r]^{\oP} & \C \; .
}
\end{equation}
Choose a point $\ey \in \eX$ with $\pi (\ey) = \ey_0$. Using Corollary \ref{t29n}, there is a multiplicative map $\tP_\ey : \kappa (\ey) \to k$ with $\tP_\ey (1) = 1$ and $\tP_\ey (0) = 0$ such that $\tP_{ 0 \ey_0}$ is the composition
\begin{equation}
\label{eq:24n}
\tP_{0 \ey_0} : W_{\rat} (\kappa (\ey_0)) \longrightarrow W_{\rat} (\kappa (\ey)) = \uZ \kappa (\ey) \xrightarrow{\tP_\ey} k \; .
\end{equation}
Hence the outer rectangle and the left square of the following diagram commute
\[
\xymatrix{
\tP_{0\ey_0} : W_{\rat} (\kappa (\ey_0)) \ar@{^{(}->}[r] & \uZ \kappa (\ey) \ar[r]^-{\tP_\ey} & k \\
P_{0\ey_0} : W_{\rat} (\Oh_{\eX_0 , \ey_0}) \ar[u] \ar@{^{(}->}[r] & \uZ \Oh_{\eX , \ey_0} \ar[r]^-{P_{\ey_0}} \ar[u] & \eo \ar@{->>}[u] \; .
}
\]
Hence the compositions
\[
\vcenter{
\xymatrix{
\uZ \kappa (\ey) \ar[r]^-{\tP_\ey} & k \\
\uZ \Oh_{\eX , \ey_0} \ar[u] & 
}} \qquad \text{and} \qquad
\vcenter{
\xymatrix{ 
& k \\
\uZ \Oh_{\eX , \ey_0} \ar[r]^-{P_{\ey_0}} & \eo \ar@{->>}[u]
}}
\]
agree on
\[
W_{\rat} (\Oh_{\eX_0 , \ey_0}) = W_{\rat} (\Oh_{\eX , \ey_0})^G = (\uZ \Oh_{\eX , \ey_0})^G \; .
\]
Using Theorem \ref{t27n} for $A = \uZ \Oh_{\eX , \ey_0}$ and $\C = k$ it follows that there is some $\sigma \in G$ such that the following diagram commutes
\begin{equation}
\label{eq:25n}
\xymatrix{
\uZ \Oh_{\eX , \ey_0} \ar[r] & \uZ \kappa (\ey) \ar[r]^{\tP_\ey} & k \\
\uZ \Oh_{\eX , \ey_0} \ar[u]^{\sigma} \ar[rr]^{P_{\ey_0}} & & \eo \ar@{->>}[u] \; .
}
\end{equation}
Thus the diagram
\begin{equation}
\label{eq:26n}
\xymatrix{
\uZ \kappa (\ey^{\sigma}) \ar[r]^{\sigma} & \uZ \kappa (\ey) \ar[r]^{\tP_\ey} & k \\
\uZ \Oh_{\eX , \ey_0} \ar[u] \ar[rr]^{P_{\ey_0}} && \eo \ar@{->>}[u] \; .
}
\end{equation}
commutes as well. Combining \eqref{eq:23nn} and \eqref{eq:26n} and replacing $\ey \in \pi^{-1} (\ey_0)$ by $\ey^{\sigma} \in \pi^{-1} (\ey_0)$ and $\tP_\ey$ by $\tP_\ey \verk \sigma$ we obtain a commutative diagram \eqref{eq:13n} as in Definition \ref{t217n} and hence a point $(\ex , \ey_0 , \oP^{\times}) \in \ddeX (\eo)$. Because of the uniqueness assertion \ref{t218n} it depends only on $\ey_0$ and the choices of $\ex$ over $\ex_0$ and of $\oP^{\times}$. Noting that $G$ acts transitively on $\pi^{-1} (\ex_0)$ and that in the construction $\oP^{\times}$ was uniquely determined $\mod G_{\ex}$, it follows that the class
\[
\psi (\ex_0 , \ey_0 , \oP_0) := (\ex , \ey_0 , \oP^{\times}) \mod G \in \ddeX_0 (\eo)
\]
is well defined. This defines the map \eqref{eq:16n}.
\end{proof}

Theorem \ref{t53} is a consequence of Proposition \ref{t49} and the next result.

\begin{theorem}
\label{t219n}
In the above situation, the maps
\[
\ddeX_0 (\eo) \overset{\displaystyle \xrightarrow{\;\;\varphi\;\;}}{\underset{\psi}{\longleftarrow}} W_{\rat} (\eX_0) (\eo)
\]
are mutually inverse $\Nh$-equivariant bijections. 
\end{theorem}

\begin{proof}
Like $\varphi$, the map $\psi$ is $\Nh$-equivariant by construction. Consider the factorizations \eqref{eq:20n}, \eqref{eq:21nn}, \eqref{eq:22nn} and the following one which results from \eqref{eq:24n} since $\kappa (\ey_0) \to \kappa (\ey)$ equals $\kappa (\ey_0) \to \kappa (\ey^{\sigma}) \xrightarrow{\sigma} \kappa (\ey)$
\[
\tP_{0 \ey_0} : W_{\rat} (\kappa (\ey_0)) \longrightarrow W_{\rat} (\kappa (\ey^{\sigma})) = \uZ \kappa (\ey^{\sigma}) \xrightarrow{\tP_\ey \verk \sigma} k \; .
\]
By definition of $\varphi$, these factorizations show that $\varphi \verk \psi = \id$. It is clear that $\psi \verk \varphi = \id$.
\end{proof}

\begin{rem}
In the definition of $\psi$ we used that $K / K_0$ was Galois to get \eqref{eq:18n}, so that Corollary \ref{t213n} could be applied to get \eqref{eq:26n}. However, with more effort one should be able to show that the definition of $\psi$ and Theorem \ref{t219n} and hence Theorem \ref{t53} are valid in general. Note that for fields Corollaries \ref{t28n} or \ref{t29n} do not assume that $\kappa / \kappa_0$ is Galois, and those arguments might extend to rings. 
\end{rem}
We end this section with some remarks on the map $W_{\rat} (\eX) (\eo) \to W_{\rat} (\eX) (\C)$.

Let $\eo$ be a rank one valuation ring with algebraically closed quotient field $\C$. Let $\eX_0$ be an integral normal scheme with function field $K_0$ and $\eX$ its normalization in an algebraic closure $K$ of $K_0$.

\begin{cor}
\label{t46}
a) For $\eX = \spec R$, the natural map
\[
W_{\rat} (\eX) (\eo) \longrightarrow W_{\rat} (\eX) (\C)
\]
is injective.\\
b) Let $(\ex , \ey , \oP^{\times})$ and $(\ex' , \ey' , \oP^{'\times})$ be two points of $W_{\rat} (\eX) (\eo) = \deX (\eo)$ with the same image in $W_{\rat} (\eX) (\C) = \deX (\C)$, i.e. $\ex = \ex' , \oP^{\times} = \oP^{'\times}$. If $\ey$ and $\ey'$ lie in a common open affine subscheme of $\eX$ then $\ey = \ey'$. 
\end{cor}

\begin{proof}
Recall from Remark \ref{t34} that for $\eX_0 = \spec R_0$ and $\eX = \spec R$ the set \\
$W_{\rat} (\eX) (\C) = \deX (\C)$ can be described as the multiplicative maps $Q : R \to \C$ satisfying the analogue of a) in Corollary \ref{t45}. The natural map $W_{\rat} (\eX) (\eo) \to W_{\rat} (\eX) (\C)$ corresponds to the map $\deX (\eo) \to \deX (\C)$ sending $Q$ to $\iota \verk Q$ where $\iota : \eo \hookrightarrow \C$ is the inclusion. Hence part a) of the result follows. Part b) is a consequence of a), noting that every open subset of $\eX$ containing $\ey$ also contains $\ex$ since $\ey \in \overline{\{ \ex \} }$. 
\end{proof}

\begin{cor}
\label{t47}
For $\eX_0$ as above, assume that $\eX_0$ is affine or more generally that there exists an ample invertible sheaf on $\eX_0$. Then the map $W_{\rat} (\eX) (\eo) \longrightarrow W_{\rat} (\eX) (\C)$ is injective. If $K_0$ is perfect, the map $W_{\rat} (\eX_0) (\eo) \to W_{\rat} (\eX_0) (\C)$ is injective as well. 
\end{cor}

\begin{proof}
By \cite[Lemma 28.29.5, (3)]{stacks}, any two points of $\eX_0$ lie in a common open affine subscheme of $\eX_0$. Since the normalization morphism $\eX \to \eX_0$ is affine, the same is true for pairs of points of $\eX$. Therefore, by Corollary \ref{t46}, b) the first map is injective. Combining this with Theorem \ref{t53} and Corollary \ref{t210n} we see that the second map is injective as well. 
\end{proof}

\section{$A_{\inf}$-rings as $\eo$-valued functions on subspaces of $\cW_{\rat} (\eX) (\eo)$} \label{sec:9}
Let $\eo$ be a $p$-adically complete rank one valuation ring with quotient field $\C$, maximal ideal $\emm$ and residue field $k$ of characteristic $p$. Let $\eX_0$ be an integral normal $\Z_p$-scheme which maps surjectively to $\spec \Z_p$. Let $K_0$ be the function field of $\eX_0$ and as usual let $K$ be an algebraic closure of $K_0$ and let $\eX$ be the normalization of $\eX_0$ in $K$. The extension $K / K_0$ is Galois since $\car K_0 = 0$ and we let $G$ be its Galois group. In the following we only consider the monoid $\Nh_0$ generated by $p$ i.e. the $F_p = (\;)^p$ action. In section \ref{sec:4gn} we have seen different descriptions of $\deX (\eo) = W_{\rat} (\eX) (\eo)$. In this section we will modify $\deX (\eo)$ and $\ceX (\eo) = \colim_{F_p} \deX (\eo)$ by first imposing certain continuity conditions on the multiplicative maps that make up $\deX (\eo)$ and $\ceX (\eo)$. In a second step the resulting subspaces will be completed. For these modifications it matters which definition of $\deX (\eo)$ we take as our starting point. We will use the equivalent versions of Definition \ref{t42nn} resp. Proposition \ref{t44} involving multiplicative maps $P'_{\ey} : \Oh_{ \overline{\{ \ex \}}  , \ey} \to \eo$. To simplify the notation, we write $P_{\ey}$ instead of $P'_{\ey}$ though. For clarity let us write down the two ways to view $\deX (\eo)$ again, that we will use in the sequel. 


\begin{defn}
\label{t121}
Let $\deX (\eo)$ be the set of triples $(\ex , \ey , \oP^{\times})$ with $\ex \in \eX, \ey \in \overline{\{ \ex \} }$ and $\oP^{\times} : \kappa (\ex)^{\times} \to \C^{\times}$ a homomorphism whose extension by zero $\oP$ fits into a commutative diagram of multiplicative maps
\[
\xymatrix{
\kappa (\ey) \ar[r]^{\tP_{\ey}} & k \\
\Oh_{\overline{\{\ex \} } , \ey} \ar[r]^{P_{\ey}} \ar@{->>}[u] \ar@{_{(}->}[d] & \eo \ar@{->>}[u] \ar@{_{(}->}[d] \\
\kappa (\ex) \ar[r]^{\oP} & \C \; .
}
\]
Equivalently we view $\deX (\eo)$ as the set of triples $(\ex , \ey , P_{\ey})$ with $\ex , \ey$ as above and $P_{\ey} : \Oh_{\overline{\{ \ex \} } , \ey} \to \eo$ a multiplicative map sending $1$ to $1$ and $0$ to $0$ and such that for $f \in \Oh_{\overline{ \{ \ex \} } , \ey }$ we have \\
1) $P_{\ey} (f) \neq 0$ if $f \neq 0$\\
2) the value of $P_{\ey} (f) \mod \emm$ in $k$ depends only on $f (\ey) \in \kappa (\ey)$. 
\end{defn}

For a set $X$ let $C (X , \eo)$ be the $\eo$-algebra of maps from $X$ to $\eo$. 

We have a multiplicative $F_p$- and $G$-equivariant map of monoids
\begin{equation}
\label{eq:81}
[\;] : \Gamma (\eX , \Oh) \longrightarrow C (\deX (\eo) , \eo) \quad \text{where} \; [\alpha] (\ex , \ey, \oP^{\times}) = \oP (\alpha (\ex)) \; .
\end{equation}
Here $\alpha (\ex) = \alpha \mod \emm_{\eX , \ex}$ is the evaluation of $\alpha \in \Gamma (\eX, \Oh)$ in $\ex$. Note here that if $\alpha$ also denotes the image of $\alpha$ in $\Oh_{\overline{\{ \ex \} } , \ey}$, we have
\[
\oP (\alpha (\ex)) = P_{\ey} (\alpha)
 \in \eo \; .
\]
Using \eqref{eq:81} we get a multiplicative map from $\varprojlim_{(\,)^p} \Gamma (\eX , \Oh)$ to the space of functions on $\ceX (\eo)$. Let $\wedge$ denote $p$-adic completion. We can interpret the elements of $\Gamma (\eX , \Oh)^{\wedge}$ and of the monoid $\varprojlim_{(\;)^p} \Gamma (\eX , \Oh)^{\wedge}$ as functions on sub-dynamical systems of $\deX (\eo)$ resp. $\ceX (\eo)$ as follows.

\begin{defn}
\label{t91}
a) In the above situation let $\deX_c (\eo)$ be the $F_p$- and $G$-invariant subspace of $\deX (\eo)$ consisting of triples $(\ex , \ey , \oP^{\times})$ as in Definition \ref{t121} for which the multiplicative maps $P_{\ey}$ factor (necessarily uniquely) over $p$-adically continuous (necessarily multiplicative) maps $\hP_{\ey}$ as follows
\[
P_{\ey} : \Oh_{\overline{\{ \ex \} } , \ey} \longrightarrow \hOh_{\overline{\{ \ex \} } , \ey} \xrightarrow{\hP_{\ey}} \eo \; .
\]
b) We set
\[
\ceX_c (\eo) = \colim_{F_p} \deX_c (\eo) \subset \ceX (\eo) \; .
\]
\end{defn}

Note that since $P_{\ey}$ maps $1$ to $1$ and $0$ to $0$ we must have $1 \neq 0$ in the $p$-adic completion $\hOh_{\overline{\{ \ex \} } , \ey}$ of $\Oh_{\overline{\{ \ex \} } , \ey}$. This is equivalent to $\car \kappa (\ey) = p$. 

A {\it ring} homomorphism $\Oh_{\overline{\{ \ex \} }, \ey} \to \eo$ extends to a continuous homomorphism $\hOh_{\overline{\{ \ex \} } , \ey} \to \eo$ since $\eo$ is $p$-adically complete, but I do not see why this should be the case for the maps $P_{\ey}$ which are only multiplicative. More precisely, if $\car \kappa (\ey) = p$ then the commutative diagram
\[
\xymatrix{
\kappa (\ey) \ar[r]^{\tP_{\ey}} & k \\
\Oh_{\overline{\{ \ex \} } , \ey} \ar[u] \ar[r]^{P_{\ey}} & \eo \ar@{->>}[u]
}
\]
shows that $\omega = P_{\ey} (p)$ lies in the maximal ideal of $\eo$, i.e. $|\omega| < 1$. Using multiplicativity of $P_{\ey}$ it follows that
\[
P_{\ey} (p^n \Oh_{\overline{\{ \ex \} } , \ey}) \subset \omega^n \eo \; .
\]
Hence $P_{\ey}$ is continuous at $0$ for the $p$-adic topologies on $\Oh_{\overline{\{ \ex \} } , \ey}$ and $\eo$ but as far as I can see this does not imply continuity at the other points of $\Oh_{\overline{\{ \ex \} } , \ey}$ if $P_{\ey}$ is not additive. 

\begin{prop}
\label{t123}
Let $\eX_0$ be an integral normal $\Z_p$-scheme of finite type. In the situation of Definition \ref{t91} with $\car \kappa (\ey) = p$ the $p$-adic topology on $\Oh_{\overline{ \{ \ex \} } , \ey}$ is separated i.e. the natural map $\Oh_{\overline{ \{ \ex \} } , \ey} \to \hOh_{\overline{ \{ \ex \} } , \ey}$ is injective.
\end{prop}

\begin{proof}
For $\car \kappa (\ex) = p$ the assertion is trivial, so assume that $\car \kappa (\ex) = 0$. Let $\ex_0$ and $\ey_0$ be the images of $\ex$ and $\ey$ under the morphism $\pi : \eX \to \eX_0$. Then $\ey_0 \in \overline{\{ \ex_0 \} }$ and since $\Oh_{\eX , \ey}$ is integral over $\Oh_{\eX_0 , \ey_0}$ it follows that $\Oh_{\overline{ \{ \ex \} } , \ey}$ is integral over $\Oh_{\overline{ \{ \ex_0 \} } , \ey_0}$. For $f \in \Oh_{\overline{ \{ \ex \} } , \ey}$ let
\[
P (t) = t^d + a_{d-1} t^{d-1} + \ldots + a_0
\]
be its minimal polynomial over the quotient field $\kappa (\ex_0)$ of $\Oh_{\overline{ \{ \ex_0 \} } , \ey_0}$. Then $a_0 , \ldots , a_{d-1}$ are in the normalization $\tilde{\Oh}_0$ of $\Oh_{\overline{ \{ \ex_0 \} } , \ey_0}$ in $\kappa (\ex_0)$. The minimal polynomial of $p^{-n} f \in \kappa (\ex)$ is the polynomial
\[
P_n (t) = t^d + p^{-n} a_{d-1} t^{d-1} + \ldots + p^{-nd} a_0 \; .
\]
If $p^{-n} f \in \Oh_{\overline{ \{ \ex_0 \} } , \ey_0}$ then $p^{-n} a_{d-1} , \ldots , p^{-nd} a_0$ are in $\tilde{\Oh}_0$. Hence, in order to see that $\bigcap_{n \ge 0} p^n \Oh_{\overline{ \{ \ex \} } , \ey} = 0$, it suffices to show that
\begin{equation} \label{eq:171n}
\bigcap_{n \ge 0} p^n \tilde{\Oh}_0 = 0 \; .
\end{equation}
Since $\eX_0$ is of finite type over the Nagata ring $\Z_p$, it follows that the domain $\tilde{\Oh}_0$ is Noetherian. We have $p \tilde{\Oh}_0 \neq \tilde{\Oh}_0$ since $\spec \tilde{\Oh}_0$ contains a point over the characteristic $p$ point $\ey_0$ of $\spec \Oh_{ \overline{ \{ \ex_0 \} } , \ey_0}$. Hence \eqref{eq:171n} follows from Krull's intersection theorem \cite{AM} Corollary 10.18.
\end{proof}

Because of Proposition \ref{t123} we will assume that $\eX_0 \twoheadrightarrow \spec \Z_p$ is of finite type in the following. Any morphism $\spec \eo \to \eX$ lands in an open affine subscheme $U = \spec A$ of $\eX$ and therefore corresponds to a homomorphism $\varphi : A \to \eo$. As a ring homomorphism, $\varphi$ is $p$-adically continuous and since $A$ is a $\Z_p$-algebra, $\varphi$ is therefore $\Z_p$-linear. Thus the $\eo$-valued points $\eX (\eo)$ of $\eX$ over $\spec \Z$ and over $\spec \Z_p$ are the same. Moreover the following assertion holds. 

\begin{prop} 
\label{t124}
Under the natural inclusion $\eX (\eo) \subset W_{\rat} (\eX) (\eo) = \deX (\eo)$ we have $\eX (\eo) \subset \deX_c (\eo) \subset \deX (\eo)$.
\end{prop}

\begin{proof}
The points of $\eX (\eo)$ correspond to those triples $(\ex , \ey, P_{\ey})$ where $P_{\ey}$ is a $\Z_p$-linear ring homomorphism. As mentioned above, a ring homomorphism $P_{\ey}$ extends to a continuous ring homomorphism
\[
\hP_{\ey} : \hOh_{\overline{ \{ \ex \} } , \ey} = \varprojlim_n \Oh_{\overline{ \{ \ex \} } , \ey} / p^n \xrightarrow{\varprojlim (P_{\ey} \mod p^n)} \varprojlim_n \eo / p^n = \eo \; .
\]
By definition $P_{\ey}$ factors over $\hP_{\ey}$. 
\end{proof}

The restriction $\Gamma (\eX , \Oh) \to \Oh_{\overline{\{ \ex \} } , \ey}$ induces a map $\Gamma (\eX , \Oh)^{\wedge} \to \hOh_{\overline{\{ \ex \} } , \ey}$. Hence the composition
\[
[ \; ] : \Gamma (\eX , \Oh) \xrightarrow{\text{\eqref{eq:81}}} C (\deX (\eo) , \eo) \xrightarrow{\res} C (\deX_c (\eo) , \eo)
\]
extends to a multiplicative map
\begin{equation}
\label{eq:82}
[\;] : \Gamma (\eX , \Oh)^{\wedge} \longrightarrow C (\deX_c (\eo) , \eo) \quad \text{where} \; [\hat{\alpha}] (\ex , \ey, \oP^{\times}) = \hP_{\ey} (\hat{\alpha}) \; .
\end{equation}
We have an isomorphism of $\eo$-modules
\begin{equation}
\label{eq:83}
C (\ceX_c (\eo) , \eo) = \varprojlim_{F_p} C (\deX_c (\eo) , \eo)
\end{equation}
and hence a multiplicative map
\begin{equation}
\label{eq:84}
\varprojlim_{(\;)^p} \Gamma (\eX , \Oh)^{\wedge} \longrightarrow C (\ceX_c (\eo) , \eo) \; .
\end{equation}
Using Proposition \ref{t44} we may view the elements of $\ceX_c (\eo)$ as triples $(\ex , \ey , F^{-i}_p \hP_{\ey})$ where $i \ge 0$ and $(\ex , \ey , P_{\ey}) \in \deX_c (\eo)$ and where by definition $F^{-i}_p \hP_{\ey} = F^{-j}_p \hP'_{\ey}$ if and only if $F^j_p \hP_{\ey} = F^i_p \hP'_{\ey}$. In the latter equation, $F_p \hP := \hP \verk (\;)^p = \hP^p$ in $C (\hOh_{\overline{\{ \ex \} } , \ey}, \eo)$. The $p$-power map $(\;)^p$ is surjective on $\Oh_{\overline{\{ \ex \} } , \ey}$ which is dense in $\hOh_{\overline{\{ \ex \} } , \ey}$. Hence $F_p$ is injective on $C (\hOh_{\overline{\{ \ex \} } , \ey} , \eo)$ and if $j \ge i$ the equation $F^j_p \hP_{\ey} = F^i_p \hP'_{\ey}$ is therefore equivalent to $F^{j-i}_p \hP_{\ey} = \hP'_{\ey}$. We may view $\cP_{\ey} = F^{-i}_p \hP_{\ey}$ as a multiplicative function:
\begin{equation}
\label{eq:85}
\cP_{\ey} : \varprojlim_{(\;)^p} \hOh_{\overline{\{ \ex \} } , \ey} \longrightarrow \eo \quad \text{by setting} \; \cP_{\ey} ((f_n)_{n \ge 0}) = \hP_{\ey} (f_i) \; .
\end{equation}
This is well defined since with notations as above, if $j \ge i$ 
\[
\hP_{\ey} (f_j) = (F^{j-i}_p \hP_{\ey}) (f_j) = \hP_{\ey} (f^{p^{j-i}}_j) = \hP_{\ey} (f_i) \; .
\]
Note that $\cP_{\ey}$ sends $1$ to $1$ and $0$ to $0$. By definition the maps $\cP_{\ey}$ in \eqref{eq:85} are the ones which factor as follows
\[
\cP_{\ey} : \varprojlim_{(\;)^p} \hOh_{\overline{\{ \ex \} } , \ey} \xrightarrow{\pr_i} \hOh_{\overline{\{ \ex \} } , \ey} \xrightarrow{\chi} \eo \; , 
\]
for some $i \ge 0$ and some continuous multiplicative map $\chi$ (namely $\hP_{\ey}$ if $\cP_{\ey} = F^{-i}_p \hP_{\ey}$). 

In this description of $\ceX_c (\eo)$ the map \eqref{eq:84} is given as follows:
\begin{equation}
\label{eq:86}
(f_n)_{n \ge 0} \longmapsto ((\ex , \ey , \cP_{\ey}) \mapsto \cP_{\ey} ((f_n)) \quad \text{with $\cP_{\ey}$ as in \eqref{eq:85}} \; .
\end{equation}
Here we have tacitly used the canonical map
\[
\varprojlim_{(\;)^p} \Gamma (\eX , \Oh)^{\wedge} \longrightarrow \varprojlim_{(\;)^p} \hOh_{\overline{\{ \ex \} } , \ey} \; .
\]
The following facts are due to Fontaine and very well known. For any $p$-adically complete ring $A$ the natural map
\begin{equation}
\label{eq:87}
\varprojlim_{(\;)^p} A \silo A^b := \varprojlim_{(\;)^p} A / p \; , \; (a_n) \longmapsto (a_n \mod p)
\end{equation}
is a multiplicative topological isomorphism, sending $1$ to $1$ and $0$ to $0$. The inverse map sends $(\oa_n)$ to $(a_n)$ where
\[
a_n = \lim_{i \to \infty} \ta^{p^i}_{n+i} \; .
\]
Here $\ta_{n+i} \in A$ is a lift of $\oa_{n+i} \in A / p$. On $A^b$ the multiplicative monoid structure is clearly part of the richer structure of a perfect $\F_p$-algebra. The automorphism $F_p$ on the monoid $\varprojlim_{(\,)^p} A$ defined by $F_p ((a_n)) = (a^p_n) = (a_{n-1})$, (the left shift) corresponds to the Frobenius $F_p = (\,)^p$ on the $\F_p$-algebra $A^b$. 

We apply the Fontaine isomorphism to $A = \Gamma (\eX , \Oh)^{\wedge}$ and $A = \hOh_{\overline{\{ \ex \} } , \ey}$. Then the multiplicative map \eqref{eq:84} induces a ring homomorphism from the monoid algebra of the perfect $\F_p$-algebra $\Gamma (\eX , \Oh)^{\wedge b}$ with multiplication to the ring of $\eo$-valued functions on $\ceX_c (\eo)$
\begin{equation}
\label{eq:88}
\Z \Gamma (\eX , \Oh)^{\wedge b} \longrightarrow C (\ceX_c (\eo) , \eo) \; .
\end{equation}
Consider the $p$-typical Witt ring of the perfect $\F_p$-algebra $\Gamma (\eX , \Oh)^{\wedge b}$
\[
A_{\inf} (\eX) := W_p (\Gamma (\eX , \Oh)^{\wedge b}) \; .
\]
The Teichm\"uller map induces a ring homomorphism $\Z \Gamma (\eX , \Oh)^{\wedge b} \to A_{\inf} (\eX)$ and we will now define natural $F^{\Z}_p$- and $G$-invariant subsystems of $\ceX_c (\eo)$ on which even the elements of the ring $A_{\inf} (\eX)$ may be viewed as $\eo$-valued functions. We recall the description of $W_p (R)$ for a perfect $\F_p$-algebra $R$ given in \cite{CD}. Consider the natural exact sequence
\[
0 \longrightarrow I \longrightarrow \Z R \xrightarrow{\pi} R \longrightarrow 0 \quad \text{where} \; \pi \big( \textstyle{\sum}_r n_r [r] \big) = \textstyle{\sum}_r n_r r \; .
\]
The ring automorphism $F_p$ on $\Z R$ is defined by setting $F_p ([r]) = [r^p]$. It satisfies $F_p (I) = I$ and hence it induces automorphisms $F_p$ of $\Z R / I^n$ for all $n \ge 0$ and of the $I$-adic completion $\varprojlim_n \Z R / I^n$ of $\Z R$. In \cite{CD} it is shown that the natural map $\Z R \to W_p (R)$ induces $F_p$-equivariant isomorphisms
\[
\Z R / I^n \silo W_p (R) / p^n W_p (R) \quad \text{for all} \; n \ge 1 \quad \text{and} \quad \varprojlim_n \Z R / I^n \silo W_p (R) \; .
\]
It is not difficult to see that as a $\Z$-module the ideal $I$ is generated by the elements of the form $[r+s] - [r] - [s]$ with $r,s \in R$. Returning to our situation, set
\[
I = \Ker (\Z \Gamma (\eX , \Oh)^{\wedge b} \longrightarrow \Gamma (\eX , \Oh)^{\wedge b})
\]
and
\[
I_{\ey} = \Ker (\Z \hOh^b_{\overline{\{ \ex \} } , \ey} \to \hOh^b_{\overline{\{ \ex \} } , \ey}) \; .
\]

\begin{defn}
\label{t92}
For any real number $0 < \alpha < 1$ define the $G$-invariant subset $\cY_{\alpha}$ of $\ceX_c (\eo)$ as follows, where $f (\cP_{\ey}) := \cP_{\ey} (f) \in \eo$ for $f \in \Z \hOh^b_{\overline{ \{ x \} } , \ey}$,
\begin{align*}
\cY_{\alpha} & = \{ (\ex , \ey , \cP_{\ey}) \in \ceX_c (\eo) \mid |f (\cP_{\ey})| \le \alpha \quad \text{for all} \; f \in I_{\ey} \} \\
& = \{ (\ex , \ey , \cP_{\ey}) \in \ceX_c (\eo) \mid |\cP_{\ey} (r+s) - \cP_{\ey} (r) - \cP_{\ey} (s)| \le \alpha \quad \text{for} \; r,s \in \hOh^b_{\overline{\{ \ex \} } , \ey} \}  \; .
\end{align*}
\end{defn}

By definition, we have $\cY_{\alpha_1} \subset \cY_{\alpha_2}$ if $\alpha_1 \le \alpha_2$. For $\alpha = 1 / p$ we are looking at multiplicative maps $\cP_{\ey}$ which $\mod p$ are also additive. Set $\cY = \cY_{1/p}$. 

\begin{remark} \label{t126}
We have canonical inclusions $\eX (\eo) \subset \cY \subset \ceX_c (\eo)$. Hence by Proposition \ref{t93} below, $\cY$ contains at least the $F^{\Z}_p$-orbits of the classical $\eo$-valued points of $\eX$. 
\end{remark}

Choose an element $\omega_{\alpha} \in \eo$ with $\alpha \le |\omega_{\alpha}| < 1$. Then $(\ex , \ey , \cP_{\ey}) \in \cY_{\alpha}$ gives a ring homomorphism
\[
\cP_{\ey} : \Z \hOh^b_{\overline{\{ \ex \} } , \ey} \longrightarrow \eo \quad \text{with} \; \cP_{\ey} (I_{\ey}) \subset \omega_{\alpha} \eo
\]
and hence an induced ring homomorphism
\begin{equation}
\label{eq:89}
W_p (\cP_{\ey}) : W_p (\hOh^b_{\overline{\{ \ex \} } , \ey}) = \varprojlim_n \Z \hOh^b_{\overline{\{ \ex \} } , \ey} / I^n_{\ey} \longrightarrow \varprojlim_n \eo / \omega^n_{\alpha} = \eo \; .
\end{equation}
Restricting functions on $\ceX_c (\eo)$ to $\cY_{\alpha}$, the ring homomorphism \eqref{eq:88} gives a ring homomorphism
\begin{equation}
\label{eq:90}
e : \Z \Gamma (\eX , \Oh)^{\wedge b} \longrightarrow C (\cY_{\alpha} , \eo) \; .
\end{equation}
It extends to $A_{\inf} (\eX)$ as the following map
\begin{equation}
\label{eq:91}
e : A_{\inf} (\eX) \longrightarrow C (\cY_{\alpha} , \eo) \; , \; a \longmapsto ((\ex , \ey , \cP_{\ey}) \mapsto W_p (\cP_{\ey}) (a)) \; .
\end{equation}
Here we have tacitly used the maps
\[
A_{\inf} (\eX) = W_p (\Gamma (\eX , \Oh)^{\wedge b}) \longrightarrow W_p (\hOh^b_{\overline{\{ \ex \} } , \ey}) \; .
\]
Recall that $\cY = \cY_{1/p}$.

\begin{prop}
\label{t93}
For $0 < \alpha < 1$ we have $F_p (\cY_{\alpha}) = \cY_{\alpha}$ and $\cY_{\alpha} \subset \cY$, and hence $\cY_{\alpha} = \cY$ for $\alpha \ge 1 / p$. 
\end{prop}

\begin{proof}
For $(\ex , \ey , \cP_{\ey}) \in \cY_{\alpha}$ and $f \in I_{\ey}$ we have
\[
f (F_p (\cP_{\ey})) = F_p (f) (\cP_{\ey}) \; .
\]
Since the isomorphism $F_p$ of $\Z \hOh^b_{\overline{\{ \ex \} } , \ey}$ satisfies $F_p (I_{\ey}) = I_{\ey}$ it follows that the condition
\[
|f (\cP_{\ey})| \le \alpha \quad \text{for all} \; f \in I_{\ey}
\]
is equivalent to the condition that
\[
|F_p (f) (\cP_{\ey})| \le \alpha \quad \text{for all} \; f \in I_{\ey} \; .
\]
This implies that $F_p (Y_{\alpha}) = Y_{\alpha}$. On the other hand, note that for $r,s \in \hOh^b_{\overline{\{ \ex \} } , \ey}$ and $\nu \ge 0$ we have
\begin{align*}
 & |F^{\nu}_p (\cP_{\ey}) (r+s) - F^{\nu}_p (\cP_{\ey}) (r) - F^{\nu}_p (\cP_{\ey}) (s)| \\
 = & |\cP_{\ey} (r+s)^{p^{\nu}} - \cP_{\ey} (r)^{p^{\nu}} - \cP_{\ey} (s)^{p^{\nu}}| \\
 = & |(\cP_{\ey} (r+s) - \cP_{\ey} (r) - \cP_{\ey} (s))^{p^{\nu}} + pc| \quad \text{for some} \; c \in \eo \\
\le & \max (|\cP_{\ey} (r+s) - \cP_{\ey} (r) - \cP_{\ey} (s)|^{p^{\nu}} , |p|) \\
\le & \max (\alpha^{p^{\nu}} , 1 / p) \; .
\end{align*}
Thus we find
\[
F^{\nu}_p (\cY_{\alpha}) \subset \cY_{\max (\alpha^{p^{\nu}} , 1 / p)} \; .
\]
Combining this with $F_p (\cY_{\alpha}) = \cY_{\alpha}$ which we saw above, we get
\[
\cY_{\alpha} \subset \cY_{\max (\alpha^{p^{\nu}} , 1 / p)} \quad \text{for all} \; \nu \ge 0 \; .
\]
Since $\alpha^{p^{\nu}} < 1 / p$ if $\nu$ is large enough, we get $\cY_{\alpha} \subset \cY_{1 / p}$. For $\alpha \ge 1 / p$ we have $\cY_{\alpha} \supset \cY_{1 / p}$ by definition and therefore $\cY_{\alpha} = \cY_{1 / p}$. 

Here is a different way to see that $\cY_{\alpha} \subset \cY_{1 /p}$. For $(\ex , \ey , \cP_{\ey}) \in \cY_{\alpha}$ consider the commutative diagram
\begin{equation} \label{eq:93}
\xymatrix{
\cP_{\ey} : \hOh^b_{\overline{\{ \ex \} } , \ey} \ar[r]^{[\;]} \ar[dr]^{\sim} & W_p (\hOh^b_{\overline{\{ \ex \} } , \ey}) \ar@{->>}[d] \ar[r]^-{W_p (\cP_{\ey})} & \eo \ar@{->>}[d] \\
 & W_p (\hOh^b_{\overline{\{ \ex \} } , \ey}) / p \ar[r] & \eo / p \; .
}
\end{equation}
All maps in the diagram are ring homomorphisms except for the Teichm\"uller map \cite{}. It follows that the composition
\begin{equation}
\label{eq:94}
\hOh^b_{\overline{\{ \ex \} } , \ey} \xrightarrow{\cP_{\ey}} \eo \longrightarrow \eo / p
\end{equation}
is additive i.e. that $\cP_{\ey} \in \cY_{1/p}$. 
\end{proof}

In the case $\eX_0 = \spec \Z_p$ we will calculate $\cY$ and it turns out to be very small. The ultimate reason is this: in the definition of $\ceX_c (\eo)$ we considered only those continuous multiplicative maps $\cP_{\ey} : \varprojlim_{(\,)^p} \hOh_{\overline{\{ \ex \} } , \ey} \to \eo$ which factored as $\cP_{\ey} = \chi \verk \pr_i$ for some $i \ge 0$ and some continuous multiplicative map $\chi : \hOh_{\overline{\{ \ex \} } , \ey} \to \eo$. Let us call such maps $\cP_{\ey}$ ``$\varprojlim$-locally constant''. However there are many more continuous maps $\cP_{\ey}$ than only the $\varprojlim$-locally constant ones. If we include them, the quotient $\hcY\!_0 = \hcY / G$ of the corresponding space $\hcY \supset \cY$ for $\eX_0 = \spec \Z_p$ turns out to be in natural bijection with the points of the Fontaine-Fargues curve. Before introducing these ``completions'' $\hceX_c (\eo)$ of $\ceX_c (\eo)$ and $\hcY\!_{\alpha}$ of $Y_{\alpha}$ we note the following simple facts.

For a ring $C$ consider $T_p \mu (C) = \varprojlim_n \mu_{p^n} (C)$ with transition maps $(\,)^p$.

\begin{prop} \label{t125}
Let $A$ be an integral domain in which every element is a $p$-th power. Let $S \subset A$ be a multiplicative subset with $a \in S \iff a^p \in S$ for $a \in A$, e.g. $S = A \setminus \ep$ for a prime ideal $\ep$. Then we have:\\
a) The image of the natural map $\varprojlim_{(\,)^p} A \to \varprojlim_{(\,)^p} \hA$ is dense.\\
b) Assume that $\mu_{p^{\infty}} (A) = \mu_{p^{\infty}} (K)$ where $K = \Quot A$. Then every element $x$ in $\varprojlim_{(\,)^p} S^{-1} A$ has the form $x = yz^{-1}$ for some $x \in \varprojlim_{(\,)^p} A$ and $z \in \varprojlim_{(\,)^p} S$. \\
c) Assume that the map $T_p \mu (A) \to T_p \mu (\kappa)$ is surjective where $\kappa = \Quot (A / \ep)$. Then the map
\[
\varprojlim_{(\,)^p} A \longrightarrow \varprojlim_{(\,)^p} A / \ep 
\]
is also surjective.
\end{prop}


\begin{proof}
a) Let $\hpr_i : \varprojlim_{(\,)^p} \hA \to \hA$ be the $i$-th projection. Any non-empty open set of $\varprojlim_{(\,)^p} \hA$ contains a set of the form $\hpr^{-1}_i (U)$ for some $i$ where $\emptyset \neq U \subset \hA$ is open. Since $A$ is dense in $\hA$ there is some $a \in U \cap A$. Since $(\,)^p$ is surjective on $A$, the $i$-th projection $\pr_i : \varprojlim_{(\,)^p} A \to A$ is surjective and hence $\emptyset \neq \pr^{-1}_i (a) \subset \hpr^{-1}_i (U)$ which implies a). \\
b) We may assume that $x \neq 0$. Writing $x = (a_n s^{-1}_n)$ with $a_n \in A \setminus 0$ and $s_n \in S$, choose $y' = (a'_n) \in \varprojlim_{(\,)^p} A $ and $z = (s'_n) \in \varprojlim_{(\,)^p} S$ with $a'_0 = a_0$ and $s'_0 = s_0$. Since $x$ and $y' z^{-1}$ have the same zeroeth component $a_0 s^{-1}_0 \neq 0$ there is an element $\zeta = (\zeta_n) \in T_p \mu (K) = T_p \mu (A)$ with $x = \zeta y' z^{-1}$. For $y = \zeta y'$ we get $x = y z^{-1}$ as desired.\\
c) Let $\alpha = (\alpha_n)$ be a non-zero element of $\varprojlim_{(\,)^p} A / \ep$. Choose an element $(a_n)$ of $\varprojlim_{(\,)^p} A$ with $a_0 \mod \ep = \alpha_0$. Then we have $(a_n \mod \ep) = \ozeta \alpha$ in $\varprojlim_{(\,)^p} \kappa$ for some $\ozeta \in T_p \mu (\kappa)$. Choose a lift $\zeta \in T_p \mu (A)$ of $\ozeta$. Then the element $\zeta (a_n) \in \varprojlim_{(\,)^p} A$ maps to $\alpha$. 
\end{proof}

We now define ``completed'' versions of our spaces $\ceX_c (\eo)$ and $\cY$. Instead of $\overset{\wedge}{\scriptstyle\vee}$ we write $\diamond$. 

\begin{defn} \label{t105}
Let $\hceX_c (\eo)$ be the set of triples $(\ex , \ey , \oP^{\times})$ where $\ex \in \eX , \ey \in \overline{\{ \ex \}}$ and $\oP^{\times} : \varprojlim_{(\,)^p} \kappa (\ex)^{\times} \to \C^{\times}$ is a homomorphism whose extension by zero $\oP$ fits into a commutative diagram of multiplicative maps where $\hP_{\ey}$ is continuous, and $\lambda$ is the natural map
\begin{equation}
\label{eq:97}
\xymatrix{
\varprojlim_{(\,)^p} \kappa (\ey) \ar[rr]^{\tP_{\ey}} & & k \\
\varprojlim_{(\,)^p} \Oh_{\overline{\{ \ex \} } , \ey} \ar[u]^{\lambda} \ar@{_{(}->}[d] \ar[r] & \varprojlim_{(\,)^p} \hOh_{\overline{\{ \ex \} } , \ey} \ar[r]^-{\hP_{\ey}} & \eo \ar@{->>}[u] \ar@{_{(}->}[d] \\
\varprojlim_{(\,)^p} \kappa (\ex) \ar[rr]^{\oP} && \C \; .
}
\end{equation}
\end{defn}

Note that since $\oP (0) = 0$ and $\oP (1) = 1$, we have $\hP_{\ey} (0) = 0$ and $\hP_{\ey} (1) = 1$ hence $1 \neq 0$ in $\varprojlim_{(\,)^p} \hOh_{\overline{\{ \ex \} } , \ey}$ hence $1 \neq 0$ in $\hOh_{\overline{\{ \ex \} } , \ey}$ and therefore $\car \kappa (\ey) = p$. It follows that the map $\Oh_{\overline{\{ \ex \} } , \ey} \to \kappa (\ey)$ factors naturally over the induced map $\hOh_{\overline{\{ \ex \} } , \ey} \twoheadrightarrow \kappa (\ey)$ (which is continuous for the discrete topology on $\kappa (\ey)$). Hence $\lambda$ factors over the induced map
\[
\hlambda : \varprojlim_{(\,)^p} \hOh_{\overline{\{ \ex \} } , \ey} \longrightarrow \varprojlim_{(\,)^p} \kappa (\ey) \; .
\]
Our notation is not consistent since previously $\oP$ etc. denoted maps on $\kappa (\ex)$ etc. Note however that an old $\oP$ gives a new $\oP$ via $\oP \verk \pr_0$ etc. There are obvious commuting operations by $G$ and $F^{\Z}_p$ on $\hceX_c (\eo)$ and canonical $F_p , G$-equivariant inclusions
\begin{equation}
\label{eq:98}
\deX_c (\eo) \hookrightarrow \ceX_c (\eo) \hookrightarrow \hceX_c (\eo) \; .
\end{equation}

\begin{supple} \label{t106}
The diagram \eqref{eq:97} is uniquely determined by $(\ex , \ey , \oP^{\times})$. Thus, given $(\ex , \ey , \oP^{\times})$ the maps $\hP_{\ey}$ and $\tP_{\ey}$ are unique. The top square in \eqref{eq:97} factors as follows:
\begin{equation}
\label{eq:109}
\xymatrix{
\varprojlim_{(\,)^p} \kappa (\ey) \ar@{=}[r] & \varprojlim_{(\,)^p} \kappa (\ey) \ar[r]^-{\tP_{\ey}} & k \\
\varprojlim_{(\,)^p} \Oh_{\overline{\{ \ex \} } , \ey} \ar[u]^{\lambda} \ar[r] & \varprojlim_{(\,)^p} \hOh_{\overline{\{ \ex \} } , \ey} \ar[u]^{\hlambda} \ar[r]^-{\hP_{\ey}} & \eo \; . \ar@{->>}[u]
}
\end{equation}
\end{supple}

\begin{proof}
Clearly $\oP^{\times}$ determines the composition
\[
P_{\ey} : \varprojlim_{(\,)^p} \Oh_{\overline{\{ \ex \} } , \ey} \xrightarrow{i} \varprojlim_{(\,)^p} \hOh_{\overline{\{ \ex \} } , \ey} \xrightarrow{\hP_{\ey}} \eo
\]
uniquely. The image of $i$ is dense by Proposition \ref{t125} a) and hence $\hP_{\ey}$, being continuous is uniquely determined by $\oP^{\times}$ as well. 
Consider the diagram
\begin{equation}
\label{eq:99}
\xymatrix{
\varprojlim_{(\,)^p} \kappa (\ey) \ar[r]^-{\tP_{\ey}} & k \\
\varprojlim_{(\,)^p} \Oh_{\overline{\{ \ex \} } , \ey} \ar[u]^{\lambda} \ar[r]^-{P_{\ey}} & \eo  \; . \ar@{->>}[u]
}
\end{equation}
Using Proposition \ref{t125} c), it follows that in \eqref{eq:99} the map $\lambda$ is surjective. It follows that $\tP_{\ey}$ is uniquely determined by $P_{\ey}$ and hence by $\oP^{\times}$. Actually, in our situation, we have $\car \kappa (\ey) = p$ as mentioned above, so that $\pr_0 : \varprojlim_{(\,)^p} \kappa (\ey) \silo \kappa (\ey)$. 

The square on the right in diagram \eqref{eq:109} commutes: If $\kappa (\ey)$ and $k$ carry the discrete topology, the composed maps
\[
\varprojlim_{(\,)^p} \hOh_{\overline{\{ \ex \} } , \ey} \xrightarrow{\hP_{\ey}} \eo \longrightarrow k \quad \text{and} \quad \varprojlim_{(\,)^p} \hOh_{\overline{\{ \ex \} } , \ey} \xrightarrow{\hlambda} \varprojlim_{(\,)^p} \kappa (\ey) \xrightarrow{\tP_{\ey}} k
\]
are continuous and agree on the dense subspace $\varprojlim_{(\,)^p} \Oh_{\overline{\{ \ex \} } , \ey}$.
\end{proof}

Since $\tP_{\ey} (1) = 1$ the multiplicative map $\tP_{\ey}$ is non-zero on
\[
\big( \varprojlim_{(\,)^p} \kappa (\ey)\big) \setminus 0 = \varprojlim_{(\,)^p} \kappa (\ey)^{\times} \; .
\]
Hence we find:
\begin{equation}
\label{eq:111}
P^{-1}_{\ey} (\emm) = \varprojlim_{(\,)^p} \emm_{\overline{ \{ \ex \} } , \ey} \; . 
\end{equation}
This will be needed later. 

Here is a simpler description of $\hceX (\eo)$ using $\hP_{\ey}$ instead of $\oP^{\times}$.

\begin{prop} \label{t117}
Mapping $(\ex , \ey , \oP^{\times})$ to $(\ex , \ey , \hP_{\ey})$ we get a bijection of $\hceX (\eo)$ with the set of triples $(\ex , \ey , \hP_{\ey})$ where $\ey \in \overline{ \{ \ex \} }$ and $\hP_{\ey} : \varprojlim_{(\,)^p} \hOh_{\overline{\{ \ex \} } , \ey} \to \eo$ is a continuous multiplicative map sending $1$ to $1$ and $0$ to $0$ and such that the following conditions hold for 
all $f \in \varprojlim_{(\,)^p} \Oh_{\overline{\{ \ex \} } , \ey}$ \\
1) $\hP_{\ey} (f) \neq 0$ if $f \neq 0$\\
2) the value of $\hP_{\ey} (f) \mod \emm$ in $k$ depends only on $f (\ey) = (f_n (\ey))$ in $\varprojlim_{(\,)^p} \kappa (\ey)$.\\
\end{prop}

\begin{proof}
The inverse map $(\ex , \ey , \hP_{\ey}) \mapsto (\ex , \ey , \oP^{\times}) \in \hceX (\eo)$ is obtained as follows. Since $\lambda$ is surjective, condition 2) means that with $\hP_{\ey}$ we get a uniquely determined diagram \eqref{eq:99}. By Proposition \ref{t125} every element of $\varprojlim_{(\,)^p} \kappa (\ex)$ is of the form $f s^{-1}$ with $f ,s \in \varprojlim_{(\,)^p} \Oh_{\overline{\{ \ex \} } , \ey}$ and $s \neq 0$.
Therefore, by condition 1) the composition
\[
\varprojlim_{(\,)^p} \Oh_{\overline{\{ \ex \} } , \ey} \hookrightarrow \varprojlim_{(\,)^p} \hOh_{\overline{\{ \ex \} } , \ey} \xrightarrow{\hP_{\ey}} \eo \hookrightarrow \C
\]
factors uniquely over a map $\oP : \multmap \kappa (\ex) \to \C$. Hence we have obtained a diagram \eqref{eq:97} and therefore a well defined point $(\ex , \ey , \oP^{\times})$ of $\hceX_c (\eo)$. 
\end{proof}

Consider the natural $F_p$- and $G$-equivariant multiplicative map
\[
e : \Gamma (\eX , \Oh)^{\wedge b} \cong \varprojlim_{(\,)^p} \Gamma (\eX , \Oh)^{\wedge} \longrightarrow C (\hceX_c (\eo) , \eo) \; , \; f \longmapsto ((\ex , \ey , \oP^{\times}) \mapsto \hP_{\ey} (f)) \; .
\]
Here we write $f$ also for the image of $f$ under the canonical map
\[
\varprojlim_{(\,)^p} \Gamma (\eX , \Oh)^{\wedge} \longrightarrow \varprojlim_{(\,)^p} \hOh_{\overline{\{ \ex \} } , \ey} \; .
\]
The resulting ring homomorphism fits into a commutative diagram of $F_p$- and $G$-equivariant maps
\begin{equation}
\label{eq:113}
\xymatrix{
\Z \Gamma (\eX , \Oh)^{\wedge b} \ar[dr]_{\eqref{eq:93}} \ar[rr]^{e} && C ( \hceX_c (\eo) , \eo) \ar[dl]^{\res} \\
 & C (\ceX_c (\eo) , \eo) \; .
}
\end{equation}

\begin{defn} \label{t1512} 
For any real number $0 < \alpha < 1$ we define a $G$-invariant subspace $\hcY\!_{\alpha}$ of $\hceX_c (\eo)$ just like we defined $\cY_{\alpha} \subset \ceX_c (\eo)$ in Definition \ref{t92}
\begin{align}
\hcY\!_{\alpha} & = \{ (\ex , \ey , \hP_{\ey}) \in \hceX_c (\eo) \mid | f(\hP_{\ey})| \le \alpha \quad \text{for all} \; f \in I_{\ey} \} \label{eq:114} \\
& = \{ (\ex , \ey , \hP_{\ey}) \in \hceX_c (\eo) \mid |\hP_{\ey} (r+s) - \hP_{\ey} (r) - \hP_{\ey} (s)| \le \alpha \quad \text{for} \; r,s \in \hOh^b_{\overline{\{ \ex \} } , \ey} \} \; . \nonumber
\end{align}
\end{defn}

\begin{rem}
We have defined $\hcY\!_{\alpha}$ in terms of $\hOh^b_{\overline{ \{ \ex \} } , \ey}$. Instead, we could work with Fontaine's formula for addition on $\varprojlim_{(\,)^p} \hOh_{ \overline{\{ \ex \} }, \ey}$. In those terms, $\hcY\!_{\alpha}$ consist of triples $(\ex , \ey , \hP_{\ey}) \in \hceX_c (\eo)$ such that
\[
\big| \hP_{\ey} \big( \big( \lim_{\nu \to \infty} (a_{\nu + n} + b_{\nu + n})^{p^{\nu}} \big)_n \big) - \hP_{\ey} ((a_n)) - \hP_{\ey} ((b_n)) \big| \le \alpha
\]
for all $(a_n) , (b_n) \in \varprojlim_{(\,)^p} \hOh_{ \overline{ \{ \ex \} } , \ey}$. I do not know how to transport such conditions to the points of $\ceX (\C)$, where $\eX$ is a scheme of finite type over $\spec \Z$ and $\C$ is the complex number field.
\end{rem}

For $0 < \alpha_1 \le \alpha_2 < 1$ we have $\hcY\!_{\alpha_1} \subset \hcY\!_{\alpha_2}$. The map $e$ restricted to functions on $\hcY\!_{\alpha}$ extends to an $F_p$- and $G$-equivariant ring homomorphism
\begin{equation}
\label{eq:115}
e : A_{\inf} (\eX) \longrightarrow C (\hcY\!_{\alpha} , \eo) \; .
\end{equation}
It is the analogue of \eqref{eq:91} and defined as before, replacing $\cP_{\ey}$ by $\hP_{\ey}$. By construction we have a commutative diagram
\begin{equation}
\label{eq:116}
\xymatrix{
A_{\inf} (\eX) \ar[rr]^e \ar[dr]_{\eqref{eq:91}} && C (\hcY\!_{\alpha} , \eo) \ar[dl]^{\res} \\
 & C (\cY_{\alpha} , \eo) \; .
}
\end{equation}

The following analogue of Proposition \ref{t93} holds with the same proofs, just replacing $\cP_{\ey} , \cY_{\alpha}$ by $\hP_y , \hcY\!_{\alpha}$. We set $\hcY := \hcY\!_{1/p}$.

\begin{prop} \label{t119}
For $0 < \alpha < 1$ we have $F_p (\hcY\!_{\alpha}) = \hcY\!_{\alpha}$ and $\hcY\!_{\alpha} \subset \hcY$ and hence $\hcY\!_{\alpha} = \hcY$ for $\alpha \ge 1 / p$. 
\end{prop}

The spaces $\ceX_c (\eo) \subset \hceX_c (\eo)$ are hard to describe because we do not have a good understanding of the multiplicative continuous maps $\multmap \hOh_{\overline{\{ \ex \} } , \ey} \to \eo$. The subspaces $\cY_{\alpha} \subset \hcY\!_{\alpha}$ are more accessible because of the following fundamental observations which we will apply to $A = \hOh_{\overline{\{ \ex \} } , \ey}$.

\begin{prop} \label{t1110}
Let $A$ be a $p$-adically complete ring. Then there is a natural bijection between (continuous) multiplicative maps $\chi : \multmap A \to \eo$ with $\chi (1) = 1 , \chi (0) = 0$ for which the composition 
\begin{equation}
\label{eq:117}
\ochi^b : A^b \cong \multmap A \xrightarrow{\chi} \eo \longrightarrow \eo / p
\end{equation}
is additive and (continuous) ring homomorphisms $\chi^b : A^b \to \eo^b$. Explicitely, $\chi$ corresponds to the composition
\begin{equation}
 \label{eq:118}
 \chi^b : A^b \cong \multmap A \xrightarrow{\hat{\ochi}} \eo^b \, , 
\end{equation}
where $\hochi ((a_{\nu})) = (\ochi ((a_{\nu + n})_{\nu}))_n$ and $\ochi = \chi \mod p$. On the other hand, given $\chi^b$, the map $\chi$ is the composition 
\[
\chi : \multmap A \cong A^b \xrightarrow{\chi^b} \eo^b \xrightarrow{\sharp} \eo \; ,
\]
where $\sharp ((x_n \mod p)_n) = \lim_{n \to \infty} x^{p^n}_n$. For such $\chi$'s the following diagram commutes
\begin{equation} \label{eq:119n}
\xymatrix{
A^b \ar[r]^{[\,]} \ar@{=}[d]^{\wr} & W_p (A^b) \ar[r]^{W_p (\chi^b)} & W_p (\eo^b) \ar[d]^{\theta} \\
\multmap A \ar[rr]^{\chi} && \eo
}
\end{equation}
Here $\theta$ is Fontaine's ring homomorphism which sends $\sum^{\infty}_{\nu = 0} [x_{\nu}] p^{\nu}$ to $\sum^{\infty}_{\nu = 0} x^{\sharp}_{\nu} p^{\nu}$. Moreover, there is a unique ring homomorphism $W_p (\chi)$ such that the map $\chi$ factors as follows
\begin{equation}
\label{eq:120n}
\chi : \multmap A \cong A^b \xrightarrow{[\,]} W_p (A^b) \xrightarrow{W_p (\chi)} \eo  \; .
\end{equation}
We have
\begin{equation}
\label{eq:121n}
W_p (\chi) = \theta \verk W_p (\chi^b) \; .
\end{equation}
Alternatively, let $I = \Ker (\Z A^b \xrightarrow{\pi} A^b)$. Then $\chi (I) \subset p \eo$ for the natural extension of $\chi$ to a ring homomorphism $\chi : \Z A^b \to \eo$. In these terms $W_p (\chi)$ is the induced map
\begin{equation}
\label{eq:122n}
W_p (\chi) : W_p (A^b) = \varprojlim_n \Z A^b / I^n \longrightarrow \varprojlim \eo / p^n = \eo \; .
\end{equation}
\end{prop}

\begin{rem}
In a special case the map $W_p (\chi)$ has been introduced before, in \eqref{eq:89}. 
\end{rem}

\begin{proof}
We write $\Hom$ for sets of (continuous) multiplicative maps sending $1$ to $1$ and $0$ to $0$. Since $A^b$ is perfect we have canonical bijections
\begin{equation}
\label{eq:119}
\Hom (A^b , \eo / p) = \Hom (\varinjlim_{(\,)^p} A^b , \eo / p) = \Hom (A^b , \multmap \eo / p) = \Hom (A^b , \eo^b) 
\end{equation}
and similarly
\begin{equation}
\label{eq:120}
\Hom (\multmap A , \eo) = \Hom (A^b , \eo) = \Hom (A^b , \multmap \eo) = \Hom (A^b , \eo^b) \; .
\end{equation}
Note that the bijections \eqref{eq:119} are also valid if $\Hom$ denotes (continuous) {\it ring} homomorphisms. 

Consider the following diagram:
\begin{equation}
\label{eq:121}
\xymatrix{
\Hom (\multmap A , \eo) \ar@{=}[r]^-{\eqref{eq:120}} \ar@{=}[d]^{\mod p} & \Hom (A^b , \eo^b) \ar@{}[r]|{\supset} \ar@{=}[d]_{\eqref{eq:119}} & \Hom_{\ring} (A^b , \eo^b) \ar@{=}[d]_{\eqref{eq:119}} \\
\Hom (\multmap A , \eo / p) \ar@{=}[r]_-{\eqref{eq:87}} & \Hom (A^b , \eo / p) \ar@{}[r]|{\supset} &\Hom_{\ring} (A^b , \eo / p) \; .
}
\end{equation}
Defining $\ochi = \chi \mod p$, and $\chi^b , \ochi^b$ by \eqref{eq:118}, \eqref{eq:117}, the maps in the left square are
\[
\xymatrix{
\chi \ar@{|->}[r]\ar@{|->}[d] & \chi^b \ar@{|->}[d] \\
\ochi \ar@{|->}[r] & \ochi^b
}
\]
Given $\chi \in \Hom (\varprojlim_{(\,)^p} A, \eo)$, the map $\chi^b : A^b \to \eo^b$ is mapped by \eqref{eq:119} to the composition $A^b \xrightarrow{\chi^b} \eo^b \xrightarrow{\pr_0} \eo / p$ and this is equal to $\ochi^b : A^b \cong \multmap A \xrightarrow{\ochi} \eo / p$. Hence diagram \eqref{eq:121} commutes. It also shows that for $\chi$, the map $\ochi^b$ is additive i.e. a ring homomorphism if and only if $\chi^b$ is a ring homomorphism. As for \eqref{eq:121n}, let $(a_{\nu}) \in \multmap A$, and calculate:
\begin{align*}
\theta (W_p (\chi^b)) [(a_{\nu} \mod p)] & = \theta ([\chi^b (a_{\nu} \mod p)]) \\
& = \theta ([\hochi ((a_{\nu}))]) \quad \text{by definition of} \; \chi^b \\
& = \hochi ((a_{\nu}))^{\sharp} \quad \text{by definition of} \; \theta \\
& = \lim_{n\to \infty} \chi ((a_{\nu+n})_{\nu})^{p^n} \quad \text{by definition of} \; \hochi \; \text{and} \; \sharp \\
& = \chi ((a_{\nu})) \quad \text{since} \; \chi \; \text{is multiplicative and} \; a^{p^n}_{\nu + n} = a_{\nu} \; .
\end{align*}
The composition
\[
\Z A^b \longrightarrow W_p (A^b) \longrightarrow W_p (A^b) / p = A^b
\]
is the map $\pi$ and hence $\pi (I) \subset p W_p (A^b)$. The induced map
\[
\Z A^b / I^n \silo W_p (A^b) / p^n
\]
is an isomorphism for all $n \ge 1$ by \cite{CD}. Since $\chi : A^b \to \eo$ is additive $\mod p$, we have $\chi (I) \subset p \eo$ and hence we get the induced map $W_p (\chi)$ in \eqref{eq:122n}. It satisfies the factorization \eqref{eq:120n} by definition. A ring homomorphism $\alpha : W_p (A^b) \to \eo$ with $\chi = \alpha \verk [\,]$ is uniquely determined because $\alpha$ is the projective limit of the maps
\[
\Z A^b / I^n = W_p (A^b) / p^n \xrightarrow{\alpha \mod p^n} \eo / p^n
\]
which are determined by their values on Teichm\"uller representatives. Hence $W_p (\chi)$ is uniquely determined in the factorization \eqref{eq:120n}. Moreover formula \eqref{eq:121n} now follows from the commutative diagram \eqref{eq:119n}. 
\end{proof}

\begin{rem}
If $A$ is a $p$-adically complete integral domain with $\mu_{p^{\infty}} (A) = \mu_{p^{\infty}} (K)$ where $K = \Quot A$, any continuous multiplicative map
\[
\ochi : \multmap A \longrightarrow \eo / p
\]
factors
\[
\ochi : \multmap A \xrightarrow{\pr_{\nu}} A \xrightarrow{\ochi_{\nu}} \eo / p
\]
for some $\nu$. Namely, the compact subgroup $T_p \mu (A)$ of the multiplicative monoid $\multmap A$ is sent to a compact, hence finite subgroup of the discrete multiplicative monoid $\eo / p$. If $p^{\nu}$ is its order the map $\ochi$ factors over $\pr_{\nu}$. However, the map $\chi : \multmap A \to \eo$ corresponding to $\ochi$ does not factor over $\pr_{\nu}$ for any $\nu$ in general if $\car \C = 0$.
\end{rem}

We also need the following simple facts.

\begin{prop}
\label{t1213}
For a local ring $A$ with maximal ideal $\emm$ containing $p$ consider the local ring $\oA = A / pA$ with maximal ideal $\oemm = \emm / pA$ and residue field $\oA / \oemm = A / \emm$. \\
1) The $p$-adic completion $\hA$ of $A$ is a local ring with maximal ideal $\hemm := \varprojlim_n \emm / p^n A$. The natural maps $\oA \silo \ohA , \oemm \silo \overline{\hemm}$ and $A / \emm \silo \hA / \hemm$ are isomorphisms.\\
2) The tilt $A^b = \multmap \oA$ is a local ring with maximal ideal $\emm^b := \multmap \oemm$, and $(A^b , \emm^b) = (\hA^b , \hemm^b)$. If the $p$-power map is surjective on $\oA$, then the projection onto the zeroeth component, $\pr_0 : A^b \to \oA$ induces an isomorphism
\[
\pr_0 : A^b / \emm^b \silo \oA / \oemm = A / \emm \; .
\]
\end{prop}

\begin{proof}
1) The sequence $0 \to \emm / p^n A \to A / p^n A \to A / \emm \to 0$ is exact and $(\emm / p^n A)_n$ is Mittag-Leffler. Hence we have $A / \emm = \hA / \hemm$. If $a = (a_n \mod p^n A) \in \hA$ is not in $\hemm$ then $a_n \notin \emm$ for some and hence all $n$. Note here that all $a_n$'s are congruent $\mod pA \subset \emm$. Hence $a_n \in A^{\times}$ for all $n$ and therefore $a \in \hA^{\times}$. The rest is equally easy.\\
2) It is clear that $\emm^b$ is an ideal in $A^b$. If $x = (x_n) \in A^b \setminus \emm^b$ then $x_n \notin \oemm$ for some and hence all $n$. Note here that for $y \in \oA$ we have $y \in \oemm$ if and only if $y^p \in \oemm$ since $\oemm$ is maximal and hence prime in $\oA$. It follows that $x_n \in \oA^{\times}$ for all $n$ and hence $x \in (A^b)^{\times}$. The isomorphism $A^b / \emm^b \silo \oA / \oemm$ follows by a similar argument. The rest is clear.
\end{proof}

\begin{example}
For each local ring $\Oh_{\overline{\{ \ex \} } , \ey}$ as in Definition \ref{t91} the completion $\hOh_{\overline{\{ \ex \} } , \ey}$ and its tilt $\Oh^b_{\overline{\{ \ex \} } , \ey} = \hOh^b_{\overline{\{ \ex \} } , \ey}$ are both local rings with algebraically closed residue field $\kappa (\ey)$ of characteristic $p$. 
\end{example}

By definition, $\hcY = \hcY\!_{1/p}$ consists of those points $(\ex , \ey , \hP_{\ey}) \in \hceX_c (\eo)$ for which $\hP_{\ey} : \hOh^b_{\overline{\{ \ex \} } , \ey} \to \eo$ is additive $\mod p$. Using the preceeding propositions, we get the following descriptions of $\hcY$:

\begin{prop}
\label{t1214}
With notations as in Definitions \ref{t105} and \ref{t1512} we have
\[
\hcY = \{ (\ex , \ey , \hP_{\ey}) \in \hceX_c (\eo) \mid \hP^b_{\ey} : \Oh^b_{\overline{\{ \ex \} } , \ey} \longrightarrow \eo^b \; \text{is a ring homomorphism} \, \}  \; .
\]
We may identify $\hcY$ with the set of triples $(\ex , \ey , \hP^b_{\ey})$ where $\ey \in \overline{\{ \ex \} }$ with $\car \kappa (\ey) = p$ and where $\hP^b_{\ey} : \hOh^b_{\overline{\{ \ex \} } , \ey} \to \eo^b$ is a continuous local homomorphism of local rings such that $\hP^b_{\ey} (f) \neq 0$ for all $0 \neq f \in \multmap \Oh_{\overline{\{ \ex \} } , \ey} \subset \multmap \hOh_{\overline{\{ \ex \} } , \ey} \cong \hOh^b_{\overline{\{ \ex \} } , \ey}$.
\end{prop}

\begin{proof}
The first assertion follows from Proposition \ref{t1110}. As for the second note that for $(\ex , \ey , \hP_{\ey}) \in \hcY$ the induced ring homomorphism $\hP^b_{\ey}$ is continuous by definition. Using the above example, locality of $\hP^b_{\ey}$ follows from diagrams \eqref{eq:97} and \eqref{eq:109} since $\pr_0 : \multmap \kappa (\ey) \silo \kappa (\ey)$ is an isomorphism. The formula after \eqref{eq:118} implies that $\hP_{\ey}$ is the composition
\begin{equation}
\label{eq:201n}
\hP_{\ey} : \multmap \hOh_{ \overline{ \{ \ex \} } , \ey} \cong \hOh^b_{ \overline{ \{ \ex \} } , \ey} \xrightarrow{\hP^b_{\ey}} \eo^b \xrightarrow{\sharp} \eo \; .
\end{equation}
By Proposition \ref{t117} we know that $\hP_{\ey} (f) \neq 0$ for $0 \neq f \in \multmap \hOh_{ \overline{ \{ \ex \} } , \ey}$. Hence the factorization \eqref{eq:201n} implies that $\hP^b_{\ey} (f) \neq 0$. Conversely, let $\hP^b_{\ey}$ be a continuous local homomorphism which is non-zero on $(\multmap \hOh_{ \overline{ \{ \ex \} } , \ey}) \setminus 0$. Let $\hP_{\ey}$ be the continuous multiplicative and $\mod p$ also additive map \eqref{eq:201n} with $1 \mapsto 1$ and $0 \mapsto 0$, which corresponds to $\hP^b_{\ey}$ via Proposition \ref{t1110}. Since $\sharp^{-1} (0) = 0$, condition 1) in Proposition \ref{t117} is satisfied for $\hP_{\ey}$. The factorization \eqref{eq:201n} fits into a diagram with evident vertical maps:
\begin{equation}
\label{eq:202n}
\xymatrix{
\hP_{\ey} : \multmap \hOh_{ \overline{ \{ \ex \} } , \ey} \ar@{->>}[d] \ar@{=}[r]^-{\sim} & \hOh^b_{ \overline{ \{ \ex \} } , \ey} \ar@{->>}[d] \ar[r]^{\hP^b_{\ey}} & \eo^b \ar@{->>}[d] \ar[r]^{\sharp} & \eo \ar@{->>}[d] \\
{}\quad \multmap \kappa (\ey) \ar@{=}[r]^-{\sim} & \kappa (\ey) \ar[r] & k \ar@{=}[r] & k
}
\end{equation}
Using e.g. Proposition \ref{t1213} it follows that the left and right squares are commutative. The commutative middle square exists because $\hP^b_{\ey}$ is a local homomorphism. Hence the diagram commutes and it follows that $\hP_{\ey}$ also satisfies 2) of Proposition \ref{t117}. Hence we see that $(\ex , \ey, \hP_{\ey}) \in \hcY$. 
\end{proof}

\begin{rem}
The condition on $\hP^b_{\ey}$ to vanish only in the zero element of the non-complete subring $\multmap \Oh_{ \overline{ \{ \ex \} } , \ey}$ of $\hOh^b_{ \overline{ \{ \ex \} } , \ey}$ looks a bit unnatural to me. As we will see below, the condition is fine for $\eX_0 = \spec \eo_{K_0}$ with $K_0 / \Q_p$ finite, but in higher dimensions, perhaps it has to be amended. The guesswork in defining $\hceX (\eo)$ and $\hcY$ starting from $W_{\rat} (\eX) (\eo)$ may not be complete.
\end{rem}

\begin{remark} \label{t1215}
In Remark \ref{t37} we have noted that $\Aut (X) \times \Aut (S)$ operates naturally and $\Nh$-equivariantly on $W_{\rat} (X) (S)$. In particular $G \times \Aut (\eo)$ operates $F_p$-equivariantly on $\deX (\eo) = W_{\rat} (\eX) (\eo)$ and $\ceX (\eo)$. Since automorphisms of $\eo$ are $p$-adically continuous we have compatible $G \times \Aut (\eo)$-operations on $\deX_c (\eo) , \ceX_c (\eo), \hceX_c (\eo)$ and $\hcY$ in the obvious way. 
\end{remark}

\section{A relation to the Fargues-Fontaine curve} \label{sec:13}
For a finite extension $K_0$ of $\Q_p$ and $\eX_0 = \spec \eo_{K_0}$ we will now establish a canonical Frobenius-equivariant bijection between (the generic fibre of) $\hcY_0 = \hcY / G$ and the points of the Fargues-Fontaine curve. Here $\eX = \spec \eo_K$ where $K$ is an algebraic closure of $K_0$ with Galois group $G = \Gal (K / K_0)$ and $\hcY = \hcY_{1/p}$ was defined in \eqref{eq:114}. In order to construct the bijection we have to recall well known results by Fontaine-Wintenberger and Scholze on tilting for which a wonderful source is \cite[I.3, I.4, I.6]{schneider2}. Choose a prime element $\pi_0$ of $\eo_{K_0}$ and let $\kappa_0$ be the residue field of $\eo_{K_0}$. We normalize the absolute value on $K_0$ by $| \pi_0| = q^{-1}$ where $q = |\kappa_0|$. Let $\phi \in \eo_{K_0} [[X]]$ be a Frobenius power series for $\pi_0$ i.e. $\phi (X) = \pi_0 X$ + higher terms and $\phi (X) \equiv X^q \mod \pi_0$. Let $F_{\phi}$ be the corresponding Lubin-Tate formal group law. Then we have a canonical ring homomorphism
\[
\eo_{K_0} \longrightarrow \End_{\eo_{K_0}} (F_{\phi}) \; , \; a \longmapsto [a]_{\phi} (X) \quad \text{with} \; [\pi_0]_{\phi} = \phi \; .
\]
Let $\Mh$ and $\hMh$ be the maximal ideals of $\eo_K$ and $\heo_K = \eo_{\hK}$ respectively. They become $\eo_{K_0}$-modules via $z_1 + z_2 = F_{\phi} (z_1 , z_2)$ and $a \cdot z = [a]_{\phi} (z)$ for $a \in \eo_{K_0}$ and $z_1 , z_2 , z \in \Mh$ resp. $\hMh$. The kernel $\Fh_n$ of $\pi^n_0$-multiplication on $\Mh$ is a free rank-one $\eo_{K_0} / \pi^n_0$-module and the Galois extension fields
\[
K_n = K_0 (\Fh_n) \quad \text{and} \quad K_{\infty} = \bigcup_{n \ge 1} K_n
\]
depend only on $\pi_0$ and not on $\phi$. The fields $K_n$ are totally ramified over $K_0$. The completion $\hK_{\infty}$ is a perfectoid field. The Galois group $G = \Gal (K / K_0)$ acts $\eo_{K_0} / \pi^n_0$-linearly on $\Fh_n$, and there is a unique character
\[
\chi_n : G \longrightarrow (\eo_{K_0} / \pi^n_0)^{\times}
\]
such that $\sigma (z) = \chi_n (\sigma) z$ for $z \in \Fh_n$ and $\sigma \in G$. It induces an isomorphism
\[
\chi_n : \Gal (K_n / K_0) \silo (\eo_{K_0} / \pi^n_0)^{\times} \; .
\]
Any generator $z$ of the $\eo_{K_0} / \pi^n_0$-module $\Fh_n$ generates $\eo_{K_n}$ as an $\eo_{K_0}$-algebra and $z$ is a prime element of $\eo_{K_n}$. Taking the projective limit, the $\chi_n$ induce a continuous character
\[
\chi : G \longrightarrow \eo^{\times}_{K_0}
\]
which factors over an isomorphism
\begin{equation} \label{eq:121a}
\chi : \Gamma := \Gal (K_{\infty} / K_0) \silo \eo^{\times}_{K_0} \; .
\end{equation}
Let
\[
T (\Fh) := \varprojlim_n \Fh_n
\]
with $\pi_0$-multiplications as transition maps, a free $\eo_{K_0}$-module of rank one with a continuous $G$-action. Let
\[
\heo^b_{K_{\infty}} := \multmap \heo_{K_{\infty}} / p \overset{(!)}{=} \varprojlim_{(\,)^q} \heo_{K_{\infty}} / \pi_0
\]
(and similarly for $\heo^b_K$) be the tilts of $\heo_{K_{\infty}}$ (and $\heo_K$). These rings are naturally $\kappa_0 = \eo_{K_0} / \pi_0$-algebras. Their quotient fields $\hK^b_{\infty}$ and $\hK^b$ are perfect of characteristic $p$ and complete for a rank one valuation $|\,|_b$ ultimately coming from the one on $K_0$ specified above. The valuation topology and the pro-discrete topology on $\heo^b_{K_{\infty}}$ resp. $\heo^b_K$ agree. The $G$-operations by transport of structure on these tilted rings and fields are continuous and preserve $|\,|_b$. The subgroup $G_{\infty} = \Gal (K / K_{\infty})$ of $G$ acts trivially on $\hK^b_{\infty}$ and we get an induced action by $\Gamma = G / G_{\infty}$ on $\hK^b_{\infty}$. Consider the $G$- and hence $\Gamma$-equivariant injection
\begin{equation} \label{eq:122}
i : T (\Fh) \hookrightarrow \heo^b_{K_{\infty}} \; , \; i ((y_n)) = (y_n \mod \pi_0) \; .
\end{equation}
Fix a generator $t$ of the $\eo_{K_0}$-module $T (\Fh)$ and let $t^b = i (t) \in \heo^b_{K_{\infty}}$ be its image. Then $|t^b|_b = |\pi_0|^{q / q-1} = q^{-(q/q-1)} < 1$ and we get a continuous homomorphism of $\kappa_0$-algebras
\[
\kappa_0 [[X]] \longrightarrow \heo^b_{K_{\infty}} \; , \; f \longmapsto f (t^b) \; .
\]
Since $t^b$ is invertible in $\hK^b_{\infty}$, we obtain an injection of fields
\[
\kappa_0 (( X)) \hookrightarrow \hK^b_{\infty} \; .
\]
We denote its image by $E_0$. With the induced valuation $|\,|_b$ it is a complete discretely valued field with residue class field $\kappa_0$, and $t^b$ is a prime element of its ring of integers $\eo_{E_0}$, the image of $\kappa_0 [[X]]$. The field $E_0$ does not depend on the choice of generator $t$ of $T (\Fh)$ and $E_0$ is preserved by the action of $\Gamma$ on $\hK^b_{\infty}$. This follows from the formula
\begin{equation} \label{eq:122a}
\gamma (t^b) = \overline{[\chi (\gamma)]}_{\phi} (t^b) \quad \text{for} \; \gamma \in \Gamma
\end{equation}
where
\[
\overline{[a]}_{\phi} (X) := [a]_{\phi} (X) \mod \pi_0 \in \kappa_0 [[X]] \quad \text{for} \; a \in \eo_{K_0} \; .
\]
The field $\hK^b$ is algebraically closed and complete. Let $E^{\sep}_0$ be the separable closure of $E_0$ in $\hK^b$. The topological closure of $E^{\sep}_0$ in $\hK^b$ can be identified with the completion $\widehat{E^{\sep}_0}$ of $E^{\sep}_0$ and we have
\[
\widehat{E^{\sep}_0} = \hK^b \; .
\]
The $G$-action on $\hK$ induces a $G$-action on $\hK^b$ and $G_{\infty} \subset G$ acts trivially on $\hK^b_{\infty} \subset \hK^b$ and hence on $E_0$. Thus $G_{\infty}$ map $E^{\sep}_0$ into itself. A main result of the theory asserts that the resulting homomorphism
\begin{equation} \label{eq:123}
G_{\infty} = \Gal (K / K_{\infty}) \silo \Gal (E^{\sep}_0 / E_0) 
\end{equation}
is a topological isomorphism. This concludes our review of the tilting correspondence. 

From now on let $\eo$ be a $p$-adically complete rank one valuation ring with (complete) algebraically closed quotient field $\C$ of characteristic zero. Let $\emm$ be its maximal ideal and $k$ its algebraically closed residue field of characteristic $p$. Then $\eo^b$ is a complete rank one valuation ring of equicharacteristic $p$ with complete algebraically closed quotient field $\C^b$ of characteristic $p$, maximal ideal $\emm^b = \multmap \emm / p \eo$ and residue field $k$. Then $k$ is also a subfield of $\eo^b$ such that $k \subset \eo^b \to k$ is the identity. One example is $\C = \C_p := \hoQ_p$ etc. The pro-discrete topology on $\eo^b = \multmap \eo / p$ agrees with the valuation topology on $\eo^b$. 

Let $H_{\eo_{K_0}} = \Spf \eo_{K_0} [[X]]$ be the formal $\eo_{K_0}$-module over $\spec \eo_{K_0}$ corresponding to $F_{\phi}$ and $[a]_{\phi}$ for $a \in \eo_K$. It is defined using the maps
\[
\eo_{K_0} [[X]] \longrightarrow \eo_{K_0} [[X]] \hotimes \eo_{K_0} [[X]] = \eo_{K_0} [[X_1 , X_2]]
\]
sending $X$ to $F_{\phi} (X_1 , X_2)$ and
\[
\eo_{K_0} [[X]] \longrightarrow \eo_{K_0} [[X]]
\]
sending $X$ to $[a]_{\phi} (X)$. Consider the reduction $\mod \pi_0$
\[
H_{\kappa_0} = H_{\eo_{K_0}} \hotimes_{\eo_{K_0}} \kappa_0 = \Spf \kappa_0 [[X]] \; .
\]
It is a formal $\eo_{K_0}$-module over $\spec \kappa_0$. Viewing $\spec \eo^b$ as a formal scheme over $\F_p$, the points
\[
H_{\kappa_0} (\eo^b) := \Mor_{\F_p} (\spec \eo^b , H_{\kappa_0}) = \Hom_{\cont} (\kappa_0 [[X]] , \eo^b)
\]
are a union of $\eo_{K_0}$-modules. In particular, the set $H_{\kappa_0} (\eo^b)$ carries an action by the group $\eo^{\times}_{K_0}$. More explicitly we have
\begin{equation}
\label{eq:124}
H_{\kappa_0} (\eo^b) = \coprod_{\tau_0 \in \Hom (\kappa_0 , k)} \emm^b_{\tau_0} = \emm^b \times \Hom (\kappa_0 , k) \; .
\end{equation}
Here $\emm^b_{\tau_0}$ is $\emm^b$ as a set, equipped with addition and $\eo_{K_0}$-multiplication via $\tau_0 \overline{F}_{\phi} \in k [[ X_1 , X_2 ]]$ and $\tau_0 \overline{[a]}_{\phi} \in k [[X]]$ where $\overline{F}_{\phi} = F_{\phi} \mod \pi_0$ and $\overline{[a]}_{\phi} = [a]_{\phi} \mod \pi_0$. On $\emm^b \times \Hom (\kappa_0 , k)$ the $\eo_{K_0}$-multiplication is given by the formula
\begin{equation}
\label{eq:124a}
a \cdot (x , \tau_0) = ((\tau_0 \overline{[a]}_{\phi}) (x) , \tau_0) \quad \text{for} \; a \in \eo_{K_0} \; .
\end{equation}
Note that $\Aut (\eo)$ also acts naturally on $H_{\kappa_0} (\eo^b)$ via $\Aut (\eo) \to \Aut_{\cont} (\eo^b)$. 

Fix a generator $t$ of the $\eo_{K_0}$-module $T (\Fh)$. We will now define a map
\begin{equation}
\label{eq:125}
\Psi_t : \Hom_{\cont} (\heo^b_K , \eo^b) / G_{\infty} \longrightarrow H_{\kappa_0} (\eo^b) \; .
\end{equation}
Here $G$ and hence $G_{\infty}$ act from the right on the set of continuous ring homomorphisms $\varphi : \heo^b_K \to \eo$ by setting $\varphi^{\sigma} := \varphi \verk \sigma$ for $\sigma \in G$. Since the extension $K_{\infty} / K_0$ is totally ramified, the residue field of $K_{\infty}$ is $\kappa_0$ and the natural map $G_{\infty} \to \Gal (\kappa / \kappa_0)$ is surjective. Here $\kappa$ is the common residue field of $\eo_K , \heo_K$ and $\heo^b_K$, an algebraic closure of $\kappa_0$. It is also a subfield of $\heo^b_K$, namely the algebraic closure of $\kappa_0$ in $\hK^b$, such that the composition $\kappa \subset \heo^b_K \to \kappa$ is the identity. Using the element $t^b = i (t) \in \heo^b_{K_{\infty}} \subset \heo^b_K$ on which $G_{\infty}$ acts trivially, we now define
\[
\Psi_t (\varphi G_{\infty}) = (\varphi (t^b) , \varphi \, |_{\kappa_0}) \in \emm^b \times \Hom (\kappa_0 , k) = H_{\kappa_0} (\eo^b) \; .
\]
Note that since $\varphi$ is continuous it maps topologically nilpotent elements to topologically nilpotent elements. Hence we have $\varphi (t^b) \in \emm^b$ recalling that $|t^b|_b < 1$. 

\begin{theorem}
\label{t131}
For any $\eo_{K_0}$-generator $t \in T (\Fh)$ the map $\Psi_t$ is a bijection:
\[
\Psi_t : \Hom_{\cont} (\heo^b_K , \eo^b) / G_{\infty} \silo H_{\kappa_0} (\eo^b) \; .
\]
\end{theorem}

\begin{rem}
Since $\heo^b_K$ is a rank one valuation ring, its only ideals are $0$ and $\hMh^b$. Hence a ring homomorphism $\varphi : \heo^b_K \to \eo^b$ is either injective or it factors over $\heo^b_K / \hMh^b = \kappa$, in which case it is a composition
\[
\varphi : \heo^b_K \longrightarrow \kappa \xrightarrow{\tau} k \subset \eo^b \; .
\]
Under the bijection $\Psi_t$ these two cases correspond to the elements in $(\emm^b \setminus 0) \times \Hom (\kappa_0 , k)$ resp. $\{ 0 \} \times \Hom (\kappa_0 , k)$ of $H_{\kappa_0} (\eo^b)$. 
\end{rem}

\begin{proof}
We construct a map in the other direction:
\[
\Phi_t : \emm^b \times \Hom (\kappa_0 , k) \longrightarrow \Hom_{\cont} (\heo^b_K , \eo^b) / G_{\infty} \; .
\]
For expliciteness we distinguish two cases although this is not really necessary. For an element
\[
(0 , \tau_0) \in \{ 0 \} \times \Hom (\kappa_0 , k)
\]
choose an extension $\tau : \kappa \to k$ of $\tau_0$ and set
\[
\Phi_t (0 , \tau_0) = (\heo^b_K \to \kappa \xrightarrow{\tau} k \subset \eo^b) \mod G_{\infty} \; .
\]
This is well-defined since $G_{\infty} \to \Gal (\kappa / \kappa_0)$ is surjective.

We will now define $\Phi_t$ on the elements
\[
(x , \tau_0) \in (\emm^b \setminus 0) \times \Hom (\kappa_0 , k) \; .
\]
Let
\[
\varphi_{t , x , \tau_0} : E_0 \hookrightarrow \C^b
\]
be the unique continuous homomorphism which equals $\tau_0$ on $\kappa_0 \subset E_0$ and sends $t^b = i (t) \in E_0$ to $x$. It exists because $0 < |x| < 1$. Let
\[
\tvarphi_{t , x, \tau_0} : E^{\sep}_0 \hookrightarrow \C^b
\]
be an extension of $\varphi_{t , x , \tau_0}$ to the separable closure of $E_0$ in $\hK^b$. Because of the fundamental isomorphism \eqref{eq:123} the extension is unique up to conjugation by an element of $G_{\infty}$. We have
\begin{equation}
\label{eq:127}
|\tvarphi_{t , x , \tau_0} (e)| = |e|^{\alpha}_b \quad \text{for all} \; e \in E^{\sep}_0 \; .
\end{equation}
Here
\[
\alpha = \big( 1 - \frac{1}{q} \big) \log_q |x|^{-1} > 0 \; .
\]
This holds because both sides of \eqref{eq:127} give (the unique) valuation on $E^{\sep}_0$ extending the valuation $|\;|^{\alpha}_b$ on the complete field $E_0$, since they agree on the generating prime element $t^b$ of $E_0$. In particular, $\tvarphi_{t , x , \tau_0}$ is equicontinuous and extends uniquely to a continuous embedding of complete and algebraically closed fields
\[
\htvarphi_{t , x, \tau_0} : \hK^b = \widehat{E^{\sep}_0} \hookrightarrow \C^b \; .
\]
We have
\begin{equation}
 \label{eq:128}
 |\htvarphi_{t , x , \tau_0} (e)| = |e|^{\alpha}_b \quad \text{for all} \; e \in \hK^b \; .
\end{equation}
Any continuous homomorphism $\hK^b \hookrightarrow \C^b$ sends $\eo_{\hK^b}$ to $\eo^b$. For $\htvarphi_{t , x , \tau_0}$ this is also evident from \eqref{eq:128}. Restriction therefore allows us to view $\htvarphi_{t , x , \tau_0}$ as a continuous injective homomorphism
\[
\htvarphi_{t , x , \tau_0} : \heo^b_K \hookrightarrow \eo^b \; .
\]
By construction its $G_{\infty}$-orbit is well defined, and we set
\[
\Phi_t (x, \tau_0) = \htvarphi_{t , x , \tau_0} \mod G_{\infty} \in \Hom_{\cont , \inj} (\heo^b_K , \eo^b) / G_{\infty} \; .
\]
It is straightforward to check that $\Phi_t$ is inverse to $\Psi_t$, which proves the theorem and the subsequent remark. Instead of working with the separable closure of $E_0$ and the formula $\hK^b = \widehat{E^{\sep}_0}$, we could also have used the isomorphism $G_{\infty} \silo \Gal (\oE_0 / E^{\perf}_0)$ and the equality $\hK^b = \hoE_0$ where $\oE_0$ is the algebraic closure of $E_0$ in $\hK^b$. Since $\C^b$ is perfect, the embedding $\varphi_{t , x, \tau_0} : E_0 \hookrightarrow \C^b$ has a unique extension $\varphi^{\perf}_{t , x, \tau_0} : E^{\perf}_0 \hookrightarrow \C$ and one extends to $\oE_0$ and then to $\hoE_0 = \hK^b$ as before. We could also have defined $\varphi_{t, x , \tau_0} : \eo_{E_0} = \kappa_0 [[t^b]] \to \eo^b$ to be $\tau_0$ on $\kappa_0$ and send $t^b$ to $x$. This works for {\it all} $x \in \emm^b$, and extending to a continuous homomorphism $\htvarphi_{t , x , \tau_0} : \heo^b_K \to \eo^b$ similarly as before one has $\Phi_t (x , \tau_0) = \htvarphi_{t , x , \tau_0} G_{\infty}$. In this way, at least in the definition of $\Phi_t$ one does not need to distinguish the cases $x = 0$ and $x \neq 0$. 
\end{proof}

Since $G_{\infty} \subset G$ is a normal subgroup, $G$ acts via $\Gamma = G / G_{\infty}$ on $\Hom_{\cont} (\heo^b_K , \eo^b) / G_{\infty}$. Via the isomorphism $\chi : \Gamma \silo \eo^{\times}_{K_0}$, and the $\eo^{\times}_{K_0}$-action on $H_{\kappa_0} (\eo^b)$, the group $\Gamma$ acts on $H_{\kappa_0} (\eo^b)$ as well. The absolute Frobenius $(\,)^p$ on $\kappa_0 [[X]]$ acts on $H_{\kappa_0}$ over $\spec \F_p$ and induces an automorphism $F_p$ on $H_{\kappa_0} (\eo^b)$. Under the bijections \eqref{eq:124} it has the following descriptions:
\begin{equation}
\label{eq:129}
F_p (x , \tau_0) = (x^p , \tau_0 \verk (\;)^p) \quad \text{for} \; (x , \tau_0) \in \emm^b \times \Hom (\kappa_0 , k)
\end{equation}
and
\begin{equation}
\label{eq:130}
F_p (x) = x^p \in \emm_{\tau_0 \verk (\,)^p} \quad \text{for} \; x \in \emm^b_{\tau_0}  \; .
\end{equation}
By definition
\[
\overline{[\pi_0]}_{\phi} (X) = X^q \in \kappa_0 [[X]] \; ,
\]
and hence we have
\[
\overline{[\pi_0]}_{\phi} = (\;)^q \quad \text{on} \; \kappa_0 [[X]] \; .
\]
It follows that $F_q := F^r_p$ for $q = p^r$ acts on $H_{\kappa_0} (\eo^b)$ by multiplication with $\pi_0$. Since $F_p$ is a bijection of $H_{\kappa_0} (\eo^b)$, it follows that the action of $\eo_{K_0}$ on $H_{\kappa_0} (\eo^b)$ by multiplication extends to an action of $K_0$ by multiplication. In particular, the group $K^{\times}_0 = \eo^{\times}_{K_0} \times \pi^{\Z}_0$ acts on $H_{\kappa_0} (\eo^b)$. We define $F_p$ and $F_q = F^r_p$ on $\Hom_{\cont} (\eo^b_K , \eo^b)$ by $F_p (\varphi) = \varphi \verk (\;)^p = \varphi^p$. Then $F_p$ is an isomorphism which commutes with the $G$-action and hence passes to the quotients by $G_{\infty}$ and $G$. 

\begin{prop}
\label{t132}
The bijections in Theorem \ref{t131}
\[
\Psi_t : \Hom_{\cont} (\heo^b_K , \eo^b) / G_{\infty} \silo H_{\kappa_0} (\eo^b)
\]
are $\Gamma \times \Aut (\eo) \times \langle F_p \rangle$-equivariant. For $\gamma \in \Gamma$ we have
\begin{equation}
\label{eq:131}
\Psi_t \verk \gamma = \Psi_{\gamma (t)} \; .
\end{equation}
Hence the induced $\Aut (\eo) \times \langle F_p \rangle$-equivariant bijection
\begin{equation}
\label{eq:132}
\Psi = \Psi_t / \Gamma : \Hom_{\cont} (\heo^b_K , \eo^b) / G \silo H_{\kappa_0} (\eo^b) / \Gamma = H_{\kappa_0} (\eo^b) / \eo^{\times}_{K_0}
\end{equation}
is independent of the choice of $t$.
\end{prop}

\begin{rem}
Dividing out the $F_q$- resp. $F_p$-actions in \eqref{eq:132} we get bijections
\begin{equation}
\label{eq:133}
\Hom_{\cont} (\heo^b_K , \eo^b) / (G \times \langle F_q \rangle) \silo H_{\kappa_0} (\eo^b) / K^{\times}_0 = \coprod_{\tau_0 \in \Hom (\kappa_0 , k)} \emm^b_{\tau_0} / K^{\times}_0
\end{equation}
and after choosing an embedding $\tau_0 : \kappa_0 \hookrightarrow k$,
\begin{equation}
\label{eq:134}
\Hom_{\cont} (\heo^b_K , \eo^b) / (G \times \langle F_p \rangle ) \overset{\sim}{\longleftarrow} \emm^b_{\tau_0} / K^{\times}_0 \; .
\end{equation}
\end{rem}

\begin{proof}
For $\gamma \in \Gamma$ we have $\gamma (t^b) = \overline{[\chi (\gamma)]}_{\phi} (t^b)$ by \eqref{eq:122a}. For $\varphi \in \Hom_{\cont} (\heo^b_K , \eo^b)$, since $\Gamma$ acts trivially on $\kappa_0$ we therefore get
\begin{align*}
\Psi_t (\varphi G_{\infty} \gamma) & = (\varphi (\gamma (t^b)) , \varphi \, |_{\kappa_0}) \\
& = (\varphi (\overline{[\chi (\gamma)]}_{\phi} (t^b)) , \varphi \, |_{\kappa_0}) \\
& = (( \varphi \, |_{\kappa_0} \overline{[\chi (\gamma)]}_{\phi}) (\varphi (t^b)) , \varphi \, |_{\kappa_0}) \\
& = \chi (\gamma) \cdot (\varphi (t^b) , \varphi \, |_{\kappa_0}) \quad \text{by \eqref{eq:124a}} \\
& = \chi (\gamma) \cdot \Psi_t (\varphi G_{\infty}) \; .
\end{align*}
Since the $\Gamma$-operation on $H_{\kappa_0} (\eo^b)$ was defined via $\chi$ it follows that $\Psi_t$ is $\Gamma$-equivariant. We have
\begin{align*}
\Psi_t (F_p (\varphi G_{\infty})) & = (F_p (\varphi) (t^b) , F_p (\varphi) \, |_{\kappa_0}) \\
& = (\varphi (t^b)^p , \varphi \, |_{\kappa_0} \verk (\;)^p) \\
& = F_p (\varphi (t^b) , \varphi \, |_{\kappa_0}) \quad \text{by \eqref{eq:129}.}
\end{align*}
Hence $\Psi_t$ commutes with the $F_p$-action. For $\gamma \in \Gamma$ we have $\gamma (t^b) = \gamma (t)^b$ and therefore \eqref{eq:131} follows from the equations:
\begin{align*}
(\Psi_t \verk \gamma) (\varphi G_{\infty}) & = \Psi_t (\varphi G_{\infty} \gamma) = (\varphi (\gamma (t^b)) , \varphi \, |_{\kappa_0}) \\
& = (\varphi (\gamma (t)^b) , \varphi \, |_{\kappa_0}) \\
& = \Psi_{\gamma (t)} (\varphi G_{\infty}) \; .
\end{align*}
Compatibility with the left action by $\Aut (\eo)$ on both sides via $\Aut (\eo) \to \Aut_{\cont} (\eo^b)$ follows similarly.
\end{proof}

\begin{exmp}  \label{t133}
Let us specialize to the case $K_0 = \Q_p$ where $K = \oQ_p$ and $\heo_K = \eo_p := \eo_{\C_p}$, $\pi_0 = p$ and $\phi (X) = (1 + X)^p - 1$. Under the $\Z_p$-equivariant bijection
\[
\alpha : \emm^b \silo D^b := 1 + \emm^b \cong \multmap (1 + \emm) \; , \; x \longmapsto 1 +x
\]
the formal $\Z_p$-module (even $\Q_p$-vector space) $\emm^b$ becomes the group $D^b$ of $1$-units in $\eo^b$ with exponentiation by elements in $\Z_p$ (even in $\Q_p)$. Moreover $G = G_{\Q_p}$, and $G_{\infty} = G_{\Q_p (\mu_{p^{\infty}})}$ and 
\[
\chi : \Gamma = \Gal (\Q_p (\mu_{p^{\infty}}) / \Q_p) \silo \Z^{\times}_p
\]
is the cyclotomic character. By Theorem \ref{t131}, for every $\Z_p$-generator $\varepsilon = (\varepsilon_n) \in T \mu = \varprojlim \mu_{p^n} (K)$ setting $t = (\varepsilon_n - 1) \in T \Fh$ we get a bijection $\tPsi_{\varepsilon} := \alpha \verk \Psi_t$:
\[
\tPsi_{\varepsilon} : \Hom_{\cont} (\eo^b_p , \eo^b) / G_{\infty} \silo D^b \quad \text{where} \; \tPsi_{\varepsilon} (\varphi G_{\infty}) = \varphi (\varepsilon^b) \quad \text{with} \; \varepsilon^b = (\varepsilon_n \mod p) \; .
\]
By Proposition \ref{t132} we have $\tPsi_{\varepsilon} \verk \gamma = \tPsi_{\gamma (\varepsilon)} $ for $\gamma \in \Gamma$ and $\tPsi_{\varepsilon}$ is $\Gamma$- and $F_p$-equivariant. The resulting bijection
\begin{equation}
\label{eq:135}
\tPsi : \Hom_{\cont} (\eo^b_p , \eo^b) / G \silo D^b / \Z^{\times}_p
\end{equation}
is independent of $\varepsilon$ and dividing by the $F^{\Z}_p$-action gives a bijection
\begin{equation}
\label{eq:136}
\Hom_{\cont} (\eo^b_p , \eo^b) / (G \times \langle F_p \rangle) \silo D^b / \Q^{\times}_p \; .
\end{equation}
\end{exmp}

\begin{rem}
For $\eo = \eo_p$, the map \eqref{eq:136} is a bijection
\begin{equation}
\label{eq:137}
\End_{\cont} (\eo^b_p) / (G \times \langle F_p \rangle) \silo D^b / \Q^{\times}_p \; .
\end{equation}
By the remark after Theorem \ref{t131} it restricts to a bijection
\begin{equation}
\label{eq:138}
\End_{\cont} (\C^b_p) / (G \times \langle F_p \rangle) \silo (D^b \setminus 0) / \Q^{\times}_p \; .
\end{equation}
Note here that continuous ring endomorphisms of $\C_p$ preserve $|\;|_p$ hence $\eo_p$. 
It follows from \cite{kedlaya} that there exist continuous embeddings $\C^b_p \hookrightarrow \C^b_p$ which are not isomorphisms. I do not know which elements in $(D^b \setminus 0) / \Q^{\times}_p$ correspond to the classes of {\em automorphisms} of $\C^b_p$ under \eqref{eq:138}.
\end{rem}

For an $\eo_{K_0}$-algebra $B$ the ring of ramified $p$-typical Witt vectors will be denoted by $W_{\eo_{K_0}} (B)$. If $B$ is a perfect $\kappa_0$-algebra, there is a natural isomorphism which can be taken as a definition
\[
W_{\eo_{K_0}} (B) = W_p (B) \otimes_{W (\kappa_0)} \eo_{K_0} \; .
\]
Here the Witt vector Frobenius $F_q$ on the left corresponds to $F^r_p \otimes \id$ on the right if $q = p^r$. Given an embedding $\tau_0 : \kappa_0 \hookrightarrow k$, we may view $\eo^b$ and $\C^b$ as perfect $\kappa_0$-algebras which we denote by $\eo^b_{\tau_0}$ and $\C^b_{\tau_0}$. The element $\xi \in W_{\eo_{K_0}} (\eo^b_{\tau_0})$ is called primitive of degree one if in the representation
\[
\xi = \sum_{n \ge 0} [\xi_n] \pi^n_0 \quad \text{with} \quad \xi_n \in \eo^b_{\tau_0}
\]
we have $\xi_0 \neq 0$ and $\xi_1 \in (\eo^b_{\tau_0})^{\times}$. Let $|Y_{\tau_0}|^{\deg=1}$ denote the set (actually a space) of principal ideals $\ep_{\xi}$ of $W_{\eo_{K_0}} (\eo^b_{\tau_0})$ generated by the primitive elements $\xi$ of degree one, cf. \cite[Ch. II, 3]{FF}. The ideals $\ep = \ep_{\xi}$ are prime. A $K_0$-untilt $(C , \iota)$ of $\C^b_{\tau_0}$ consists of a perfectoid field $C$ which is also a $K_0$-algebra, together with a $\kappa_0$-linear isomorphism $\iota : C^b \silo \C^b_{\tau_0}$ of the tilts. Two untilts $(C, \iota)$ and $(C' , \iota')$ are equivalent if there is an isomorphism $\alpha : C \silo C'$ of $K_0$-algebras such that $\iota' \verk \alpha^b = \iota$. One obtains a bijection of $|Y_{\tau_0}|^{\deg = 1}$ with the space of equivalence classes of $K_0$-untilts of $\C^b_{\tau_0}$ as follows: The prime ideal $\ep$ is mapped to the perfectoid field
\[
C = (W_{\eo_{K_0}} (\eo^b_{\tau_0}) / \ep) [\opi^{-1}_0] \quad \text{where} \; \opi_0 = \pi_0 \mod \ep \]
and $\iota : C^b \silo \C^b_{\tau_0}$ is the induced isomorphism, see \cite[Ch. II, 3]{FF} for details.

Fontaine's map $\theta_{\hK} : W_{\eo_{K_0}} (\eo^b_{\hK}) \to \eo_{\hK}$ sends $\sum_{n\ge 0} [\xi_n] \pi^n_0$ to $\sum_{n \ge 0} \xi^{\sharp}_n \pi^n_0$. It is a surjective homomorphism of $\eo_{K_0}$-algebras whose kernel $\ker \theta_{\hK}$ is generated by a primitive element of degree one. Now let $\varphi : \eo^b_{\hK} = \heo^b_K \to \eo^b$ be a continuous injective homomorphism and set $\tau_0 = \varphi \, |_{\kappa_0}$. The map $\varphi$ may be viewed as a map of $\kappa_0$-algebras $\varphi : \eo^b_{\hK} \to \eo^b_{\tau_0}$. It therefore induces an injective homomorphism of $\eo_{K_0}$-algebras
\[
W_{\eo_{K_0}} (\varphi) : W_{\eo_{K_0}} (\eo^b_{\hK}) \longrightarrow W_{\eo_{K_0}} (\eo^b_{\tau_0}) \; .
\]

\begin{theorem} \label{t134}
We obtain a canonical bijection
\begin{equation} \label{eq:141n}
\arah : \Hom_{\cont, \inj} (\heo^b_K , \eo^b) / G \silo \coprod_{\tau_0 \in \Hom (\kappa_0 , k)} |Y_{\tau_0}|^{\deg = 1}
\end{equation}
by mapping $\varphi G$ to the ideal $\arah (\varphi G) \in |Y_{\tau_0}|^{\deg = 1}$ which is generated by \\
$W_{\eo_{K_0}} (\varphi) (\ker \theta_{\hK})$ in $W_{\eo_{K_0}} (\eo^b_{\tau_0})$. Here as above $\tau_0 = \varphi \, |_{\kappa_0}$. 
\end{theorem}

\begin{proof} Consider the map
\begin{equation}
\label{eq:139}
\Lambda : (\emm^b_{\tau_0} \setminus 0) / \eo^{\times}_{K_0} \longrightarrow |Y_{\tau_0}|^{\deg = 1}
\end{equation}
which sends $x \eo^{\times}_{K_0}$ to the principal ideal of $W_{\eo_{K_0}} (\eo^b_{\tau_0})$ generated by $y_x := [x]_{\phi} / [x^{1/q}]_{\phi}$. Here
\[
[\,]_{\phi} : \emm^b_{\tau_0} \longrightarrow W_{\eo_{K_0}} (\eo^b_{\tau_0})
\]
is the Teichm\"uller map for the Frobenius power series $\phi$, c.f. \cite[Ch. I, Proposition 1.2.7]{FF} or \cite{schneider2} where $[\,]_{\phi}$ is the map $\tau_{\phi}$ before Proposition 2.1.12. In \cite[Ch. I, Proposition 2.3.10]{FF} it is shown that the map \eqref{eq:139} is a bijection if $\phi \in \eo_{K_0} [[X]]$ is a polynomial, but the result is true in general. Theorem \ref{t134} is proved by checking that $\arah$ is the composition of the bijections $\Lambda$ in \eqref{eq:139} and $\Psi$ in \eqref{eq:132}
\begin{equation}
\label{eq:140}
\arah : \Hom_{\cont , \inj} (\heo^b_K , \eo^b) / G \overset{\Psi}{\silo} \coprod_{\tau_0} (\emm^b_{\tau_0} \setminus 0) / \eo^{\times}_{K_0} \overset{\Lambda}{\silo} \coprod_{\tau_0 \in \Hom (\kappa_0 , k)} |Y_{\tau_0}|^{\deg = 1} \; .
\end{equation}
In \cite[Proposition 2.1.19]{schneider2} it is proved that for any $\eo_{K_0}$-generator $t$ of $T \Fh$ and corresponding element $t^b \in \heo^b_{K_{\infty}} \subset \heo^b_K$, the element $y_{t^b} = [t^b]_{\phi} / [(t^b)^{1/q}]_{\phi}$ of $W_{\eo_{K_0}} (\heo^b_K)$ satisfies the conditions of \cite[Lemma 1.4.19]{schneider2} to be a generator of $\ker \theta_{\hK}$. For $\varphi  \in \Hom_{\cont, \inj} (\heo^b_K , \eo^b)$ we have
\[
W_{\eo_{K_0}} (\varphi) (y_{t^b}) = y_{\varphi (t^b)} \; .
\]
Hence the ideal of $W_{\eo_{K_0}} (\eo^b_{\tau_0})$ generated by $W_{\eo_{K_0}} (\varphi) (\ker \theta_{\hK}) = W_{\eo_{K_0}} (\varphi) ((y_{t^b}))$ is the principal ideal $(y_{\varphi (t^b)})$. On the other hand we have
\[
\Lambda (\Psi (\varphi G)) = \Lambda (\varphi (t^b) \eo^{\times}_{K_0}) = (y_{\varphi (t^b)}) \; .
\]
This shows that $\arah (\varphi G)$ as defined in Theorem \ref{t134} equals $\Lambda (\Psi (\varphi G))$. Hence $\arah$ is well defined and $\arah = \Lambda \verk \Psi$, so that $\arah$ is a bijection which factors as in \eqref{eq:140}. 
\end{proof}

The bijection \eqref{eq:139} is $F_q$-equivariant where on the left the $F_q$-action equals multiplication by $\pi_0$. Dividing out the $F_q$-action in \eqref{eq:139} and composing with the bijection \eqref{eq:134} we obtain the following result:

\begin{cor} \label{t135}
For any embedding $\tau_0 : \kappa_0 \to k$ the map $\arah$ in Theorem \ref{t134} induces an $\Aut (\eo)$-equivariant bijection
\[
\Hom_{\cont , \inj} (\heo^b_K , \eo^b) / G \times \langle F_p \rangle \silo |Y_{\tau_0}|^{\deg = 1} / \langle F_q \rangle \; .
\]
\end{cor}

\begin{example}
The case $K_0 = \Q_p$. With notations as in Example \ref{t133}, $|Y|^{\deg = 1}$ is the set of principal ideals in $W_p (\eo^b)$ generated by primitive elements of degree one. The bijection
\[
\arah : \Hom_{\cont , \inj} (\eo^b_p , \eo^b) / G \silo |Y|^{\deg = 1}
\]
maps $\varphi G$ to the extension ideal of $W_p (\varphi) (\Ker \theta_{\C_p})$ in $W_p (\eo^b)$. The map $\arah$ equals the composition of bijections
\[
\arah : \Hom_{\cont, \inj} (\eo^b_p , \eo^b) / G \overset{\tPsi}{\silo} (D^b \setminus \{ 1 \} ) / \Z^{\times}_p \overset{\tLambda}{\silo} |Y|^{\deg = 1} \; .
\]
Here $D^b = 1 + \emm^b = \multmap 1 + \emm$ on which $\Z^{\times}_p$ acts by exponentiation, and $\tPsi$ is the map \eqref{eq:135}. For $\phi (X) = (1 + X)^p - 1$ we have
\[
[x]_{\phi} = [1 + x] - 1 \quad \text{for} \; x \in \emm^b \setminus 0 \quad \text{c.f. \cite[Ch. I, Example 1.2.8]{FF}} \; .
\]
The bijection $\tLambda$ is the composition $\tLambda := \Lambda \verk \alpha^{-1}$ where $\alpha (x) = 1 + x$ and $\Lambda (x \Z^{\times}_p) = ([x]_{\phi} / [x^{1/p}]_{\phi})$. Hence we have
\[
\tLambda (u^{\Z^{\times}_p}) = \big( \frac{[u]-1}{[u^{1/p}]-1} \big) = (1 + [u^{1/p}] + \ldots + [u^{(p-1)/p}])
\]
for $u \in D^b \setminus \{ 1 \}$. In the proof of Theorem \ref{t134} we needed to know that $y_{t^b}$ is a primitive element of degree one which generates $\ker \theta_{\hK}$. In the case $K_0 = \Q_p$ this amounts to the following fact. For any $\Z_p$-generator $\varepsilon$ of $T\mu = \varprojlim_n \mu_{p^n} (\oQ_p)$ let $\varepsilon^b$ be its image under the canonical injective map $T\mu \hookrightarrow \multmap \eo_p / p = \eo^b_p$. Then the element
\[
\frac{[\varepsilon^b]-1}{[\varepsilon^{b^{1/p}}]-1} = 1 + [\varepsilon^{b^{1 / p}}] + \ldots + [\varepsilon^{b (p-1)/p}] \quad \text{in} \; W_p (\eo^b_p)
\]
is primitive of degree one and generates $\ker \theta_{\C_p}$. Note that $t \ent \varepsilon - 1$ and $t^b \ent \varepsilon^b -1$. Viewing $\arah (\varphi G)$ as an untilt of $\C^b$, it is given by the perfectoid field
\[
C = (W_p (\eo^b) / (W_p (\varphi) (\Ker \theta_{\C_p}))) [p^{-1}] \; ,
\]
together with the induced isomorphism $\iota : C^b \silo \C^b$. 
\end{example}

\begin{rem}
The bijection $\arah$ in Theorem \ref{t134} is only defined on (classes of) {\em injective} maps $\varphi : \heo^b_K \to \eo^b$. Viewing $\arah$ as a map to untilts, we can extend $\arah$ to all of $\Hom_{\cont} (\heo^b_K , \eo^b) / G$ by generalizing the notion of an untilt as follows: We consider pairs $(C , \iota)$ with $C$ a perfectoid field as before but with an $\eo_{K_0}$-algebra structure instead of a $K_0$-algebra structure. The equivalence classes of such untilts with $\car C = p$ correspond to classes $\varphi G$ for the non-injective maps $\varphi$.
\end{rem}

After these preparations it is easy to describe the $F_p$-dynamical system $\hcY_0 = \hcY / G$ of the previous section in the case $\eX_0 = \spec \eo_{K_0}$. We use the characterization of $\hcY = \hcY_{1 / p}$ in Proposition \ref{t1214} as a set of triples $(\ex , \ey , \hP^b_{\ey})$. There are only two pairs $(\ex , \ey)$ of points of $\eX = \spec \eo_K$ with $\ey \in \overline{ \{ \ex \} }$ and $\car \kappa (\ey) = p$. These are $(\ex , \ey) = (\eta , s)$ and $(\ex , \ey) = (s,s)$ where $\eta \ent (0)$ and $s \ent \emm_K$ are the generic and the closed point of $\spec \eo_K$. The corresponding local rings $\Oh_{\overline{ \{ \ex \} } , \eta}$ are $\eo_K$ resp. $\kappa = \eo_K / \emm_K$ with $p$-adic completions $\hOh_{ \overline{ \{ \ex \} } , \eta}$ given by $\heo_K$ resp. $\kappa$ and tilts $\hOh^b_{ \overline{ \{ \ex \} } , \ey}$ equal to $\heo^b_K$ and $\multmap \kappa \overset{\pr_0}{\silo} \kappa$. We now determine the maps $\hP^b_{\ey} : \hOh^b_{ \overline{ \{ \ex \} } , \ey} \to \eo^b$ in Proposition \ref{t1214}. For $(\ex , \ey) = (\eta , s)$ they are the continuous local ring homomorphisms $\hP^b_{\ey} : \heo^b_K \to \eo^b$ such that $\hP^b_{\ey} (f) \neq 0$ for all $0 \neq f \in \multmap \eo_K \subset \heo^b_K$. A continuous ring homomorphism from $\heo^b_K$ to $\eo^b$ sends topologically nilpotent elements to topological nilpotent elements. It is therefore local. If $\ker \hP^b_{\ey} = 0$ then the above non-vanishing condition is trivially satisfied. If $\ker \hP^b_{\ey} = \emm^b_{\hK}$, choose some $1 \neq u \in \multmap 1 + \emm_K \subset 1 + \emm^b_{\hK}$. Then $f := u - 1 \neq 0$ in $\heo^b_K$ and $\hP^b_{\ey} (f) = 0$. Hence the non-vanishing condition is violated if $\hP^b_{\ey}$ is not injective. For $(\ex , \ey) = (\eta , s)$, the relevant maps $\hP^b_{\ey}$ are therefore the continuous injective ring homomorphisms $\heo^b_K \to \eo^b$. For $(\ex , \ey) = (s,s)$ the continuous local ring homomorphisms $\hP^b_{\ey} : \kappa \to \eo^b$ with $\hP^b_{\ey} (f) \neq 0$ for all $0 \neq f \in \multmap \kappa \cong \kappa$ are simply the ring homomorphisms $\hP^p_{\ey} : \kappa \hookrightarrow k \subset \eo^b$, or equivalently the non-injective continuous homomorphisms $\hP^b_{\ey} : \heo^b_K \to \eo^b$. Consider the map
\begin{equation}
\label{eq:141}
\pr_{\eX} : \hcY \longrightarrow \eX_{\top} \quad \text{with} \; \pr_{\eX} (\ex , \ey , \hP_{\ey}) = \ex \; .
\end{equation}
It is $G$- and $F_p$-equivariant where $F_p$ acts trivially on $\eX_{\top}$. Hence it induces a map
\begin{equation}
\label{eq:142}
\pr_{\eX_0} : \hcY\!_0 \longrightarrow \eX_{0 \top} \; .
\end{equation}
The next result sums up what we know about $\hcY$. Note that $G \times \Aut (\eo)$ operates on $\Hom_{\cont} (\heo^b_K , \eo^b)$ and by Remark \ref{t1215} also on $\hcY$. 

\begin{theorem}
\label{t136}
Let $\eX_0 = \spec \eo_{K_0}$ with generic point $\eta_0$ and closed point $s_0$. \\
1) There is a natural $G \times \Aut (\eo) \times \langle F_p \rangle$-equivariant identification 
\[
\hcY = \Hom_{\cont} (\heo^b_K , \eo^b) \quad \text{sending} \; (\ex , \ey , \hP_{\ey}) \; \text{to} \; \hP^b_{\ey} \; .
\]
Under the projection $\pr_{\eX} : \hcY \to \eX_{\top}$ we have
\[
\hcY\!_{\eta} := \pr^{-1}_{\eX} (\eta) = \Hom_{\cont , \inj} (\heo^b_K , \eo^b)
\]
and
\[
\hcY\!_s := \pr^{-1}_{\eX} (s) = \Hom (\kappa , k ) \; . 
\]
2) For any choice of $\eo_{K_0}$-generator $t$ of $T\Fh$ there is the $\Gamma \times \Aut (\eo) \times \langle F_p \rangle$- equivariant bijection of Proposition \ref{t132}
\begin{equation}
\label{eq:143}
\Psi_t : \hcY / G_{\infty} \silo H_{\kappa_0} (\eo^b) = \coprod_{\tau_0 \in \Hom (\kappa_0 , k)} \emm^b_{\tau_0} \; .
\end{equation}
3) Dividing by $\Gamma$ in \eqref{eq:143}, the induced $F_p$-equivariant bijection
\begin{equation}
\label{eq:144} 
\Psi : \hcY\!_0 = \hcY / G \silo H_{\kappa_0} (\eo^b) / \eo^{\times}_{K_0} = \coprod_{\tau_0} \emm^b_{\tau_0} / \eo^{\times}_{K_0}
\end{equation}
is independent of $t$. We have
\begin{equation}
\label{eq:145}
\hcY\!_{0 \eta_0} := \pr^{-1}_{\eX_0} (\eta_0) \silo \coprod_{\tau_0} (\emm^b_{\tau_0} \setminus 0) / \eo^{\times}_{K_0}
\end{equation}
and
\begin{equation}
\label{eq:146}
\hcY\!_{0 s_0} := \pr^{-1}_{\eX_0} (s_0) \silo \coprod_{\tau_0} \{ 0 \} / \eo^{\times}_{K_0} = \Hom (\kappa_0 , k) \; .
\end{equation}
4) The map $\arah$ of Theorem \ref{t134} provides a bijection
\begin{equation}
\label{eq:147}
\arah : \hcY\!_{0 \eta_0} \silo \coprod_{\tau_0} |Y_{\tau_0}|^{\deg = 1} \; .
\end{equation}
Setting $\tau_0 = \hP^b_{\ey} \, |_{\kappa_0}$ it sends $(\ex, \ey , \hP_{\ey}) G$ to the ideal generated by $W_{\eo_{K_0}} (\hP^b_{\ey}) (\Ker \theta_{\hK})$ in $W_{\eo_{K_0}} (\eo^b_{\tau_0})$.\\
5) For any fixed embedding $\tau_0 : \kappa_0 \hookrightarrow k$, the bijection $\arah$ of 4) induces a bijection
\begin{equation}
\label{eq:148}
\arah : \hcY\!_{0 \eta_0} / \langle F_p \rangle \silo |Y_{\tau_0}|^{\deg = 1} / \langle F_q \rangle \; .
\end{equation}
Here the right hand side can be identified with the Frobenius-equivalence classes of untilts of $\C^b_{\tau_0}$. \\
6) The only periodic (i.e. finite) orbit of the $F_p$-action on $\hcY\!_0$ is $\hcY\!_{0 s_0}$. It has order $\log_p N (\pi_0) = r$ if $q = p^r$.
\end{theorem}

\begin{rem}
I think the bijection $\arah$ in \eqref{eq:147} is remarkable for two reasons. Firstly, $\hcY\!_{0\eta_0}$ has a purely characteristic zero definition if we take into account the remark after Definition \ref{t1512}. In fact $\hcY\!_{0 \eta_0}$ is the ``completion at $p$'' of a global construction. Secondly, on both sides $A_{\inf} (\eX)$ appears as a ring of functions. But on the left the functions take values in $\C_{\tau_0}$ whereas on the right they take values in the untilts of $\C^b_{\tau_0}$ which may not even be isomorphic to $\C_{\tau_0}$. 
\end{rem}

\begin{proof}
Apart from 6) everything has been proved. It is clear that $\hcY\!\!_{0 s_0} \silo \Hom (\kappa_0 , k)$ is a periodic orbit of order $r$ for the $F_p$-action. Using the description \eqref{eq:145} of $\hcY\!\!_{0  \eta_0}$ it remains to show that $F_q$ has no periodic orbits on $(\emm^b_{\tau_0} \setminus 0) / \eo^{\times}_K$. The Frobenius $F_q$-acts on $\emm^b_{\tau_0}$ by (Lubin-Tate-)multiplication with $\pi_0$. Assume that there is some $x \in \emm^b_{\tau_0} \setminus 0$ and some $n \ge 1$ with
\[
F^n_q (x \eo^{\times}_{K_0}) = x \eo^{\times}_{K_0} \; .
\]
Then we would have $\pi^n_0 x = ux$ for some $u \in \eo^{\times}_{K_0}$. Since $x \neq 0$ and since $\emm^b_{\tau_0}$ is a $K_0$-vector space this would imply $\pi^n_0 = u$ in $\eo^{\times}_{K_0}$ which is a contradiction. 
\end{proof}

\begin{prop}
\label{t137}
For $\eX_0 = \spec \eo_{K_0}$, recall the homomorphism of rings \eqref{eq:115}
\[
e : A_{\inf} (\eX) = W_p (\heo^b_K) \longrightarrow C (\hcY , \eo) \; .
\]
Under the identification of Theorem \ref{t136}, 1)
\begin{equation}
\label{eq:151}
\hcY = \Hom_{\cont} (\heo^b_K , \eo^b) \; ,
\end{equation}
it has the following description:\\
For $f \in A_{\inf} (\eX)$ and $\varphi \in \hcY$, the element $e (f) (\varphi) \in \eo$ is the image of $f$ under the composition
\begin{equation}
\label{eq:152}
W_p (\heo^b_K) \xrightarrow{W_p (\varphi)} W_p (\eo^b) \xrightarrow{\theta} \eo \; .
\end{equation}
\end{prop}

\begin{proof}
For $(\ex , \ey , \hP_{\ey}) \in \hcY$ consider the composition
\[
\chi : \multmap \heo_K = \multmap \Gamma (\eX , \Oh)^{\wedge} \longrightarrow \multmap \hOh_{\overline{\{ \ex \} } , \ey} \xrightarrow{\hP_{\ey}} \eo \; .
\]
The map $\chi$ is multiplicative and $\mod p$ also additive (on $\heo^b_K$). Therefore $\chi$ induces a ring homomorphism as in \eqref{eq:120n}--\eqref{eq:122n}
\[
W_p (\chi) : W_p (\heo^b_K) = A_{\inf} (\eX) \longrightarrow \eo \; .
\]
In formula \eqref{eq:115} which extends \eqref{eq:91} we defined the map $e$ above by setting (with the present notations)
\[
e (f) ((\ex , \ey , \hP_{\ey})) := W_p (\chi) (f) \; .
\]
Under the identification \eqref{eq:151} the element $(\ex , \ey , \hP_{\ey})$ corresponds to the ring homomorphism
\[
\chi^b : \heo^b_K = \Gamma (\eX , \Oh)^{\wedge b} \longrightarrow \hOh^b_{\overline{\{ \ex \} } , \ey} \xrightarrow{\hP^b_{\ey}} \eo^b \; .
\]
The maps $\chi$ and $\chi^b$ correspond under the bijection in Proposition \ref{t1110} and by formula \eqref{eq:121n} we therefore have
\[
W_p (\chi) = \theta \verk W_p (\chi^b) \; .
\]
This gives
\[
e (f) ((\ex , \ey , \hP_{\ey})) = (\theta \verk W_p (\chi^b)) (f) \quad \text{for} \; (\ex , \ey , \hP_{\ey}) \in \hcY
\]
and hence
\[
e (f) (\varphi) = (\theta \verk W_p (\varphi)) (f) \quad \text{for} \; \varphi \in \Hom_{\cont} (\heo^b_K , \eo^b) \equiv \hcY \; .
\]
\end{proof}

The following proposition shows that for $\eX_0 = \spec \Z_p$ the sub-system $\cY_0$ of $\ceX_0 (\eo_p)$ is quite small. Recall that by Theorem \ref{t136}, 1) we have an identification
\[
\hcY\!_0 = \End_{\cont} (\eo^b_p) / G
\]
where $\sigma \in G$ acts on $\varphi \in \End_{\cont} (\eo^b_p)$ via $\varphi^{\sigma} = \varphi \verk \sigma$.

\begin{prop}
\label{t138}
For $\eX_0 = \spec \Z_p$ and $\eo = \eo_p$, the set $\cY_0 \subset \hcY\!_0$ consists of the $F_p$-fixed point $s_0 = (\eo^b_p \to \oF_p \to \eo^b_p) G$ and the (infinite) $F^{\Z}_p$-orbit of $\eta_0 = (\eo^b_p \xrightarrow{\id} \eo^b_p) G$.
\end{prop}

\begin{proof}
Recalling the correspondence $\chi \leftrightarrow \chi^b$ in Proposition \ref{t1110} and the definition of $\cY$, we see the following: A point $\varphi = \chi^b \in \hcY = \End_{\cont} (\eo^b_p)$ is in $\cY$ if and only if the corresponding continuous multiplicative map
\[
\chi : \multmap \eo_p \cong \eo^b_p \xrightarrow{\varphi} \eo^b_p \xrightarrow{\sharp} \eo_p
\]
factors over the $i$-th projection $\pr_i : \multmap \eo_p \to \eo_p$ for some $i \ge 0$. This means that $\chi$ is constant on each fibre of $\pr_i$. Since $\pr^{-1}_i (1) = \varepsilon^{p^i \Z_p}$ for any $\Z_p$-generator $\varepsilon$ of $T_p \mu (\eo_p)$ we must have $\chi (\varepsilon^{p^i}) =  1$ and since $\chi$ is multiplicative and continuous this condition is also sufficient for $\chi$ to be constant on the fibres of $\pr_i$. In terms of $\varphi = \chi^b$, the condition asserts that $(\varphi (\varepsilon^b)^{p^i})^{\sharp} = 1$ where $\varepsilon^b \in \eo^b_p$ is the image of $\varepsilon$ under the canonical map $T_p \mu (\eo_p) \hookrightarrow \eo^b_p$. By Example \ref{t133} we have a $\Gamma \times \langle F_p \rangle$-equivariant bijection
\[
\tPsi_{\varepsilon} : \End_{\cont} (\eo^b_p) / G_{\infty} \silo 1 + \emm^b_p \; \, \quad \text{where} \; \tPsi_{\varepsilon} (\varphi G_{\infty}) = \varphi (\varepsilon^b) \; .
\]
Under the canonical bijection of $\Q_p$-vector spaces
\[
1 + \emm^b_p = \multmap 1 + \emm_p \; .
\]
The elements $u \in 1 + \emm^b_p$ with $(u^{p^i})^{\sharp} = 1$ correspond to the elements $(u_n) \in \multmap 1 + \emm_p$ with $u^{p^i}_0 = 1$ i.e. with $(u^{p^i}_n) \in T_p \mu (\eo_p)$. Thus we have
\[
\{ u \in 1 + \emm^b_p \mid (u^{p^i})^{\sharp} = 1 \} = (\varepsilon^b)^{p^{-i} \Z_p} \; ,
\]
and therefore
\[
\{ u \in 1 + \emm^b_p \mid (u^{p^i})^{\sharp} = 1 \; \text{for some} \; i \ge 0 \} = (\varepsilon^b)^{\Q_p} \; .
\]
Hence we have a commutative diagram induced by $\tPsi_{\varepsilon}$
\[
\xymatrix{
\hcY_0 \ar@{=}[r] & \End_{\cont} (\eo^b_p) / G \ar[r]^{\overset{\tPsi}{\sim}} & (1 + \emm^b_p) / \Z^{\times}_p \\
\cY_0 \ar@{^{(}->}[u] \ar[rr]^{\sim} & & (\varepsilon^b)^{\Q_p} / \Z^{\times}_p \ar@{^{(}->}[u]
}
\]
The natural bijection $\{ 0 \} \dcup p^{\Z} \silo \Q_p / \Z^{\times}_p$ shows that
\[
\cY_0 = \{ \varphi G \mid \varphi (\varepsilon^b) = 0 \quad \text{or} \quad \varphi (\varepsilon^b) = (\varepsilon^b)^{p^i} \; \text{for some} \; i \in \Z \} \; .
\]
The map $\varphi_0 = (\eo^b_p \to \oF_p \to \eo^b_p)$ satisfies $\varphi_0 (\varepsilon^b) = 0$ and the map $F^i_p (\id_{\eo^b_p})$ satisfies $(F^i_p (\id_{\eo^b_p})) (\varepsilon^b) = (\varepsilon^b)^{p^i}$. Hence we have
\[
\cY_0 = \{ \varphi_0 G \} \dcup \{ (F^i_p (\id_{\eo^b_p})) G \mid i \in \Z \}
\]
as claimed. Finally, $F_p (\varphi_0 G) = \varphi_0 G$ since $G \to \Gal (\oF_p / \F_p)$ is surjective. 
\end{proof}

\end{document}